\documentclass[aihp,reqno]{imsart}
\RequirePackage{amsthm,amsmath,amsfonts,amssymb}
\RequirePackage[numbers]{natbib}
\RequirePackage[colorlinks,citecolor=blue,urlcolor=blue]{hyperref}
\RequirePackage{graphicx}
\RequirePackage[colorlinks,citecolor=blue,urlcolor=blue]{hyperref}
\usepackage[title]{appendix}
\usepackage{xkeyval}
\usepackage[utf8]{inputenc}
\usepackage[T1]{fontenc}
\usepackage[english]{babel}
\usepackage{csquotes}
\usepackage{enumerate}
\usepackage{bm}
\usepackage{verbatim}
\usepackage{indentfirst}
\usepackage{xcolor}
\usepackage{bbold}
\usepackage{mathtools}
\usepackage{framed}
\usepackage{amsmath}
\usepackage{amsthm}	
\usepackage{amsfonts}
\usepackage{mathrsfs}	
\usepackage{amssymb}
\usepackage{nicefrac}
\usepackage{dsfont}
\usepackage{bbm}
\allowdisplaybreaks

\newcommand{\fdm}[1]{\textcolor{purple}{#1}}

\newenvironment{frcseries}{\fontfamily{frc}\selectfont}{}
\newcommand{\textfrc}[1]{{\frcseries#1}}
\newcommand{\mathfrc}[1]{\text{\textfrc{#1}}}

\DeclarePairedDelimiterX{\inp}[2]{\langle}{\rangle}{#1, #2}
\newcommand{\intrd}[1][]{ \int_{\bT^d}  }

\newcommand{\hyp}[1]{{\rm{(}\bf{#1}\rm{)}}}

\newcommand{\curss}{\textrm{\rm \mathfrc{s}}}
\newcommand{\cursu}{\textrm{\rm \mathfrc{u}}}
\newcommand{\cursr}{\textrm{\rm \mathfrc{r}}}
\newcommand{\essup}{\textrm{\rm essup}_{\omega^0 \in \Omega^0}}

\newcommand{\bT}{\mathbb{T}}

\def\ud{\mathrm{d}}
\def\ud{\mathrm{d}}

\makeatletter
\DeclareFontFamily{OMX}{MnSymbolE}{}
\DeclareSymbolFont{MnLargeSymbols}{OMX}{MnSymbolE}{m}{n}
\SetSymbolFont{MnLargeSymbols}{bold}{OMX}{MnSymbolE}{b}{n}
\DeclareFontShape{OMX}{MnSymbolE}{m}{n}{
    <-6>  MnSymbolE5
   <6-7>  MnSymbolE6
   <7-8>  MnSymbolE7
   <8-9>  MnSymbolE8
   <9-10> MnSymbolE9
  <10-12> MnSymbolE10
  <12->   MnSymbolE12
}{}
\DeclareFontShape{OMX}{MnSymbolE}{b}{n}{
    <-6>  MnSymbolE-Bold5
   <6-7>  MnSymbolE-Bold6
   <7-8>  MnSymbolE-Bold7
   <8-9>  MnSymbolE-Bold8
   <9-10> MnSymbolE-Bold9
  <10-12> MnSymbolE-Bold10
  <12->   MnSymbolE-Bold12
}{}

\let\llangle\@undefined
\let\rrangle\@undefined
\DeclareMathDelimiter{\llangle}{\mathopen}%
                     {MnLargeSymbols}{'164}{MnLargeSymbols}{'164}
\DeclareMathDelimiter{\rrangle}{\mathclose}%
                     {MnLargeSymbols}{'171}{MnLargeSymbols}{'171}
\makeatother

\newcounter{daggerfootnote}

\newcounter{starfootnote}

\theoremstyle{plain}
\newtheorem{theorem}{Theorem}[section]
\newtheorem{lemma}[theorem]{Lemma}
\newtheorem{proposition}[theorem]{Proposition}

\newtheorem{corollary}[theorem]{Corollary}
\theoremstyle{definition}
\newtheorem{definition}[theorem]{Definition}

\theoremstyle{definition}
\newtheorem{remark}[theorem]{Remark}
\newtheorem*{theorem*}{Theorem (Main result)}
\numberwithin{equation}{section}
\numberwithin{figure}{section}

\begin{document}
\begin{frontmatter}\title{Major-minor mean field games: common noise helps}
\runtitle{Major Minor MFGs}

\begin{aug}
\author[A]{Fran\c{c}ois Delarue}
\and
\author[B]{Chenchen Mou}
\address[A]{Universit\'e C\^ote d'Azur, CNRS, Laboratoire J.A.Dieudonn\'{e},
Parc Valrose,
France-06108 NICE Cedex 2,
 francois.delarue@univ-cotedazur.fr}

\address[B]{City University of Hong Kong, Department of Mathematics,
China-Hong Kong SAR,
chencmou@cityu.edu.hk}
\end{aug}



\begin{abstract}
The objective of this work is to study the existence, uniqueness, and stability of equilibria in mean field games involving a major player and a continuum of minor players over finite intervals of arbitrary length. Following earlier articles addressing similar questions in the context of classical mean field games, the cost functions for the minor players are assumed to satisfy the Lasry-Lions monotonicity condition.
In this contribution, we demonstrate that if, in addition to the monotonicity condition, the intensity of the (Brownian) noise driving the major player is sufficiently high, then—under further mild regularity assumptions on the coefficients—existence, uniqueness, and stability of equilibria are guaranteed.
A key challenge is to show that the threshold (beyond which the noise intensity must be taken) can be chosen independently of the length of the time interval over which the game is defined. Building on the stability properties thus established, we further show that the associated system of master equations admits a unique classical solution. To the best of our knowledge, this is the first result of its kind for major-minor mean field games defined over intervals of arbitrary length.
 \end{abstract}

\begin{keyword}[class=MSC]
\kwd[Primary ]{49N80}
\kwd{91A16}
\kwd[; secondary ]{35R60}
\kwd{60H10}
\kwd{60H15}
\end{keyword}

\begin{keyword}
\kwd{Major-minor mean field games}
\kwd{Master equation}
\kwd{BMO martingale}
\end{keyword}

\end{frontmatter}

\section{Introduction} 

Mean field games were introduced simultaneously by Lasry and Lions \cite{Lasry2006,LasryLions2,LasryLions}, and by Huang, Caines, and Malhamé  \cite{HuangCainesMalhame1,Huang2006}, with the aim of providing an asymptotic description of differential games for large populations of weakly interacting agents. While the theory for the most popular form of mean field games has reached a certain level of maturity (as evidenced by Lions' lectures at the Collège de France, \cite{Lionscollege}, 
as well as --among other references-- the books by Cardaliaguet et al.  \cite{CardaliaguetDelarueLasryLions}, Carmona and Delarue  
\cite{CarmonaDelarue_book_I,CarmonaDelarue_book_II} and Gomes et al. \cite{MR3559742}, along with the lecture notes by Cardaliaguet and Porretta \cite{MR4214773} and Delarue \cite{MR4368182}, and the surveys by Caines et al.  \cite{MR4139292} and Cardaliaguet and Delarue \cite{MR4680378}), many questions remain open. These arise either in regimes where the assumptions are too weak to be covered by the existing literature or in situations where the very structure of the game is significantly more complex and thus escapes known results.

In this regard, one example currently generating considerable interest is the case of mean field games subject to common noise (see, for some of the earlier works on the subject, the books
 \cite{CardaliaguetDelarueLasryLions,CarmonaDelarue_book_II}
 and
 the papers
by Bertucci, Lasry and Lions
\cite{MR3941633},  
Carmona, Delarue and Lacker 
\cite{CarmonaDelarueLacker}, 
Carmona, Fouque and Sun
\cite{MR3325083}
and 
Lacker and Zariphopoulou
\cite{MR4014625}). By common noise, we mean a noise that impacts all the agents in the population in a common manner, thereby leading to a randomization of the population state. This contrasts with the traditional version of mean field games, where the population state—obtained by taking the limit over a population of asymptotically identically distributed and independent players—is deterministic. In the presence of common noise, the mathematical challenges are numerous: they may involve adapting existing concepts (see for instance the works by Bertucci \cite{MR4687277}, 
Bertucci, Lasry and Lions 
\cite{MR4749404}, 
Bertucci and Maynard 
\cite{bertucci2024noiseadditionalvariablemean}, Cardaliaguet and Souganidis
\cite{MR4451309}, 
Gangbo et al. \cite{MR4499277}, and 
Mou and Zhang \cite{MR4813035} for works in connection with the master equation, now a cornerstone of the field; see also 
Cardaliaguet and Souganidis
\cite{MR4430014} for the analysis of the mean field game system itself) or exploring additional effects induced solely by the noise (see Bayraktar et al. \cite{BCCD2020} and 
Foguen Tchuendom \cite{fog2018}).

In our work, we consider a more "advanced" version of mean field games with common noise: games involving a major player and minor players. In the regime we study, the major player is subject to a Brownian diffusion, thereby forcing the entire population of minor players to become random. However, mean field games with a major player introduce additional difficulties not present in games involving only minor players (even with common noise). In games with a major player, the population of minor agents competes with the major player. In addition to the standard mean field game played among the minor agents once the major player has determined her strategy, there is also an equilibrium condition stemming from the ability of the major player to react to the strategies of the minor players. This 
entanglement 
makes the formalization of major-minor games quite subtle. For this reason, different concepts of solution have been proposed in the literature, each reflecting a particular form of equilibrium among the various actors.

The first article introducing the concept of a major player is due to Huang \cite{HuangMajor}. Initially formulated in infinite horizon and for a linear-quadratic structure, the paper was later adapted by Nourian and Huang \cite{NguyenHuang1} to the finite horizon case, and by Nourian and Caines \cite{NourianCaines} to the nonlinear case. In these works, the complexity inherent to the presence of two types of players is resolved by imposing a sequential structure on the equilibria, where the major player's response is computed for a fixed state of the population. An alternative sequential approach is proposed in the paper of Bensoussan, Chau, and Yam \cite{BensoussanChauYam}, where the major player anticipates the rational behaviour of the minor players. For a review of other types of equilibria considered in the literature, we refer to the works of Cardaliaguet, Cirant, and Porretta \cite{Cardaliaguet:Cirant:Porretta:PAMS}
 and Bergault et al. \cite{MR4757544}. 

In the current paper, we focus on full-fledged Nash equilibria, which 
implicitly require updating the state of the minor player population when changing the major player’s strategy. In this regard, two regimes are possible, depending on whether the major player employs open-loop or closed-loop strategies and leading to possibly distinct solutions. The open-loop case has been studied by Carmona and Zhu \cite{CarmonaZhu}, as well as Carmona and Wang \cite{CarmonaWang} and 
Huang and Tang 
\cite{Huang12112024}, primarily using the stochastic Pontryagin principle.
The closed-loop case has led to several articles by Lasry and Lions \cite{Lasry:Lions:Major:Minor}, and Cardaliaguet, Cirant, and Porretta \cite{Cardaliaguet:Cirant:Porretta:PAMS,Cardaliaguet:Cirant:Porretta:JEMS}. This is precisely the framework we consider in our contribution. In this context, the primary work to which we compare our results is that of Cardaliaguet, Cirant, and Porretta \cite{Cardaliaguet:Cirant:Porretta:JEMS}, where the authors address the
existence of classical solutions to the Nash system over short time horizons. This system is expressed as two partial differential equations (called master equations) posed on the space of probability measures.

In comparison, the objective of our work is to study stability of the Nash equilibria and then existence of classical solutions to the Nash system, for games set over finite but arbitrarily long time horizons. The main idea to prevent the emergence, over intervals of arbitrary length, of multiple Nash equilibria or, equivalently, of singular solutions to the Nash system, 
 is twofold: (i) we assume that the costs to the minor players satisfy the Lasry-Lions monotonicity condition; this provides stability for the system governing the minor players when the major player is fixed—this is well understood in the case of common noise models (see, for example \cite{CardaliaguetDelarueLasryLions})--; (ii) 
we assume that the major player's dynamics is subjected to a non-degenerate Brownian noise, which is referred to as a common noise in the rest of the text; 
the diffusion structure induces a regularizing effect that is expected to contribute to the uniqueness and stability of equilibria. In fact, we show that this regularizing effect indeed exists if the intensity of the common noise exceeds a certain threshold. Furthermore, we provide conditions under which this threshold is independent of the duration $T$ of the game. A concise version of this result, requiring detailed mild regularity assumptions outlined in Section \ref{se:2}, is presented in Theorem \ref{thm:1}. As a corollary, Theorem \ref{thm:2} establishes the existence and uniqueness of classical solutions to the corresponding Nash system. All  the results are stated for minor players evolving in the torus (of dimension $d$). This simplification is already present in \cite{CardaliaguetDelarueLasryLions}. Here (and similarly to what occurs in the long-term study of classical mean field games, see 
Cardaliaguet et al. 
\cite{CLLP}
and
Cardaliaguet and Porretta \cite{cardaliaguet:porretta:longtime:master}), the periodic structure allows for easier control of certain key quantities, independently of the length $T$. 

The proof structure relies on a natural representation of the major and minor players'  values
in the form of two coupled forward-backward stochastic equations. The forward-backward system
representing the minor players 
 can be interpreted as a mean field game system embedded in the random environment induced by the major player. 
 It is made of a forward Kolmogorov equation with random coefficients and a backward stochastic Hamilton-Jacobi-Bellman equation, accounting respectively for the evolution and the value of the minor player population. 
 We thus call this system a Forward-Backward Stochastic Partial Differential Equation (FBSPDE). 
 As 
  for the major player, the forward-backward equation is a standard finite-dimensional 
  Forward-Backward Stochastic Differential Equation (FBSDE) with random coefficients depending on the state of the minor player population.
Stochastic control theory
says that
the major player's strategy is represented by the martingale component of the corresponding forward-backward equation. 
Our approach then  combines stability methods for each of the two systems. 
 On the one hand, the Lasry-Lions condition is crucial to
 apply to the FBSPDE the techniques 
  introduced in \cite{CardaliaguetDelarueLasryLions} for solving mean field game
 systems with common noise.
 On the other hand, we use tools from stochastic analysis and the theory of backward stochastic differential equations (BSDEs)
 to address the FBSDE.  
%
 
 In fact, the key challenge is to study the stability properties of the FBSDE, which has here the peculiarity to be strongly coupled with the FBSPDE. As is often the case in stochastic analysis, Girsanov’s transformation simplifies the forward equation of the FBSDE, reducing it to a standard Brownian motion and thus decoupling it from the FBSPDE system for the minor players. Quantitatively, a major step is to estimate the impact of this Girsanov transformation on the overall stability properties of the system comprising both the FBSDE and the FBSPDE. For this, we make systematic use of Bounded Mean Oscillation (BMO) martingale theory introduced in \cite{Kazamaki}. In short, these BMO bounds decrease as the intensity of the common noise increases, which explains the condition on the intensity of the noise. 

This quantitative analysis ultimately yields \textit{a priori} bounds on the spatial gradient of the game's value functions. Once these bounds are established, existence and uniqueness of equilibria are obtained by iterating contraction principles applied in small time intervals. 
Remarkably, the resulting solutions are strong, meaning that they 
are adapted to the filtration generated by the common noise. 
The resolution of the master equations is also achieved iteratively by demonstrating that, over short time intervals, the system of forward-backward equations mentioned earlier serves as the characteristic system for the master equations. The derivation of classical solutions from this characteristic system follows a linearization procedure similar to that developed by  \cite{CardaliaguetDelarueLasryLions}. In this regard, our approach differs (even in short time) from the splitting method proposed in the work \cite{Cardaliaguet:Cirant:Porretta:JEMS}.

The paper is organized as follows. In Section \ref{se:2}, we introduce the problem and the equilibrium representation in the form of two coupled forward-backward equations, the aforementioned FBSDE and the FBSPDE. We also state the main results. In Section \ref{se:3}, we derive a priori estimates for solutions to the FBSDE and FBSPDE, including the essential BMO bounds necessary for measure changes. Section \ref{se:4} is devoted to obtaining a priori bounds on the gradients of the value functions by linearizing the forward-backward equations. Section \ref{se:5} contains several short-time results required to conclude, though their novelty is of lesser significance.

\section{Set-up and main results}
\label{se:2}
\subsection{A brief presentation of the model} 
\label{subse:1.1}
We introduce the Major/Minor MFG through a probabilistic formulation 
that comprises 
 two filtered probability spaces 
$(\Omega^0,{\mathcal F}^0,{\mathbb F}^0,{\mathbb P}^0)$
and
$(\Omega,{\mathcal F},{\mathbb F},{{\mathbb P}})$ equipped with two 
Brownian motions $(B_t^0)_{0 \leq t \leq T}$ 
and $(B_t)_{0 \leq t \leq T}$ with values in ${\mathbb R}^d$.
 In brief, 
the space labelled with a `$0$' carries the `common noise' underpinning 
the major player and the other spaces carries the `idiosyncratic noise' underpinning one 
tagged minor player.

 We
then consider random variables constructed on 
$\Omega^0$, $\Omega$ or $\Omega^0 \times 
\Omega$. When dealing with a
random variable $X$ defined on  
the product space $\Omega^0 \times \Omega$, 
the
expectation of $X$ (if well-defined, and possibly with values in 
${\mathbb R}^n$ for some $n \geq 1$) is denoted 
${{\mathbb E}^0 {\mathbb E}}(X)$ and the conditional law of $X$ given  
${\mathcal F}^0$ is denoted ${\mathcal L}^0(X)$ (i.e., 
for $\omega^0 \in \Omega^0$, 
${\mathcal L}^0(X)(\omega^0)$ is the law of the random variable 
$\omega \in  \Omega \mapsto X(\omega^0,\omega)$). 

Given an ${\mathbb F}^0$-adapted continuous random flow of measures ${\boldsymbol \mu} = (\mu_t)_{0 \leq t \leq T}$ with values in 
${\mathcal P}({\mathbb T}^d)$, 
the game is equipped with two kinds of players: 
\begin{description}
\item[\hspace{5pt} \textrm{\rm 1.} Major player\textrm{\rm :}]
For
a (fixed) ${\mathcal F}^0_0$-measurable 
initial condition $X_0^0 : \Omega^0 \rightarrow  \mathbb R^d$
and for  a
 (square-integrable) 
 ${\mathbb F}^0$-progressively measurable control process 
${\boldsymbol \alpha}^0=(\alpha_t^0)_{0 \le t \le T}$
with values in ${\mathbb R}^d$, the major player has dynamics
\begin{equation*} 
\ud X_t^0 = \alpha_t^0   \ud t + \sigma_0 \ud B_t^0, \quad t \in [0,T],
\end{equation*} 
where 
$\alpha^0$ is a 
velocity field from ${\mathbb R}^d$ into itself 
and $\sigma_0$ is 
a (strictly) positive diffusion coefficient,
and has cost functional 
\begin{equation*} 
J^0({\boldsymbol \alpha}^0;{\boldsymbol \mu}) 
= {\mathbb E}^0 \biggl[ g^0(X_T^0,\mu_T) + \int_0^T   \Bigl( f^0_t(X_t^0,\mu_t) + L^0(X_t^0, \alpha_t^0) \Bigr) \ud t \biggr].
\end{equation*} 
Here, $f^0$ (resp. $g^0$) is a cost coefficient from $[0,T]\times {\mathbb R}^d \times {\mathcal P}({\mathbb T}^d)$ 
(resp. ${\mathbb R}^d \times {\mathcal P}({\mathbb T}^d)$) 
to ${\mathbb R}$
and $L^0$ is a Lagrangian from ${\mathbb R}^d \times {\mathbb R}^d$ to ${\mathbb R}$.
\item[\hspace{5pt} \textrm{\rm 2.} Minor player\textrm{\rm :}]
For 
a  (fixed)
${\mathcal F}_0$-measurable 
initial condition 
 $X_0  :  \Omega \rightarrow {\mathbb T}^d$, 
 and for 
a  (square-integrable) ${\mathbb F}^0 \otimes {\mathbb F}$-progressively measurable control process 
${\boldsymbol \alpha}=(\alpha_t)_{0 \le t \le T}$
with values in ${\mathbb R}^d$, the minor player has dynamics
\begin{equation*} 
\ud X_t = \alpha_t \ud t + \ud B_t,\quad t \in [0,T], 
\end{equation*} 
and has cost functional 
\begin{equation*} 
J({\boldsymbol \alpha}; {\boldsymbol \alpha}^0, {\boldsymbol \mu}) 
= {\mathbb E} \biggl[ g(X_T^0,X_T,\mu_T) +  \int_0^T 
\Bigl( f_t(X_t^0,X_t,\mu_t) + L(X_t,\alpha_t) 
\Bigr) 
\ud t \biggr].
\end{equation*} 
 Here, $f$ (resp. $g$) is a cost coefficients from $[0,T] \times {\mathbb R}^d \times {\mathbb T}^d \times {\mathcal P}({\mathbb T}^d)$
(resp. $ {\mathbb R}^d \times {\mathbb T}^d \times {\mathcal P}({\mathbb T}^d)$)
 to ${\mathbb R}$
and $L$ is a Lagrangian from ${\mathbb T}^d \times {\mathbb R}^d$ to ${\mathbb R}$.
\end{description}

As far as the major player is concerned, we look for strategies in Markov feedback form $\alpha^0 : [0,T] \times \mathbb R^d \times 
 {{\mathcal P}({\mathbb T}^d)} \rightarrow {\mathbb R}^d$,  namely ${\boldsymbol \alpha}^0$  reads in the form 
$\alpha_t^0 = \alpha^0(t,X_t^0,\mu_t)$, $t \in [0,T]$.
This implicitly puts some constraints on the state equation for $(X_t^0)_{0 \le t \le T}$, which are addressed in Subsection  
\ref{subse:weak:formulation:MFG}.
As for the minor player, we also look for strategies in Markov feedback form $\alpha : [0,T] \times {\mathbb R}^d \times {\mathbb T}^d \times {\mathcal P}({\mathbb T}^d) \rightarrow {\mathbb R}^d$, namely $
\alpha_t = \alpha(t,X_t^0,X_t,\mu_t)$, $t \in [0,T]$,
which also puts some constraints on the state equation for $(X_t)_{0\le t \le T}$. 
The goal is to find a Nash equilibrium in the sense of \cite{Cardaliaguet:Cirant:Porretta:PAMS,Cardaliaguet:Cirant:Porretta:JEMS,Lasry:Lions:Major:Minor}. A clean definition (though in a weaker setting) is given 
in Subsection \ref{subse:weak:formulation:MFG} below. In a nutshell, an equilibrium has the following three features: 
(i) the process ${\boldsymbol \mu}$ solves the Fokker-Planck equation associated with the random velocity field 
$x \mapsto \alpha(t,X_t^0,x,\mu_t)$; (ii) with ${\boldsymbol \mu}$ as in (i), 
the strategy $\alpha$  minimizes
the cost $J$ when ${\boldsymbol \alpha}^0$ and ${\boldsymbol \mu}$ are fixed; (iii) the strategy $\alpha^0$ minimizes the cost $J^0$ with the peculiarity that 
${\boldsymbol \mu}$ therein depends on $\alpha^0$ itself through the state ${\boldsymbol X}^0$ of the major player. 
The latter point makes Major/Minor MFGs more difficult to solve than MFGs with common noise. 


\subsection{General notations} 

We first state with several notations that are necessary in our analysis. 
\vskip 5pt

\textit{Functional and distributional spaces on ${\mathbb T}^d$.}  
For an index $s \in {\mathbb N}$, we call ${\mathcal C}^s({\mathbb T}^d)$ the space of functions 
from ${\mathbb T}^d$ into ${\mathbb R}$ that have $\lfloor s \rfloor$ continuous derivatives. We equip 
the space ${\mathcal C}^s({\mathbb T}^d)$ with the norm 
\begin{equation*} 
\| f \|_s = \sup_{k=0,\cdots,\lfloor s \rfloor}
\sup_{x \in {\mathbb T}^d} | \nabla^k f(x) |.
\end{equation*} 
In particular, $\| f \|_0$ coincides with the $L^\infty$ norm, sometimes denoted $\| f \|_{L^\infty}$,
of $f$. When $s>0$ and $s \not \in {\mathbb N}$, 
we call ${\mathcal C}^s({\mathbb T}^d)$ the space of functions 
from ${\mathbb T}^d$ into ${\mathbb R}$ that have $\lfloor s \rfloor$ derivatives and such that the derivative
of order $\lfloor s \rfloor$ is $s - \lfloor s \rfloor$ H\"older continuous. 
We equip the space ${\mathcal C}^s({\mathbb T}^d)$ with the norm 
\begin{equation*} 
\| f \|_s = \sup_{k=0,\cdots,\lfloor s \rfloor}
\sup_{x \in {\mathbb T}^d} | \nabla^k f(x) |
 + \sup_{x,x' \in {\mathbb T}^d : x \not = x'}
 \frac{\vert \nabla^{\lfloor s \rfloor} f(x) - \nabla^{\lfloor s \rfloor} f(x') \vert}{\vert x- x' \vert^{s- \lfloor s \rfloor}}.  
\end{equation*} 
 We also make use of the (topological) dual space of ${\mathcal C}^{s}({\mathbb T}^d)$, which we denote
 ${\mathcal C}^{-s}({\mathbb T}^d)$ and which we equip with the dual norm:
 \begin{equation*} 
 \| q \|_{-s} = \sup_{\| f \|_s \leq 1} (f,q)_{s,-s}, 
 \end{equation*} 
 where $(\cdot,\cdot)_{s,-s}$ is here used to denote the duality product between elements of 
 ${\mathcal C}^{-s}({\mathbb T}^d)$ and ${\mathcal C}^s({\mathbb T}^d)$. 
Quite often the precise index $s$ in the duality $(f,q)_{s,-s}$ between a function $f \in {\mathcal C}^s({\mathbb T}^d)$ and 
a distribution $q \in {\mathcal C}^{-s}({\mathbb T}^d)$ is implicitly understood and thus omitted, and the duality product between $f$ and 
$q$ 
is merely written $(f,q)$.

We next introduce the Sobolev space ${\mathcal H}^{s}(\bT^d)$. 
We denote $(e_k)_{k \in {\mathbb Z}^d}$ the standard (complex valued) Fourier basis
of ${\mathbb T}^d$
and $( \cdot , \cdot)_{0,2} $ the standard inner product in $L^2({\mathbb T}^d)$
For $s>0$, we call ${\mathcal H}^{s}(\bT^d)$ the space of functions $f \in L^2(\bT^d)$
such that $\| f \|_{s,2}^2:=\sum_{k \in {\mathbb Z}^d} (1+\vert k \vert^2)^{s} \vert (f,e_k)_{0,2}  \vert^2 < \infty$. The ${\mathcal H}^{s}(\bT^d)$-norm is $\| \cdot \|_{s,2}$.  See
for instance \cite{dinezza}.
\vskip 5pt

\textit{Distances on the space  ${\mathcal P}({\mathbb T}^d)$.}  
For two probability measures $\mu,\nu \in {\mathcal P}({\mathbb T}^d)$, 
$\| \mu - \nu \|_{-1}$ coincides with 
${\mathbb W}_1(\mu,\nu)$, where we recall that, for  any $p \geq 1$, 
\begin{equation*}
{\mathbb W}_p(\mu,\nu): = \inf_{\pi} \biggl[ \int_{{\mathbb T}^d \times {\mathbb T}^d} \vert x-y \vert^p \ud \pi(x,y) \biggr]^{1/p}, 
\end{equation*}
with the infimum being taken over `couplings' $\pi$, i.e. 
over elements $\pi \in {\mathcal P}({\mathbb T}^d \times {\mathbb T}^d)$, 
 whose images by the first and second projection mappings from ${\mathbb T}^d \times {\mathbb T}^d$ into  ${\mathbb T}^d$ are respectively $\mu$ and $\nu$.
 
 Similarly, 
$\| \mu - \nu \|_{-0}$ coincides with 
$d_{\textrm{\rm TV}}(\mu,\nu)$, where we recall that 
\begin{equation*}
d_{\textrm{\rm TV}}(\mu,\nu) := \sup_{ \| \ell \|_\infty \leq 1}
\biggl\vert \int_{{\mathbb T}^d} \ell (x) \ud \bigl( \mu - \nu \bigr)(x) \biggr\vert,
\end{equation*} 
with the supremum in the right-hand side being taken over functions 
$\ell : {\mathbb T}^d \rightarrow {\mathbb R}$ that are bounded by 1. 
\vskip 5pt

\textit{Derivatives on  ${\mathcal P}({\mathbb T}^d)$.} 
Throughout, we use two standard notions of derivatives on ${\mathcal P}({\mathbb T}^d)$. We refer the reader to 
\cite[Chapter 5]{CarmonaDelarue_book_I} and \cite{CardaliaguetDelarueLasryLions}. Briefly, we say that a function 
$\ell : {\mathcal P}({\mathbb T}^d) \rightarrow {\mathbb R}$ is continuously differentiable in the flat sense if there exists a jointly continuous 
function  
$ \delta_\mu   \ell : {\mathcal P}({\mathbb T}^d) \times {\mathbb R}^d  \rightarrow
{\mathbb R} $ 
such that, for any two $\mu,\nu \in {\mathcal P}({\mathbb T}^d)$,
\begin{equation*} 
\ell(\nu) - \ell(\mu) = \int_0^1 \delta_\mu \ell \bigl( r \nu + (1-r) \mu , y \bigr) \ud  \bigl( \nu - \mu \bigr)(y).
\end{equation*} 
Because the flat derivative is just defined up to an additive constant, we require (by convention) that 
\begin{equation}
\label{eq:convention:deltam}
\forall \mu \in {\mathcal P}({\mathbb T}^d), 
\quad 
\int_{{\mathbb T}^d} 
 \delta_\mu \ell (\mu,y) \ud \mu(y) =0.  
\end{equation} 
When the function $(\mu,y) \mapsto \delta_\mu \ell(\mu,y)$ is differentiable with respect to $y$, we let 
\begin{equation*}
\partial_\mu \ell(\mu,y) = \nabla_y \bigl(  \delta_\mu \ell \bigr) (\mu,y), \quad (\mu,y) \in {\mathcal P}({\mathbb T}^d) \times {\mathbb T}^d,
\end{equation*} 
which we sometimes refer to as the `Wasserstein' derivative of $\ell$. It is standard to observe that, if 
$\partial_\mu \ell$ is bounded (with respect to $\mu$ and $y$), then 
the function $\ell$ is Lipschitz continuous with respect to the $1$-Wasserstein distance ${\mathbb W}_1$. 
 \vskip 5pt
 
\textit{Stochastic processes.} 
On a filtered probability space $(\Omega,{\mathcal F},{\mathbb F},{\mathbb P})$, we introduce the following spaces. 
For a Euclidean space $(E,  | \cdot |)$ and an exponent $p \in [1,\infty]$, we call ${\mathscr S}^p(\Omega,{\mathcal F},{\mathbb F},{\mathbb P};E)$ (or just 
${\mathscr S}^p(E)$) the collection of ${\mathbb F}$-adapted processes with continuous $E$-valued trajectories $(S_t)_{0 \le t \le T}$
such that 
\begin{equation*} 
  \sup_{0 \le t \le T}  | S_t |  \in L^p(\Omega,{\mathcal F},{\mathbb P};{\mathbb R}),
\end{equation*} 
and, for $p \in [1,\infty)$, we call ${\mathscr H}^p(\Omega,{\mathcal F},{\mathbb F},{\mathbb P};E)$ (or just 
${\mathscr H}^p(E)$) the collection of ${\mathbb F}$-progressively-measurable $E$-valued processes $(H_t)_{0 \le t \le T}$
such that 
\begin{equation*} 
{\mathbb E} \biggl[ \biggl( \int_0^T  |  H_t |^2 \ud t \biggr)^{p/2}\biggr] < \infty. 
\end{equation*} 
For a  ${\mathbb F}$-continuous local martingale ${\boldsymbol M} = (M_t)_{0 \le t \le T}$ with values in 
${\mathbb R}$, 
we call $\langle {\boldsymbol M} \rangle_\cdot 
= (\langle {\boldsymbol M} \rangle_t)_{0 \le t \le T}$ the standard bracket of 
${\boldsymbol M}$
and $({\mathscr E}_t({\boldsymbol M})_{0 \le t \le T}$ the Doléans-Dade exponential 
\begin{equation*} 
{\mathscr E}_t( {\boldsymbol M} ) 
:= 
\exp \Bigl( M_t - \frac12 \langle M \rangle_t 
\Bigr),  
\quad t \in [0,T]. 
\end{equation*} 
If in addition, ${\boldsymbol M}$ is uniformly integrable, we let 
\begin{equation*} 
\| {\boldsymbol M} \|_{\textrm{\rm BMO}}^2 := \sup_{\tau}  \bigl\|
{\mathbb E} \bigl[ \vert M_T - M_{\tau} \vert^2 \vert {\mathcal F}_\tau \bigr] \bigr\|_{L^\infty(\Omega)},
\end{equation*} 
with the supremum being taken over the collection of ${\mathbb F}$-stopping times 
$\tau$ and with $\| \cdot \|_{L^\infty(\Omega)}$ here denoting the $L^\infty$ norm on 
$(\Omega,{\mathcal F},{\mathbb P})$. When ${\boldsymbol M} \in {\mathscr S}^2({\mathbb R})$, the above is the same as 
\begin{equation*} 
\| {\boldsymbol M} \|_{\textrm{\rm BMO}}^2 = 
\sup_{\tau}  \bigl\|
{\mathbb E} \bigl[  \langle M \rangle _T - \langle M \rangle_{\tau} \vert {\mathcal F}_\tau \bigr]
\bigr\|_{L^\infty(\Omega)}.
\end{equation*} 
We refer to
\cite{Kazamaki} for more details on BMO martingales. 
\subsection{Assumptions on the coefficients}
 In order to state the assumptions, 
 we introduce a set of four generic conditions, which will be also very useful in the rest of the paper. 
For any three (strictly) positive reals 
 $\Lambda$, $\cursr$ and $\curss$, we let:
 \vskip 5pt
 
  \noindent \textbf{Condition ${\mathscr C}^0(\Lambda,\cursr)$.}
  We say that a function $h^0 : (x_0,\mu) \in {\mathbb R}^d \times {\mathcal P}({\mathbb T}^d) \mapsto
  h^0(x_0,\mu) \in {\mathbb R}$ satisfies 
 ${\mathscr C}^0(\Lambda,\cursr)$
 if
 \begin{enumerate}[(i)]
\item  $h^0$ is continuously differentiable with respect to $x_0$ and $\mu$;
\item for any 
 $(x_0,\mu) \in {\mathbb R}^d \times {\mathcal P}({\mathbb T}^d)$, the function $\delta_\mu h^0(x_0,\mu,\cdot) : y \in {\mathbb T}^d \mapsto 
 \delta_\mu h^0(x_0,\mu,y) \in 
 {\mathbb R}$ belongs to ${\mathcal C}^{\cursr}({\mathbb T}^d)$;
 \item the functions 
 $\nabla_{x_0} h^0$ and $\delta_\mu h^0$ satisfy 
 \begin{equation*} 
\forall x_0 \in {\mathbb R}^d, \ \mu \in {\mathcal P}({\mathbb T}^d), \quad
\bigl\vert 
  h^0(x_0,\mu)\bigr\vert +
\bigl\vert 
\nabla_{x_0} h^0(x_0,\mu)\bigr\vert \leq \Lambda, 
\quad 
\| \delta_\mu h^0(x_0,\mu,\cdot) 
 \|_{\cursr} \leq \Lambda. 
 \end{equation*} 
 \end{enumerate}
 \vskip 4pt
 
 \noindent \textbf{Condition ${\mathscr C}(\Lambda,\cursr,\curss)$.}
 We say that a function $h : (x_0,x,\mu) \in {\mathbb R}^d \times {\mathbb T}^d \times {\mathcal P}({\mathbb T}^d) 
 \mapsto h(x_0,x,\mu) \in {\mathbb R}$ satisfies ${\mathscr C}(\Lambda,\cursr,\curss)$ if 
 \begin{enumerate}[(i)]
 \item  $h$ is continuously differentiable with respect to $x_0$, $x$ and $\mu$;
 \item for any $(x_0,\mu) \in {\mathbb R}^d \times {\mathcal P}({\mathbb T}^d)$, 
 the function $\nabla_{x_0} h(x_0,\cdot,\mu) : x \mapsto \nabla_{x_0} h(x_0,x,\mu) \in {\mathbb R}^d$ belongs to 
 ${\mathcal C}^{\curss}({\mathbb T}^d)$; 
 \item the function 
 $\delta_\mu h(x_0,\cdot,\mu,\cdot) :  (x,y) \in  {\mathbb R}^d  \times {\mathbb T}^d \mapsto 
 \delta_\mu h(x_0,\cdot,\mu,\cdot)$
 has continuous joint derivatives up to the order $\lfloor \cursr \rfloor$ in $y$
 and $\lfloor \curss \rfloor$ in $x$; 
 \item the functions 
 $\nabla_{x_0} h$ and $\delta_\mu h$ satisfy 
  \begin{equation*} 
  \begin{split}
\forall x_0 \in {\mathbb R}^d, \ \mu \in {\mathcal P}({\mathbb T}^d), \quad
&\bigl\|
h(x_0,\cdot,\mu)\bigr\|_{\curss} + \bigl\|
\nabla_{x_0} h(x_0,\cdot,\mu)\bigr\|_{\curss} \leq \Lambda, 
\quad 
\hspace{-2pt} 
\sup_{l=0,\cdots,\lfloor \cursr \rfloor}
\sup_{y \in {\mathbb T}^d} 
\| \nabla_y^l \delta_\mu h(x_0,\cdot,\mu,y) 
 \|_{\curss} \leq \Lambda;
 \\
 \quad \forall y,y' \in {\mathbb T}^d, \quad 
& \Bigl\| 
 \nabla_y^{\lfloor \cursr \rfloor} 
 \delta_\mu h(x_0,\cdot,\mu,y') 
 - 
  \nabla_y^{\lfloor \cursr \rfloor} 
 \delta_\mu h(x_0,\cdot,\mu,y)
 \Bigr\|_{\curss} \leq \Lambda \vert y'-y \vert^{\cursr- \lfloor \cursr \rfloor}. 
 \end{split}
 \end{equation*} 
\end{enumerate}
Of course, the very last line above can be removed if $\cursr$ is an integer. 

Notice that condition (iv) right above is equivalent to 
 \begin{equation*} 
  \begin{split}
\forall x_0 \in {\mathbb R}^d, \ \mu \in {\mathcal P}({\mathbb T}^d), \quad
&\bigl\|
h(x_0,\cdot,\mu)\bigr\|_{\curss} + \bigl\|
\nabla_{x_0} h(x_0,\cdot,\mu)\bigr\|_{\curss} \leq \Lambda, 
\quad 
\sup_{k=0,\cdots,\lfloor \curss \rfloor}
\sup_{x \in {\mathbb T}^d} 
\| \nabla_x^k \delta_\mu h(x_0,x,\mu,\cdot) 
 \|_{\cursr} \leq \Lambda;
 \\
 \quad \forall x,x' \in {\mathbb T}^d, \quad 
& \Bigl\| 
 \nabla_x^{\lfloor \curss \rfloor} 
 \delta_\mu h(x_0,x',\mu,\cdot )
 - 
  \nabla_x^{\lfloor \curss \rfloor} 
 \delta_\mu h(x_0,x,\mu,\cdot)
 \Bigr\|_{\cursr} \leq \Lambda \vert x'-x \vert^{\curss- \lfloor \curss \rfloor}.
 \end{split}
 \end{equation*}

  \noindent \textbf{Condition ${\mathscr D}^0(\Lambda,\cursr)$.}
  We say that a function $h^0 : (x_0,\mu) \in {\mathbb R}^d \times {\mathcal P}({\mathbb T}^d) \mapsto
  h^0(x_0,\mu) \in {\mathbb R}$ satisfies 
 ${\mathscr D}^0(\Lambda,\cursr)$
 if
 it satisfies the first two conditions in ${\mathscr C}^0(\Lambda,\cursr)$
 and, for all  
$l \in \{0,\cdots,\lfloor \cursr \rfloor\}$, $x_0,x_0' \in {\mathbb R}^d$, $\mu,\mu' \in {\mathcal P}({\mathbb T}^d)$ and $y',y\in\mathbb T^d$,
\begin{align*}
&\big|\nabla_{x_0} h^0(x_0',\mu')-\nabla_{x_0}  h^0(x_0,  \mu)\big| \leq \kappa \Big(| x_0'-x_0|+\mathbb W_1( \mu',\mu) \Big),\\
&\big|\nabla_y^l\delta_\mu h^0(x_0',\mu',y')-\nabla_y^l\delta_\mu h^0( x_0,\mu, y)\big| \leq \kappa \Big(| x_0'-x_0|+\mathbb W_1( \mu',\mu) +|y'-y|^{\cursr -\lfloor \cursr \rfloor}\Big).
\nonumber
\end{align*}

  \noindent \textbf{Condition ${\mathscr D}(\Lambda,\cursr,\curss)$.}
  We say that a function $h : (x_0,x,\mu) \in {\mathbb R}^d \times  {\mathbb T}^d \times {\mathcal P}({\mathbb T}^d) \mapsto
  h(x_0,x,\mu) \in {\mathbb R}$ satisfies 
 ${\mathscr D}(\Lambda,\cursr,\curss)$
 if
 it satisfies the first two conditions in ${\mathscr C}(\Lambda,\cursr,\curss)$
 and, for all  
$l \in \{0,\cdots,\lfloor \cursr \rfloor\}$, $x_0,x_0' \in {\mathbb R}^d$, $\mu,\mu' \in {\mathcal P}({\mathbb T}^d)$ and $y, y'\in\mathbb T^d$,
\begin{align*}
&\big\| 
\nabla_{x_0} h(x_0',\cdot,\mu')-\nabla_{x_0}  h(x_0,\cdot,  \mu)\big\|_{\curss}
\leq \Lambda \Big(|x_0'-x_0|+\mathbb W_1( \mu',\mu) \Big),\\
&\big\|\nabla_y^l\delta_\mu h(x_0',\cdot,\mu',y')-\nabla_y^l\delta_\mu h( x_0,\cdot, \mu, y)\big\|_{\curss}
\leq \Lambda \Big(|x_0'-x_0|+\mathbb W_1( \mu',\mu)+|y'-y|^{\cursr -\lfloor \cursr \rfloor} \Big).
\nonumber
\end{align*}

 We now introduce the two sets of assumptions.
 \vskip 5pt
 
 \noindent \textbf{Assumption A.} 
 There exist three reals $\lambda>0$, $\kappa>0$ and $\curss > d/2+5$, $\curss \not \in {\mathbb N}$, such that 
\vskip 4pt

\noindent {\bf (A1)}
The functions $(f^0_t)_{0 \leq t \leq T}$ and $g^0$ satisfy 
${\mathscr D}^0(\kappa,\lfloor \curss\rfloor-(d/2+1))$;
the functions $(f_t)_{0 \le t \le T}$ and 
$g$ satisfy 
${\mathscr D}(\kappa,\lfloor \curss\rfloor-(d/2+1),\curss)$. 
\vskip 4pt

\noindent {\bf (A2)} 
The Lagrangians $L^0$ and $L$ are $\lambda$-strictly convex in the variables 
$\alpha_0$ and $\alpha$ respectively, uniformly with respect to the other variables. 
 Moreover, $L^0$ has continuous joint derivatives up to the order 2 in $x_0$ and $\lfloor \curss \rfloor +1$ in 
 $\alpha_0$, and 
 $L$ has continuous joint derivatives up to the order $\lfloor \curss \rfloor +1$ in 
 $(x,\alpha)$.
 The quantity $L^0(x_0,\alpha_0)$ is bounded by 
 $\kappa(1+\vert \alpha_0\vert^2)$ and 
 the quantity $L(x,\alpha)$ is bounded by 
 $\kappa(1+\vert \alpha \vert^2)$.  
  The gradient $\nabla_{\alpha_0} L^0(x_0,\alpha_0)$ 
  is bounded by $\kappa( 1+ \vert \alpha_0\vert)$
  and the gradient $\nabla_\alpha L(x_0,x,\alpha)$
  is bounded by $\kappa( 1 + \vert \alpha \vert)$.   
 All the other existing derivatives are bounded by $\kappa$. 
 \vskip 4pt
 
 \noindent{\bf (A3)}
 The functions $f$ and $g$ satisfy the following Lasry-Lions monotonicity conditions
 \begin{equation*} 
 \int_{{\mathbb R}^d} 
 \bigl( 
 f_t(x_0,x,\mu') -  f_t(x_0,x,\mu) \bigr) 
 \ud \bigl( \mu' - \mu \bigr)(x) \geq 0, 
 \quad 
 \int_{{\mathbb R}^d} 
 \bigl( 
 g(x_0,x,\mu') -  g(x_0,x,\mu) \bigr) 
 \ud \bigl( \mu' - \mu \bigr)(x) \geq 0, 
 \end{equation*} 
 for all $t \in [0,T]$, $x_0 \in {\mathbb R}^d$,
 and $\mu,\mu' \in {\mathcal P}({\mathbb T}^d)$. 
  \vskip 4pt
 
 \noindent{\bf (A4)}
 The coefficient $\sigma_0$ is greater than 1. 

\vskip 10pt

 \noindent \textbf{Assumption B.} 
On top of Assumption \hyp{A} and for the same parameters as therein,
there exists two functions $F^0 : {\mathbb R}^d \rightarrow {\mathbb R}$ and $F
: {\mathbb T}^d \times {\mathcal P}({\mathbb T}^d) \rightarrow 
{\mathbb R}$ such that 
\vskip 4pt

 \noindent{\bf (B1)}
 The function 
 $F^0$ is continuously differentiable and bounded by $\kappa$, and its gradient is also bounded by 
 $\kappa$ and is $\kappa$-Lipschitz continuous; 
 the function $F$, seen as a function of $(x_0,x,\mu)$ that would be constant 
 in the variable $x_0$,  
satisfies 
${\mathscr D}(\kappa,\lfloor \curss\rfloor-(d/2+1),\curss)$.
 \vskip 4pt

 \noindent{\bf (B2)}
The following two bounds hold true:
\begin{equation*} 
\int_0^T \sup_{x_0 \in {\mathbb T}^d} \sup_{\mu \in {\mathcal P}({\mathbb T}^d)}
\vert f_t^0(x_0,\mu) - F^0(x_0) \vert 
\ud t \leq \kappa,
\quad
\int_0^T \sup_{x_0 \in {\mathbb T}^d} \sup_{\mu \in {\mathcal P}({\mathbb T}^d)}
\| f_t(x_0,x,\mu) - F(x,\mu) \|_{\curss}
\ud t \leq \kappa.
\end{equation*} 

 \noindent{\bf (B3)}
The functions $(f_t^0)_{0 \le t \le T}$
satisfy 
\begin{equation*} 
\int_0^{T} \sup_{x_0 \in {\mathbb T}^d} \sup_{\mu \in {\mathcal P}({\mathbb T}^d)}
\| \delta_\mu f_t^0(x_0,\mu,\cdot)  \|_1 
\ud t \leq \kappa.
\end{equation*} 

\begin{remark}
The following remarks are in order: 
\begin{enumerate}
\item 
Assumption \hyp{B} has a simple interpretation. In long time, the 
running costs associated with the major and minor 
players become independent of the state of the other player. 
As made clear below, this assumption is very important to obtain 
a 
lower bound independent of $T$ 
for the intensity $\sigma_0$ 
beyond which the Major/Minor MFG has the desired solvability properties. 
\item Construction of examples satisfying 
\hyp{B} is quite simple. It suffices to start from given coefficients $F^0$ and $F$ and to add
 perturbations $f^0_t-F^0$ and $f_t-F$ that decay sufficiently fast 
 as $t$ tends to $\infty$. 
 \item 
 The form of 
 Assumption \hyp{B}
 explains (up to some extent) our choice to restrict the analysis 
 to Lagrangians $L^0$ and $L$ 
 that depend on the state of one player only (and not on the state of the other player). 
 If one of the two Lagrangians were depending on both states, we would need a convenient form of 
condition \hyp{B2}. 
\end{enumerate}

\end{remark}

\textit{Hamiltonians.} 
With the two Lagrangians $L^0$ and $L$, we associate the
following two Hamiltonians:
\begin{equation}
\label{eq:Hamiltonians}
\begin{split}
&H^0(x_0,p) := \sup_{\alpha \in {\mathbb R}^d} 
\bigl[ - p \cdot \alpha - L^0(x,\alpha) \bigr],
\\
&H(x,p) = \sup_{\alpha \in {\mathbb R}^d} 
\bigl[ - p \cdot \alpha - L(x,\alpha) \bigr].
\end{split}
\end{equation} 
Under the standing standing assumptions on the Lagrangian $L$, we have the following representation formula for $H$: 
\begin{equation}
\label{eq:Hamiltonians:Fenchel}
\nabla_p H(x,p) = (\nabla_\alpha L)^{\circ -1}(x,p), \quad 
H(x,p)=p \cdot  \nabla_p H(x,p) -L\bigl(x,-\nabla_p H(x,p)\bigr) , 
\end{equation} 
and similarly for $H^0$. We easily deduce that 
there exist two (strictly) positive constants $\lambda'$ and $\kappa'$
such that 
 (notice that \hyp{A5} below is not an assumption but a consequence of 
Assumption \hyp{A}):
\vskip 4pt

\noindent {\bf (A5)} 
The Hamiltonians $H^0$ and $H$ are $\lambda'$-strictly convex in the variables 
$p_0$ and $p$ respectively, uniformly with respect to the other variables. 
 Moreover, $H^0$ has continuous joint derivatives up to the order 2 in $x_0$ and $\lfloor \curss \rfloor +1$ in 
 $p$, and 
 $H$ has continuous joint derivatives up to the order $\lfloor \curss \rfloor +1$ in 
 $(x,p)$.
 The quantity $H^0(x_0,p)$ is bounded by 
 $\kappa'(1+\vert p\vert^2)$ and 
 the quantity $H(x,p)$ is bounded by 
 $\kappa'(1+\vert p \vert^2)$.  
  The gradient $\nabla_{p} H^0(x_0,p)$ 
  is bounded by $\kappa'( 1+ \vert p \vert)$
  and the gradient $\nabla_p H(x,p)$
  is bounded by $\kappa'( 1 + \vert p \vert)$.   
 All the other existing derivatives are bounded by $\kappa'$.

\subsection{Weak formulation of the Major/Minor MFG} 
\label{subse:weak:formulation:MFG}

Below, we address a weak formulation of the game in which the spaces
$\Omega^0$ and $\Omega$ are taken as the canonical spaces $\Omega^0_{\textrm{\rm canon}}:={\mathcal C}([0,T],{\mathbb R}^d) \times {\mathcal C}([0,T],{\mathcal P}({\mathbb T}^d)) \times {\mathcal C}([0,T],{\mathbb R}^d)$
and 
$\Omega_{\textrm{\rm canon}} := {\mathcal C}([0,T],{\mathbb R}^d)  \times {\mathcal C}([0,T],{\mathbb R}^d)$,
 equipped with their respective Borel $\sigma$-fields ${\mathcal B}(\Omega^0_{\textrm{\rm canon}})$ and ${\mathcal B}(\Omega_{\textrm{\rm canon}})$.
For simplicity, the canonical processes on $\Omega^0_{\textrm{\rm canon}}$ and $\Omega_{\textrm{\rm canon}}$ are still   denoted $({\boldsymbol X}^0,{\boldsymbol \mu},{\boldsymbol B}^0)=(X_t^0,\mu_t,B_t^0)_{0 \le t \le T}$ and 
$({\boldsymbol X},{\boldsymbol B})=(X_t,B_t)_{0 \le t \le T}$, and the canonical filtrations 
are still  denoted ${\mathbb F}^0=({\mathcal F}_t^0)_{0 \le t \le T}$ and 
${\mathbb F}=({\mathcal F}_t)_{0 \le t \le T}$.

\begin{definition} 
\label{def:admissibility:pair}
For fixed initial conditions $(x_0,\mu) \in {\mathbb R}^d \times {\mathcal P}({\mathbb T}^d)$, 
a pair $(\alpha^0,\alpha)$, with $\alpha^0 :  [0,T] \times \mathbb R^d \times 
{\mathcal P}({\mathbb T}^d) \rightarrow {\mathbb R}^d$
and 
$\alpha :  [0,T] \times {\mathbb R}^d \times {\mathbb T}^d \times {\mathcal P}({\mathbb T}^d) \rightarrow {\mathbb R}^d$
is said to be admissible if 
$\alpha^0$ and 
$\alpha$ are measurable, 
$\alpha$ is bounded and 
there exists a unique probability measure ${\mathbb P}^0$ on $\Omega_{\textrm{\rm canon}}^0$ such that: 
(i) under ${\mathbb P}^0$, the 
process ${\boldsymbol B}^0$ is an-${\mathbb F}^0$ Brownian motion starting from 
$0$; (ii) the 
pair
$({\boldsymbol X}^0,{\boldsymbol \mu})$
satisfies ${\mathbb P}^0(\{ (X^0_0,\mu_0) = (x_0,\mu)\})=1$; (iii) the two equations
\begin{equation}
\label{eq:Markov:state:equation}
\begin{split}
&\ud X_t^0 = \alpha^0(t,X_t^0,\mu_t) \ud t + \ud B_t^0, 
\\
&\partial_t \mu_t = \tfrac12 \Delta_{x} \mu_t - \textrm{\rm div}_{x} \bigl( \alpha(t,X_t^0,\cdot,\mu_t) \mu_t \bigr), \quad t \in [0,T], 
\end{split}
\end{equation} 
are satisfied under ${\mathbb P}^0$; (iv) the  BMO condition 
\begin{equation}
\label{eq:Markov:state:equation:BMO} 
\biggl\| 
\int_0^\cdot \alpha^0(t,X_t^0,\mu_t)
\cdot 
\ud B_t^0
\biggr\|_{\textrm{\rm BMO}} < \infty
\end{equation} 
holds true.
When needed to clarify the set-up, we 
put an additional index in ${\mathbb P}^0$ and write
 ${\mathbb P}^0_{({ \alpha}^0,{ \alpha})}$.
\end{definition} 

\begin{remark}
\label{rem:weak:form:FP}
In \eqref{eq:Markov:state:equation}, the Kolmogorov equation is understood in a weak sense, namely, 
for ${\mathbb P}^0$-almost every $\omega^0 \in \Omega_{\textrm{\rm canon}}^0$, for any 
function $\phi : [0,T] \times {\mathbb T}^d \rightarrow {\mathbb R}^d$ in the space ${\mathcal W}_{1,2}^{d+1}$
of functions with first and second order $x$-derivatives in space and first order $t$-derivative in 
$L^{d+1}({\mathbb T}^d)$, 
\begin{equation} 
\label{eq:FP:weak:form:W12d+1}
\bigl( \phi_t , \mu_t \bigr) =
\bigl( \phi_0 , \mu_0\bigr)
+ \int_0^t \Bigl( \partial_t \phi_s(\cdot)  + \tfrac12 \Delta_{x} \phi_s(\cdot) + 
\alpha(s,X_s^0,\cdot,\mu_s) \cdot \nabla_x \phi_s(\cdot), \mu_s \Bigr) \ud s, \quad t \in [0,T]. 
\end{equation} 
By an obvious separability argument, the above is equivalent to having the same expansion for any $\phi \in {\mathcal W}_{1,2}^{d+1}$
and for ${\mathbb P}^0$-almost every $\omega^0 \in \Omega^0_{\textrm{\rm canon}}$.
\end{remark} 

The definition of the costs in the weak formulation relies on the following lemma (we recall that 
${\boldsymbol X}^0$ and $\boldsymbol{\mu}$ are part of the canonical process on $\Omega^0_{\textrm{\rm canon}}$):
\begin{lemma}
\label{lem:alpha:weak:solution}
Let 
$\mu \in {\mathcal P}({\mathbb T}^d)$ and 
$\alpha:  [0,T] \times {\mathbb R}^d \times {\mathbb T}^d \times {\mathcal P}({\mathbb T}^d) \rightarrow {\mathbb R}^d$
be a bounded and mesurable function as in Definition \ref{def:admissibility:pair}. Then, there exists a measurable mapping 
\begin{equation*} 
\Omega^0_{\textrm{\rm canon}}  \ni \omega^0 \mapsto {\mathbb P}_{\omega^0} \in {\mathcal P}(\Omega_{\textrm{\rm canon}} ),
\end{equation*}
such that, for any $\omega^0 \in \Omega^0_{\textrm{\rm canon}}$, ${\mathbb P}_{\omega^0}$ is the unique probability measure on 
$\Omega_{\textrm{\rm canon}} $ satisfying the following three items: 
(i) under ${\mathbb P}_{\omega^0}$, 
the 
process ${\boldsymbol B}$ is an-${\mathbb F}$ Brownian motion starting from 
$0$; (ii) the 
pair
$({\boldsymbol X},{\boldsymbol B})$
satisfies ${\mathbb P}_{\omega^0} \circ X_0^{-1} = \mu$; (iii) the equation
\begin{equation}
\label{eq:lem:alpha:weak:solution:SDE}
\ud X_t  = \alpha(t,X_t^0,X_t,\mu_t) \ud t + \ud B_t, \quad t \in [0,T],
\end{equation}
is satisfied under ${\mathbb P}_{\omega^0}$. 
When needed to clarify the set-up, we 
put an additional index in ${\mathbb P}_{\omega^0}$ and write
 ${\mathbb P}_{\omega^0,\alpha}$.
\end{lemma} 

\begin{remark}
\label{rem:!:Fokker-Planck}
Using It\^o-Krylov formula, it is easy to check, for $\omega^0 \in \Omega^0$
and for any 
function $\phi : [0,T] \times {\mathbb T}^d \rightarrow {\mathbb R}^d$ in the space ${\mathcal W}_{1,2}^{d+1}$,
the flow of probability measures ${\boldsymbol \nu} := (\nu_t= {\mathbb P}_{\omega^0} \circ X_t^{-1})_{0 \le t \le T}$
satisfies  
 the Fokker-Planck equation 
\begin{equation} 
\label{eq:fokker-planck:nu}
\partial_t \nu_t = \tfrac12 \Delta_{x} \nu_t - \textrm{\rm div}_{x} \bigl( \alpha(t,X_t^0,\cdot,\mu_t) \nu_t \bigr), \quad t \in [0,T]; \quad \nu_0=\mu,
\end{equation} 
 in the same weak sense as 
in Remark 
\ref{rem:weak:form:FP}. 
As such, ${\boldsymbol \nu}$ is ${\mathbb F}^0$-adapted, because for any bounded test 
function $\phi : {\mathbb T}^d \rightarrow {\mathbb R}$, 
the process 
$(\int_{{\mathbb T}^d} \phi(x) \ud \nu_t(x) 
= {\mathbb E}_{\omega^0}[ \phi (X_t)])_{0 \le t \le T}$
is ${\mathbb F}^0$-progressively measurable.   

In fact, this is the unique weak solution to \eqref{eq:fokker-planck:nu}. Indeed, any weak solution, say
$\tilde{\boldsymbol \nu} = (\tilde \nu_t)_{0 \leq t \leq T}$, to 
the Fokker-Planck equation 
(understood in the same weak sense as 
\eqref{eq:FP:weak:form:W12d+1}
in Remark 
\ref{rem:weak:form:FP}) is necessarily equal to the flow of marginal laws 
of the solution to the SDE 
\eqref{eq:lem:FP:SDE:X:tildeB}. This follows from a standard duality argument that consists in solving 
the parabolic equation $\partial_t \phi_t + \tfrac12 \Delta_x \phi_t + \alpha(t,X_t(\omega^0),\cdot,\mu_t) \cdot 
\nabla_x \phi_t = 0$, with $\phi_T = \ell$ for 
a prescribed smooth function $\ell : 
{\mathbb T}^d \rightarrow {\mathbb R}$, and then in expanding $((\phi_t,\tilde \nu_t))_{0 \le t \le T}$. In this way, we get 
$(\ell,\tilde \nu_T) = (\phi_T,\tilde \nu_T) = (\phi_0,\mu)$
for any smooth $\ell$, 
which suffices to identify $\tilde \nu_T$. Replacing $T$ by $t$, we can proceed in a similar manner and identify 
$\tilde \nu_t$ for any $t \in [0,T]$. 
\end{remark}

The proof of Lemma \ref{lem:alpha:weak:solution} is postponed to the end of the section, see Subsection \ref{subse:proof:FP}. For the time being, the statement makes it possible to let: 
\begin{definition}
\label{weak:costs}
For a fixed initial condition 
$(x_0,\mu) \in {\mathbb R}^d \times {\mathcal P}({\mathbb T}^d)$
and 
with the same notation as in Definitions 
\ref{def:admissibility:pair}
and 
Lemma \ref{lem:alpha:weak:solution}, we let:
\begin{description}
\item[\hspace{5pt} \textrm{\rm 1.} Cost to the major\textrm{\rm :}] 
Let $(\alpha^0,\alpha)$ be an admissible pair
and ${\mathbb P}^0_{({\alpha}^0,{\alpha})}$ be the probability 
associated with it by Definition 
\ref{def:admissibility:pair}. Then, the cost to the major player is defined by 
\begin{equation*}
J^0_{\rm w}\bigl({\alpha}^0,{  \alpha} \bigr) = 
{\mathbb E}^0_{({ \alpha}^0,{\alpha})} \biggl[ g^0(X_T^0,\mu_T) + \int_0^T   \Bigl( f^0_t(X_t^0,\mu_t) + L^0(X_t^0, \alpha_t^0) \Bigr) \ud t \biggr];
\end{equation*} 
\item[\hspace{5pt} \textrm{\rm 2.} Cost to the minor\textrm{\rm :}]
Let $\alpha:  [0,T] \times {\mathbb R}^d \times {\mathbb T}^d \times {\mathcal P}({\mathbb T}^d) \rightarrow {\mathbb R}^d$
be a bounded and mesurable function and 
${\mathbb P}^0$  be a probability measure on $\Omega^0_{\textrm{\rm canon}} $. Then, the cost to the minor player in the environment 
${\mathbb P}^0$ is defined by 
\begin{equation*} 
J_{\rm w}\bigl({ \alpha} ; {\mathbb P}^0 \bigr)
=
\int_{\Omega^0} 
{\mathbb E}_{\omega^0,\alpha} \biggl[ g(X_T^0,X_T,\mu_T) +  \int_0^T 
\Bigl( f_t(X_t^0,X_t,\mu_t) + L(X_t,\alpha_t) 
\Bigr) 
\ud t \biggr] \ud {\mathbb P}^0(\omega^0). 
\end{equation*} 
\end{description}
\end{definition} 

The following definition is inspired from \cite{Cardaliaguet:Cirant:Porretta:PAMS,Cardaliaguet:Cirant:Porretta:JEMS,Lasry:Lions:Major:Minor}:

\begin{definition}
\label{def:Nash}
Let 
$(x_0,\mu) \in {\mathbb R}^d \times {\mathcal P}({\mathbb T}^d)$
be a fixed initial condition and 
$(\alpha^0,\alpha)$ be an admissible 
pair in the sense of Definition \ref{def:admissibility:pair}. 
Recalling the notation  
$({\boldsymbol X}^0,{\boldsymbol \mu},{\boldsymbol B}^0)$
for the canonical process on $\Omega^0_{\textrm{\rm canon}} $ and the notation 
$({\boldsymbol X},{\boldsymbol B})$ for the canonical process on 
$\Omega_{\textrm{\rm canon}} $,   
the pair $(\alpha^0,\alpha)$
is said to be a mean field equilibrium if 
\begin{description}
\item[\hspace{5pt} \textrm{\rm 1.} Deviation of the minor\textrm{\rm :}] For any 
bounded and measurable function
$\beta:  [0,T] \times {\mathbb R}^d \times {\mathbb T}^d \times {\mathcal P}({\mathbb T}^d) \rightarrow {\mathbb R}^d$,
\begin{equation*} 
J_{\rm w}\bigl({\alpha}; {\mathbb P}^0_{(\alpha^0,\alpha)}\bigr) 
\leq 
J_{\rm w}\bigl({\beta}; {\mathbb P}^0_{(\alpha^0,\alpha)}\bigr);
\end{equation*} 
\item[\hspace{5pt}  \textrm{\rm 2.} Deviation of the major\textrm{\rm :}] For any feedback function $\beta^0 : [0,T] \times {\mathbb R}^d \times {\mathcal P}({\mathbb T}^d) \rightarrow {\mathbb R}$ such that the pair 
$(\beta^0,\alpha)$ is admissible, 
\begin{equation*} 
J_{\rm w}^0({\alpha}^0, \alpha) 
\leq 
J_{\rm w}^0({ \beta}^0,   {\alpha}).
\end{equation*} 
\end{description}
\end{definition} 
\color{black}

\begin{remark}
It is worth emphasizing that
the nature of the equilibrium would be different if 
${\boldsymbol \alpha}^0$ and ${\boldsymbol \alpha}$ were required to be in open loop form. 
The analysis of equilibria over controls in {Markov} feedback form is in fact more difficult. 

As far as the control ${\beta}$ is concerned in item 1 right above, the nature 
of it (open versus closed) does not make any difference. Intuitively, this comes from 
the fact that any deviation of the minor player has no influence on the state of the 
population (encoded through ${\boldsymbol \mu}$) nor on 
the state of the major player (encoded through 
${\boldsymbol X}^0$). 
This fact is standard in MFG theory. 
\end{remark}

We now give conditions under
which a pair $(\alpha^0,\alpha)$ is admissible in the sense of Definition  
\ref{def:admissibility:pair}:

\begin{lemma}
\label{lem:weak:uniqueness:forward:equation}
Let $\alpha^0 : [0,T] \times {\mathbb R}^d \times {\mathcal P}({\mathbb T}^d ) \rightarrow {\mathbb R}^d$
and
$\alpha  : [0,T] \times {\mathbb R}^d \times {\mathbb T}^d \times {\mathcal P}({\mathbb T}^d ) \rightarrow {\mathbb R}^d$
 be two Borel measurable functions such that
 $\alpha$ is bounded and satisfies 
 \begin{equation}
\label{prop:lem:weak:uniqueness:forward:equation}
\sup_{t \in [0,T]} 
\sup_{x_0 \in {\mathbb R}^d}
\sup_{\mu,\mu' \in {\mathcal P}({\mathbb T}^d) : \mu \not = \mu'} 
\sup_{x \in {\mathbb T}^d} \frac{| \alpha(t,x_0,x,\mu') - \alpha(t,x_0,x,\mu) |}{ {\mathbb W}_1(\mu,\mu')}  < \infty.
\end{equation} 
Then:
\vskip 4pt

(a) 
For any
$(x_0,\mu) \in {\mathbb R}^d \times {\mathcal P}({\mathbb T}^d)$,
there exists at most one 
probability measure 
${\mathbb P}^0$ on $\Omega_{\textrm{\rm canon}}^0$
such that items (i), (ii), (iii) and (iv) 
in Definition 
\ref{def:admissibility:pair} are satisfied. 
\vskip 4pt

(b) 
For any
$(x_0,\mu) \in {\mathbb R}^d \times {\mathcal P}({\mathbb T}^d)$, 
 there exists a unique probability measure $\bar{\mathbb P}^0$ on $\Omega^0_{\textrm{\rm canon}}$
such that items (i), (ii) and (iii) 
in Definition 
\ref{def:admissibility:pair} are satisfied when $\alpha^0 \equiv 0$. 
\vskip 4pt

(c) If, for a given $(x_0,\mu) \in {\mathbb R}^d \times {\mathcal P}({\mathbb T}^d)$, the  BMO condition 
\begin{equation}
\label{eq:Markov:state:equation:BMO:0} 
\biggl\| 
\int_0^\cdot \alpha^0(t,X_t^0,\mu_t)
\cdot 
\ud B_t^0
\biggr\|_{\textrm{\rm BMO}} < \infty
\end{equation} 
is satisfied under $\bar{\mathbb P}^0$, then there exists a (hence unique) 
probability measure 
${\mathbb P}^0$ on $\Omega_{\textrm{\rm canon}}^0$
such that items (i), (ii), (iii) and (iv) are satisfied 
under Definition 
\ref{def:admissibility:pair}. 
Conversely, if 
there exists a  
probability measure 
${\mathbb P}^0$ on $\Omega_{\textrm{\rm canon}}^0$
such that items (i), (ii), (iii) and (iv) are satisfied, then the BMO condition 
\eqref{eq:Markov:state:equation:BMO:0} 
is satisfied under $\bar{\mathbb P}^0$.
\end{lemma} 
Similar to Lemma \ref{lem:alpha:weak:solution}, 
Lemma \ref{lem:weak:uniqueness:forward:equation} is proven in Subsection \ref{subse:proof:FP}.

\subsection{Forward-backward characterization}

Our analysis below relies on a stochastic forward-backward system, which plays the same role 
as the MFG system in the standard setting. This system reads in the form of 
two coupled  stochastic forward-backward equations, which are understood in a strong sense
and thus posed on any arbitrary 
filtered probability space 
$(\Omega^0,{\mathcal F}^0,{\mathbb F}^0,{\mathbb P}^0)$
satisfying the usual conditions and 
equipped with a 
Brownian motion $(B_t^0)_{0 \leq t \leq T}$ with values in ${\mathbb R}^d$. 
In particular, 
{\bf we NO longer regard
$({\boldsymbol X}^0,{\boldsymbol \mu},{\boldsymbol B}^0)=(X_t^0,\mu_t,B_t^0)_{0 \le t \le T}$ and 
$({\boldsymbol X},{\boldsymbol B})=(X_t,B_t)_{0 \le t \le T}$
as canonical processes in this subsection}.

The first equation  provides a Lagrangian description of the state of the major player (at equilibrium):
\begin{equation} 
\label{eq:major:FB:1}
\begin{split} 
&\ud X_t^0 = - \nabla_p H^0 \bigl( X_t^0 , Z_t^0\bigr)  \ud t + \sigma_0 \ud B_t^0, \quad t\in [0,T],
\\
&\ud Y_t^0 =   - \Bigl( f_t^0(X_t^0,\mu_t) + L^0\bigl(X_t^0,- \nabla_p H^0\bigl( X_t^0, Z_t^0\bigr) \bigr)
\Bigr) \ud t + 
\sigma_0 Z_t^0 \cdot \ud B_t^0, \quad t \in [0,T], 
\\
&X_0^0=x_0, \quad Y_T^0 = g^0(X_T^0,\mu_T).
\end{split}
\end{equation} 
Above, the 
measure argument  $(\mu_t)_{0 \le t \le T}$
 in the dynamics of 
$(Y_t^0)_{0 \le t \le T}$ corresponds to the statistical law of the minor player, 
whose evolution (at equilibrium) is described by the
 stochastic system  {(which is the second of the two aforementioned forward-backward equation)}: 
\begin{equation}
\label{eq:minor:FB:2}
\begin{split} 
&\partial_t \mu_t - \tfrac12 \Delta_x \mu_t - {\rm div}_x \Bigl( \nabla_p H \bigl(\cdot, \nabla_x u_t\bigr) \mu_t \Bigr) =0, 
\quad  \ (t,x)\in [0,T] \times {\mathbb T}^d, 
\\
&\ud_t u_t(x) =\Bigl(  -  \tfrac12 \Delta_x u_t(x) + H\bigl(x,\nabla_x u_t(x) \bigr)  - f_t(X_t^0,x,\mu_t)  \Bigr) \ud t
+ \sigma_0 v_t^0(x) \cdot \ud B_t^0, \quad (t,x) \in [0,T] \times {\mathbb T}^d, 
\\
&\mu_0=\mu,\quad u_T(x) = g(X_T^0,x,\mu_T), \quad x \in {\mathbb T}^d.  
\end{split}
\end{equation}

\begin{definition} 
\label{def:forward-backward=MFG:solution}
Solutions to the above two systems are understood in the following sense: 
\begin{enumerate}
\item 
For a given initial condition $x_0  \in {\mathbb R}^d$
and  an ${\mathbb F}^0$-adapted continuous process 
${\boldsymbol \mu} = (\mu_t)_{0 \le t \le T}$ with values in 
${\mathcal P}({\mathbb T}^d)$, we call  solution 
to 
\eqref{eq:major:FB:1}
any ${\mathbb F}^0$-progressively measurable process 
$({\boldsymbol X}^0,{\boldsymbol Y}^0,{\boldsymbol Z}^0)$ with values in ${\mathbb R}^d \times {\mathbb R}
\times {\mathbb R}^d$, such that: (i)
${\boldsymbol X}^0$ and ${\boldsymbol Y}^0$ 
have continuous trajectories; (ii) 
$\sup_{0 \le t \le T} \vert X_t^0 \vert \in L^2(\Omega^0,{\mathbb P}^0)$
and
$\sup_{0 \le t \le T} \vert Y_t^0 \vert \in L^\infty(\Omega^0,{\mathbb P}^0)$;
(iii) $\sup_{\tau} \| {\mathbb E}^0[ \int_{[\tau,T]} \vert Z_t^0 \vert^2 \ud t \vert {\mathcal F}_\tau^0] \|_{L^\infty(\Omega^0,{\mathbb P}^0)} < \infty$, the supremum being 
taken over the collection of stopping times $\tau$ with respect to the usual augmentation of the filtration 
generated by $({\boldsymbol X}^0,{\boldsymbol Y}^0,{\boldsymbol Z}^0,{\boldsymbol B}^0)$; 
(iv) the system 
\eqref{eq:major:FB:1} is satisfied ${\mathbb P}^0$- almost surely.
\item For a given initial condition $\mu \in {\mathcal P}({\mathbb T}^d)$
and  an ${\mathbb F}^0$-adapted continuous process 
${\boldsymbol X}^0 = (X^0_t)_{0 \le t \le T}$ with values in 
${\mathbb R}^d$ such that 
$\sup_{0 \le t \le T} \vert X_t^0 \vert \in L^2(\Omega^0,{\mathbb P}^0)$, we call  solution 
to 
\eqref{eq:minor:FB:2}
any ${\mathbb F}^0$-progressively measurable process 
$({\boldsymbol \mu},{\boldsymbol u},{\boldsymbol v}^0)$ with values in ${\mathcal P}({\mathbb T}^d) \times {\mathcal C}^\mathfrc{s}({\mathbb T}^d) \times 
{\mathcal C}^{\lfloor \mathfrc{s} \rfloor -d/2-1}({\mathbb T}^d)$, such that: (i)
${\boldsymbol \mu}$ and ${\boldsymbol u}$ 
have continuous trajectories in 
${\mathcal P}({\mathbb T}^d) \times   {\mathcal C}^{\cursr}({\mathbb T}^d)$
for any $\mathfrc{r} < \mathfrc{s}$
; (ii) $\sup_{0 \le t \le T} \| u_t \|_{\mathfrc{s}} \in L^\infty(\Omega^0,{\mathbb P}^0)$;
(iii) 
${\mathbb E}^0 \int_{[0,T]} \| v_t^0 \|_{\lfloor \mathfrc{s} \rfloor -d/2-1}^2 \ud s < \infty$;
(iv) the forward equation
\eqref{eq:minor:FB:2} is satisfied ${\mathbb P}^0$ almost surely in the same weak sense as in Remark
\ref{rem:weak:form:FP}; 
(v) the backward equation is satisfied ${\mathbb P}^0$-almost surely (in the classical sense).
\item For a given initial condition $(x_0,\mu) \in {\mathbb R}^d \times {\mathcal P}({\mathbb T}^d)$,
we call a solution to the coupled systems 
\eqref{eq:major:FB:1}--\eqref{eq:minor:FB:2} a pair $(({\boldsymbol X}^0,{\boldsymbol Y}^0,{\boldsymbol Z}^0),({\boldsymbol \mu},{\boldsymbol u},{\boldsymbol v}^0))$
satisfying items 1 and 2 right above. 
\end{enumerate}
\end{definition}

\begin{remark} 
\label{rem:def:2:9}
The following remarks are in order: 
\begin{enumerate}
\item 
The process $(Y_t^0)_{0 \le t \le T}$ 
is expected to describe the evolution of the equilibrium cost to the major player. Intuitively, 
$(-\nabla_p H^0(X_t^0,Z_t^0))_{0 \le t \le T}$ is the corresponding equilibrium feedback.
 
 \item 
 In the stochastic forward-backward system \eqref{eq:minor:FB:2}, the (random) function $(t,x) \mapsto u_t(x)$ 
 is the equilibrium value of the minor player. The optimal feedback is 
 $(t,x) \mapsto - \nabla_p H(x,\nabla_x u_t(x))$. 
 Obviously, $(u_t)_{0 \le t \le T}$
 is random under the presence of the noise ${\boldsymbol B}^0$ 
 acting on the major player. 
 The term $(v_t^0)_{0 \le t \le T}$ is here to ensure that 
 $(u_t)_{0 \le t \le T}$ is indeed ${\mathbb F}^0$-adapted. 
 
\hspace{5pt}  
It is worth mentioning that, in comparison with the two forward and backward equations (4.2) and (4.3) 
introduced in \cite[Chapter 4]{CardaliaguetDelarueLasryLions}
in the analysis of MFGs with a common noise, our own system is simpler because the dynamics of the minor player are NOT forced by the common noise 
$B^0$. In particular, 
the forward equation is not a stochastic Fokker-Planck equation
(like
\cite[(4.2)]{CardaliaguetDelarueLasryLions})
 but  a Fokker-Planck equation with random coefficients.
For the same reason, the backward equation does not contain any It\^o-Wentzell correction comparable to the one appearing 
in 
\cite[(4.3)]{CardaliaguetDelarueLasryLions}.
\item Below, the analysis of the backward SPDE in 
\eqref{eq:minor:FB:2}
is inspired by the study carried out in 
the monograph 
\cite{CardaliaguetDelarueLasryLions}, in which a similar equation is treated 
within the framework of mean field games with common noise (see Chapter 4 therein). 
However, we have slightly changed the spaces in which solutions are taken:
we feel clearer to see them as random processes with values in non-integer H\"older 
spaces, whilst they are regarded as random 
process with values  
in integer H\"older spaces in \cite{CardaliaguetDelarueLasryLions}. This requires some care because, for $\curss \not \in {\mathbb N}$, 
the space ${\mathcal C}^{\curss}({\mathbb T}^d)$ is not separable. 
Working with (non-separable) Banach-valued random
variables
is indeed an issue, see for instance 
\cite{LedouxTalagrand}. One standard way to overcome the 
lack of separability is to strengthen the notion of measurability 
and to work with \textit{Bochner} measurable random variables. 
As explained in 
 \cite{LiQueffelec} (see also 
 \cite{JansonKaijser} for an overview), 
 a random variable
with values in a Banach space $E$ (measurability being understood 
with respect to the standard Borel $\sigma$-field on $E$) is \textit{Bochner} measurable 
if it takes values in a separable subspace of $E$. 
 Any such \textit{Bochner} measurable random variable has a tight distribution and can be approximated by simple 
 random variables, which makes its manipulation easier. 
 
 In the specific framework of  \eqref{eq:minor:FB:2}, 
 one can typically choose ${\mathcal C}^{\curss'}({\mathbb T}^d)$, for $\curss' > \curss$, as separable subspace of 
 ${\mathcal C}^{\curss}({\mathbb T}^d)$. 
Indeed,   
 Schauder estimates make it possible to gain some extra regularity on $u_t$, for $t < T$, and then to regard
 the latter as an element of  ${\mathcal C}^{\curss'}({\mathbb T}^d)$, for $\curss' > \curss$. 
 At time $t=T$, \textit{Bochner} measurability can be checked directly thanks to the 
 properties of $g$.
 Combined with the continuity properties stated in Definition 
  \ref{def:forward-backward=MFG:solution} and  an  interpolation inequality in H\"older spaces, this
  argument says even more:   
  ${\boldsymbol u}$ has continuous trajectories 
  from $[0,T-\varepsilon]$ to ${\mathcal C}^{\curss'}$, for any 
  $\varepsilon >0$ and some $\curss' > \curss$. As such, ${\boldsymbol u}$ can be written as the almost everywhere 
  limit in $[0,T] \times \Omega$ of simple processes of the form  $\sum_{i=0}^{n-1} X_{t_i} {\mathbf 1}_{(t_i,t_{i+1}]}$, 
  with $n \geq 1$, $t_0=0<t_1<\cdots<t_n=T$ and $X_{t_i}$ an ${\mathcal F}_{t_i}$ 
  \textit{Bochner} measurable ${\mathcal C}^\curss({\mathbb T}^d)$-valued random variable for each 
  $i =0,\cdots,n-1$.  This extends the notion of Bochner measurability to processes.  
   
 We will come back to these measurability questions when needed, but the message is 
 clear: measurability
properties stated 
 in Definition
\ref{def:forward-backward=MFG:solution}
are 
 in fact
understood in the \textit{Bochner} sense.  
  Of course, there is no similar difficulty 
with the forward component of the system 
\eqref{eq:minor:FB:2} because 
${\mathcal P}({\mathbb T}^d)$
is compact when equipped with any standard distance metricizing 
the weak topology.
\item In the same vein as above, notice that, 
in 
the second item 
of Definition 
\ref{def:forward-backward=MFG:solution}, 
$\sup_{0 \le t \le T} \| u_t \|_{\mathfrc{s}}$ is necessarily measurable if 
${\boldsymbol u}$ 
has continuous trajectories in 
${\mathcal C}^{\cursr}({\mathbb T}^d)$
for any $\mathfrc{r} < \mathfrc{s}$. 
Indeed, for 
$\lfloor \curss \rfloor < \mathfrc{r} < \mathfrc{s}$, the norm 
$\| \cdot \|_{\curss}$ is lower semi-continuous on 
${\mathcal C}^{\cursr}({\mathbb T}^d)$, which shows that 
$\sup_{0 \le t \le T} \| u_t \|_{\mathfrc{s}}$
is then equal to 
$\sup_{t \in [0,T] \cap {\mathbb Q}} \| u_t \|_{\mathfrc{s}}$. 
  \end{enumerate}
\end{remark}

The 
following statement clarifies the 
connection with Definition 
\ref{def:forward-backward=MFG:solution}:
\begin{proposition}
\label{prop:characterization:1}
Assume that,
on any filtered probability space 
$(\Omega^0,{\mathcal F}^0,{\mathbb F}^0,{\mathbb P}^0)$ 
equipped with 
an ${\mathbb F}^0$-Brownian motion ${\boldsymbol B}^0=(B_t^0)_{0 \leq t \leq T}$ 
with 
values in ${\mathbb R}^d$, for any initial condition $(t,x_0,\mu) \in [0,T] \times {\mathbb R}^d \times {\mathcal P}({\mathbb T}^d)$, the system 
\eqref{eq:major:FB:1}--\eqref{eq:minor:FB:2} has a unique solution (in the sense of Definition 
\ref{def:forward-backward=MFG:solution}), denoted 
$(X_s^{0,t,x_0,\mu},Y_s^{0,t,x_0,\mu},Z_s^{0,t,x_0,\mu})_{t \leq s \leq T}$
and
$(\mu_s^{t,x_0,\mu},u_s^{t,x_0,\mu},v_s^{0,t,x_0,\mu})_{t \leq s \leq T}$, and satisfying 
\begin{equation}
\label{eq:prop:characterization:1}
\sup_{t \in [0,T]} 
\sup_{x_0 \in {\mathbb R}^d}
\sup_{\mu,\mu' \in {\mathcal P}({\mathbb T}^d) : \mu \not = \mu'} 
\frac{\| u_t^{t,x_0,\mu}(\cdot) - u_t^{t,x_0,\mu'}(\cdot) \|_1}{ {\mathbb W}_1(\mu,\mu')}  < \infty. 
\end{equation} 
Then,
there exists a pair $(\alpha^0,\alpha)$ satisfying the Definition \ref{def:Nash} of a mean field equilibrium such that, for any 
initial condition 
$(x_0,\mu) \in {\mathbb R}^d \times {\mathcal P}({\mathbb T}^d)$ to the Major/Minor MFG 
 at time $0$, 
the law of the forward path
$(X_s^{0,0,x_0,\mu},\mu_s^{0,x_0,\mu},{B}_s^0)_{0 \leq s \leq T}$ coincides with 
the measure ${\mathbb P}_{(\alpha^0,\alpha)}$ defined in Definition \ref{def:Nash}. 

Moreover,
assume that
for a fixed $(x_0,\mu)  \in {\mathbb R}^d \times {\mathcal P}({\mathbb T}^d)$, 
there exists another 
mean field equilibrium $(\tilde \alpha^0,\tilde \alpha)$ to the Major/Minor initialized at 
$(x_0,\mu)$ at time $0$
such that: (i) $\tilde \alpha$ satisfies
\begin{equation}
\label{eq:prop:characterization:2}
\sup_{t \in [0,T]} 
\sup_{x_0 \in {\mathbb R}^d}
\sup_{\mu,\mu' \in {\mathcal P}({\mathbb T}^d) : \mu \not = \mu'} 
\sup_{x \in {\mathbb T}^d}
\frac{|\tilde \alpha(t,x_0,x,\mu') -\tilde \alpha(t,x_0,x,\mu) |}{ {\mathbb W}_1(\mu,\mu')}  < \infty; 
\end{equation} 
and (ii) 
the state equation
\eqref{eq:Markov:state:equation}
driven by $(\tilde \alpha^0,\tilde \alpha)$ and 
defined on the canonical space $\Omega^0_{\textrm{\rm canon}}$
has, 
for any starting point $( t,\tilde x_0,\tilde \mu) \in [0,T] \times {\mathbb R}^d \times {\mathcal P}({\mathbb T}^d)$,
 a solution
that is adapted with respect to the (augmentation of the) filtration generated by 
$(B_s^0 - B_t^0)_{t \leq s \leq T}$ (which is here the third component of the canonical process). Then,  
${\mathbb  P}_{(\tilde \alpha^0,\tilde \alpha)} =  {\mathbb  P}_{(\alpha^0,\alpha)}$.
\end{proposition}

\begin{remark}
Our statement may seem rather complicated at first sight. In fact, the main idea is to limit the analysis to the existence and uniqueness of Nash equilibria that are 
adapted to the common noise. By analogy with the terminology used in the theory of SDEs, those equilibria should be called ``strong''. 
As for the existence of such strong equilibria, 
the key point here is that the system 
\eqref{eq:major:FB:1}--\eqref{eq:minor:FB:2} 
is assumed to be uniquely strongly solvable. This forces the solutions (to the system) to be adapted to the common noise and also 
implies the existence of a feedback function, see the first step of the proof below.
As for uniqueness, the main assumption consists of the two items (i) and (ii) in the second part of the statement. As shown in the third step 
of the proof, the combination of both forces the equation \eqref{eq:Markov:state:equation} 
(when driven by driven by $(\tilde \alpha^0,\tilde \alpha)$) to be uniquely strongly solvable. 
\end{remark}

Here is now the main statement of our article:
\begin{theorem}
\label{thm:1}
Under Assumption \hyp{A}, for any $T>0$, there exists a threshold $\sigma_0^*(T) \in (0,+\infty)$ such that, for 
$\sigma_0 \geq \sigma_0^*(T)$, 
for any initial condition $(x_0,\mu) \in {\mathbb R}^d \times {\mathcal P}({\mathbb T}^d)$, 
the system 
\eqref{eq:major:FB:1}--\eqref{eq:minor:FB:2} has a unique solution in the sense of Definition 
\ref{def:forward-backward=MFG:solution} (on any filtered probability space 
$(\Omega^0,{\mathcal F}^0,{\mathbb F}^0,{\mathbb P}^0)$ 
equipped with 
an ${\mathbb F}^0$-Brownian motion ${\boldsymbol B}^0=(B_t^0)_{0 \leq t \leq T}$ 
with 
values in ${\mathbb R}^d$). Moreover, \eqref{eq:prop:characterization:1} holds true. 

If in addition, Assumption \hyp{B} is also in force, then we can
choose $\sigma_0^*(T)$ independently of $T$. Namely, 
we can find a threshold $\sigma_0^* \in (0,+\infty)$ such that existence and uniqueness 
hold true on any interval $[0,T]$, $T>0$, and for any $\sigma_0 \geq \sigma_0^*$.
\end{theorem}

Notice that 
\eqref{eq:prop:characterization:1}
implicitly requires to solve 
\eqref{eq:major:FB:1}--\eqref{eq:minor:FB:2} when the initial condition is fixed at any time 
$t \in [0,T]$. 
Observe in particular that, combined with Proposition 
\ref{prop:characterization:1}, Theorem \ref{thm:1}
says the 
the Major/Minor MFG has a unique ``strong'' equilibrium when 
$\sigma_0
\geq \sigma_0^*(T)$ (and \hyp{A} holds true) 
or 
$\sigma_0
\geq \sigma_0^*$ (and \hyp{B} holds true).
\vskip 4pt

The proof of Theorem \ref{thm:1} is deferred
to Sections \ref{se:3} and \ref{se:4}, see in particular 
Theorem \ref{thm:4.10} for a refined version of it.
The proof of the latter makes explicit use of the condition 
$\curss > d/2+5$.  
Notice in fact that it suffices to prove the second part of Theorem \ref{thm:1}, namely the claim under 
Assumption \hyp{B}. Indeed, once the conclusion has been proved to hold true under 
Assumption \hyp{B}, one can easily the derive the
first part of the statement by modifying the constant 
$\kappa$ in \hyp{A} in such a way that 
\hyp{B2} and \hyp{B3} hold true. This is possible to do so by replacing $\kappa$ by a new constant 
that is allowed to depend on $T$.

We now turn to the proof of 
Proposition 
\ref{prop:characterization:1} (the reader may skip it on a first reading):

\begin{proof}[Proof of Proposition \ref{prop:characterization:1}.]
The proof is divided into three steps. 
\vspace{5pt}
\\
\textit{First Step.} 
The first step is to prove that,
for a fixed initial condition $(x_0,\mu) \in {\mathbb R}^d \times {\mathcal P}({\mathbb T}^d)$ 
at time $0$ and
 under the standing unique 
solvability property of the system 
\eqref{eq:major:FB:1}--\eqref{eq:minor:FB:2}, the latter induces 
a Nash equilibrium to the Major/Minor MFG. 
To do so, we 
consider a probability space $(\Omega^0,{\mathcal F}^0,{\mathbb F}^0,{\mathbb P}^0)$ 
equipped with 
a Brownian motion ${\boldsymbol B}^0=(B_t^0)_{0 \leq t \leq T}$.
We assume that ${\mathbb F}^0$ is the ${\mathbb P}^0$-completion of 
the filtration.generated by ${\boldsymbol B}^0$, from which we deduce that 
solutions to
\eqref{eq:major:FB:1}--\eqref{eq:minor:FB:2}, when constructed on this space, 
are necessarily adapted to the (completion) of the filtration generated by ${\boldsymbol B}^0$.

Following \cite[Proposition 1.31]{CarmonaDelarue_book_II}, 
strong uniqueness of the solution to 
\eqref{eq:major:FB:1}--\eqref{eq:minor:FB:2} implies that the mapping 
that sends the initial condition $(t,x_0,\mu)$ onto the law of 
$(X_{s \vee t}^{t,x_0,\mu},\mu_{s \vee t}^{t,x_0,\mu})_{0 \leq s \leq T}$ 
on ${\mathcal C}([0,T],{\mathbb R}^d) \times {\mathcal C}([0,T],{\mathcal P}({\mathbb T}^d))$
is measurable, from which we deduce that the 
solution to 
\eqref{eq:major:FB:1}--\eqref{eq:minor:FB:2}
forms a strong Markov process. Then, 
following  
\cite[Proposition 3.2 \& Theorem 3.4]{ImkellerReveillacRichter}
(which relies on Theorem 6.27 in 
\cite{CinlarJacodProtterSharpe}),  we can find a Borel function 
$\psi^0 : [0,T] \times {\mathbb R}^d \times {\mathcal P}({\mathbb T}^d) \rightarrow {\mathbb R}^d \otimes {\mathbb R}^d$  such that, for
any initial time $t \in [0,T]$
and  any time $s \in [t,T]$,  
\begin{equation*} 
\begin{split}
&{\mathbb P}^0 \Bigl( \bigl\{ Z_s^{0,t,x_0,\mu}  = \psi^0(s,X_s^{0,t,x_0,\mu},\mu_s^{0,t,x_0,\mu}) \bigr\} \Bigr) = 1, 
\end{split}
\end{equation*} 
which shows the existence of a Markov feedback function for the process ${\boldsymbol X}^{0,t,x_0,\mu}$ (understood below as the state of the major player). 

Also,
by observing from the unique strong solvability 
of \eqref{eq:major:FB:1}--\eqref{eq:minor:FB:2}
that
the process ${\boldsymbol u}^{t,x_0,\mu}$ 
is, for any $(t,x_0,\mu) \in [0,T] \times {\mathbb R}^d \times {\mathcal P}({\mathbb T}^d)$, 
adapted to the (completion of the) filtration generated by 
$(B^0_s-B^0_t)_{t \leq s \leq T}$, 
we can define the value to the minor player as 
\begin{equation*}
V(t,x_0,x,\mu) = u_t^{t,x_0,\mu}(x), \quad x \in {\mathbb T}^d,
\quad (t,x_0,\mu) \in [0,T] \times {\mathbb R}^d \times {\mathcal P}({\mathbb T}^d). 
\end{equation*} 
Since the process ${\boldsymbol u}^{t,x_0,\mu}$ takes values in 
${\mathcal C}^{\mathfrc{s}}({\mathbb T}^d)$, we deduce that  $V$ is differentiable in $x$. The gradient 
$\nabla_x V$ induces a measurable mapping from $[0,T] \times {\mathbb R}^d \times {\mathbb T}^d \times {\mathcal P}({\mathbb T}^d)$ to 
${\mathbb R}^d$. 
By unique strong solvability 
of \eqref{eq:major:FB:1}--\eqref{eq:minor:FB:2}, 
we know that, for any 
$(t,x_0,\mu) \in [0,T] \times {\mathbb R}^d \times 
{\mathcal P}({\mathbb T}^d)$ and 
any
$s \in [t,T]$,  
\begin{equation*} 
{\mathbb P}^0 \Bigl( \bigl\{ \forall x \in {\mathbb T}^d, \ u_s^{0,t,x_0,\mu}(x)  = V(s,X_s^{0,t,x_0,\mu},x,\mu_s^{0,t,x_0,\mu}) \bigr\} \Bigr) = 1,
\end{equation*} 
and then,
\begin{equation*} 
{\mathbb P}^0 \Bigl( \bigl\{ \forall x \in {\mathbb T}^d, \  \nabla_x u_s^{0,t,x_0,\mu}(x)  = \nabla_x V(s,X_s^{0,t,x_0,\mu},x,\mu_s^{0,t,x_0,\mu}) \bigr\} \Bigr) = 1,
\end{equation*}
which shows the existence of a Markov feedback function for the minor player. 
Moreover, under the assumption \eqref{eq:prop:characterization:1}, the partial derivative $\nabla_x V$ appearing in the above event is Lipschitz continuous in the measure argument, uniformly with respect to the other parameters.  
This supplies us with the following two feedback functions:
\begin{equation} 
\label{eq:alpha0:alpha:equilibria}
\begin{split}
&\alpha^0 : [0,T] \times {\mathbb R}^d \times {\mathcal P}({\mathbb T}^d) \ni (x_0,\mu) 
\mapsto 
- \nabla_p H^0\bigl(x_0,\psi^0(t,x_0,\mu) \bigr), 
\\
&\alpha : [0,T] \times 
{\mathbb R}^d \times {\mathbb T}^d  \times {\mathcal P}({\mathbb T}^d) \ni (x_0,x,\mu) 
\mapsto - \nabla_p H \bigl( x, \nabla_x V(t,x_0,x,\mu) \bigr).
\end{split} 
\end{equation} 
The next point is to show that, for any fixed  initial condition $(x_0,\mu) \in {\mathbb R}^d \times {\mathcal P}({\mathbb T}^d)$ (choose, to simplify, 
$0$ as initial time),  the pair $(\alpha^0,\alpha)$ 
is admissible in the sense of Definition 
\ref{def:admissibility:pair}. 
By 
transferring 
 the law of $({\boldsymbol X}^{0,0,x_0,\mu},{\boldsymbol \mu}^{0,x_0,\mu},{\boldsymbol B}^{0})$
 onto the canonical space $\Omega^0_{\textrm{\rm canon}}$, we get the existence of a probability 
 measure ${\mathbb P}^0_{(\alpha^0,\alpha)}$ under 
 which the system 
  \eqref{eq:Markov:state:equation} (with $(\alpha^0,\alpha)$ given by 
  \eqref{eq:alpha0:alpha:equilibria})
  is satisfied. 
Denoting here by $({\boldsymbol X}^0,{\boldsymbol \mu},{\boldsymbol B}^0)$ 
   the canonical process 
   on  $\Omega^0_{\textrm{\rm canon}}$, 
   the  
   pair $({\boldsymbol X}^0,{\boldsymbol \mu})$
is adapted to the completion under ${\mathbb P}^0_{(\alpha^0,\alpha)}$ of the filtration 
generated by ${\boldsymbol B}^0$.   Hence, under the completion of ${\mathbb P}^0_{(\alpha^0,\alpha)}$ (still denoted ${\mathbb P}^0_{(\alpha^0,\alpha)}$), one can solve the backward equation in 
   \eqref{eq:major:FB:1}
 and then denote the solution by $({\boldsymbol Y}^0,{\boldsymbol Z}^0)=(Y_t^0,Z_t^0)_{0 \le t \le T}$, 
 see 
\cite{Kobylanski} for standard solvability results for quadratic BSDEs. Then, 
 observing that 
 the 
law of the solution to  
 the backward equation in 
   \eqref{eq:major:FB:1}
is uniquely determined by the law of the 
input $({\boldsymbol X}^0,{\boldsymbol \mu},{\boldsymbol B}^0)$, 
 one necessarily 
 has, for $\textrm{\rm Leb} \times {\mathbb P}_{(\alpha^0,\alpha)}$-almost every $(t,\omega^0) \in [0,T] \times  \Omega^0_{\textrm{\rm canon}}$,
 $Z_t^0 = \psi^0(t,X_t^0,\mu_t)$, which proves that $({\boldsymbol X}^0,{\boldsymbol Y}^0,{\boldsymbol Z}^0)$
 solves the forward-backward system 
    \eqref{eq:major:FB:1}
    under ${\mathbb P}^0_{(\alpha^0,\alpha)}$.
  The BMO condition   \eqref{eq:Markov:state:equation:BMO}
is then established by means of standard results for backward SDEs, see e.g. 
\cite{Tevzadze}.

Uniqueness of this probability measure is a consequence of Lemma
\ref{lem:weak:uniqueness:forward:equation},
which says that the law of the solution to 
  \eqref{eq:Markov:state:equation} (with $(\alpha^0,\alpha)$ given by 
  \eqref{eq:alpha0:alpha:equilibria}) is uniquely determined. 
This shows that the pair $(\alpha^0,\alpha)$ fits the requirements of  Definition 
\ref{def:admissibility:pair}. 
\vspace{5pt}
\\
\textit{Second Step.} 
We now prove that the pair 
$(\alpha^0,\alpha)$ in 
\eqref{eq:alpha0:alpha:equilibria}
defines an equilibrium in the sense of Definition 
\ref{def:Nash}.
As in the first step, we do so for a fixed initial condition 
$(x_0,\mu) \in {\mathbb R}^d \times {\mathcal P}({\mathbb T}^d)$
at time $0$. 

We first check item 1 in Definition 
\ref{def:Nash}. For ${\beta}$ as in item 1
and for
$({\mathbb P}_{\omega^0,\beta})_{\omega^0 \in \Omega^0_{\textrm{\rm canon}}}$
as in Lemma \ref{lem:alpha:weak:solution}, 
and with 
$({\boldsymbol X}^0,{\boldsymbol \mu},{\boldsymbol B}^0)=(X_t^0,\mu_t,B_t^0)_{0 \le t \le T}$ and 
$({\boldsymbol X},{\boldsymbol B})=(X_t,B_t)_{0 \le t \le T}$
denoting the canonical processes on 
$\Omega^0_{\textrm{\rm canon}}$ and $\Omega_{\textrm{\rm canon}}$, we know that, for any $\omega^0 \in \Omega^0$, 
the following equation holds true under 
${\mathbb P}^0_{\omega^0,\beta}$: 
\begin{equation*} 
\ud X_t = \beta(t,X_t^0,X_t,\mu_t) \ud t + \ud B_t, \quad t \in [0,T],
\end{equation*}
with ${\mathbb P}^0_{\omega^0,\beta} \circ X_0^{-1} =\mu$. 
We then introduce, on $\Omega^0_{\textrm{\rm canon}} \times \Omega_{\textrm{\rm canon}}$, 
the probability measure
${\mathbb P}_{(\alpha^0,\alpha)}^0 \otimes 
  {\mathbb P}_{\cdot,\beta}$ defined by 
\begin{equation*} 
{\mathbb P}_{(\alpha^0,\alpha)}^0 \otimes 
  {\mathbb P}_{\cdot,\beta}
  \bigl( A^0 \times A \bigr) 
  = 
  \int_{\Omega^0_{\textrm{\rm canon}}}
  {\mathbf 1}_{A^0}(\omega^0)
  {\mathbb P}_{\omega^0,\beta}(A) 
  \ud 
  {\mathbb P}^0_{(\alpha^0,\alpha)}(\omega^0), 
  \quad A^0 \in {\mathcal B}(\Omega^0_{\textrm{\rm canon}}), \ A \in {\mathcal B}(\Omega_{\textrm{\rm canon}}). 
\end{equation*}

We then expand $(u_t(X_t))_{0 \le t \le T}$ by using It\^o-Wentzell formula (which does not raise any difficulty in this setting because 
$(u_t)_{0 \le t \le T}$ is independent of $(B_t)_{0 \leq t \le T}$, see\footnote{\label{foo:1}The proof is as follows. 
For a time $h>0$, we write $u_{t+h}(X_{t+h}) - u_t(X_t) = u_{t+h}(X_{t+h}) - u_{t+h}(X_t) + u_{t+h}(X_{t}) - u_t(X_t)$. And then, 
by standard It\^o's formula, we get on the one hand, for any fixed $\omega^0 \in \Omega^0$, under 
${\mathbb P}_{\omega^0,\beta}$, 
\begin{equation*} 
\label{eq:expansion:ito:wentzell:1}
\begin{split} 
u_{t+h}(X_{t+h}) - u_{t+h}(X_{t})
&= 
\tfrac12 \int_{t}^{t+h} 
\Delta_{x} u_{t+h}(X_s) \ud s + 
\int_t^{t+h} 
\nabla_xu_{t+h}(X_s)\cdot \Bigl( \beta(s,X_s^0,X_s,\mu_s) \ud s 
+ 
  \ud B_s\Bigr), \quad t \in [0,T-h].  
\end{split}
\end{equation*}
On the other hand, using the fact that $X_t$ is independent of $(B^0_s - B^0_t)_{t \leq s \leq t+h}$, we 
can formally replace $x$ by $X_t$ in the backward SPDE of 
$(u_s)_{t \leq s \leq t+h}$ and then get, with probability 1 under 
${\mathbb P}_{(\alpha^0,\alpha)}^0 \otimes 
  {\mathbb P}_{\cdot,\beta}$, 
  \begin{equation*} 
\label{eq:expansion:ito:wentzell:2}
  \begin{split}
 u_{t+h}(X_t) - u_t(X_t) &= 
 - \tfrac12  \int_t^{t+h}  \Delta_x u_s(X_t) \ud s +
 \int_t^{t+h} 
  H\bigl(X_t,\nabla_x u_s(X_t) \bigr)  \ud s
   -
 \int_t^{t+h}    f_s(X_s^0,X_t,\mu_s)   \ud s
+ \sigma_0 \int_t^{t+h} v_s^0(X_t) \cdot \ud B_s^0. 
 \end{split} 
 \end{equation*}  
We then sum the last two displays
over a mesh of stepsize $h$. 
We handle the Lebesgue integrals by using the fact that 
$\sup_{0 \le t \le T} \| u_t \|_{\mathfrc{s}} \in L^\infty(\Omega^0,{\mathbb P}^0)$. In fact, 
the most difficult term to handle is the stochastic integral. 
We have, for any deterministic exponent 
$\eta \in (0,1)$, with $\eta < \lfloor \curss \rfloor - d/2-1$,
$${\mathbb E}^0 \int_t^{t+h} \vert v^0_s(X_t) - v^0_s(X_s) 
\vert^2 \ud 
s \leq 
{\mathbb E}^0 \biggl[ 
\Bigl( 
1 \wedge 
\sup_{\vert s-r \vert \le h} 
\vert X_s - X_r \vert^{2 \eta}
\Bigr) 
\int_t^{t+h} \| v^0_s 
\|^2_{\eta} \ud s\biggr].$$
Using the fact that 
$\beta$ satisfies 
\eqref{eq:Markov:state:equation:BMO} 
together with the bound
${\mathbb E}^0 \int_{[0,T]} \| v_s^0 \|_{\lfloor \mathfrc{s} \rfloor-d/2-1}^2 \ud s < \infty$, we can easily 
sum the above right-hand side 
over a mesh of stepsize $h$
and 
let 
$h$ tend to $0$. 
We obtain 
\eqref{eq:verif:argument:ito:wentzell}.}). We obtain, under 
${\mathbb P}_{(\alpha^0,\alpha)}^0 \otimes 
  {\mathbb P}_{\cdot,\beta}$, for all $t \in [0,T]$, 
\begin{align}
\label{eq:verif:argument:ito:wentzell} 
&u_T(X_T) - u_t(X_t) + \int_t^T \bigl[  f_s(X_s^0,X_s,\mu_s) + L(X_s,\beta_s) \bigr] \ud s  
\\
&\geq \int_t^T \Bigl[ H\bigl(X_s,\nabla_x u_s(X_s)\bigr) +   \beta_s   \cdot \nabla_x u_s(X_s) + L(X_s,\beta_s) \Bigr]
\ud s
  + \int_t^T \nabla_x u_s(X_s) \cdot   \ud B_s
+ \int_t^T v_s^0(X_s) \cdot   \ud B_s^0,
\nonumber
\end{align} 
with the short-hand notation $(\beta_s  := \beta(s,X_s^0,X_s,\mu_s))_{0 \leq s \leq T}$.

By construction, see
\eqref{eq:Hamiltonians}, the integrand in the $\ud s$ integral is non-negative. Taking expectation 
(which is licit thanks to 
the properties of 
${\boldsymbol u}$ stated in 
Definition 
\ref{def:forward-backward=MFG:solution}), we deduce that 
$J_{\rm w}({ \beta} ; {\mathbb P}^0_{(\alpha^0,\alpha)} ) \geq 
{\mathbb E}_{(\alpha^0,\alpha)}^0 \otimes {\mathbb E}_{\cdot,\beta} [u_0(X_0)]=
{\mathbb E}_{(\alpha^0,\alpha)}^0 [(u_0,\mu_0)]$, 
with the inequality becoming an equality 
when $\beta$ is equal to 
$\alpha$ (see 
\eqref{eq:alpha0:alpha:equilibria} for the definition of the latter), i.e., 
$J_{\rm w}({ \beta} ; {\mathbb P}^0_{(\alpha^0,\alpha)} ) \geq 
J_{\rm w}({ \alpha} ; {\mathbb P}^0_{(\alpha^0,\alpha)} ) $. 

Next, we prove item 2 in Definition 
\ref{def:Nash}. For this, we consider 
a new Markov feedback function $\beta^0$ such that the pair 
$(\beta^0,\alpha)$ is admissible. Then, we rewrite
the forward-backward system 
   \eqref{eq:major:FB:1} solved by 
   $({\boldsymbol X}^0,{\boldsymbol \mu},{\boldsymbol B}^0,{\boldsymbol Y}^0,{\boldsymbol Z}^0)$   
   on $(\Omega^0_{\textrm{\rm canon}},{\mathbb P}^0_{(\alpha^0,\alpha)})$ (see the first step) in the form
\begin{equation}
\label{eq:proof:!:Nash:auxiliary:bsde:1}
\begin{split}
&\ud X_t^0 = \beta^0(t,X_t^0,\mu_t) dt + \sigma_0 \ud \tilde B_t^0
\\
&\ud Y_t^0 =  - \Bigl( f_t^0(X_t^0,\mu_t) + L^0\bigl(X_t^0,- \nabla_p H^0\bigl( X_t^0, Z_t^0\bigr) \bigr)
\Bigr) \ud t 
 +
  \Bigl( \nabla_p H^0 (X_t^0,Z_t^0)  + \beta^0(t,X_t^0,\mu_t) \Bigr) \cdot Z_t^0 \ud t
  \\
&\hspace{20pt} +
\sigma_0 Z_t^0 \cdot \ud \tilde{B}_t^0,
\end{split}
\end{equation} 
where 
\begin{equation}
\label{eq:proof:!:Nash:auxiliary:bsde:1:2} 
\tilde B_t^0 := B_t^0 - \sigma_0^{-1} \Bigl( \nabla_p H^0 (X_t^0,Z_t^0)  + \beta^0(t,X_t^0,\mu_t) \Bigr) \ud t, \quad t \in [0,T].
\end{equation}
We justify in the fourth step below that
we can apply Girsanov theorem 
to prove that 
$\tilde{\boldsymbol B}^0$ 
 is a Brownian motion under the probability measure 
\begin{equation}
\label{eq:proof:!:Nash:auxiliary:bsde:1:3} 
\frac{\ud \tilde{{\mathbb P}}^0_{(\alpha^0,\alpha)}}{\ud {\mathbb P}_{(\alpha^0,\alpha)}^0} 
= 
{\mathscr E}_T \biggl(  \sigma_0^{-1} \int_0^\cdot \Bigl( \nabla_p H^0 (X_t^0,Z_t^0)  + \beta^0(t,X_t^0,\mu_t) \Bigr) \cdot \ud B_t^0 
 \biggr). 
 \end{equation}  
And then, we 
observe that the $\ud t$ term in the 
backward equation appearing in
\eqref{eq:proof:!:Nash:auxiliary:bsde:1}
 is greater than 
$- [ f_t^0(X_t^0,\mu_t) + L^0(X_t^0,\beta^0(t,X_t^0,\mu_t))]$, which 
follows from
the convexity of $L^0$: 
\begin{equation*}
\begin{split} 
& L^0\bigl(X_t^0,\beta^0(t,X_t^0,\mu_t)\bigr)
\\
&\geq L^0\bigl(X_t^0,- \nabla_p H^0\bigl( X_t^0, Z_t^0\bigr) \bigr)
 +
  \Bigl( \nabla_p H^0 (X_t^0,Z_t^0)  + \beta^0(t,X_t^0,\mu_t) \Bigr) \cdot \nabla_\alpha L^0\bigl(X_t^0, - \nabla_p H^0( X_t^0, Z_t^0)\bigr)
  \\
  &= 
  L^0\bigl(X_t^0,- \nabla_p H^0\bigl( X_t^0, Z_t^0\bigr) \bigr)
 -
  \Bigl( \nabla_p H^0 (X_t^0,Z_t^0)  + \beta^0(t,X_t^0,\mu_t) \Bigr) \cdot  
  Z_t^0.
\end{split}
\end{equation*} 
Observe now that
the stochastic integral 
in 
\eqref{eq:proof:!:Nash:auxiliary:bsde:1}
has zero expectation 
under $\tilde{\mathbb P}^0_{(\alpha^0,\alpha)}$
because 
${\mathbb E}^0[ 
\vert \int_{[0,T]} \vert Z_t^0 \vert^2 \ud t \vert^p ] < \infty$
for any $p \geq 1$ and the change of measure 
in \eqref{eq:proof:!:Nash:auxiliary:bsde:1:3} has a finite exponential moment, see \cite[Theorem 2.2]{Kazamaki} and the fourth step below for the proof of the related BMO property). This suffices to say that 
\begin{equation*}
Y_0^0 \leq \tilde{\mathbb E}^0_{(\alpha^0,\alpha)}  
\biggl[ g^0(X_T^0,\mu_T) + \int_0^T\Bigl(  f^0(X_t^0,\mu_t) + L^0\bigl(X_t^0,\beta^0(t,X_t^0,\mu_t)\bigr) \Bigr) \ud t \biggr].  
\end{equation*} 
Observing that the law of $({\boldsymbol X}^0,{\boldsymbol \mu})$ under $\tilde{\mathbb P}^0_{(\alpha^0,\alpha)}$
is the same as the law of $({\boldsymbol X}^0,{\boldsymbol \mu})$ under ${\mathbb P}^0_{(\beta^0,\alpha)}$ (because 
$(\beta^0,\alpha)$ is admissible, see Definition 
\ref{def:admissibility:pair}), 
this shows that 
$J_{\rm w}^0({\alpha}^0, \alpha) = Y_0^0  
\leq 
J_{\rm w}^0({ \beta}^0,   {\alpha})$, as required. 
\vspace{5pt} 
\\
\textit{Third Step.} 
The last step is to prove uniqueness. 
To do so, we take an admissible pair $(\alpha^0,\alpha)$ 
(different from the one constructed right above, but denoted in the same manner) 
satisfying the requirements of Definition 
\ref{def:Nash}. By item (ii) in the statement
and with   $({\boldsymbol X}^0,{\boldsymbol \mu},{\boldsymbol B}^0)$  denoting again the canonical process   on $\Omega^0_{\textrm{\rm canon}}$, 
we know that, under the completion of 
${\mathbb P}_{(\alpha^0,\alpha)}^0$ (still denoted ${\mathbb P}_{(\alpha^0,\alpha)}^0$), the process 
 $({\boldsymbol X}^0,{\boldsymbol \mu})$
 is adapted to the augmentation of the filtration 
 generated by ${\boldsymbol B}^0$. This makes it possible to solve the backward stochastic HJB equation
within the same class as in Definition 
\ref{def:forward-backward=MFG:solution}
(solvability of this equation is explained in 
Subsection 
\ref{subse:HJB:system:analysis:apriori}):
\begin{equation}
\label{eq:verif:backward:HJB} 
\begin{split} 
&\ud_t u_t(x) =\Bigl(  -  \tfrac12 \Delta_x u_t(x) + H\bigl(x,\nabla_x u_t(x) \bigr)  - f_t(X_t^0,x,\mu_t)  \Bigr) \ud t
+ \sigma_0 v_t^0(x) \cdot \ud B_t^0, \quad (t,x) \in [0,T] \times {\mathbb T}^d, 
\\
&u_T(x) = g(X_T^0,x,\mu_T).
\end{split} 
\end{equation} 
Then, by expanding the duality product $((u_t,\mu_t))_{0 \le t \le T}$ (or by expanding 
$(u_t(X_t))_{0 \le t \le T}$ 
under ${\mathbb P}_{(\alpha^0,\alpha)}^0 \otimes 
  {\mathbb P}_{\cdot,\alpha}$), we can reproduce the verification argument used in the second step to show that necessarily, 
for
$\textrm{\rm Leb} \times 
{\mathbb P}^0_{(\alpha^0,\alpha)}$-almost every $(t,\omega^0) \in [0,T] \times \Omega^0_{\textrm{\rm canon}}$, 
\begin{equation*} 
\alpha(t,X_t^0,x,\mu_t) = - \nabla_p H \bigl( X_t^0,  \nabla_x u_t(x) \bigr), \quad x \in {\mathbb T}^d.
\end{equation*} 
Combined with the second equation in 
\eqref{eq:Markov:state:equation}, this suffices to show that the pair $({\boldsymbol X}^0,{\boldsymbol \mu})$
together with $({\boldsymbol u},{\boldsymbol v}^0)$ introduced in 
\eqref{eq:verif:backward:HJB} solves 
\eqref{eq:minor:FB:2} under ${\mathbb P}_{(\alpha^0,\alpha)}^0$. 

We now handle the major player. As above, we can solve,
under 
${\mathbb P}_{(\alpha^0,\alpha)}^0$,
the BSDE
\begin{equation}
\label{eq:proof:!:Nash:auxiliary:bsde:2} 
\begin{split}
&\ud Y_t^0 =  - \Bigl( f_t^0(X_t^0,\mu_t) + L^0\bigl(X_t^0,\alpha^0(t, X_t^0, \mu_t ) \bigr)
\Bigr) \ud t + 
\sigma_0 Z_t^0 \cdot \ud B_t^0, \quad t \in [0,T], 
\\
&Y_T^0 = g^0(X_T^0,\mu_T).
\end{split}
\end{equation} 
Under 
the standing assumption 
(see in particular 
\eqref{eq:prop:characterization:2} in item (i) of the statement)
and by Lemma 
\ref{lem:weak:uniqueness:forward:equation}, the state equation
\eqref{eq:Markov:state:equation}
has at most one weak solution, for any starting point in $(t,x_0,\mu) \in [0,T] \times {\mathbb R}^d \times {\mathcal P}({\mathbb T}^d)$. 
By item (ii) in the statement, we know that 
equation \eqref{eq:Markov:state:equation}
is already assumed to have at least one strong solution, and this for any 
starting point in $(t,x_0,\mu) \in [0,T] \times {\mathbb R}^d \times {\mathcal P}({\mathbb T}^d)$. 
By a straightforward modification of the Yamada-Watanabe theorem, we deduce that 
the state equation
\eqref{eq:Markov:state:equation}
is uniquely strongly solvable for any starting point in $(t,x_0,\mu) \in [0,T] \times {\mathbb R}^d \times {\mathcal P}({\mathbb T}^d)$. And then, 
by 
adapting 
the proof of 
\cite[Proposition 1.31]{CarmonaDelarue_book_II} (which is itself 
inspired from the remark below \cite[Theorem 6.2.2]{StroockVaradhan}), 
strong uniqueness implies that the 
solution to 
\eqref{eq:Markov:state:equation}
forms a strong Markov process. Then, 
following once again 
\cite[Proposition 3.2 \& Theorem 3.4]{ImkellerReveillacRichter},  we can find a new 
Borel function 
$\varphi^0 : [0,T] \times {\mathbb R}^d \times {\mathcal P}({\mathbb T}^d) \rightarrow {\mathbb R}^d \otimes {\mathbb R}^d$  such that, for any $s \in [0,T]$,  
\begin{equation}
 \label{eq:proof:!:Nash:auxiliary:bsde:2:4} 
\begin{split}
&{\mathbb P}^0 \Bigl( \bigl\{ Z_s^{0}  = \varphi^0(s,X_s^{0},\mu_s) \bigr\} \Bigr) = 1.
\end{split}
\end{equation} 

Recalling the assumption 
\eqref{eq:Markov:state:equation:BMO}, it is quite easy to prove that 
$\sup_{0 \leq t \leq T} \vert Y_t^0 \vert  \in L^\infty(\Omega^0,{\mathbb P}^0)$. And then,
expanding $(\vert Y_t^0\vert^2)_{0 \le t \le T}$ by means of It\^o's formula, 
we obtain $\| \int_0^{\cdot} Z_t^0 \ud B_t^0 \|_{\textrm{\rm BMO}} < \infty$, and 
then 
$\| \int_0^{\cdot} \nabla_p H^0(X_t^0,Z_t^0) \cdot  \ud B_t^0 \|_{\textrm{\rm BMO}} < \infty$.

We then define the tilted noise 
\begin{equation}
\label{eq:proof:!:Nash:auxiliary:bsde:2:1}  
\tilde{B}_t^0 := B_t^0 + \sigma_0^{-1} \Bigl( \nabla_p H^0 (X_t^0,Z_t^0)  + \alpha^0(t,X_t^0,\mu_t) \Bigr) \ud t,
\quad t \in [0,T]. 
\end{equation} 
We justify in the fourth step below
that
we can apply Girsanov theorem 
to prove that 
$\tilde{\boldsymbol B}^0$ 
 is a Brownian motion under the probability measure 
\begin{equation}
\label{eq:proof:!:Nash:auxiliary:bsde:2:2}  
\frac{\ud \tilde{{\mathbb P}}_{(\alpha^0,\alpha)}^0}{\ud {\mathbb P}_{(\alpha^0,\alpha)}^0} 
= 
{\mathscr E}_T \biggl( - \sigma_0^{-1} \int_0^\cdot \Bigl( \nabla_p H^0 (X_t^0,Z_t^0)  + \alpha^0(t,X_t^0,\mu_t) \Bigr) \cdot \ud B_t^0 
 \biggr). 
 \end{equation}  
And we rewrite the BSDE in \eqref{eq:proof:!:Nash:auxiliary:bsde:2} as 
\begin{equation}
\label{eq:proof:!:Nash:auxiliary:bsde:2:2:11}  
\begin{split}
&\ud Y_t^0 =  - \Bigl( f_t^0(X_t^0,\mu_t) + L^0\bigl(X_t^0,\alpha^0(t, X_t^0, \mu_t ) \bigr) 
\Bigr) \ud t 
-
\Bigl( \nabla_p H^0 (X_t^0,Z_t^0)  + \alpha^0(t,X_t^0,\mu_t) \Bigr)
\cdot 
Z_t^0
 \ud t
\\
&\hspace{30pt} + 
\sigma_0 Z_t^0 \cdot \ud \tilde{B}_t^0, \quad t \in [0,T]. 
\end{split}
\end{equation} 
As before, we observe from the  (strict) convexity property of $L^0$ that 
\begin{equation*} 
\begin{split}
&L^0\bigl(X_t^0,\alpha^0(t, X_t^0, \mu_t ) \bigr) 
+
\Bigl( \nabla_p H^0 (X_t^0,Z_t^0)  + \alpha^0(t,X_t^0,\mu_t) \Bigr)
\cdot 
Z_t^0
 \geq 
L^0\bigl(X_t^0, - \nabla_p H^0 (X_t^0,Z_t^0) \bigr),
\end{split} 
\end{equation*}
with equality if and only if $\nabla_p H^0 (X_t^0,Z_t^0)  + \alpha^0(t,X_t^0,\mu_t) = 0$. 
We then have (the BMO condition, which can be transferred from the original probability measure 
${\mathbb P}_{(\alpha^0,\alpha)}^0$
to the new probability measure 
$\tilde{\mathbb P}_{(\alpha^0,\alpha)}^0$
--see Theorem \cite[Theorems 2.3 and 3.6]{Kazamaki}--,
 makes it possible to take expectation in 
\eqref{eq:proof:!:Nash:auxiliary:bsde:2:2:11}) 
\begin{equation}
\label{eq:J0:uniqueness:cost:3rd:part} 
J^0_{\textrm{\rm w}}({\alpha}^0,\alpha)
= Y_0^0 
\geq 
\tilde{\mathbb E}^0_{(\alpha^0,\alpha)}
\biggl[ g^0(X_T^0,\mu_T) + \int_0^T \Bigl( f_t(X_t^0,\mu_t) + L^0\bigl(X_t^0,-\nabla_p H^0(X_t^0,Z_t^0) \bigr) \Bigr)
\ud t \biggr].
\end{equation}  
Under $\tilde{\mathbb P}_{(\alpha^0,\alpha)}^0$,
 $({\boldsymbol X}^0,{\boldsymbol \mu},\tilde{\boldsymbol B}^0)=(X_t^0,\mu_t,\tilde{B}^0_t)_{0 \le t \le T}$ solves the system 
\begin{equation*}
\begin{split}
&\ud X_t^0 = -\nabla_p H^0\bigl( X_t^0, 
 \varphi^0(s,X_s^{0},\mu_s) 
 \bigr) \ud t + \ud \tilde{B}_t^0, 
\\
&\partial_t \mu_t = \tfrac12 \Delta_{x_0} \mu_t - \textrm{\rm div}_{x_0} \bigl( \alpha(t,X_t^0,\cdot,\mu_t) \mu_t \bigr), \quad t \in [0,T].
\end{split}
\end{equation*}
Here, we observe that 
$
\| \int_0^{\cdot} \nabla_p H^0(X_t^0,-\nabla_p H^0\bigl( X_t^0, 
 \varphi^0(s,X_s^{0},\mu_s) ) \cdot  \ud \tilde B_t^0 \|_{\textrm{\rm BMO}} 
=
\| \int_0^{\cdot} \nabla_p H^0(X_t^0,Z_t^0) \cdot  \ud \tilde B_t^0 \|_{\textrm{\rm BMO}} < \infty$
(again, this follows from 
Theorem \cite[Theorem 3.6]{Kazamaki}). 
This provides one weak solution to the state equation
\eqref{eq:Markov:state:equation}, driven by 
$\beta^0(t,x_0,\mu) : = 
-  \nabla_p H^0(x_0,\varphi^0(t,x_0,\mu))$, 
that satisfies the BMO condition 
\eqref{eq:Markov:state:equation:BMO}.
By Lemma
\ref{lem:weak:uniqueness:forward:equation}, this 
weak solution is necessarily unique and, therefore, the pair 
$(\beta^0,\alpha)$ is
admissible in the sense of Definition 
\ref{def:admissibility:pair}. Then, the right-hand side  in 
\eqref{eq:J0:uniqueness:cost:3rd:part}
coincides 
with $J^0_{\textrm{\rm w}}(\beta^0,\alpha)$, 

 By item 2 in 
 Definition \ref{def:Nash}, 
the inequality in 
\eqref{eq:J0:uniqueness:cost:3rd:part}
must become an equality and then, for almost every $(t,\omega^0) \in [0,T] \times \Omega^0$ under the measure 
$\textrm{\rm Leb} \times {\mathbb P}^0$, it holds 
$\nabla_p H^0 (X_t^0,Z_t^0)  + \alpha^0(t,X_t^0,\mu_t) = 0$, which proves that 
$({\boldsymbol X}^0,{\boldsymbol \mu})$ coincides with the solution
of 
\eqref{eq:major:FB:1}--\eqref{eq:minor:FB:2}.
\vspace{5pt}
\\
\noindent \textit{Fourth Step.} 
We now justify the application of Girsanov theorem in 
\eqref{eq:proof:!:Nash:auxiliary:bsde:1},
\eqref{eq:proof:!:Nash:auxiliary:bsde:1:2}
and 
\eqref{eq:proof:!:Nash:auxiliary:bsde:1:3}. 
We start from 
\eqref{eq:Markov:state:equation}
for an admissible pair $(\alpha^0,\alpha)$ as in the statement of Definition 
\ref{def:admissibility:pair},  with 
$\alpha$ satisfying 
\eqref{prop:lem:weak:uniqueness:forward:equation}
(which is the case in 
\eqref{eq:alpha0:alpha:equilibria}
because of 
\eqref{eq:prop:characterization:1}). 
Then,  Lemma 
\ref{lem:weak:uniqueness:forward:equation} says that the BMO condition 
\eqref{eq:Markov:state:equation:BMO:0} is satisfied under 
$\bar{\mathbb P}^0$. This observation applies here 
to both 
$\alpha^0$ 
as in \eqref{eq:alpha0:alpha:equilibria} and $\alpha^0\equiv \beta^0$ (with 
$\beta^0$ as in \eqref{eq:proof:!:Nash:auxiliary:bsde:1}). 
In particular, choosing now 
$\alpha^0$ as in 
 \eqref{eq:alpha0:alpha:equilibria}, 
the process $(\int_{[0,t]} [ \beta^0(s,X_s^0,\mu_s) - \alpha^0(s,X_s^0,\mu_s)] \cdot \ud B_s^0)_{0 \le t \le T}$ is 
BMO under $\bar{\mathbb P}^0$. Let now 
\begin{equation*} 
\check{\mathbb P}^0 : = 
{\mathscr E}_T\biggl( \int_0^\cdot 
\alpha^0(s,X_s^0,\mu_s) \cdot \ud B_s^0
\biggr) \cdot \overline{\mathbb P}^0. 
\end{equation*} 
By \cite[Theorem 3.6]{Kazamaki}, we know that 
$(\int_{[0,t]} [\beta^0(s,X_s^0,\mu_s) - \alpha^0(s,X_s^0,\mu_s)] \cdot \ud \check{B}_s^0)_{0 \le t \le T}$ is BMO under 
$\check{\mathbb P}^0$, where 
$\check{\boldsymbol B}^0=(\check{ B}^0_t := B_t^0 - \int_0^t \alpha^0(s,X_s^0,\mu_s)  \ud s)_{0 \le t \le T}$
is a Brownian motion under $\check{\mathbb P}^0$. 
Since $\check{\mathbb P}^0 \circ ({\boldsymbol X}^0,{\boldsymbol \mu},\check{\boldsymbol B}^0)^{-1} 
=  {\mathbb P}^0_{(\alpha^0,\alpha)}$, 
this shows that $(\int_{[0,t]} [\beta^0(s,X_s^0,\mu_s) - \alpha^0(s,X_s^0,\mu_s)] \cdot \ud B_s^0)_{0 \le t \le T}$ is BMO under 
${\mathbb P}_{(\alpha^0,\alpha)}^0$, which suffices to apply Girsanov theorem 
in 
\eqref{eq:proof:!:Nash:auxiliary:bsde:1},
\eqref{eq:proof:!:Nash:auxiliary:bsde:1:2}
and 
\eqref{eq:proof:!:Nash:auxiliary:bsde:1:3}. 

We now proceed in a similar manner to 
justify the Girsanov transformation in 
\eqref{eq:proof:!:Nash:auxiliary:bsde:2:1}  and 
\eqref{eq:proof:!:Nash:auxiliary:bsde:2:2}. 
In fact, it suffices to apply the same argument as above but with $\alpha^0$ 
as in the third step, see \eqref{eq:proof:!:Nash:auxiliary:bsde:2},
and with $\beta^0(t,x_0,\mu) = 
-  \nabla_p H^0(x_0,\varphi^0(t,x_0,\mu))$, see \eqref{eq:proof:!:Nash:auxiliary:bsde:2:4}. 
%
%
\end{proof}

\subsection{System of master equations}

Following the analysis performed in 
\cite{Cardaliaguet:Cirant:Porretta:JEMS}, we
associate with the Major/Minor MFG 
a system of master equations. Formally, it reads as a pair of two equations for the value $V^0$ to the major player 
and the value $V$ to the minor player. 

The equation for $V^0$ reads
\begin{equation} 
\label{eq:major:1}
\begin{split} 
&\partial_t V^0(t,x_0,\mu) + 
\tfrac12 \sigma_0^2 \Delta_{x_0} 
V^0(t,x_0,\mu) - H^0 \bigl( x_0, 
  \nabla_{x_0} V^0(t,x_0,\mu) \bigr)  + f_t^0(x_0,\mu)
\\
&\hspace{15pt}+ 
\int_{{\mathbb T}^d}
\Bigl\{ \tfrac12
{\rm div}_y(\partial_{\mu}V^0(t,x_0,\mu,y)) -\partial_\mu V^0(t,x_0,\mu,y) \cdot \nabla_p H \bigl(y, \nabla_x V(t,x_0,y,\mu) \bigr) \Bigr\} \ud \mu(y) =0, 
\\
&V^0(T,x_0,\mu) = g^0(x_0,\mu),
\end{split}
\end{equation}
for $(t,x_0,\mu) \in [0,T] \times {\mathbb R}^d \times {\mathcal P}({\mathbb T}^d)$. 

The equation for $V$ is
\begin{equation} 
\label{eq:minor:1}
\begin{split} 
&\partial_t V(t,x_0,x,\mu) + 
\tfrac{1}2 
\Delta_x 
V(t,x_0,x,\mu) +\tfrac12 {\sigma_0^2} \Delta_{x_0}V(t,x_0,x,\mu)- H \bigl( x_0,x, \nabla_x V(t,x_0,x,\mu) \bigr)+  f_t(x_0,x,\mu)
\\
&\hspace{15pt} - \nabla_p H^0\big( x_0,  \nabla_{x_0} V^0(t,x_0,\mu)\big) \cdot \nabla_{x_0} V(t,x_0,x,\mu) 
\\
&\hspace{15pt} + \int_{{\mathbb T}^d}
\Bigl\{ \tfrac12 {\rm div}_y(\partial_\mu V(t,x_0,x,\mu,y)) -\partial_\mu V(t,x_0,x,\mu,y) \cdot 
\nabla_p H \bigl( y, \nabla_x V(t,x_0,y,\mu)\bigr) \Bigr\} \ud \mu(y) =0, 
\\
&V(T,x_0,x,\mu) = g(x_0,x,\mu),
\end{split}
\end{equation} 
for $(t,x_0,x,\mu) \in [0,T] \times {\mathbb R}^d \times {\mathbb T}^d \times {\mathcal P}({\mathbb T}^d)$.

The following statement clarifies the connection between 
\eqref{eq:major:1}--\eqref{eq:minor:1}
and 
\eqref{eq:major:FB:1}--\eqref{eq:minor:FB:2}.

\begin{proposition}
\label{prop:verif}
Assume that the master equation
\eqref{eq:major:1}--\eqref{eq:minor:1}
has a classical solution $(V^0,V)$ in the sense that 
\begin{enumerate}
\item $(t,x_0,\mu) \mapsto (\partial_t V^0(t,x_0,\mu),  \nabla_{x_0} V^0(t,x_0,\mu), \nabla^2_{x_0} V^0(t,x_0,\mu))$ is continuous on $[0,T] \times {\mathbb R}^d \times {\mathcal P}({\mathbb T}^d)$ (with the latter 
factor being equipped with any distance metricizing weak convergence on ${\mathcal P}({\mathbb T}^d)$, for instance ${\mathbb W}_1$); 
$(t,x_0,\mu,y) \mapsto (\partial_\mu V^0(t,x_0,\mu,y), \nabla_y \partial_\mu V^0(t,x_0,\mu,y))$ is continuous on $[0,T] \times {\mathbb R}^d \times {\mathcal P}({\mathbb T}^d) \times {\mathbb T}^d$; 
\item $(t,x_0,x,\mu) \mapsto (\partial_t V(t,x_0,x,\mu),  \nabla_{x_0} V(t,x_0,x,\mu), \nabla_{x} V(t,x_0,x,\mu),  \nabla^2_{x_0} V(t,x_0,x,\mu),\nabla_{x}^2 V(t,x_0,x,\mu) )$ is continuous on $[0,T] \times {\mathbb R}^d \times {\mathbb T}^d \times {\mathcal P}({\mathbb T}^d)$; 
$(t,x_0,x,\mu,y) \mapsto (\partial_\mu V(t,x_0,x,\mu,y), \nabla_y \partial_\mu V(t,x_0,x,\mu,y))$ is continuous on $[0,T] \times {\mathbb R}^d \times {\mathbb T}^d \times {\mathcal P}({\mathbb T}^d) \times {\mathbb T}^d$. 
\end{enumerate}
Assume also that 
$(t,x_0,\mu) \mapsto \nabla_{x_0} V^0(t,x_0,\mu)$ and $(t,x_0,x,\mu) \mapsto (\nabla_{x_0} V(t,x_0,x,\mu), \nabla_x V(t,x_0,x,\mu))$ are Lipschitz 
continuous with respect to $(x_0,\mu)$ and $(x_0,x,\mu)$ respectively 
(using the distance ${\mathbb W}_1$ to handle the argument $\mu$) 
and that the initial condition 
$X^0_0$ (in Subsection 
\ref{subse:1.1})
 is square-integrable.
Then, 
the triplets $(X_t^0,Y_t^0,Z_t^0)_{0 \le t \le T}$ and 
$(\mu_t,u_t,v_t^0)_{0 \le t \le T}$
obtained by solving, 
on a product structure comprising 
two filtered probability spaces 
$(\Omega^0,{\mathcal F}^0,{\mathbb F}^0,{\mathbb P}^0)$
and
$(\Omega,{\mathcal F},{\mathbb F},{{\mathbb P}})$ equipped with two 
Brownian motions $(B_t^0)_{0 \leq t \leq T}$ 
and $(B_t)_{0 \leq t \leq T}$ with values in ${\mathbb R}^d$, the (coupled forward) equations
\begin{equation}
\label{eq:prop:verif:forward}
\begin{split}
&\ud X_t^0 = - \nabla_p H^0 \bigl( X_t^0, \nabla_{x_0} V^0(t,X_t^0,\mu_t) \bigr)  \ud t + \sigma_0 \ud B_t^0, \quad t \in [0,T], 
\\
&\partial_t \mu_t - \tfrac12 \Delta_x \mu_t - \textrm{\rm div}_x \bigl( \nabla_p H (x, \nabla_x V(t,X_t^0,x,\mu_t)) \mu_t  \bigr) =0, \quad (t,x)\in [0,T] \times {\mathbb T}^d,  
\end{split} 
\end{equation} 
and then by letting 
\begin{equation}
\label{eq:prop:verif:backward}
\begin{split} 
&Y_t^0 := V^0(t,X_t^0,\mu_t), \quad Z_t^0 := \nabla_{x_0} V^0(t,X_t^0,\mu_t), \quad t \in [0,T], 
\\
&u_t(x) := V(t,X_t^0,x,\mu_t), \quad v_t^0(x) :=\nabla_{x_0} V(t,X_t^0,x,\mu_t), \quad (t,x) \in [0,T] \times {\mathbb T}^d, 
\end{split}
\end{equation}
are solutions of 
\eqref{eq:major:FB:1}--\eqref{eq:minor:FB:2}. 
\end{proposition}

 As the proof shows (see \eqref{eq:verif:integrability:Y0:v0}), the solutions to \eqref{eq:major:FB:1}--\eqref{eq:minor:FB:2} that we obtain in this manner just supply us with `true' martingales in 
the two equations.
Notice that, to simplify, we do not 
prove that the 
the triplets $(X_t^0,Y_t^0,Z_t^0)_{0 \le t \le T}$ and 
$(\mu_t,u_t,v_t^0)_{0 \le t \le T}$
satisfy all the conditions in Definition 
\ref{def:forward-backward=MFG:solution}, as this would be useless at this stage of the paper. Obviously, this would require further assumptions on $V^0$ and $V$.

Our main statement regarding the solvability of the master equation is 
\begin{theorem}
\label{thm:2}
In addition to Assumption \hyp{A}, assume that 
\begin{enumerate}[i.]
\item  the coefficient $(t,x_0,\mu) \mapsto 
f_t^0(x_0,\mu)$ is H\"older continuous in time, uniformly in 
$(x_0,\mu)$; the coefficient $(t,x_0,x,\mu) \mapsto 
f_t(x_0,x,\mu)$ is H\"older continuous in time, 
uniformly in $(x_0,x,\mu)$;
\item $x_0 \mapsto g^0(x_0,\mu)$ has 
H\"older continuous second-order derivatives,
uniformly in $\mu$; $x_0 \mapsto g(x_0,x,\mu)$
has H\"older continuous second-order derivatives,  
uniformly in $(x,\mu)$. 
 \end{enumerate}
Then, for any $T>0$ and for $\sigma_0 \geq \sigma_0^*(T) \in (0,+\infty)$ 
(with the latter being defined as in the statement of Theorem \ref{thm:1}), 
the system 
\eqref{eq:major:1}--\eqref{eq:minor:1}
admits a unique solution in the class of functions 
$(V^0,V)$ 
that satisfy items (1) and (2) in the statement of Proposition  
\ref{prop:verif}
and such that: (3) 
$(t,x_0,\mu) \mapsto \nabla_{x_0} V^0(t,x_0,\mu)$ and $(t,x_0,x,\mu) \mapsto (\nabla_{x_0} V(t,x_0,x,\mu), \nabla_x V(t,x_0,x,\mu))$ are Lipschitz 
continuous with respect to $(x_0,\mu)$ and $(x_0,x,\mu)$ respectively 
(using the distance ${\mathbb W}_1$ to handle the argument $\mu$); 
(4) $(t,x_0,\mu) \mapsto (V^0(t,x_0,\mu),\nabla_{x_0} V^0(t,x_0,\mu))$ is globally bounded, 
$(t,x_0,\mu) \mapsto( \| V(t,x_0,\cdot,\mu)\|_{\curss}, \| \nabla_{x_0} V(t,x_0,\cdot,\mu)\|_{\cursr})$ is globally
bounded for any $\cursr \in [1,\lfloor \curss \rfloor -(d/2+1)] \setminus {\mathbb N}$.

If  Assumption \hyp{B} is also in force, then existence and uniqueness 
hold true on any interval $[0,T]$, $T>0$, and for any $\sigma_0 \geq \sigma_0^*$
(with the latter being defined as in the statement of Theorem \ref{thm:1}).
\end{theorem}

The proof of Theorem \ref{thm:2} is deferred to Section \ref{se:4}, see Theorem 
\ref{thm:4.10:b} for a refined version.

We now turn to the proof of Proposition \ref{prop:verif}:

\begin{proof}[Proof of Proposition \ref{prop:verif}.]
To establish the solvability of \eqref{eq:prop:verif:forward}, we consider the conditional McKean-Vlasov equation 
\begin{equation*} 
\begin{split}
&\ud X_t^0 = - \nabla_p H^0 \Bigl( X_t^0, \nabla_{x_0} V^0\bigl(t,X_t^0,
{\mathcal L}^0(X_t)\bigr) 
  \Bigr)  \ud t + \sigma_0 \ud B_t^0, 
\\
&\ud X_t = - \nabla_p H 
\Bigl( X_t^0, X_t , \nabla_{x} V\bigl(t,X_t^0,X_t,
{\mathcal L}^0(X_t)\bigr) 
  \Bigr)  \ud t + \ud B_t, \quad t \in [0,T], 
\end{split}
\end{equation*} 
with the same initial conditions $X_0^0$ and $X_0$ as in Subsection \ref{subse:1.1}. 

Because $\nabla_{x_0} V^0$ and $\nabla_x V$ are 
jointly continuous in all the arguments 
and 
Lipschitz continuous 
with respect to $(x_0,\mu)$ and $(x_0,x,\mu)$ respectively 
(Lipschitz continuity in the argument $\mu$ holding true with respect to ${\mathbb W}_1$), the above system has 
a unique solution, see \cite[Proposition 2.8]{CarmonaDelarue_book_II}. 
It satisfies 
\begin{equation}
\label{eq:verif:integrability:X0:X} 
{\mathbb E}^0 \Bigl[ \sup_{0 \le t \le T} \vert X_t^0 \vert^2 \Bigr] < \infty. 
\end{equation} 
(A similar bound holds true for $\sup_{0 \le t \le T} \vert X_t \vert^2$ under 
${\mathbb E}^0 {\mathbb E}$ when ${\boldsymbol X}$ is implicitly regarded as being 
${\mathbb R}^d$-valued, but this bound has little interest since 
${\boldsymbol X}$ is regarded as being 
${\mathbb T}^d$-valued.)
Letting $(\mu_t:={\mathcal L}^0(X_t))_{0 \le t \le T}$, 
this makes it possible to define $(Y_t^0,Z_t^0)_{0 \le t \le T}$ and 
$(u_t,v_t^0)_{0 \le t \le T}$ as in
\eqref{eq:prop:verif:backward}. Combining the integrability condition 
\eqref{eq:verif:integrability:X0:X}
with 
the regularity properties of $\nabla_{x_0} V^0$, $\nabla_{x_0} V$ and $\nabla_x V$, we get 
\begin{equation}
\label{eq:verif:integrability:Y0:v0} 
{\mathbb E}^0 \Bigl[ \sup_{0 \leq t \leq T} \bigl( \vert Z_t^0 \vert^2 +\sup_{x \in {\mathbb T}^d} \vert   \nabla_x V(t,X_t^0,x,\mu_t) \vert^2
+  \sup_{x \in {\mathbb T}^d} \vert v_t^0(x) \vert^2
\bigr) \Bigr] < \infty, 
\end{equation}
The derivation of the forward equation in 
\eqref{eq:minor:FB:2} is straightforward. 
As for the two backward equations
in 
\eqref{eq:major:FB:1}
and
\eqref{eq:minor:FB:2}, they are obtained by combining the chain rule proved in Appendix, see Proposition \ref{prop:ito:formula}, with
the two PDEs 
\eqref{eq:major:1} and \eqref{eq:minor:1}. 
The bound 
\eqref{eq:verif:integrability:Y0:v0} 
 shows that the martingale terms are `true' martingales. 
\end{proof}

\subsection{About Fokker-Planck equations with random coefficients}
\label{subse:proof:FP}
The purpose of this subsection is to prove 
Lemmas 
\ref{lem:alpha:weak:solution}
and
\ref{lem:weak:uniqueness:forward:equation}.

\begin{proof}[Proof of Lemma \ref{lem:alpha:weak:solution}]
We call $\tilde{\mathbb P}$ the (completion of the) probability measure on $\Omega_{\textrm{\rm canon}}$ under which the canonical 
process $({\boldsymbol X},{\boldsymbol B})$ satisfies:
(i) $\tilde{\mathbb P} \circ X_0^{-1} = \mu$; 
(ii) ${\boldsymbol B}$ is an ${\mathbb F}$-Brownian motion starting from $0$; 
(iii) $X_t-X_0=B_t$ for all $t \in [0,T]$. 

For $\omega \in \Omega^0$, 
we then let  
\begin{equation*} 
\tilde{\mathbb P}_{\omega^0} 
:= {\mathscr E}_T \biggl( \int_0^\cdot \alpha\bigl(t,X_t^0(\omega^0),X_t,\mu_t(\omega^0)\bigr) \cdot \ud B_t \biggr) \cdot \tilde {\mathbb P}. 
\end{equation*} 
Since $\alpha$
is bounded, 
$\tilde{\mathbb P}_{\omega^0}$ is a probability measure. 
It is standard to check that, under 
$\tilde{\mathbb P}_{\omega^0}$, the process 
$(\tilde B_t := B_t - \int_0^t 
\alpha(s,X_s^0(\omega^0),X_s,\mu_s(\omega^0)) \ud s)_{0 \le t \le T}$ is an ${\mathbb F}$-Brownian motion starting from $0$, $X_0$
is distributed according to $\mu$ and, for any $t \in [0,T]$, 
\begin{equation}
\label{eq:lem:FP:SDE:X:tildeB} 
\ud X_t = 
\alpha\bigl(t,X_t^0(\omega^0),X_t,\mu_t(\omega^0)\bigr)
\ud t + \ud \tilde B_t, \quad t \in [0,T]. 
\end{equation} 
We then define
\begin{equation*} 
{\mathbb P}_{\omega^0} := \tilde{\mathbb P}_{\omega^0} 
\circ \bigl( {\boldsymbol X},\tilde{\boldsymbol B} \bigr)^{-1}.
\end{equation*}
Using Fubini's theorem, 
it is easy to prove that, for any event $C$ of ${\mathcal B}(\Omega_{\textrm{\rm canon}})$, 
the mapping $\omega^0 \mapsto {\mathbb P}_{\omega^0}(C)$ is measurable, which proves the measurability 
of the mapping 
$\Omega_{\textrm{\rm canon}}^0 \ni \omega^0 \mapsto {\mathbb P}_{\omega^0}$.

Using the fact that the SDE 
\eqref{eq:lem:alpha:weak:solution:SDE}
(with a prescribed initial condition) 
is uniquely strongly (and thus weakly) solvable for any $\omega^0 \in \Omega^0_{\textrm{\rm canon}}$ (because $\alpha$ is bounded), we get that 
${\mathbb P}_{\omega^0}$ is unique. 
\end{proof}

\begin{proof}[Proof of Lemma \ref{lem:weak:uniqueness:forward:equation}]

\textit{First Step.} 
With ${\boldsymbol x}^0 : = (x_t^0)_{0 \leq t \leq T}$ an element of ${\mathcal C}([0,T],{\mathbb R}^d)$, we
associate the Fokker-Planck equation 
\begin{equation*} 
\begin{split}
&\partial_t \tilde \mu_t^{{\boldsymbol x}^0} = \tfrac12 \Delta_{x} \tilde \mu_t^{{\boldsymbol x}^0} - \textrm{\rm div}_{x} \Bigl( \alpha\bigl(t,x_t^0,\cdot,\tilde \mu_t^{{\boldsymbol x}^0} \bigr) \tilde \mu_t^{{\boldsymbol x}^0} \Bigr), \quad t \in [0,T],
\end{split}
\end{equation*}
with $\mu_0$ as initial condition (at time $0$).

By 
\cite{Lacker_ECP} and 
Remark 
\ref{rem:!:Fokker-Planck},
we know that, for any  ${\boldsymbol x}^0\in {\mathcal C}([0,T],{\mathbb R}^d)$, the Fokker-Planck equation right above admits a unique 
solution in the 
same weak sense as in 
Remark
\ref{rem:weak:form:FP}, which can be obtained by iterating the mapping, denoted $\Phi({\boldsymbol x}^0,\cdot)$, that sends an element 
${\boldsymbol \nu} = (\nu_t)_{0 \le t \le T} \in {\mathcal C}([0,T],{\mathcal P}({\mathbb T}^d))$ onto the solution 
${\boldsymbol \mu}=(\mu_t)_{0 \le t \le T}$ 
of the 
equation 
\begin{equation*} 
\begin{split}
&\partial_t  \mu_t = \tfrac12 \Delta_{x}   \mu_t - \textrm{\rm div}_{x} \bigl( \alpha\bigl(t,x_t^0,\cdot,\nu_t \bigr)  \mu_t \bigr), \quad t \in [0,T],
\end{split}
\end{equation*} 
with $\mu_0$ as initial condition. 
By 
Lemma 
\ref{lem:alpha:weak:solution}, $(\mu_t)_{0 \le t \le T}$ is the flow of marginal laws of the 
process ${\boldsymbol X}$ under 
the tilted measure 
\begin{equation*} 
{\mathbb P}_{{\boldsymbol x}^0}
:= 
 {\mathscr E}_T \biggl( \int_0^\cdot \alpha(t,x_t^0,X_0+B_t,\nu_t) \cdot \ud B_t 
\biggr) \cdot \tilde{\mathbb P}, 
\end{equation*} 
where 
$\tilde{\mathbb P}$ is the same probability measure on $\Omega_{\textrm{\rm canon}}$
as in the proof of Lemma \ref{lem:alpha:weak:solution}. Proceeding as in the latter proof, 
we easily deduce that 
${\boldsymbol \mu}$, seen as an element of ${\mathcal C}([0,T],{\mathcal P}({\mathbb T}^d))$, is 
the image by a 
 measurable function $\Phi$ of 
the pair $({\boldsymbol x}^0,{\boldsymbol \nu})$, seen as an element of
${\mathcal C}([0,T],{\mathbb R}^d) \times {\mathcal C}([0,T],{\mathcal P}({\mathbb T}^d))$, i.e., 
${\boldsymbol \mu} = \Phi({\boldsymbol x}^0,{\boldsymbol \nu})$. 
And then, writing
\begin{equation*}
{\boldsymbol \mu}^{{\boldsymbol x}^0} 
= \lim_{n \rightarrow \infty} \bigl[ \Phi({\boldsymbol x}^0,\cdot) \bigr]^{\circ n}({\boldsymbol \nu}^0), 
\end{equation*} 
for an arbitrarily fixed element 
${\boldsymbol \nu}^0 \in {\mathcal C}([0,T],{\mathcal P}({\mathbb T}^d))$,  
we deduce that the mapping 
${\mathcal C}([0,T],{\mathbb R}^d) \ni 
{\boldsymbol x}^0 
\mapsto 
{\boldsymbol \mu}^{{\boldsymbol x}^0} 
\in {\mathcal C}([0,T],{\mathcal P}({\mathbb T}^d))$ is measurable. 

Considering on an arbitrary probability space a $d$-dimensional Brownian 
motion $(\bar B_t^0)_{0 \le t \le T}$ and replacing ${\boldsymbol x}^0$ by 
$(x_0+\bar B_t^0)_{0 \le t \le T}$, we get the existence of a (measurable) solution to 
\begin{equation*}
\begin{split} 
&\ud \bar X_t^0 = \ud \bar B_t^0, 
\\
&\partial_t  \bar \mu_t = \tfrac12 \Delta_{x}   \bar \mu_t  - \textrm{\rm div}_{x} \bigl( \alpha(t,\bar X_t^0,\cdot,\bar \mu_t)  \bar\mu_t \bigr), \quad t \in [0,T],
\end{split} 
\end{equation*} 
with $(x_0,\mu)$ as initial condition (at time $0$). In fact, the solution is (replacing the interval $[0,T]$ by the interval $[0,S]$, for $S$ running between $0$ and 
$T$ in the above measurability argument) progressively-measurable with respect to the filtration  generated by 
$\bar{\boldsymbol B}^0$. Also, it is pathwise unique. 
Below, we call 
$\bar{\mathbb P}^0$ the law of $(\bar{\boldsymbol X}^0,\bar{\boldsymbol \mu},\bar{\boldsymbol B}^0)$
on $\Omega^0_{\textrm{\rm canon}}$. 
This proves (b) in the statement. 
\vspace{5pt}
\\
\textit{Second Step.} 
Assume now that the BMO condition 
\eqref{eq:Markov:state:equation:BMO:0} 
is satisfied 
under 
$\bar{\mathbb P}^0$.
Under 
the latter probability, 
the canonical process $({\boldsymbol X}^0,{\boldsymbol \mu},{\boldsymbol B}^0)$
on $\Omega^0_{\textrm{\rm canon}}$ satisfies 
\begin{equation*}
\begin{split} 
&\ud   X_t^0 = \ud   B_t^0, 
\\
&\partial_t    \mu_t = \tfrac12 \Delta_{x}    \mu_t  - \textrm{\rm div}_{x} \bigl( \alpha(t, X_t^0,\cdot, \mu_t)  \mu_t \bigr), \quad t \in [0,T],
\end{split} 
\end{equation*} 
with $(x_0,\mu)$ as initial condition (at time $0$). We introduce the tilted measure 
\begin{equation*} 
\tilde{\mathbb P}^0 := 
{\mathscr E}_T \biggl(  \int_0^\cdot \alpha^0(t,X_t^0,\mu_t) \cdot \ud B_t^0 \biggr) \cdot \bar{\mathbb P}^0,
\end{equation*} 
which is a probability measure thanks to the BMO condition \eqref{eq:Markov:state:equation:BMO}.
Letting $(\tilde B^0_t := B_t^0 -  \int_0^t \alpha^0(s,X_s^0,\mu_s)\ud s)_{0 \le t \le T}$, we have 
\begin{equation*}
\begin{split}
&\ud X_t^0 = 
\alpha^0(t,X_t^0,\mu_t)\ud t
+
 \ud \tilde B_t^0, 
\\
&\partial_t \mu_t = \tfrac12 \Delta_{x_0} \mu_t - \textrm{\rm div}_{x_0} \bigl( \alpha(t,X_t^0,\cdot,\mu_t) \mu_t \bigr), \quad t \in [0,T], 
\end{split}
\end{equation*} 
with $(x_0,\mu)$ as initial condition (at time $0$), 
and $\tilde{\boldsymbol B}^0$ is a Brownian motion under $\tilde{\mathbb P}^0$. 
It then remains to let
${\mathbb P}^0:= 
\tilde{\mathbb P}^0 \circ ({\boldsymbol X}^0,{\boldsymbol \mu},\tilde{\boldsymbol B}^0)^{-1}$. 
It is a probability measure on
$\Omega^0_{\textrm{\rm canon}}$
and it
 satisfies the requirements of 
Lemma \ref{lem:weak:uniqueness:forward:equation}.
Notice in particular that item (iv) in Definition 
\ref{def:admissibility:pair}
follows from 
\cite[Theorems 2.3 and 3.3]{Kazamaki}.
This proves the existence part in item (c) in the statement. 
\vskip 5pt

\noindent \textit{Third Step.} 
Uniqueness is proven in a similar manner. Assuming that we are given a probability measure 
(still denoted) ${\mathbb P}^0$ 
 satisfying items (i), (ii), (iii) and (iv) in 
 Definition  
\ref{def:admissibility:pair} and considering without any loss of generality its completion, we introduce the tilted measure 
\begin{equation*} 
\tilde{\mathbb P}^0 := 
{\mathscr E}_T \biggl( - \int_0^\cdot \alpha^0(t,X_t^0,\mu_t) \cdot \ud B_t^0 \biggr) \cdot {\mathbb P}^0,
\end{equation*} 
which a probability measure thanks to the BMO condition
\eqref{eq:Markov:state:equation:BMO}
under ${\mathbb P}^0$.

Letting $(\tilde B^0_t := X_t^0 - x_0)_{0 \le t \le T}$, we have 
\begin{equation*}
\begin{split}
&\ud X_t^0 =  \ud \tilde B_t^0, 
\\
&\partial_t \mu_t = \tfrac12 \Delta_{x_0} \mu_t - \textrm{\rm div}_{x_0} \bigl( \alpha(t,X_t^0,\cdot,\mu_t) \mu_t \bigr), \quad t \in [0,T], 
\end{split}
\end{equation*} 
and $\tilde{\boldsymbol B}^0$ is a Brownian motion under $\tilde{\mathbb P}^0$. 
And by the uniqueness result established in the first step, we see that the law of 
$({\boldsymbol X}^0,{\boldsymbol \mu},\tilde{\boldsymbol B}^0)$ under $\tilde{\mathbb P}^0$ is equal to the probability 
$\bar {\mathbb P}^0$ constructed in the second step. 
By \cite[Theorems 2.3 and 3.3]{Kazamaki} again, 
the BMO condition \eqref{eq:Markov:state:equation:BMO:0} is 
satisfied under $\bar{\mathbb P}^0$. 
Also, we have
\begin{equation*} 
{\mathbb P}^0 
 = 
  {\mathscr E}_T \biggl( \int_0^\cdot \alpha^0(t,  X_t^0 , \mu_t) \cdot \ud \tilde B_t^0 \biggr) 
  \cdot \tilde{\mathbb P}^0.
\end{equation*} 
In the end,
${\mathbb P}^0$, which is tautologically equal to the law of 
$({\boldsymbol X}^0,{\boldsymbol \mu},{\boldsymbol B}^0)$
under ${\mathbb P}^0$, 
can be regarded as the law of 
$({\boldsymbol X}^0,{\boldsymbol \mu},
(\tilde B_t^0 -  \int_0^\cdot \alpha^0(s,X_s^0,\mu_s)\ud s)_{0 \leq t \leq T}
)$
under ${\mathscr E}_T  (\int_0^\cdot \alpha^0(t,  X_t^0 , \mu_t) \cdot \ud \tilde B_t^0 ) 
  \cdot \tilde{\mathbb P}^0$. 
Since 
$\tilde{\mathbb P}^0 \circ ({\boldsymbol X}^0,{\boldsymbol \mu},\tilde{\boldsymbol B}^0)^{-1}
=
\bar{\mathbb P}^0 \circ  ({\boldsymbol X}^0,{\boldsymbol \mu},{\boldsymbol B}^0)^{-1}$, 
we deduce that 
${\mathbb P}^0$ coincides with the law of 
$({\boldsymbol X}^0,{\boldsymbol \mu},
( B_t^0 -  \int_0^\cdot \alpha^0(s,X_s^0,\mu_s)\ud s)_{0 \leq t \leq T}
)$
under $  {\mathscr E}_T  (\int_0^\cdot \alpha^0(t,  X_t^0 , \mu_t) \cdot \ud  B_t^0 ) 
  \cdot \bar{\mathbb P}^0$. 
This is exactly the construction achieved in the previous step.  
\end{proof}

\section{A priori estimates for the forward-backward system}
\label{se:3}

The objective of this section is to obtain a series of \textit{a priori} estimates for the solution(s)
to the forward-backward system
\eqref{eq:major:FB:1}--\eqref{eq:minor:FB:2}, when 
posed on an arbitrary 
filtered probability space 
$(\Omega^0,{\mathcal F}^0,{\mathbb F}^0,{\mathbb P}^0)$
satisfying the usual conditions and 
equipped with a 
Brownian motion $(B_t^0)_{0 \leq t \leq T}$ with values in ${\mathbb R}^d$, 
with ${\mathbb F}^0$ being 
generated by ${\mathcal F}_0^0$ 
and ${\boldsymbol B}^0$. 


\subsection{HJB equation for the minor player}
\label{subse:HJB:system:analysis:apriori}
The main result of this subsection concerns the regularity of the solution to the 
stochastic HJB equation in 
\eqref{eq:minor:FB:2}. We proceed very much as in the monograph \cite{CardaliaguetDelarueLasryLions}. 
 We also refer to 
 \cite{Du:Chen} for a related analysis but  in Sobolev (instead of H\"older) spaces.

Throughout this subsection, we fix an 
 ${\mathbb F}^0$-adapted 
 continuous path ${\boldsymbol X}^0=(X_t^0)_{0 \le t \le T}$ with values in 
 ${\mathbb R}^d$ (not necessarily solving the forward equation in 
 \eqref{eq:major:FB:1}) and an ${\mathbb F}^0$-adapted continuous path 
 ${\boldsymbol \mu} = (\mu_t)_{0 \le t \leq T}$ with values in ${\mathcal P}({\mathbb T}^d)$
 (not necessarily solving the forward equation in 
 \eqref{eq:minor:FB:2}). With the two of them, 
 we associate the (backward) stochastic Hamilton-Jacobi-Bellman equation  
\begin{equation}
\label{eq:minor:with:X0:frozen}
\begin{split} 
&\ud_t u_t(x) = \Bigl( -  \tfrac12 \Delta_x u_t(x) + 
H\bigl(x,\nabla_x u_t(x)\bigr) 
- f_t(X_t^0,x,\mu_t) 
\Bigr) \ud t
+ \ud m_t(x), \quad (t,x) \in [0,T] \times {\mathbb T}^d, 
\\
&u_T(x) = g(X_T^0,x,\mu_T), \quad x \in {\mathbb T}^d,
\end{split}
\end{equation}
where $(m_t(x))_{0\leq t\leq T}$ is an $\mathbb F^0$-martingale for any $x\in\mathbb T^d$.

The class within which the equation \eqref{eq:minor:with:X0:frozen} is solved is clarified in the following statement, which is taken 
from   \cite[Proposition 4.3.8]{CardaliaguetDelarueLasryLions}:
 \begin{lemma}\label{lem:minor:regularity}
 Under Assumption {\rm (}{\bf A}{\rm )} and within the framework described above, 
 the equation \eqref{eq:minor:with:X0:frozen}
 has a unique solution $({\boldsymbol u},{\boldsymbol m})=(u_t,m_t)_{0 \le t \le T}$, such that 
\begin{enumerate}
\item $(u_t)_{0 \le t \le T}$ is an 
${\mathbb F}^0$-adapted process with 
values in ${\mathcal C}^{\mathfrc{s}}({\mathbb T}^d)$, with continuous 
paths in
${\mathcal C}^{\mathfrc{r}}({\mathbb T}^d)$ for any 
$\mathfrc{r} < \mathfrc{s}$, 
satisfying 
$\sup_{t \in [0,T]} \| u_t \|_{\mathfrc{s}} \in L^\infty(\Omega^0,{\mathcal F}^0,{\mathbb P}^0)$.
\item $(m_t)_{0 \le t \le T}$ is an 
${\mathbb F}^0$-adapted process with 
values in $ {\mathcal C}^{\mathfrc{s}-2}({\mathbb T}^d)$, 
with continuous paths in 
${\mathcal C}^{\mathfrc{r}}({\mathbb T}^d)$ for any 
$ \mathfrc{r} < \mathfrc{s}-2$,
satisfying 
$
\sup_{t \in [0,T]} \| m_t \|_{\mathfrc{s}-2} \in L^\infty(\Omega^0,{\mathcal F}^0,{\mathbb P}^0)$, with $m_0 \equiv 0$, and
with $(m_t(x))_{0 \le t \le T}$ being an ${\mathbb F}^0$-martingale 
for any $x \in {\mathbb T}^d$.  
\end{enumerate}
In fact, uniqueness holds in a wider class of solutions $({\boldsymbol u},{\boldsymbol m})$ for which the above holds true 
with respect to some $\curss' >2$ in lieu of $\curss$. 
 \end{lemma}

 \begin{proof} 
The
result mainly follows from 
 \cite[Proposition 4.3.8]{CardaliaguetDelarueLasryLions}, 
 with two main differences: (i) one must here handle a quadratic HJB equation, whilst the nonlinearity 
 is of linear growth in \cite{CardaliaguetDelarueLasryLions};  (ii) one 
 here claims that continuity of ${\boldsymbol u}$ holds in 
 ${\mathcal C}^{\mathfrc{r}}({\mathbb T}^d)$ for any 
$\mathfrc{r} < \mathfrc{s}$, whilst continuity 
 in 
\cite{CardaliaguetDelarueLasryLions}
is obtained up to the order $\lfloor \curss\rfloor$; (iii) 
as explained in 
the third item of 
Remark 
\ref{rem:def:2:9}, the pair $({\boldsymbol u},{\boldsymbol m})$
is directly seen as 
a random variable with values in 
non-integer H\"older spaces, which are not separable. Generally speaking, the argument to treat (ii) 
is as follows: using the fact that solutions are proven to satisfy 
$\sup_{t \in [0,T]} \| u_t \|_{\mathfrc{s}} \in L^\infty(\Omega^0,{\mathcal F}^0,{\mathbb P}^0)$, 
continuity with values in any 
  ${\mathcal C}^{\mathfrc{r}}({\mathbb T}^d)$, with 
  $\mathfrc{r} < \mathfrc{s}$, 
  follows from the fact that the embedding from 
  ${\mathcal C}^{\mathfrc{r}}({\mathbb T}^d)$
 to 
  ${\mathcal C}^{\mathfrc{s}}({\mathbb T}^d)$
  is compact. 
  As for (iii), we follow the outline 
  given in 
  Remark 
\ref{rem:def:2:9}. Once 
 ${\boldsymbol u}$ is known to be an ${\mathbb F}^0$-adapted 
(continuous) process with values in ${\mathcal C}^{\lfloor \curss \rfloor -1}({\mathbb T}^d)$, 
the same argument based on Schauder's estimates as the one 
used in \cite{CardaliaguetDelarueLasryLions} permits to show that, for any $t \in [0,T)$, 
$u_t$ takes values in ${\mathcal C}^{\curss'}({\mathbb T}^d)$ and 
$m_t$ in ${\mathcal C}^{\curss' -2}({\mathbb T}^d)$, for a certain ${\curss}' > {\curss}$. 
Then, using the same argument as in 
the proof of \cite[Proposition II.2, (1b)]{LiQueffelec}, we deduce that, for any 
$t \in [0,T)$, $u_t$ and $m_t$ are ${\mathcal F}_t^0$-\textit{Bochner} measurable 
with values in ${\mathcal C}^{\curss}({\mathbb T}^d)$ and 
${\mathcal C}^{\curss-2}({\mathbb T}^d)$ respectively. 
At time $t=T$, the field $x \mapsto g(X_T^0,x,\mu_T)$ can be proven to 
be \textit{Bochner} measurable: $X_T^0$ and $\mu_T$ are the almost sure limits of 
simple random variables with values in ${\mathbb R}^d$ and 
${\mathcal P}({\mathbb T}^d)$ respectively. Using the fact that the function 
$(x_0,\mu) \in {\mathbb R}^d \times {\mathcal P}({\mathbb T}^d) 
\mapsto g(x_0,\cdot,\mu) \in {\mathcal C}^{\curss}({\mathbb T}^d)$ 
is continuous, we deduce that 
$u_T$ is the almost sure limit of 
simple random variables with values in ${\mathcal C}^{\curss}({\mathbb T}^d)$, which 
shows \textit{Bochner} measurability. 
It remains to prove that 
$m_T$ is also 
 \textit{Bochner} measurable with values in 
 ${\mathcal C}^{\curss-2}({\mathbb T}^d)$. In fact, it suffices to prove that 
 the integral from $0$ to $T$ of the driver in the backward equation of 
 \eqref{eq:minor:with:X0:frozen}
 is  \textit{Bochner} measurable with values in 
 ${\mathcal C}^{\curss-2}({\mathbb T}^d)$.
 By  
\cite[Proposition II.2, (2)]{LiQueffelec}, it suffices to prove 
the same property but for the integral from $0$ to $T-\varepsilon$, and this for any 
$\varepsilon \in (0,T)$. The latter is a mere consequence of the fact that
the integral from $0$ to $T-\varepsilon$ takes in fact values in 
${\mathcal C}^{\curss'-2}({\mathbb T}^d)$ for a certain 
$\curss'>\curss$.

In order to tackle (i), we proceed by considering first a truncated version of the Hamiltonian, namely we consider a function $H^R$ on ${\mathbb R}^d$ such that 
 \begin{equation*} 
\begin{split}
& H^R(x,p) = H(x,p), \quad \textrm{\rm if} \  \vert p \vert \leq  R \ ;  \quad \sup_{x \in {\mathbb T}^d} \sup_{p \in {\mathbb R}^d} 
 \vert \nabla_p H^R(x,p) \vert < \infty,  
 \\
 &\sup_{x \in {\mathbb T}^d} \sup_{p \in {\mathbb R}^d} 
 \vert \nabla_x H^R(x,p) \vert \leq 
   \sup_{x \in {\mathbb T}^d} \sup_{p \in {\mathbb R}^d} 
 \vert \nabla_x H(x,p) \vert + 1. 
 \end{split} 
 \end{equation*} 
 Then, \cite[Proposition 4.3.8]{CardaliaguetDelarueLasryLions} applies to the following equation: 
 \begin{equation}
 \label{eq:minor:with:X0:frozen:HR}
\begin{split} 
&\ud_t u_t^R(x) = \Bigl( -  \tfrac12 \Delta_x u_t^R(x) +
H^R \bigl( x, \nabla_x u_t^R(x) \bigr) 
- f_t(X_t^0,x,\mu_t) 
\Bigr) \ud t
+ \ud m_t^R(x), \quad (t,x) \in [0,T] \times {\mathbb T}^d, 
\\
&u_T^R(x) = g(X_T^0,x,\mu_T), \quad x \in {\mathbb T}^d,
\end{split}
\end{equation}
and supplies us with the existence of a unique solution $({\boldsymbol u}^R,{\boldsymbol m}^R)$ satisfying the prescriptions of 
items 1 and 2 in the statement of Lemma \ref{lem:minor:regularity}.

In order to pass from the equation \eqref{eq:minor:with:X0:frozen:HR} to the original equation \eqref{eq:minor:with:X0:frozen}, it suffices to show that 
$| \nabla_x u_t^R |$ can be bounded by a deterministic constant independent of $R$, which can be done by the classical Bernstein 
argument. Differentiating with respect to $x$ (which is licit from the results in \cite[Proposition 4.3.8]{CardaliaguetDelarueLasryLions}), we have 
 \begin{equation}
 \label{eq:HJB:derivee}
\begin{split} 
&\ud_t \partial_{x_i} u_t^R(x) = \Bigl( -  \tfrac12 \Delta_x \partial_{x_i} u_t^R(x) + \nabla_p H^R \bigl(x,  \nabla_x u_t^R(x) \bigr)\cdot \partial_{x_i} \nabla_x u_t^R(x) 
+ \partial_{x_i} H^R\bigl(x, \nabla_x u_t^R(x) \bigr) 
\\
&\hspace{15pt} - 
\partial_{x_i} f_t(X_t^0,x,\mu_t) 
\Bigr) \ud t
+ \ud \partial_{x_i} m_t^R(x), \quad (t,x) \in [0,T] \times {\mathbb T}^d, 
\\
&\partial_{x_i} u_T^R(x) = \partial_{x_i} g(X_T^0,x,\mu_T), \quad x \in {\mathbb T}^d,
\end{split}
\end{equation}
for any coordinate $i \in \{1,\cdots,d\}$. Then, we can interpret the equation as a backward stochastic transport diffusion equation, for which we have a maximum 
principle. In turn, we get that 
\begin{equation*} 
\vert \partial_{x_i} u_t^R(x) \vert \leq   \| \nabla_x g \|_{L^\infty}
+ T 
\| \nabla_x f \|_{L^\infty}
+ 
T 
\bigl( \| \nabla_x H  \|_{L^\infty} + 1
\bigr).
\end{equation*} 
We hence get a bound $C$ for the left-hand side. Obviously, this bound $C$ is deterministic and independent of $R$, which allows us to retrieve the original equation by choosing $R$ larger than $C$. Uniqueness of the hence constructed solution 
is obvious because we restricted our analysis to solutions that are bounded in ${\mathcal C}^{\mathfrc{s}}({\mathbb T}^d)$ by a deterministic constant. 
%
%
As mentioned in the statement, we can even state uniqueness in a wider class. This follows from the fact that the parameter 
$n$ in 
\cite[Proposition 4.3.8]{CardaliaguetDelarueLasryLions}
can be taken equal to 2. 
This completes the proof. 
%
 \end{proof} 

Since ${\mathbb F}^0$ is 
generated by ${\mathcal F}_0^0$ 
and ${\boldsymbol B}^0$, we can represent, for any 
$x \in {\mathbb T}^d$, the martingale $(m_t(x))_{0 \le t \le T}$ in the form
\begin{equation}
\label{eq:representation:formula:m:v0}
m_t(x) = \sigma_0 \int_0^t v_s^0(x) \cdot \ud B_s^0, \quad 0 \le t \le T, 
\end{equation}
where $(v_t^0(x))_{0 \le t \le T}$
is an ${\mathbb R}^d$-valued ${\mathbb F}^0$-progressively measurable process. 
The regularity of $v^0$ is given by the following statement:

\begin{lemma} 
\label{lem:reg:v0}
The representation term in \eqref{eq:representation:formula:m:v0}
induces a process 
$(v_t^0)_{0 \le t \le T}$
with values in ${\mathcal H}^{\lfloor \textrm{\rm \mathfrc{s}} \rfloor}({\mathbb T}^d)$
(which is embedded in 
${\mathcal C}^{\textrm{\rm  \mathfrc{s}} - d/2-1}({\mathbb T}^d)$). Moreover, 
\begin{equation}
\label{lem:eq:v0:s-d/2-2} 
{\mathbb E}^0 \int_{0}^T \| v_t^0 \|_{\textrm{\rm \mathfrc{s}}-d/2-1}^2 \ud s < \infty,
 \end{equation}
and, for any $k \in \{0,\cdots,\lfloor \textrm{\rm \mathfrc{s}}  -d/2 \rfloor -1\}$,
${\mathbb P}^0$-almost surely, for all $x \in {\mathbb T}^d$,  
\begin{equation}
\label{lem:eq:v0:s-d/2-2:exchange:derivative}  
\nabla^k_x m_t(x) = \sigma_0 \int_0^t \nabla^k_x v_s^0(x) \cdot \ud B_s^0, 
\quad t \in [0,T]. 
\end{equation}
\end{lemma}
Notice that 
${\mathcal H}^{\lfloor \textrm{\rm \mathfrc{s}} \rfloor}({\mathbb T}^d)$
(equipped with $\| \cdot \|_{\textrm{\rm  \mathfrc{s}} - d/2-1}$) is separable:
for any $t \in [0,T]$, 
$v_t^0$, when regarded as a random variable with values in 
${\mathcal C}^{\textrm{\rm  \mathfrc{s}} - d/2-1}({\mathbb T}^d)$ is obviously 
\textit{Bochner} measurable. 

\begin{proof}
Because 
$\sup_{t \in [0,T]} \| u_t \|_{\mathfrc{s}} \in L^\infty(\Omega^0,{\mathcal F}^0,{\mathbb P}^0)$, 
we can see the equation  
\eqref{eq:minor:with:X0:frozen}
as a linear equation with a source
term in 
${\mathcal C}^{\mathfrc{s}-1}({\mathbb T}^d)$. 
This makes it possible 
to apply \cite[Theorem 2.3]{Du:Meng} (with a modicum of care because the latter result is stated on the Euclidean space,
but the adaptation to the periodic setting is straightforward).  
We deduce that
the representation term
 $(v_t^0)_{0 \le t \le T}$ in
 \eqref{eq:representation:formula:m:v0}
 takes values in 
 ${\mathcal H}^{\lfloor \mathfrc{s} \rfloor}({\mathbb T}^d)$. Moreover, it satisfies 
 \begin{equation*} 
 {\mathbb E}^0  \int_{0}^T \| v_t^0 \|_{\lfloor \mathfrc{s} \rfloor,2}^2 \ud s < \infty.
 \end{equation*}
 By Sobolev embedding, we obtain 
\eqref{lem:eq:v0:s-d/2-2} (because 
$\lfloor \curss \rfloor - d/2 > \curss - d/2-1$). 
Then, 
\eqref{lem:eq:v0:s-d/2-2:exchange:derivative}
can be easily established by induction on $k \in \{1,\cdots,\lfloor \textrm{\rm \mathfrc{s}}  -d/2 \rfloor -1\}$, by passing to the limit
(as $\varepsilon$ tends to $0$) 
 in the 
representation 
formula
\begin{equation*} 
\tfrac1{\varepsilon} 
\Bigl( \nabla_x^{k-1} m_t(x+\varepsilon \xi) -\nabla_x^{k-1} m_t(x)
\Bigr) 
= \sigma_0 \int_0^t \Bigl[ 
\tfrac1{\varepsilon} 
\Bigl( \nabla_x^{k-1} v^0_s(x+\varepsilon \xi) -\nabla_x^{k-1} v^0_s(x)
\Bigr)\Bigr] 
\cdot \ud B_s, 
\end{equation*} 
for any $x \in {\mathbb T}^d$ and unitary $\xi \in {\mathbb R}^d$. This proves 
that, for all $x \in {\mathbb T}^d$,  
${\mathbb P}^0$-almost surely,
for all $t \in [0,T]$, 
\eqref{lem:eq:v0:s-d/2-2:exchange:derivative} holds true.  
Using the regularity properties 
of the left-hand side 
in 
\eqref{lem:eq:v0:s-d/2-2:exchange:derivative}, we 
easily deduce that, for any integer 
$p \geq 1$, there exists 
a constant $C_p$ such that, for any 
$k \in \{0,\cdots,\lfloor \textrm{\rm \mathfrc{s}}  -d/2 \rfloor -1\}$
and any $x,x' \in {\mathbb T}^d$, 
\begin{equation*}
{\mathbb E} \biggl[  \sup_{0 \le t \le T} 
\biggl\vert 
\int_0^t \Bigl( \nabla^k_x v_s^0(x') -  \nabla^k_x v_s^0(x)\Bigr) \cdot \ud B_s^0
\biggr\vert^p \biggr] \
= {\mathbb E} \Bigl[ \sup_{0 \le t \le T} 
\vert \nabla^k m_t(x') - \nabla^k m_t(x) \vert^p \Bigr] \leq 
C_p \vert x'-x \vert^p,
\end{equation*} 
which proves, by Kolmogorov continuity theorem, that we can find, for each $x \in {\mathbb T}^d$, a version of the stochastic process 
$(\int_0^t   \nabla_x^{k} v^0_s(x)  
\cdot \ud B_s)_{0 \leq t \leq T}$ 
that is continuous in $x$. This permits to exchange the quantifiers
in  
\eqref{lem:eq:v0:s-d/2-2:exchange:derivative} 
and then get the formula ${\mathbb P}^0$-almost surely,
for all $(t,x) \in [0,T] \times {\mathbb T}^d$.
\end{proof} 
\subsection{Uniform in time estimates for the HJB equation}
\label{subse:uniform:HJB}
Within the same framework as  in Subsection 
\ref{subse:HJB:system:analysis:apriori}, we now provide a uniform in time estimate of the gradient of 
the solution to 
\eqref{eq:minor:with:X0:frozen}. This bound is key in our analysis. 

\begin{proposition}\label{prop:minor:regularity}
 Under Assumption \hyp{A}, there exists a constant $R_0$,  only depending on the parameters in \hyp{A} except $(\sigma_0,T)$ (in particular, $R_0$ is independent of ${\boldsymbol \mu}$),  such that 
\begin{equation}
\label{minor:linfty:bound:Du}
{\mathbb P}^0 \Bigl( \sup_{0 \le t \le T} \| \nabla_x u_t \|_{L^\infty} \leq R_0 \Bigr) = 1. 
\end{equation}
 \end{proposition} 
 

\begin{proof}
%
%
%
%
The proof of the bound \eqref{minor:linfty:bound:Du} is  inspired from the proof of Lemma 1.5 in \cite{cardaliaguet:porretta:longtime:master}, but adapted to the SPDE setting. 


The very first step of the proof is to consider, for  a given $\xi \in {\mathbb R}^d$ with $\vert \xi \vert=1$, 
the equation satisfied by $(D^2_{\xi \xi} u_t(x) := (\nabla^2_{xx} u_t(x) \xi) \cdot \xi)_{0 \leq t \leq T, x \in {\mathbb T}^d}.$
By a straightforward adaptation of the computations as in \cite{cardaliaguet:porretta:longtime:master}, we get that 
$D^2_{\xi \xi} u$ solves the following equation: 
\begin{equation*} 
\begin{split}
\ud_t \Bigl( D^2_{\xi \xi}  u_t(x) \Bigr)  &= \Bigl[ - \tfrac12 \Delta_{x} D^2_{\xi \xi} u_t(x) + 
\Bigl( \nabla^2_{xx} H
\bigl(x, \nabla_x u_t(x) \bigr) \xi \Bigr) \cdot \xi + 2 \Bigl( \nabla_{px}^2 H \bigl(x,\nabla_x u_t(x) \bigr) \xi \Bigr)\cdot \nabla_x 
D_\xi 
u_t(x) 
\\
&\hspace{15pt} 
+\bigl(  \nabla^2_{pp} H\bigl(x, \nabla_x u_t(x)\bigr)
\nabla_x D_\xi u_t(x) \bigr) \cdot \nabla_x D_\xi u_t(x) 
\\
&\hspace{15pt} + 
\nabla_p H\bigl(x, \nabla_x u_t(x) \bigr) \cdot \nabla_x D^2_{\xi \xi} u_t(x) 
- D^2_{\xi \xi} f_t(X_t^0,x,\mu_t) \Bigr] \ud t 
\\
&\hspace{15pt} + \ud \bigl(D^2_{\xi \xi} m_t\bigr) (x), 
\end{split}
\end{equation*} 
where the notations $D_{\xi}$ and $D^2_{\xi \xi}$ are extended in an obvious manner to the first and second order derivatives 
in the direction $\xi$ of (possibly random) functions different from $u_t(\cdot)$. 
Above, $(D^2_{\xi \xi } m_t(x))_{0 \leq t \leq T}$ is a  martingale. 

The key point in the proof is to observe that
there exists a constant $C>0$, only depending on the parameters in \hyp{A5}, such that, for $x \in {\mathbb T}^d$, $p,q \in {\mathbb R}^d$, 
\begin{equation} 
\label{eq:lower:bound:key:functional:HR}
\begin{split} 
&\Bigl( \nabla^2_{xx} H (x,p) \xi \Bigr) \cdot \xi + 2 \Bigl( \nabla_{px}^2 H(x,p) \xi \Bigr)\cdot q 
+\Bigl(  \nabla^2_{pp} H(x,p) 
q \Bigr) \cdot q
  - D^2_{\xi \xi} f_t(X_t^0,x,\mu_t) 
\\
&\geq C^{-1} \vert q \vert^2 - C.  
\end{split} 
\end{equation}

We next consider the same Brownian motion $(B_t)_{0 \leq t \leq T}$ as in the dynamics of the minor player. In particular, 
$(B_t)_{0 \leq t \leq T}$ is independent of the common noise $(B_t^0)_{0 \le t \le T}$. We then consider the SDE
\begin{equation*} 
\ud X_t = - \nabla_p H \bigl( X_t, \nabla_x u_t(X_t) \bigr) \ud t + \ud B_t, \quad t \in [t_0,T],
\end{equation*}  
for a given $t_0 \in [0,T]$ and for an initial condition $X_{t_0} = x$. 
Notice that the drift is random, but the solution is uniquely defined thanks to the regularity properties
granted by Lemma 
\ref{lem:minor:regularity}. 
We also let
\begin{equation*} 
U_t := D^2_{\xi \xi} u_t(X_t). 
\end{equation*} 
By applying It\^o-Wentzell formula 
to $(U_t)_{0 \le t \le T}$ 
and by invoking 
Lemma 
\ref{lem:reg:v0} to guarantee that 
${\mathbb E}^0 \int_{[0,T]} \| D^2_{\xi \xi} v^0_t\|^2_1 \ud t < \infty$
(thanks to the condition $\mathfrc{s}>d/2+5$), we get (see footnote 
\ref{foo:1} for the proof)
\begin{equation*} 
\begin{split} 
\ud U_t &= \Bigl[ \Bigl( \nabla^2_{xx} H\bigl(X_t, \nabla_x u_t(X_t) \bigr) \xi \Bigr)
\cdot \xi + 2 \Bigl( \nabla_{px}^2 H \bigl(X_t,\nabla_x u_t(X_t) \bigr) \xi \Bigr) \cdot \nabla_x D_\xi u_t  (X_t) 
\\
&\hspace{15pt}+
\bigl(  \nabla^2_{pp} H\bigl(X_t, \nabla_x u_t(X_t)\bigr)
\nabla_x D_\xi u_t(X_t) \bigr) \cdot \nabla_x  D_\xi u_t(X_t) 
 - D^2_{\xi \xi} f_t(X_t^0,X_t,\mu_t) \Bigr] \ud t + \ud n_t, 
\end{split}
\end{equation*} 
where $(n_t)_{0 \le t \le T}$ is a new generic martingale term, but now independent of the entry $x$. 
By 
\eqref{eq:lower:bound:key:functional:HR} and thanks to the lower bound 
$\vert \nabla_x D_\xi u_t  (X_t) \vert^2 \geq  \vert \nabla_x D_\xi u_t  (X_t) \cdot \xi\vert^2 
= \vert  D^2_{\xi \xi} u_t  (X_t)  \vert^2=\vert U_t \vert^2$, we get 
\begin{equation} 
\label{eq:UtR:2}
\begin{split}
\ud U_t &\geq  \bigl( C^{-1} \vert U_t \vert^2 - C \bigr)    \ud t + \ud n_t.
\\
&= C^{-1} \bigl( U_t - C) \bigl( U_t + C \bigr) \ud t + \ud n_t,
\end{split} 
\end{equation} 
which gives
\begin{equation*} 
\ud \Bigl[ \exp \biggl( - C^{-1} \int_{t_0}^t (U_s + C) \ud s \biggr) \bigl( U_t - C\bigr) \Bigr] \geq \ud n_t.
\end{equation*} 
In particular, 
assuming without any loss of generality that $\| D^2_{\xi \xi} g \|_\infty \leq C$ (i.e. $C \geq \kappa$, 
with the notations used in Assumption \hyp{A}), 
the boundary condition at time $T$ in the above left-hand side has a non-positive value, from which we deduce (by conditioning on 
${\mathcal F}_{t_0}^0 \otimes {\mathcal F}_{t_0}$) that $U_{t_0}\leq C$. 
By initializing $X$ from any $x \in {\mathbb T}^d$ at time $t_0$,
this shows that, ${\mathbb P}^0$-almost surely, 
\begin{equation*} 
\sup_{x \in {\mathbb T}^d} \sup_{\xi \in {\mathbb R}^d : \vert \xi \vert =1} \nabla^2_x u_{t_0}(x) \xi \cdot \xi \leq C.
\end{equation*}  
By (15) in \cite{cardaliaguet:porretta:longtime:master}, which says that, on the torus, the Lipschitz constant of a smooth function can be controlled by its
semi-concavity constant,
we deduce that ${\mathbb P}^0(\{\| \nabla_x u_{t_0} \|_{L^\infty} \leq R_0\})=1$, with $R_0:=C \sqrt{d}$. 
By continuity in time of the process $(\| u_t \|_2)_{0 \le t \le T}$, we complete the proof. 
\end{proof}

We eventually extend the result of 
Proposition 
\ref{prop:minor:regularity}
to higher-derivatives: 
\begin{proposition}
\label{prop:minor:higher}
 Under Assumption \hyp{A}
 and within the context described in the beginning of Subsection 
\ref{subse:HJB:system:analysis:apriori}, there exists a constant $R_{\textrm{\rm \mathfrc{s}}-1}$, only depending on the parameters in \hyp{A} except $(\sigma_0,T)$ (in particular, $R$ is independent of ${\boldsymbol \mu}$), such that   
\begin{equation}
\label{minor:linfty:bound:Dsu}
{\mathbb P}^0 \Bigl( \sup_{0 \le t \le T} \| \nabla_x u_t \|_{\textrm{\rm \mathfrc{s}}-1} \leq R_{\textrm{\rm \mathfrc{s}}-1}\Bigr) = 1. 
\end{equation}
\end{proposition}

\begin{proof}
When $t$ is close to $T$ (say $T-t \leq 1$), the result follows   from  \cite[Proposition 4.3.8]{CardaliaguetDelarueLasryLions}
(which applies since the gradient in the Hamiltonian is known to be bounded by means of 
Proposition \ref{prop:minor:regularity}).

In particular, we can assume that $T-t \geq 1$. The only difficulty is to get a bound that is independent of $T$. By 
taking one derivative in the equation 
\eqref{eq:minor:with:X0:frozen}
as done in \eqref{eq:HJB:derivee}
and then by using standard heat kernel estimates, we  claim that, for any 
integer $k \in \{1,\cdots,\lfloor \textrm{\rm \mathfrc{s}} \rfloor-1\}$, any $t,S \in [0,T]$, with $S - 1 \leq t < S$,
\begin{equation*} 
\begin{split}
\| \nabla_x u_{t} \|_{\textrm{\rm \mathfrc{s}}-k} &\leq \frac{C}{\sqrt{S-t}} {\mathbb E}^0 \bigl[ 
\| \nabla_x u_{S} \|_{\textrm{\rm \mathfrc{s}}-(k+1)} 
\, \vert \, {\mathcal F}_t^0 \bigr]
\\
&\hspace{5pt} +
C
 {\mathbb E}^0 \biggl[ 
\int_t^{S}  
\frac1{\sqrt{r-t}} 
\sup_{i=1,\cdots,d}
\bigl\|   \nabla_p H(\cdot, \nabla_x u_r(\cdot) \bigr)\cdot \partial_{x_i} \nabla_x u_r    \bigr\|_{\textrm{\rm \mathfrc{s}}-(k+1)}
\ud r 
\, \vert \, {\mathcal F}_t^0 \biggr]
\\
&\hspace{5pt} +
C
 {\mathbb E}^0 \biggl[ 
\int_t^{S}  
\frac1{\sqrt{r-t}} 
\Bigl( \bigl\| \nabla_x H(\cdot, \nabla_x u_r(\cdot) \bigr)  \bigr\|_{s-(k+1)}
+
\bigl\| \nabla_x f_r(X_r^0,\cdot,\mu_r) 
\bigr\|_{s-(k+1)}
\Bigr) \ud r 
\, \vert \, {\mathcal F}_t^0 \biggr].
\end{split} 
\end{equation*}
Assuming that we have uniform (in time and in $\omega \in \Omega^0$) bounds for
$\| \nabla_x u_t \|_{\textrm{\rm \mathfrc{s}}-(k+1)}$, we deduce that there exists a constant $C_k$ such that 
\begin{equation}
\label{:eq:heat:kernel:1}
\begin{split}
\| \nabla_x u_{t} \|_{\textrm{\rm \mathfrc{s}}-k} &\leq \frac{C_k}{\sqrt{S-t}}  +
C_k
 {\mathbb E}^0 \biggl[ 
\int_t^{S}  
\frac1{\sqrt{r-t}} 
\| \nabla_x u_{r} \|_{\textrm{\rm \mathfrc{s}}-k}
\ud r 
\, \vert \, {\mathcal F}_t^0 \biggr].
\end{split} 
\end{equation}
And then, for any $\delta \in (0,S)$, 
with $\delta \leq 1$, 
(the mapping $t \mapsto \textrm{\rm essup}_{\omega \in \Omega^0} \| \nabla_x u_{t} \|_{\textrm{\rm \mathfrc{s}}-k}$
is measurable as limit of 
$t \mapsto {\mathbb E}^0[ \| \nabla_x u_{t} \|_{\textrm{\rm \mathfrc{s}}-k}^p]^{1/p}$
as $p$ tends to $\infty$)
\begin{equation*} 
\begin{split}
&\int_{S-\delta}^S 
\frac1{\sqrt{t-(S-\delta)}}
\textrm{\rm essup}_{\omega \in \Omega^0} \| \nabla_x u_{t} \|_{\textrm{\rm \mathfrc{s}}-k}
\ud t 
\\
&\leq C_k B(\tfrac12,\tfrac12) +
C_k
\int_{S-\delta}^S 
\frac1{\sqrt{t-(S-\delta)}}
\biggl(
\int_t^{S}  
\frac1{\sqrt{r-t}} 
\textrm{\rm essup}_{\omega \in \Omega^0}
\| \nabla_x u_{r} \|_{\textrm{\rm \mathfrc{s}}-k}
\ud r \biggr) 
\ud t
\\
&= C_k B(\tfrac12,\tfrac12) +
C_k
\int_{S-\delta}^S 
\textrm{\rm essup}_{\omega \in \Omega^0}
\| \nabla_x u_{r} \|_{\textrm{\rm \mathfrc{s}}-k}
\biggl( \int_{S-\delta}^r \frac1{\sqrt{r-t}}  \frac1{\sqrt{t-(S-\delta)}} \ud t
\biggr)
\ud r 
\\
&= C_k B(\tfrac12,\tfrac12) +
C_k  B(\tfrac12,\tfrac12)
\int_{S-\delta}^S 
\textrm{\rm essup}_{\omega \in \Omega^0}
\| \nabla_x u_{r} \|_{\textrm{\rm \mathfrc{s}}-k}
\ud r
\\
&\leq C_k B(\tfrac12,\tfrac12) +
C_k  \sqrt{\delta} B(\tfrac12,\tfrac12)
\int_{S-\delta}^S 
\frac1{\sqrt{t-(S-\delta)}}
\textrm{\rm essup}_{\omega \in \Omega^0}
\| \nabla_x u_{r} \|_{\textrm{\rm \mathfrc{s}}-k}
\ud r,
\end{split} 
\end{equation*} 
where $B(\cdot,\cdot)$ is the Euler B\^eta function. 
And then,
we get a bound in the left hand-side for $\sqrt{\delta}  B(\nicefrac12,\nicefrac12) =\min(1/2,1/{2C_k})$. 
Choosing $S=t+\delta$ in 
\eqref{:eq:heat:kernel:1}, we get, for $t \leq T-1$, 
a bound for $\textrm{\rm essup}_{\omega \in \Omega^0} \| \nabla_x u_{t} \|_{\textrm{\rm \mathfrc{s}}-k}$. 
\end{proof}

\subsection{BMO estimates for the minor MFG system}
\label{subse:MFG:system:analysis:apriori}
We now fix an 
 ${\mathbb F}^0$-adapted 
 continuous path ${\boldsymbol X}^0 = (X_t^0)_{0 \le t \le T}$ with values in 
 ${\mathbb R}^d$ (not necessarily solving the forward equation in 
 \eqref{eq:major:FB:1}). With it, 
 we associate the 
 stochastic MFG system
 \eqref{eq:minor:FB:2}. Since $f$ and $g$ are assumed to be monotone in Lasry Lions sense, and since we have 
 an \textit{a priori} estimate for the gradient of the backward component, we can easily do as if the Hamiltonian 
 were at most of linear growth and hence follow the proof of 
  \cite[Theorem 4.3.1]{CardaliaguetDelarueLasryLions} to get a unique solution to the MFG system. Importantly, 
  the backward component satisfies the conclusion of 
  Lemma 
  \ref{lem:minor:regularity}
and Proposition 
\ref{prop:minor:regularity}. 
Our first step is to strengthen the estimate provided by 
Proposition 
\ref{prop:minor:higher}. 
 
Throughout the subsection, we 
consider a fixed ${\mathbb F}^0$-stopping time $\tau$ with values in $[0,T]$ and we consider the auxiliary MFG system 
\begin{equation} 
\label{eq:auxiliary:MFG} 
\begin{split}
&\partial_t \tilde \mu_t - \tfrac12 \Delta_x \tilde \mu_t - {\rm div}_x \Bigl( \nabla_p H \bigl(\cdot, \nabla_x \tilde u_t\bigr)\tilde \mu_t \Bigr) =0, 
\quad  \ (t,x)\in [\tau,T] \times {\mathbb T}^d, 
\\
&\partial_t \tilde u_t(x) =   -  \tfrac12 \Delta_x \tilde u_t(x) + H\bigl(x,\nabla_x \tilde u_t(x) \bigr) - F(x,\tilde \mu_t), \quad (t,x) \in [\tau,T] \times {\mathbb T}^d, 
\\
&\tilde u_T(x) = 0, \quad x \in {\mathbb T}^d, \quad \tilde \mu_\tau=\mu_\tau,  
\end{split}
\end{equation}
where $F$ is as in Assumption \hyp{B}. 
Except for the fact that the initial time is random, this MFG system is deterministic. 
Since $F$ satisfies the Lasry-Lions monotonicity condition, 
the system \eqref{eq:auxiliary:MFG} 
has a unique solution. 

\begin{proposition}\label{prop:minor:regularity:higher}
 Under Assumption \hyp{B}
 and with the  the same 
 notation as above, 
 there exists a constant $C$, 
 only depending on the parameters in  \hyp{B} except $(\sigma_0,T)$
 (in particular, $C$ is also independent of $\tau$), such that
 \begin{equation*}
{\mathbb P}^0 
\biggl( \biggl\{ 
{\mathbb E}^0 \biggl[ 
\int_\tau^T
\biggl( 
 \mathbb W_1( \mu_t , \tilde \mu_t)^2  + \Bigl\| \bigl( u_{t} - \tilde{u}_{t} \bigr)
- \int_{{\mathbb T}^d} \bigl( u_{t} - \tilde{u}_{t} \bigr)(x) \ \ud x \Bigr\|_{\textrm{\rm \mathfrc{s}}}^2
\biggr) \, \ud  t 
\, \vert \, {\mathcal F}^0_{\tau} \biggr] \leq C \biggr\} \biggr) = 1.
\end{equation*} 
 \end{proposition}

\begin{proof}
For simplicity, we assume in the first three steps of the proof that 
$\tau=0$. We then explain in the very last step how the proof has to be changed when 
$\tau$ is general. 
\vskip 4pt

\noindent \textit{First Step.}
By monotonicity of the coefficient $F$ and from the standard duality method of mean field games
(see \cite[Lemma 3.1.2]{CardaliaguetDelarueLasryLions}), we have, 
for any $t \in [0,T]$,
\begin{equation*} 
\begin{split}
&- \langle u_t- \tilde u_t, \mu_t - \tilde \mu_t) + c  {\mathbb E}^0 \biggl[ \int_t^T \bigl( \mu_r+\tilde \mu_r, \vert \nabla_x 
(u_r - \tilde u_r) \vert^2 \bigr) \ud r \, \vert \, {\mathcal F}_t^0 \biggr]
\\
&\leq - {\mathbb E}^0 \biggl[  \int_t^T \bigl( f_r(X_r^0,\cdot,\mu_r) - F(\cdot,\mu_r), \mu_r - \tilde \mu_r \bigr) 
\ud r + ( u_T- \tilde u_T, \mu_T - \tilde \mu_T) \, \vert \, {\mathcal F}_t^0 \biggr],
\end{split}
\end{equation*} 
for a constant $c>0$ only depending on the parameters in \hyp{B} except $(\sigma_0,T)$. 

By boundedness of $u_T$  and $\tilde u_T$ (but the latter is null), we obtain 
\begin{equation}
\label{eq:ergodicity:H} 
\begin{split}
&- ( u_t- \tilde u_t, \mu_t - \tilde \mu_t) + c {\mathbb E}^0 \biggl[ \int_t^T \bigl( \mu_r+\tilde \mu_r, \vert \nabla_x 
(u_r - \tilde u_r) \vert^2 \bigr) \ud r \, \vert \, {\mathcal F}_t^0 \biggr]
\\
&\leq C 
{\mathbb E}^0 \biggl[   1 + \int_t^T
 \sup_{x_0 \in {\mathbb R}^d} \sup_{\mu \in {\mathcal P}({\mathbb T}^d)}
\| f_r(x_0,\cdot,\mu) - F(\cdot,\mu) \|_{L^\infty} 
\ud r \, \vert \, {\mathcal F}_t^0 \biggr].
\end{split}
\end{equation} 
We give below two applications of \eqref{eq:ergodicity:H}. 
\vskip 4pt

\noindent \textit{Second Step.}
First, we study the equation for $( \mu_t - \tilde \mu_t)_{0 \le t \le T}$. We write
\begin{equation*} 
\begin{split}
\partial_t (\mu_t - \tilde \mu_t) &= \tfrac12 \Delta_x  ( \mu_t - \tilde \mu_t) + \textrm{\rm div}_x
\Bigl( \nabla_p H\bigl(\cdot,\nabla_x \tilde u_t(\cdot)\bigr) 
( \mu_t - \tilde \mu_t)\Bigr) 
\\
&\hspace{15pt} + 
\textrm{\rm div}_x
\Bigl( \bigl[ 
\nabla_p H\bigl(\cdot,\nabla_x u_t(\cdot)\bigr) 
-
\nabla_p H\bigl(\cdot,\nabla_x \tilde u_t(\cdot)\bigr)
\bigr]   \mu_t \Bigr). 
\end{split} 
\end{equation*} 
Following Lemma
\ref{le:appendix:1} and noticing that $\tilde {\boldsymbol u}$ satisfies the same bounds 
as ${\boldsymbol u}$ in Proposition \ref{prop:minor:higher}, consider now the backward PDE
\begin{equation*} 
\begin{split}
&\partial_t \varphi_t + \tfrac12 \Delta_x \varphi_t - \nabla_p H\bigl(\cdot,\nabla_x \tilde u_t(\cdot)\bigr)  \cdot \nabla_x \varphi_t = 0,
\quad t \in [0,S]; \quad \varphi_S=\phi,
\end{split}
\end{equation*}
for a smooth terminal function $\phi$ and for some $S \in [0,T]$
(since $\tilde {\boldsymbol u}$ is deterministic, so is the above PDE and there is no need 
to add an additional martingale term as it would be the case if we were
considering the same equation but driven by ${\boldsymbol u}$). 
By duality, we obtain
\begin{equation*} 
\begin{split} 
\frac{\ud}{\ud t} \bigl( \varphi_t, \mu_t - \tilde \mu_t\bigr) 
&= - \Bigl( \bigl[ 
\nabla_p H\bigl(\cdot,\nabla_x u_t(\cdot)\bigr) 
-
\nabla_p H\bigl(\cdot,\nabla_x \tilde u_t(\cdot)\bigr)
\bigr] \mu_t,\nabla_x \varphi_t \Bigr).
\end{split}
\end{equation*} 
Therefore, 
using the fact that $\mu_0-\tilde \mu_0=0$, we obtain
\begin{equation*} 
\begin{split} 
 \bigl( \phi, \mu_S - \tilde \mu_S \bigr) &\leq \int_0^S \| \nabla_x \varphi_t \|_{L^\infty} 
 \bigl(  \mu_t, \vert  \nabla_p H\bigl(\cdot,\nabla_x u_t(\cdot)\bigr) 
-
\nabla_p H\bigl(\cdot,\nabla_x \tilde u_t(\cdot)\bigr) \vert^2 \bigr)^{1/2} 
\ud t
\\
&\leq   \|\phi\|_1  \int_0^S\exp \bigl( - \gamma (S-t) \bigr) 
 \bigl( \mu_t, \vert  \nabla_p H\bigl(\cdot,\nabla_x u_t(\cdot)\bigr) 
-
\nabla_p H\bigl(\cdot,\nabla_x \tilde u_t(\cdot)\bigr) \vert^2 \bigr)^{1/2} 
\ud t, 
\end{split} 
\end{equation*} 
with the exponential decay following from 
Lemma
\ref{le:appendix:1}, and for $\gamma$ only depending on the parameters 
in \hyp{B} but not on $(\sigma_0,T)$. 
And then, 
\begin{equation*} 
\begin{split}
\| \mu_S - \tilde \mu_S \|_{-1} &\leq \int_0^S\exp \bigl( - \gamma (S-t) \bigr) 
 \bigl( \mu_t, \vert  \nabla_p H\bigl(\cdot,\nabla_x u_t(\cdot)\bigr) 
-
\nabla_p H\bigl(\cdot,\nabla_x \tilde u_t(\cdot)\bigr) \vert^2 \bigr)^{1/2} 
\ud t
\\
&\leq 
\frac1{\gamma^{1/2}}
\biggl( \int_0^S\exp \bigl( - \gamma (S-t) \bigr) 
 \bigl(  \mu_t, \vert  \nabla_p H\bigl(\cdot,\nabla_x u_t(\cdot)\bigr) 
-
\nabla_p H\bigl(\cdot,\nabla_x \tilde u_t(\cdot)\bigr) \vert^2\bigr) \ud t\biggr)^{1/2}, 
\end{split}
\end{equation*} 
which yields (replacing $S$ by $t$ in the left-hand side and $t$ by 
$r$ in the right-hand side) 
\begin{equation*} 
\begin{split}
{\mathbb E}^0 \biggl[ \int_0^T
 \| \mu_t - \tilde \mu_t \|_{-1}^2 \ud t \biggr]
 &\leq \frac1{\gamma} {\mathbb E}^0 \int_0^T \biggl( \int_0^t\exp \bigl( - \gamma (t-r) \bigr) 
 \bigl(  \mu_r, \vert  \nabla_p H\bigl(\cdot,\nabla_x u_r(\cdot)\bigr) 
-
\nabla_p H\bigl(\cdot,\nabla_x \tilde u_r(\cdot)\bigr) \vert^2 \bigr) 
\ud r \biggr) \ud t
\\
&=
\frac1{\gamma} {\mathbb E}^0 \int_0^T 
 \bigl(  \mu_r, \vert  \nabla_p H\bigl(\cdot,\nabla_x u_r(\cdot)\bigr) 
-
\nabla_p H\bigl(\cdot,\nabla_x \tilde u_r(\cdot)\bigr) \vert^2 \bigr) 
\biggl( \int_r^T\exp \bigl( - \gamma (t-r) \bigr) \ud t \biggr)
\ud r
\\
&\leq \frac1{\gamma^2}  {\mathbb E}^0 \int_0^T 
 \bigl( \mu_r, \vert  \nabla_p H\bigl(\cdot,\nabla_x u_r(\cdot)\bigr) 
-
\nabla_p H\bigl(\cdot,\nabla_x \tilde u_r(\cdot)\bigr) \vert^2 \bigr) 
\ud r.
\end{split}
\end{equation*}
We now insert \eqref{eq:ergodicity:H} (with $t=0$, recalling that $\mu_0=\tilde \mu_0$)
into the above bound. We obtain (recalling that $\| \mu - \tilde \mu\|_{-1}$ is the same 
as $\mathbb W_1(\mu,\tilde \mu)$)
\begin{equation} 
\label{eq:minor:MFG:ergodicity:forward}
\begin{split}
{\mathbb E}^0 \biggl[ \int_0^T
 \mathbb W_1( \mu_t , \tilde \mu_t)^2 \ud t \biggr]
 &\leq C \biggl( 1+ {\mathbb E}^0 
\int_0^T 
 \sup_{x_0 \in {\mathbb R}^d} \sup_{\mu \in {\mathcal P}({\mathbb T}^d)}
\| f_r(x_0,\cdot,\mu) - F(\cdot,\mu) \|_{\infty} 
\ud r \biggr).
\end{split}
\end{equation}

\noindent \textit{Third Step.} 
We provide a similar estimate for the difference of the two backward components $u-\tilde u$. We have
\begin{equation*} 
\begin{split} 
\ud_t \bigl( u_t - \tilde u_t\bigr)(x) &=
\biggl[
- \tfrac12 \Delta_x \bigl( u_t - \tilde u_t \bigr)(x) + H\bigl(x,\nabla_x u_t(x)\bigr) - H\bigl(x,\nabla_x \tilde u_t(x)\bigr) 
- \bigl( f_t(X_t^0,x,\mu_t) - F(x,\tilde \mu_t) \bigr) 
\biggr] \ud t 
\\
&\hspace{15pt} + \sigma_0 v_t^0(x) \cdot \ud B_t^0. 
\end{split}
\end{equation*} 
Rewriting the difference of the two Hamiltonians as 
\begin{equation*} 
\begin{split}
&H\bigl(x,\nabla_x u_t(x)\bigr) - H\bigl(x,\nabla_x \tilde u_t(x)\bigr) 
\\
&= \biggl( \int_0^1 \nabla_p H \bigl( x, \theta \nabla_x u_t(x) + (1-\theta) \nabla_x \tilde u_t(x) \bigr) \ud \theta \biggr)
\cdot \bigl( \nabla_x u_t(x) - \nabla_x \tilde u_t(x) \bigr), 
\end{split} 
\end{equation*} 
we apply Lemma 
\ref{le:appendix:2} with 
\begin{equation}
\label{eq:b:Hu-Htildeu}
b_t(x) = - \biggl( \int_0^1 \nabla_p H \bigl( x, \theta \nabla_x u_t(x) + (1-\theta) \nabla_x \tilde u_t(x) \bigr) \ud \theta \biggr). 
\end{equation} 
For a smooth initial condition $q$ satisfying $\int_{{\mathbb T}^d} q(x) \ud x=0$
and for an initial time $t_0 \in [0,T]$, we solve the Fokker-Planck equation
\begin{equation*} 
\partial_t q_t - \tfrac12 \Delta_x q_t + \textrm{\rm div}_x \bigl( b_t(\cdot) q_t \bigr) =0, \quad t \in [t_0,T] \ ; \quad q_{t_0}=q. 
\end{equation*}
We observe that this equation is a random conservative equation, due to the random nature of 
$(b_t)_{0 \leq t\leq T}$.  
By decomposing $q$ in positive and negative parts, it suffices to 
construct solutions for $q$ a probability density. 
Following the proof of Lemma \ref{lem:alpha:weak:solution}
and  
the first step of the proof of Lemma
\ref{lem:weak:uniqueness:forward:equation},
we deduce that the solution is measurable with respect to $\omega^0$: 
$(q_t)_{0 \le t \le T}$ can be regarded as an ${\mathbb F}^0$-progressively measurable process with values in the space of finite signed measures on ${\mathbb T}^d$. 
Then, by duality, we expand
(see \cite[Lemma 4.3.11]{CardaliaguetDelarueLasryLions}) 
\begin{equation}
\label{eq:duality:1st:time}
\begin{split}
\ud_t \bigl( u_t - \tilde u_t,q_t \bigr) &= 
- \bigl( f_t(X_t^0,\cdot,\mu_t) - F(\cdot,\tilde \mu_t), q_t \bigr) \ud t + \ud n_t, \quad t \in [t_0,T], 
\end{split}
\end{equation} 
where $(n_t)_{t_0 \le t \le T}$ is a generic martingale term. 
We obtain, for a constant $C$ depending on the parameters in \hyp{B} except $(\sigma_0,T)$,
\begin{equation*} 
\begin{split}
\bigl( u_{t_0} - \tilde{u}_{t_0}, q \bigr) &\leq 
{\mathbb E}^0 \biggl[ \| q_T \|_{-\curss}\sup_{x_0 \in {\mathbb R}^d} \sup_{\mu \in {\mathcal P}({\mathbb T}^d)}
  \| g(x_0,\cdot,\mu) \|_{\curss} 
  + C \int_{t_0}^T \| q_t \|_{-\curss} \mathbb W_1(\mu_t,\tilde \mu_t) \ud t
  \\
  &\hspace{15pt} 
+ \int_{t_0}^T \| q_t \|_{-\curss} \sup_{x_0 \in {\mathbb R}^d} \sup_{\mu \in {\mathcal P}({\mathbb T}^d)}
\| f_t(x_0,\cdot,\mu) - F(\cdot,\mu) \|_{\curss} \ud t \, \vert \, {\mathcal F}_{t_0}^0 \biggr]. 
\end{split}
\end{equation*} 
By Lemma 
\ref{le:appendix:2} (which is stated for deterministic conservative equations but which applies here because 
$b$ is bounded by a deterministic constant), we can find new values of $C$ and $\gamma$ such that $\| q_t \|_{-\curss} \leq C \exp(-\gamma (t-t_0)) \| q \|_{-\curss}$. Therefore, 
choosing $q$ as 
$q(\cdot) = (-1)^l  \partial^l \rho(\cdot-x)$
and 
then
as
$q(\cdot) = (-1)^l [ \partial^l \rho(\cdot-x)-   \partial^l \rho(\cdot-x')]$, for $x,x' \in {\mathbb T}^d$ and $\rho$ a smooth density
on ${\mathbb R}^d$, 
where $\partial^l \rho$ denotes the derivative of $\rho$ along $l$ (arbitrary) directions of ${\mathbb R}^d$ (with possible repetitions), 
and then letting $\rho$ tend to the Delta mass in $0$, we get 
\begin{equation*} 
\begin{split}
\biggl\|u_{t_0} - \tilde{u}_{t_0} 
- \int_{{\mathbb T}^d} \bigl( u_{t_0} - \tilde{u}_{t_0}  \bigr)(x) \ud x 
\biggr\|_{\curss}
 &\leq C \exp(-\gamma(T-t_0)) + C {\mathbb E}^0 
 \biggl[ \int_{t_0}^T \exp (-\gamma(t-t_0)) \mathbb W_1(\mu_t,\tilde \mu_t) \ud t \, \vert \, {\mathcal F}_{t_0}^0 \biggr]
 \\
&\hspace{15pt} + C \int_{t_0}^T \exp (-\gamma(t-t_0)) \sup_{x_0 \in {\mathbb R}^d} \sup_{\mu \in {\mathcal P}({\mathbb T}^d)}
\| f_t(x_0,\cdot,\mu) - F(\cdot,\mu) \|_{\curss} \ud t. 
\end{split}
\end{equation*} 
Squaring as done in the second step, we obtain (allowing the constant $C$
to vary from line to line as long as it only depends on the various parameters in   
\hyp{B}  except $(\sigma_0,T)$)
\begin{equation*} 
\begin{split}
\Bigl\|u_{t_0} - \tilde{u}_{t_0} 
- \int_{{\mathbb T}^d} \bigl( u_{t_0} - \tilde{u}_{t_0}  \bigr)(x) \ud x 
\Bigr\|_\curss^2
 &\leq C \exp(-\gamma(T-t_0)) + C {\mathbb E}^0 
 \biggl[ \int_{t_0}^T \exp (-\gamma(t-t_0)) \mathbb W_1(\mu_t,\tilde \mu_t)^2 \ud t \, \vert \, {\mathcal F}_{t_0}^0 \biggr] 
 \\
&\hspace{15pt} + C \int_{t_0}^T \exp (-\gamma(t-t_0)) \sup_{x_0 \in {\mathbb R}^d} \sup_{\mu \in {\mathcal P}({\mathbb T}^d)}
\| f_t(x_0,\cdot,\mu) - F(\cdot,\mu) \|_{\curss} \ud t,
\end{split}
\end{equation*} 
where we used the fact that $f_t$ and $F$ are uniformly bounded. Above, 
$\gamma$ is implicitly hidden in the multiplicative constant $C$. Hence, 
by integrating in $t_0$ and using the same Fubini argument as before, we deduce that 
\begin{equation*} 
\begin{split}
&{\mathbb E}^0 \biggl[ \int_0^T \Bigl\| \bigl( u_{t} - \tilde{u}_{t} \bigr)
- \int_{{\mathbb T}^d} \bigl( u_{t} - \tilde{u}_{t} \bigr)(x) \ \ud x \Bigr\|_{\curss}^2 \ud  t 
\biggr]
 \\
 &\hspace{15pt} \leq C   + C {\mathbb E}^0 \Big[\int_{0}^T   \mathbb W_1(\mu_t,\tilde \mu_t)^2 \ud t\Big]
 + C \int_{0}^T \sup_{x_0 \in {\mathbb T}^d} \sup_{\mu \in {\mathcal P}({\mathbb T}^d)}
\| f_t(x_0,\cdot,\mu) - F(\cdot,\mu) \|_{\curss} \ud t.
\end{split}
\end{equation*} 
Substituting 
the last term on the first line by the upper bound in 
\eqref{eq:minor:MFG:ergodicity:forward}, we obtain 
\begin{equation*} 
\begin{split}
{\mathbb E}^0 \biggl[ \int_0^T  \Bigl\| \bigl( u_{t} - \tilde{u}_{t} \bigr)
- \int_{{\mathbb T}^d} \bigl( u_{t} - \tilde{u}_{t} \bigr)(x) \ \ud x \Bigr\|_{\curss}^2 \ud  t 
\biggr]
 &\leq C  + C \int_{0}^T \sup_{x_0 \in {\mathbb T}^d} \sup_{\mu \in {\mathcal P}({\mathbb T}^d)}
\| f_t(x_0,\cdot,\mu) - F(\cdot,\mu) \|_{\curss} \ud t.
\end{split}
\end{equation*} 
The above right-hand side is bounded thanks to \hyp{B2}. 

\noindent \textit{Fourth Step.} 
Replacing the initial time $0$ by a more general stopping time $\tau$ with values in 
$[0,T]$ (as given in the assumption of the lemma), we get in a very similar manner (using conditional expectation given ${\mathcal F}_{\tau}^0$
instead of expectation) that there exists a constant $C$, independent of 
$\sigma_0$, $\tau$ and $T$, such that 
\begin{equation*} 
\begin{split}
{\mathbb P}^0 
\biggl( \biggl\{ 
{\mathbb E}^0 \biggl[ \int_\tau^T \Bigl\| \bigl( u_{t} - \tilde{u}_{t} \bigr)
- \int_{{\mathbb T}^d} \bigl( u_{t} - \tilde{u}_{t} \bigr)(x) \ \ud x
 \Bigr\|_{\curss}^2 \ud  t 
\, \vert \,  {\mathcal F}^0_{\tau} \biggr] \leq C \biggr\} \biggr) = 1. 
\end{split}
\end{equation*}
Reformulating 
\eqref{eq:minor:MFG:ergodicity:forward}
in a similar manner, 
we complete the proof. 
\end{proof} 

%
%
%

An important application is (recall 
\eqref{eq:representation:formula:m:v0} for the definition of the term $v^0$ below) 
\begin{proposition}
\label{prop:bmo:minor}
Under Assumption  \hyp{B}, 
there exist an exponent $\gamma>0$ and a constant $C$, only depending on the parameters in 
\hyp{B} except $(\sigma_0,T)$, such that, for any ${\mathbb F}^0$-stopping time $\tau$ with values in $[0,T]$, 
\begin{equation*} 
{\mathbb P}^0
\biggl( \biggl\{ {\mathbb E}^0 \biggl[ \exp \biggl( \gamma \sigma_0^2 \int_\tau^T \Bigl\| v^0_r(\cdot)
- \int_{{\mathbb T}^d} 
v^0_r(x) \, \ud x 
\Bigr\|^2_{\lfloor \curss \rfloor -(d/2+1)} \ \ud r
\biggr) \, \vert \, {\mathcal F}_\tau^0 \biggr] \leq C \biggr\} \biggr) = 1. 
\end{equation*} 
\end{proposition} 
In the sequel, we let 
\begin{equation}
\label{eq:barv0} 
\bar v_t^0(x) := v_t^0(x) - \int_{{\mathbb T}^d} v_t^0(y) \ud y, \quad 
(t,x) \in [0,T] \times {\mathbb T}^d. 
\end{equation}

\begin{proof}
The proof relies on BMO estimates similar to those used in Section 
\ref{se:2}, but there is a subtlety due to the fact that the argument in 
the exponential involves a functional norm. 

We proceed by Sobolev embeddings. We recall that 
${\mathcal C}^\curss({\mathbb T}^d)$ is included in the Sobolev space 
${\mathcal H}^\curss({\mathbb T}^d)$. 
Also, because
$H^{\curss-1}({\mathbb T}^d)$ embeds continuously in 
${\mathcal C}^{\lfloor \curss \rfloor -(d/2+1)}({\mathbb T}^d)$, we obtain
\begin{equation*}
\begin{split}  
\bigl\| \bar v^0_r
 \bigr\|^2_{\lfloor \curss \rfloor-(d/2+1)}
&\leq 
C \bigl\| \bar v^0_r
\bigr\|^2_{H^{\curss-1}({\mathbb T}^d)}
=
C  \bigl\| \bar v^0_r
\bigr\|^2_{\curss-1,2}
 = C  \sum_{k \in {\mathbb Z}^d \setminus \{0\}} 
\bigl( 1 + \vert k \vert^2 \bigr)^{\curss-1}
\vert (  v^0_r,e_k ) \vert^2,
\end{split} 
\end{equation*} 
where we recall $(e_k)_{k \in {\mathbb Z}^d}$ is the standard (complex valued) Fourier basis
of ${\mathbb T}^d$
and $( \cdot , \cdot)$ denotes here the standard inner product in $L^2({\mathbb T}^d)$ (understood coordinate-wise since 
$v^0$ takes values in ${\mathbb R}^d$). 
Therefore, 
for $\tau$ as in the statement, 
\begin{equation}
\label{eq:fourier:bmo}
\begin{split} 
  {\mathbb E}^0 \biggl[ \int_\tau^T 
\big\| \bar v^0_r  \big\|^2_{\lfloor \curss \rfloor -(d/2+1)} \ud r \, \vert \, {\mathcal F}_\tau^0\biggr]
&\leq C  \sum_{k \neq 0} 
\bigl( 1 + \vert k \vert^2 \bigr)^{\curss-1} 
{\mathbb E}^0 \biggl[ \int_\tau^T 
\vert ( v^0_r,e_k ) \vert^2 \ud r \, \vert \, {\mathcal F}_\tau^0 \biggr].
\end{split}
\end{equation} 
Next, we compute the dynamics of the Fourier coefficients. For $k \in {\mathbb Z}^d$, 
\begin{equation*} 
\begin{split}
\ud_t \bigl(u_t-\tilde u_t,e_k\bigr) &= 2 \pi^2 \vert k \vert^2 \bigl(u_t-\tilde u_t,e_k\bigr) \ud t 
+ \bigl( e_k, H(\cdot,\nabla_x u_t(\cdot)) - H(\cdot,\nabla_x \tilde u_t(\cdot)) 
\bigr) \ud t 
\\
&\hspace{15pt} - \bigl( e_k, f_t(X^0_t,\cdot,\mu_t) - F(\cdot,\tilde \mu_t) \bigr) \ud t
 + \sigma_0 \bigl( v_t^0,e_k\bigr) \ud B_t^0, \quad t \in [0,T]. 
\end{split}
\end{equation*} 
Taking squared norm, we obtain
\begin{equation*} 
\begin{split}
\ud_t \vert \bigl(u_t-\tilde u_t,e_k\bigr) \vert^2&\geq 
 2 \bigl(u_t-\tilde u_t,e_k\bigr)  \bigl( e_k, H(\cdot,\nabla_x u_t(\cdot)) -  H(\cdot,\nabla_x \tilde u_t(\cdot)) 
\bigr) \ud t 
\\
&\hspace{15pt} 
-  2 \bigl(u_t-\tilde u_t,e_k\bigr) \bigl( e_k, f_t(X_t^0,\cdot,\mu_t) - F(\cdot,\tilde \mu_t) \bigr) \ud t
+ \sigma_0^2 \bigl\vert  \bigl( v_t^0,e_k\bigr) \bigr\vert^2 \ud t + \ud n_t, 
\end{split}
\end{equation*} 
where $(n_t)_{0 \le t \le T}$ is a generic martingale term. 
Therefore, 
\begin{equation*} 
\begin{split}
&\sigma^2_0
{\mathbb E}^0 \biggl[ \int_{\tau}^T 
\sum_{k \neq 0} \bigl( 1 + \vert k \vert^2 \bigr)^{\curss-1} \bigl\vert  \bigl( v_t^0,e_k\bigr) \bigr\vert^2 \ud t \, \vert \, {\mathcal F}^0_{\tau} 
\biggr]
\\
&\leq {\mathbb E}^0 \biggl[ 
\sum_{k \neq 0} \bigl( 1 + \vert k \vert^2 \bigr)^{\curss-1} \bigl\vert  \bigl( u_T,e_k\bigr) \bigr\vert^2  \, \vert \, {\mathcal F}^0_{\tau} 
\biggr]
\\
&\hspace{15pt}  -
2 {\mathbb E}^0 \biggl[ 
\int_{\tau}^T 
\sum_{k \neq 0} \bigl( 1 + \vert k \vert^2 \bigr)^{\curss-1} \bigl(u_t-\tilde u_t,e_k\bigr) \bigl( e_k, f_t(X_t^0,\cdot,\mu_t) 
- 
F(\cdot,\tilde \mu_t)
\bigr) \ud t \, \vert \, {\mathcal F}^0_{\tau} 
\biggr]
\\
&\hspace{15pt} +
2 {\mathbb E}^0 \biggl[ 
\int_{\tau}^T 
\sum_{k \neq 0} \bigl( 1 + \vert k \vert^2 \bigr)^{\curss-1}\bigl(u_t-\tilde u_t,e_k\bigr)    \bigl( e_k, H(\cdot,\nabla u_t(\cdot)) -  H(\cdot,\nabla \tilde u_t(\cdot)) 
\bigr) \ud t \, \vert \, {\mathcal F}^0_{\tau} 
\biggr].
\end{split}
\end{equation*} 
And then, 
by 
\eqref{eq:fourier:bmo} and for a constant $C$ as in the statement,
\begin{equation*} 
\begin{split}
&\sigma_0^2 {\mathbb E}^0 \biggl[ \int_\tau^T 
\big\| \bar v^0_r \big\|^2_{\lfloor \curss \rfloor -(d/2+1)} \ud r \, \vert \, {\mathcal F}_\tau^0\biggr]
\\
&\leq C
{\mathbb E}^0 \bigl[ \| u_T \|^2_{\curss-1,2}  \, \vert \, {\mathcal F}^0_{\tau} 
\bigr]
\\
&\hspace{15pt} 
+
C {\mathbb E}^0 \biggl[ 
\int_{\tau}^T  
\Bigl\| \bigl(u_t-\tilde u_t\bigr) - \int_{{\mathbb T}^d} 
\bigl(u_t-\tilde u_t\bigr)(x) \, \ud x
\Bigr\|_{\curss-1,2}
\| f_t(X_t^0,\cdot,\mu_t) - F(\cdot,\tilde \mu_t) 
\|_{\curss-1,2}
 \ud t \, \vert \, {\mathcal F}^0_{\tau} 
\biggr]
\\
&\hspace{15pt} + 
C {\mathbb E}^0 \biggl[ 
\int_{\tau}^T 
\Bigl\| \bigl(u_t-\tilde u_t\bigr) - \int_{{\mathbb T}^d} 
\bigl(u_t-\tilde u_t\bigr)(x) \, \ud x
\Bigr\|_{\curss-1,2}
\Bigl\| 
H(\cdot,\nabla_x u_t(\cdot)) -  H(\cdot,\nabla_x \tilde u_t(\cdot)) 
\Bigr\|_{\curss-1,2}  \ud t \, \vert \, {\mathcal F}^0_{\tau} 
\biggr]
\\
&\leq 
 C
{\mathbb E}^0 \bigl[ \| u_T \|^2_{\curss-1}  \, \vert \, {\mathcal F}^0_{\tau} 
\bigr]
+
C  {\mathbb E}^0 \biggl[ 
\int_{\tau}^T  
\|
 f_t(X_t^0,\cdot,\mu_t) - F(\cdot,\tilde \mu_t) 
\|_{\curss-1}^2
 \ud t \, \vert \, {\mathcal F}^0_{\tau} 
\biggr]
\\
&\hspace{15pt} 
+
C {\mathbb E}^0 \biggl[ 
\int_{\tau}^T  
\Bigl\| \bigl(u_t-\tilde u_t\bigr) - \int_{{\mathbb T}^d} 
\bigl(u_t-\tilde u_t\bigr)(x) \, \ud x
\Bigr\|_{\curss-1}^2
\ud t \, \vert \, {\mathcal F}^0_{\tau} 
\biggr]
\\
&\hspace{15pt} + 
C {\mathbb E}^0 \biggl[ 
\int_{\tau}^T 
\Bigl\| 
H(\cdot,\nabla_x u_t(\cdot)) -  H(\cdot,\nabla_x \tilde u_t(\cdot)) 
\Bigr\|_{\curss-1}^2  \ud t \, \vert \, {\mathcal F}^0_{\tau} 
\biggr].
\end{split}
\end{equation*} 
The key point to treat the very last term is to observe that 
\begin{equation*} 
\begin{split}
H(\cdot,\nabla_x u_t(\cdot)) -  H(\cdot,\nabla_x \tilde u_t(\cdot))  = - b_t(\cdot) \cdot \bigl( \nabla_x u_t - \nabla_x \tilde u_t \bigr)(\cdot),
\end{split}
\end{equation*}
with $b_t(\cdot)$  as in 
\eqref{eq:b:Hu-Htildeu}. 
By Proposition 
\ref{prop:minor:higher} (with a similar result holding true for the derivatives of $\tilde u$), 
we have a bound for $\| b_t \|_{\curss-1}$, independently of $(\sigma_0,T)$. Thanks to this, 
\begin{equation*} 
\Bigl\| 
H(\cdot,\nabla_x u_t(\cdot)) -  H(\cdot,\nabla_x \tilde u_t(\cdot)) 
\Bigr\|_{\curss-1}
\leq C \bigl\| 
\nabla_x (u_t -    \tilde u_t )
\bigr\|_{\curss-1}
\leq C \Bigl\| \bigl(u_t-\tilde u_t\bigr) - \int_{{\mathbb T}^d} 
\bigl(u_t-\tilde u_t\bigr)(x) \, \ud x
\Bigr\|_{\curss}. 
\end{equation*}
And then, using the Lipschitz property of $F$ w.r.t. $\mu$ in ${\mathcal C}^{\curss}$, we obtain 
\begin{equation*} 
\begin{split}
&\sigma_0^2  {\mathbb E}^0 \biggl[ \int_\tau^T 
\big\| \bar v^0(r,\cdot) \big\|^2_{\lfloor \curss \rfloor -(d/2+1)} \ud r \, \vert \, {\mathcal F}_\tau^0\biggr]
\\
&\leq 
 C
\sup_{x_0 \in {\mathbb R}^d} \sup_{\mu \in {\mathcal P}({\mathbb T}^d)}
\| g(x_0,\cdot,\mu) \|_{\curss}
+
C  
\int_{\tau}^T  
\sup_{x_0 \in {\mathbb R}^d} \sup_{\mu \in {\mathcal P}({\mathbb T}^d)}
\| f_t(x_0,\cdot,\mu) - F(\cdot,\mu) \|_{\curss}
 \ud t 
\\
&\hspace{15pt} 
+
C  {\mathbb E}^0 \biggl[ 
\int_{\tau}^T  
\mathbb W_1(\mu_t,\tilde \mu_t)^2 
 \ud t \, \vert \, {\mathcal F}^0_{\tau} 
\biggr]
+
C {\mathbb E}^0 \biggl[ 
\int_{\tau}^T  
\Bigl\| \bigl(u_t-\tilde u_t\bigr) - \int_{{\mathbb T}^d} 
\bigl(u_t-\tilde u_t\bigr)(x) \, \ud x
\Bigr\|_{\curss}^2
\ud t \, \vert \, {\mathcal F}^0_{\tau} 
\biggr],
\end{split}
\end{equation*} 
where we removed the square in the $f_t-F$ term by using boundedness of the latter. 
Thanks to 
 \hyp{B2}. 
and Proposition \ref{prop:minor:regularity:higher}, we can bound the 
left-hand side by $C$ (for a possibly new value of $C$). 
The conclusion follows from the theory of BMO martingale, see 
\cite[Theorem 2.2]{Kazamaki}. 
\end{proof}

\begin{remark}
\label{rem:BMO:v0}
The same proof shows that
there exist an exponent $\gamma_T>0$ and a constant $C_T$, possibly depending on $T$, such that, for any ${\mathbb F}^0$-stopping time $\tau$ with values in $[0,T]$, 
\begin{equation*} 
{\mathbb P}^0
\biggl( \biggl\{ {\mathbb E}^0 \biggl[ \exp \biggl( \gamma_T \sigma_0^2 \int_\tau^T \bigl\| v^0_r 
 \bigr\|^2_{\lfloor \curss \rfloor -(d/2+1)} \ \ud r
\biggr) \, \vert \, {\mathcal F}_\tau^0 \biggr] \leq C_T \biggr\} \biggr) = 1. 
\end{equation*} 
The only difference is that, without the additional centring, we loose the exponential 
decay in the second and third steps of the proof of 
Proposition 
\ref{prop:minor:regularity:higher}. 
\end{remark}

\subsection{Forward-backward SDE for the major player} 

We here provide similar results for the system 
\eqref{eq:major:FB:1}
satisfied by the major player, using now the following auxiliary deterministic HJB equation:
\begin{equation} 
\label{eq:auxiliary:HJB}
\begin{split}
&\partial_t w^0(t,x_0) +  \tfrac{1}2 \sigma_0^2 \Delta_{x_0} w^0(t,x_0) + F^0(x_0) - H^0 \bigl( x_0,\nabla_{x_0} w^0(t,x_0) \bigr) =0, 
\quad (t,x_0) \in [0,T] \times {\mathbb R}^d,
\\ 
&w^0(T,x_0) = 0. 
\end{split}
\end{equation}  

Existence and uniqueness of a (classical) solution is standard. 
The key point is that, by 
\cite[Theorem 1.20]{TranHJB}, we have a bound for $\sup_{(t,x_0) \in  [0,T] \times {\mathbb R}^d} \vert \nabla_{x_0} w^0
(t,x_0)\vert$ that only depends on the parameters in  \hyp{B},
except $T$. 
Using the fact that $F^0$ has bounded first-order derivatives, we deduce from 
standard regularization properties of the 
heat kernel that $\sup_{(t,x_0) \in  [0,T] \times {\mathbb R}^d} \vert \nabla^2_{x_0} w^0
(t,x_0)\vert$ is bounded with a bound that only depends on the parameters in  \hyp{B},
except $T$. Notice that when $F^0 \equiv 0$ and $H^0(\cdot,0) \equiv 0$, then $w^0 \equiv 0$.

\begin{lemma} 
\label{lem:BMO:major}
Under Assumption  \hyp{B}, for any ${\mathbb F}^0$-adapted continuous path 
 $(\mu_t)_{0 \le t \leq T}$ with values in ${\mathcal P}({\mathbb T}^d)$
 and  any ${\mathbb F}^0$-stopping time $\tau$ with values in $[0,T]$, any solution to  
 the 
forward-backward SDE
\eqref{eq:major:FB:1} (in the sense of item 1 in Definition 
\ref{def:forward-backward=MFG:solution}) satisfies 
 \begin{equation} 
 \label{eq:BMO:major}
 {\mathbb P}^0
\biggl( \biggl\{
{\mathbb E}^0 \biggl[ \int_\tau^T \sigma_0^2\vert Z_r^0 - \nabla_{x_0} w^0(X_t^0) \vert^2 \ud r \, \vert \, {\mathcal F}_\tau^0 \biggr] \leq C \biggr\} 
\biggr) = 1, 
 \end{equation} 
 for a constant $C$ only depending on the parameters  
 in 
\hyp{B}
 except the parameters $(\sigma_0,T)$. 
 \end{lemma}
\begin{proof}
Using smoothness of $w^0$, we can expand $(Y_t^0 - w^0(t,X_t^0))_{0 \leq t \leq T}$ by means of It\^o's formula. We get
\begin{equation*} 
\begin{split} 
\ud \bigl[ Y_t^0 - w^0(t,X_t^0) \bigr]
&= - \bigl( f_t^0(X_t^0,\mu_t) - F^0(X_t^0) \bigr) 
\ud t 
\\
&\hspace{-25pt} - \Bigl( L^0\bigl(X_t^0, - \nabla_p H^0(X_t^0,Z_t^0)\bigr) + H^0\bigl( X_t^0,\nabla_{x_0} w^0(t,X_t^0)\bigr) 
- \nabla_p H^0(X_t^0,Z_t^0) \cdot \nabla_{x_0} w^0(t,X_t^0) \Bigr) 
\ud t 
\\
&\hspace{-25pt} + \sigma_0 \bigl( Z_t^0 - \nabla_{x_0} w^0(t,X_t^0) \bigr) \cdot \ud B_t^0, \quad t \in [0,T]. 
\end{split}
\end{equation*} 
We recall the standard formula
(see \eqref{eq:Hamiltonians:Fenchel}) 
\begin{equation}\label{eq:H0}
H^0(x_0,p_0) = p_0 \cdot \nabla_p H^0(x_0,p_0) - L^0\bigl(x_0,- \nabla_p H^0(x_0,p_0)\bigr),
\end{equation}
from which we obtain (using $p_0=Z_t^0$)
\begin{equation*} 
\begin{split} 
&\ud \bigl[ Y_t^0 - w^0(t,X_t^0)   \bigr]
=  - \bigl( f_t^0(X_t^0,\mu_t) - F^0(X_t^0) \bigr) 
\ud t 
\\
&\hspace{15pt} - \Bigl( - H^0(X_t^0,Z_t^0) + H^0\bigl( X_t^0,\nabla_{x_0} w^0(t,X_t^0)\bigr) 
- \nabla_p H^0(X_t^0,Z_t^0) \cdot \bigl[ \nabla_{x_0} w^0(t,X_t^0) - Z_t^0 \bigr] \Bigr) 
\ud t 
\\
&\hspace{15pt} + \sigma_0 \bigl( Z_t^0 - \nabla_{x_0} w^0(t,X_t^0) \bigr) \cdot \ud B_t^0, \quad t \in [0,T]. 
\end{split}
\end{equation*} 
Here, we write 
\begin{equation*} 
\begin{split}   
&H^0\bigl( X_t^0,\nabla_{x_0} w^0(t,X_t^0)\bigr) 
- H^0(X_t^0,Z_t^0) 
- \nabla_p H^0(X_t^0,Z_t^0) \cdot \bigl[ \nabla_{x_0} w^0(t,X_t^0) - Z_t^0 \bigr]  
\\
&= \int_0^1 \bigl[ \nabla_p H^0\bigl(X_t^0,\theta \nabla_{x_0} w^0(t,X_t^0) + (1-\theta) Z_t^0\bigr) 
- \nabla_p H^0(X_t^0, Z_t^0) 
\bigr] \cdot \bigl( \nabla_{x_0} w^0(t,X_t^0) - Z_t^0 \bigr) \ud \theta
\\
&= \bigl[ A_t^0
\bigl( \nabla_{x_0} w^0(t,X_t^0) - Z_t^0 \bigr) 
\bigr] \cdot \bigl( \nabla_{x_0} w^0(t,X_t^0) - Z_t^0 \bigr), 
\end{split}
\end{equation*} 
where we have let 
\begin{equation*}
A_t^0 := \int_0^1 \int_0^1   \nabla^2_{pp} H^0\bigl(X_t^0,\theta \varphi \nabla_{x_0} w^0(t,X_t^0) + (1-\varphi \theta) Z_t^0\bigr)   \ud \theta \ud \varphi.
\end{equation*} 
So, with the notations
\begin{equation*} 
P_t^0 := Y_t^0 - w^0(t,X_t^0), \quad Q_t^0 := Z_t^0 - \nabla_{x_0} w^0(t,X_t^0), 
\end{equation*} 
we obtain the following BSDE: 
\begin{equation*} 
\ud P_t^0 = \Bigl[   - \bigl(f_t^0(X_t^0,\mu_t) - F^0(X_t^0)\bigr) - (A_t^0 Q_t^0) \cdot Q_t^0 \Bigr] \ud t + 
\sigma_0 Q_t^0 \cdot \ud B_t^0, \quad t \in [0,T]. 
\end{equation*} 
Using Girsanov transformation, 
it is standard  (see 
 \cite{BRIAND20132921} and the references therein about quadratic BSDEs)
to deduce that 
\begin{equation} 
\label{eq:bound:P0}
\begin{split} 
\vert P_t^0 \vert &\leq 
\sup_{x_0 \in {\mathbb R}^d} \sup_{\mu \in {\mathcal P}({\mathbb T}^d)}
\vert g^0(x_0,\mu) \vert
+
\int_{t}^T  
\sup_{x_0 \in {\mathbb R}^d} \sup_{\mu \in {\mathcal P}({\mathbb T}^d)}
| f_r^0(x_0,\mu) - F^0(\mu) |
 \ud r 
\leq  C,
\end{split} 
\end{equation} 
with $C$ independent of $(\sigma_0,T)$, see \hyp{B2}. 

We 
now 
let ${\boldsymbol E}^0:= (E_t^0 := \exp(\nu P_t^0))_{0 \le t \le T}$, for some  parameter $\nu>0$. 
Notice  from \eqref{eq:bound:P0} that we have a global bound for both ${\boldsymbol P}^0$ and ${\boldsymbol E}^0$. This bound is independent
of $\sigma_0$ and $T$ (and of $\tau$). 
We then expand 
\begin{equation*}
\begin{split} 
\ud E_t^0 &= \nu E_t^0 \ud P_t^0 + \tfrac{1}2 \nu^2 E_t^0 \ud \langle P^0 \rangle_t  
\\
&= \nu E_t^0 \Bigl[ 
 - \bigl(f_t^0(X_t^0,\mu_t) - F^0(X_t^0)\bigr) - (A_t^0 Q_t^0) \cdot Q_t^0 \Bigr] \ud t
 + \tfrac1{2} \nu^2 \sigma_0^2 E_t^0 \vert Q_t^0 \vert^2 \ud t 
+ \nu \sigma_0 E_t^0 Q_t^0 \cdot \ud B_t^0. 
\end{split}
\end{equation*} 
Recall that $\sigma_0 \geq 1$, choose $\nu \geq 2 (\| \nabla^2_{pp} H^0 \|_\infty+1)$ and deduce that, 
for any ${\mathbb F}^0$-stopping time $\tau$ with values in $[0,T]$, 
\begin{equation*} 
\begin{split} 
 \sigma_0^2 {\mathbb E}^0 \biggl[ \int_{\tau}^T E_t^0 \vert Q_t^0 \vert^2 \ud t 
 \, \vert \, {\mathcal F}_\tau^0 \biggr] \leq 
  {\mathbb E}^0 \biggl[
   E_T^0 + \int_\tau^T  E_t^0 \bigl\vert f_t^0(X_t^0,\mu_t) - F^0(X_t^0)\bigr\vert 
  \ud t  \, \vert \, {\mathcal F}_\tau^0 \biggr]. 
    \end{split}
\end{equation*} 
Thanks to the bounds for ${\boldsymbol P}^0$ and ${\boldsymbol E}^0$ and to \hyp{B2} again,
we obtain 
\begin{equation*} 
\begin{split} 
 \sigma_0^2 {\mathbb E}^0 \biggl[ \int_{\tau}^T E_t^0 \vert Q_t^0 \vert^2 \ud t 
 \, \vert \, {\mathcal F}_\tau^0 \biggr] &\leq 
  C,
      \end{split}
\end{equation*} 
with the constant $C$ being independent of $\sigma_0$ and $T$. 
\end{proof} 

Bound
 \eqref{eq:BMO:major}
is a BMO bound. It ensures that the Doléans-Dade exponential 
\begin{equation*}
\begin{split}
{\mathcal E}_t^0 &:= 
{\mathscr E}_t \biggl( \int_0^\cdot \sigma_0^{-1} \bigl( \nabla_p H^0(X_r^0, Z_r^0)
- \nabla_p H^0(X_r^0, \nabla_{x_0} w^0(r,X_r^0)) \bigr) \cdot \ud B_r^0 \biggr), \quad t \in [0,T], 
\end{split} 
\end{equation*} 
is a martingale with respect to ${\mathbb F}^0$, see \cite[Theorem 2.3]{Kazamaki}. In particular, the measure 
\begin{equation}
\label{eq:widetilde:P0}
\widetilde{\mathbb P}^0 := {\mathcal E}_T^0 \cdot {\mathbb P}^0
\end{equation} 
is a probability measure. In fact, the very benefit of 
Lemma \ref{lem:BMO:major}
is to provide an $L^{1+}$ bound for 
${\boldsymbol {\mathcal E}}^0$, 
independently of $\sigma_0$ and $T$. Indeed, from  
\cite[Theorem 3.1]{Kazamaki} (together with the fact that $\sigma_0 \geq 1$), we claim
\begin{lemma}
\label{lem:Girsanov:major}
Under Assumption \hyp{B}, 
there exist two constants $\gamma_0 >0$ and $C_0 \geq 0$, only depending on the parameters in \hyp{B} except $(\sigma_0,T)$ such that, for any ${\mathbb F}^0$-stopping time $\tau$ with values in $[0,T]$,  
\begin{equation*}
{\mathbb P}^0 \biggl( \biggl\{ {\mathbb E}^0 \Bigl[ \bigl( {\mathcal E}_T^0 ({\mathcal E}_\tau^0)^{-1} \bigr)^{1+\gamma_0} \, \vert \, {\mathcal F}_{\tau}^0 \Bigr] \leq C_0 \biggr\} \biggr) = 1,
\end{equation*} 
and 
\begin{equation*}
{\mathbb P}^0 \biggl( \biggl\{ 
{\mathbb E}^0 \biggl[  
\exp \biggl( \gamma_0 \sigma_0^2 \int_\tau^T \vert Z_r^0 - \nabla_{x_0} w^0(r,X^0_r) \vert^2 \ud r \biggr) 
\, \vert \, {\mathcal F}_{\tau}^0 \biggr] \leq C_0 \biggr\} \biggr)=1.
\end{equation*} 
\end{lemma}

\section{Weak formulation of the major-minor MFG} 
\label{se:4}

\subsection{Girsanov transformation of the coupled major-minor forward-backward system} 
The preliminary estimates we obtained in the previous section allow us to make 
a change of measure that leads to a new formulation of the major-minor stochastic MFG coupled system, 
which we call weak formulation (precisely because it is formulated on a new probability space depending on 
the solution itself). On an arbitrary 
filtered probability space 
$(\Omega^0,{\mathcal F}^0,{\mathbb F}^0,{\mathbb P}^0)$
satisfying the usual conditions and 
equipped with a 
Brownian motion $(B_t^0)_{0 \leq t \leq T}$ with values in ${\mathbb R}^d$, we introduce the tilted Brownian motion
\begin{equation*}
\widetilde{B}_t^0 := B_t^0 - \sigma_0^{-1} \int_0^t \bigl[ \nabla_p H^0(X_r^0,Z_r^0)
- 
\nabla_p H^0(X_r^0,\nabla_{x_0} w(r,X_r^0)) 
\bigr]
 \ud r, \quad t \in [0,T].  
\end{equation*} 
Thanks to Lemma \ref{lem:BMO:major}, we know that $\widetilde{\boldsymbol B}^0 = (\widetilde B_t^0)_{0 \le t \le T}$ is a Brownian motion under 
$\widetilde{\mathbb P}^0$ defined in 
\eqref{eq:widetilde:P0}, and then 
\eqref{eq:major:FB:1}--\eqref{eq:minor:FB:2} become
\begin{equation} 
\label{eq:major:FBG:1:t}
\begin{split} 
\ud X_t^0 &= - \nabla_p H^0(X_t^0,\nabla_{x_0} w^0(X_t^0)) \ud t +  \sigma_0 \ud \widetilde B_t^0, 
\\
\ud Y_t^0 &=  \Bigl[
  - f^0_t(X_t^0,\mu_t) - L^0(X_t^0,- \nabla_p H^0(X_t^0,Z_t^0)) 
\\
&\hspace{15pt} + 
Z_t^0 \cdot \bigl( \nabla_p H^0(X_t^0,Z_t^0) - \nabla_p H^0(X_t^0,\nabla_{x_0} w^0(X_t^0)) \bigr)
\Bigr] \ud t
 + \sigma_0 Z_t^0 \cdot \ud \widetilde B_t^0, \quad t \in [0,T], 
\\
Y_T^0 &= g^0(X_T^0,\mu_T), 
\end{split}
\end{equation} 
and
\begin{equation}
\label{eq:minor:FBG:2:t}
\begin{split} 
&\partial_t \mu_t - \tfrac12 \Delta_x \mu_t - {\rm div}_x \Bigl( \nabla_p H(\cdot,\nabla_x u_t(\cdot))  \mu_t \Bigr) =0, 
\quad {\rm on} \ (0,T) \times {\mathbb T}^d, 
\\
&\ud_t u_t(x) = \Bigl( -  \tfrac12 \Delta_x u_t(x) + H(x, \nabla_x u_t(x))- f_t(X_t^0,x,\mu_t) 
\Bigr) \ud t
\\
&\hspace{20pt}
+ \Bigl( \nabla_p H^0(X_t^0,Z_t^0) - \nabla_p H^0(X_t^0,\nabla_{x_0} w^0(X_t^0))  \Bigr)\cdot v_t^0(x) 
\ud t 
+ \sigma_0 v_t^0(x) \cdot \ud \widetilde B_t^0, \quad (t,x) \in [0,T] \times {\mathbb T}^d, 
\\
&u_T(x) = g(X_T^0,x,\mu_T), \quad x \in {\mathbb T}^d.  
\end{split}
\end{equation}
This is the system that we use below to get estimates on the values 
of the Major/Minor MFG.
The very interest is that the forward equation of the major player is now independent of the remaining three equations.

It is important to observe that the Girsanov transformation \eqref{eq:widetilde:P0} does not impact 
the conclusion of Proposition \ref{prop:bmo:minor}: provided we change $\gamma$ and $C$ in the statement, the bound therein remains true under the tilted measure. 
This follows from the fact that, 
with $\gamma_0$ as in Lemma
\ref{lem:Girsanov:major}, 
for any ${\mathbb F}^0$-stopping time $\tau$ with values in $[0,T]$,
\begin{equation*} 
\begin{split}
&\widetilde{\mathbb E}^0 \biggl[ \exp \biggl( \gamma \frac{\gamma_0}{1+\gamma_0} \sigma_0^2 \int_\tau^T 
\bigl\| \bar v_r^0  \bigr\|^2_{\lfloor \curss \rfloor -(d/2+1)} \ud r
\biggr) \, \vert \, {\mathcal F}_\tau^0 \biggr]
\\
&=  {\mathbb E}^0 \biggl[ {\mathcal E}_T {\mathcal E}_\tau^{-1} \exp \biggl( \gamma \frac{\gamma_0}{1+\gamma_0} \sigma_0^2 \int_\tau^T 
\bigl\| \bar v_r^0  \bigr\|^2_{\lfloor \curss \rfloor-(d/2+1)}  \ud r
\biggr) \, \vert \, {\mathcal F}_\tau^0 \biggr]
\\
&\leq 
{\mathbb E}^0 \Bigl[
\Bigl( 
 {\mathcal E}_T {\mathcal E}_\tau^{-1} \Bigr)^{1+\gamma_0} \, \vert \, {\mathcal F}_\tau^0 \Bigr]^{1/(1+\gamma_0)}
{\mathbb E}^0 \biggl[  \exp \biggl( \gamma   \sigma_0^2 \int_\tau^T \bigl\| \bar v_r^0  \bigr\|^2_{\lfloor \curss \rfloor-(d/2+1)}  \ud r
\biggr) \, \vert \, {\mathcal F}_\tau^0 \biggr]^{\gamma_0/(1+\gamma_0)}
\\
&\leq C_0^{1/(1+\gamma_0)} C^{\gamma_0/(1+\gamma_0)},
\end{split}
\end{equation*} 
with $C_0$ as in 
Lemma 
\ref{lem:Girsanov:major}. 

Letting 
\begin{equation*} 
\tilde \gamma_0 :=  \gamma_0 \frac{\gamma_0}{1+\gamma_0}, \quad 
\tilde \gamma :=  \gamma\frac{\gamma_0}{1+\gamma_0}, 
\quad  \tilde C:= C_0^{1/(1+\gamma_0)} C^{\gamma_0/(1+\gamma_0)}, 
\end{equation*} 
we obtain (proceeding similarly for the second one below), for any ${\mathbb F}^0$-stopping time $\tau$ with values in $[0,T]$,  
\begin{equation}
\label{eq:Novikov} 
\begin{split}  
&\widetilde{\mathbb E}^0 \biggl[ \exp \biggl( \tilde \gamma  \sigma_0^2 \int_\tau^T \bigl\|
\bar v^0_r \bigr\|^2_{\lfloor \curss \rfloor -(d/2+1)} \ud r
\biggr) \, \vert \, {\mathcal F}_\tau^0 \biggr] \leq \tilde C, 
\\
&\widetilde{\mathbb E}^0 \biggl[ \exp \biggl( \tilde \gamma_0  \sigma_0^2 \int_\tau^T \vert Z_r^0 - \nabla_{x_0} w^0(X_r^0) \vert^2 \ud r
\biggr) \, \vert \, {\mathcal F}_\tau^0 \biggr] \leq \tilde{C}_0.
\end{split} 
\end{equation}

\subsection{Strong formulation of the tilted system and related linearized system}
We now come to the heart of the paper. 
We study  the tilted system when posed in the strong form, meaning on the original probability space, with 
the given ${\boldsymbol B}^0$ as driving Brownian motion. 
In clear, we directly address
\begin{equation} 
\label{eq:major:FBG:1}
\begin{split} 
&\ud X_t^0 = - \nabla_p H^0(X_t^0,\nabla_{x_0} w^0(X_t^0)) \ud t +  \sigma_0 \ud  B_t^0, 
\\
&\ud Y_t^0 =  \Bigl[
  - f^0_t(X_t^0,\mu_t) - L^0(X_t^0,- \nabla_p H^0(X_t^0,Z_t^0)) 
\\
&\hspace{15pt} + 
Z_t^0 \cdot \bigl( \nabla_p H^0(X_t^0,Z_t^0) - \nabla_p H^0(X_t^0,\nabla_{x_0} w^0(X_t^0)) \bigr)
\Bigr] \ud t
 + \sigma_0 Z_t^0 \cdot \ud  B_t^0, \quad t \in [0,T], 
\\
&Y_T^0 = g^0(X_T^0,\mu_T), 
\end{split}
\end{equation} 
and
\begin{equation}
\label{eq:minor:FBG:2}
\begin{split} 
&\partial_t \mu_t - \tfrac12 \Delta_x \mu_t - {\rm div}_x \Bigl( \nabla_p H(\cdot,\nabla_x u_t(\cdot))  \mu_t \Bigr) =0, 
\quad {\rm on} \ (0,T) \times {\mathbb T}^d, 
\\
&\ud_t u_t(x) = \Bigl( -  \tfrac12 \Delta_x u_t(x) + H(x, \nabla_x u_t(x))- f_t(X_t^0,x,\mu_t) 
\Bigr) \ud t
\\
&\hspace{20pt}
+ \Bigl( \nabla_p H^0(X_t^0,Z_t^0) - \nabla_p H^0(X_t^0,\nabla_{x_0} w^0(X_t^0))  \Bigr)\cdot v_t^0(x) 
\ud t 
+ \sigma_0 v_t^0(x) \cdot \ud B_t^0, \quad (t,x) \in [0,T] \times {\mathbb T}^d, 
\\
&u_T(x) = g(X_T^0,x,\mu_T), \quad x \in {\mathbb T}^d,
\end{split}
\end{equation}
on the space $(\Omega^0,{\mathcal F}^0,{\mathbb F}^0,{\mathbb P}^0)$. 
\vskip 4pt

The
following statement clarifies the passage from 
strong existence and uniqueness for 
\eqref{eq:major:FB:1}--\eqref{eq:minor:FB:2}
to 
strong existence and uniqueness for 
\eqref{eq:major:FBG:1}--\eqref{eq:minor:FBG:2}:

\begin{proposition}
\label{prop:strong:existence:uniqueness:passage}
Under Assumption \hyp{B}, 
assume further that, on any arbitrary 
probabilistic set-up $(\Omega^0,{\mathcal F}^0,{\mathbb F}^0,{\mathbb P}^0)$ and for a 
fixed initial condition $(x_0,\mu) \in {\mathbb R}^d \times {\mathcal P}({\mathbb T}^d)$, the system 
\eqref{eq:major:FB:1}--\eqref{eq:minor:FB:2}
has a unique solution in the sense of Definition 
\ref{def:forward-backward=MFG:solution}. Assume also that there exist 
two bounded and measurable mappings 
${\mathcal U}^0 : [0,T] \times {\mathbb R}^d \times {\mathcal P}({\mathbb T}^d)
\rightarrow {\mathbb R}$ and 
${\mathcal U} : [0,T] \times {\mathbb R}^d \times {\mathbb T}^d 
\times {\mathcal P}({\mathbb T}^d) \rightarrow
{\mathbb R}$, with ${\mathcal U}^0$ and 
${\mathcal U}$ being  both  
differentiable in the ${\mathbb R}^d$-variable, with 
${\mathcal U}$ being  
differentiable in the ${\mathbb T}^d$-variable 
and with 
${\mathcal U}^0$ and $\nabla_x {\mathcal U}$ being Lipschitz continuous in the ${\mathbb R}^d$ and 
${\mathcal P}({\mathbb T}^d)$-variables uniformly in the other variables, 
 such that, ${\mathbb P}^0$ almost surely,  
\begin{equation} 
\label{eq:representation:weak:solution}
\begin{split} 
&Y_t^0 = {\mathcal U}^0(t,X_t^0,\mu_t),
\
Z_t^0 = \nabla_{x_0} {\mathcal U}^0(t,X_t^0,\mu_t), \quad t \in [0,T], 
\\
&u_t(x) = {\mathcal U}(t,X_t^0,x,\mu_t),
\
v_t^0(x) = \nabla_{x_0} {\mathcal U}(t,X_t^0,x,\mu_t), 
 \quad (t,x) \in [0,T] \times 
{\mathbb T}^d. 
\end{split} 
\end{equation} 
Then, 
on the same (and thus on any) 
probabilistic set-up $(\Omega^0,{\mathcal F}^0,{\mathbb F}^0,{\mathbb P}^0)$ and for 
the same 
initial condition $(x_0,\mu) \in {\mathbb R}^d \times {\mathcal P}({\mathbb T}^d)$, 
the system 
\eqref{eq:major:FBG:1}--\eqref{eq:minor:FBG:2}
has a unique (hence strong) solution
satisfying the requirements of 
Definition 
\ref{def:forward-backward=MFG:solution}
together with 
\begin{equation} 
\label{eq:new:iii}
\sup_{\tau} \biggl\| {\mathbb E}^0 \biggl[ \int_{\tau}^T 
\| v^0_r \|^2_{\lfloor \curss \rfloor - (d/2+1)}
\ud r 
\, \vert \, {\mathcal F}_\tau \biggr]
 \biggr\|_{L^\infty(\Omega^0,{\mathbb P}^0)}
< \infty.
\end{equation} 

\end{proposition}

\begin{remark}
The fact that existence and uniqueness hold true on any 
$(\Omega^0,{\mathcal F}^0,{\mathbb F}^0,{\mathbb P}^0)$
(equipped with a Brownian motion ${\boldsymbol B}^0$) 
implies that
any solution constructed on a sub-filtration of ${\mathbb F}^0$ (still carrying ${\boldsymbol B}^0$)  must coincide with the solution 
constructed on ${\mathbb F}^0$. 
In particular, without any loss of generality, 
${\mathbb F}^0$ can always be taken as the
filtration generated by ${\mathcal F}_0^0$ and 
${\boldsymbol B}^0$. In this way, we can recover the setting of Section \ref{se:3}. 
\end{remark}

\begin{proof} 
Existence of a (weak) solution
to 
\eqref{eq:major:FBG:1}--\eqref{eq:minor:FBG:2}
 follows from 
 the aforementioned Girsanov argument (as before, the word `weak' means that the solution is 
 driven by another Brownian motion, different from the original 
 ${\boldsymbol B}^0$). 
This weak solution satisfies 
the integrability conditions in 
Definition 
\ref{def:forward-backward=MFG:solution}
 because 
 (ii) and (iii) in item 1 in Definition 
\ref{def:forward-backward=MFG:solution}
are preserved by BMO changes of measure 
(see 
\cite[Theorem 3.3]{Kazamaki})
and (ii)  in item 2 is  also preserved.
The bound 
\eqref{eq:new:iii}
is a consequence of Remark \ref{rem:BMO:v0}. 
In fact, the Lipschitz property of ${\mathcal U}^0$ gives here an $L^\infty$ bound 
on ${\boldsymbol Z}^0$ that is even stronger than the BMO condition stated in 
(iii) in item 1 in Definition \ref{def:forward-backward=MFG:solution}. 

 Thanks to (iii) in item 1 of Definition \ref{def:forward-backward=MFG:solution}
 (but for 
\eqref{eq:major:FBG:1}--\eqref{eq:minor:FBG:2}), 
the Girsanov transformation can be reverted, hence proving
weak uniqueness. 

Strong uniqueness is 
proved by means of 
\eqref{eq:representation:weak:solution}. 
By weak uniqueness, any weak solution to 
\eqref{eq:major:FBG:1}--\eqref{eq:minor:FBG:2}
is in fact obtained by changing the Brownian motion 
in the original 
\eqref{eq:major:FB:1}--\eqref{eq:minor:FB:2}. As such, it must
satisfy 
\eqref{eq:representation:weak:solution}. 
Inserting the representation 
\eqref{eq:representation:weak:solution}
into 
the forward equation of 
\eqref{eq:minor:FBG:2}
(which is the same as the forward equation of 
\eqref{eq:minor:FB:2}) and using the fact that 
$\nabla_x {\mathcal U}$ is Lipschitz in the 
${\mathcal P}({\mathbb T}^d)$-argument, one gets (following the proof of Lemma \ref{lem:weak:uniqueness:forward:equation}) that the pair 
$({\boldsymbol X}^0,{\boldsymbol \mu})$ is necessarily progressively-adapted to the 
(usual augmentation of the)
filtration generated by ${\boldsymbol B}^0$ and pathwise unique.
By 
\eqref{eq:representation:weak:solution}, 
we deduce that solutions to 
\eqref{eq:major:FBG:1}--\eqref{eq:minor:FBG:2}
are 
progressively-adapted to the  
filtration generated by ${\boldsymbol B}^0$ and pathwise unique.
 \end{proof}

Existence and uniqueness of a solution to the system 
\eqref{eq:major:FB:1}--\eqref{eq:minor:FB:2}
(which is the main purpose of this section)
 will be eventually established in Subsection \ref{subse:4.6} (using auxiliary results proven in 
Section \ref{se:5}). 
At this stage, we are given a flow of solutions
\begin{center}
$\bigl(({\boldsymbol X}^{0,x_0,\mu},{\boldsymbol Y}^{0,x_0,\mu},{\boldsymbol Z}^{0,x_0,\mu}),({\boldsymbol \mu}^{x_0,\mu},{\boldsymbol u}^{x_0,\mu},{\boldsymbol v}^{0,x_0,\mu})\bigr)_{x_0,\mu}$
\end{center}
to the systems 
\eqref{eq:major:FBG:1}
and
\eqref{eq:minor:FBG:2}, parametrized by $(x_0,\mu) \in {\mathbb T}^d \times {\mathcal P}({\mathbb T}^d)$
($x_0$ should be understood as $X_0^{0,x_0,\mu}$ and $\mu$ as $\mu_0^{x_0,\mu}$). 
 Implicitly, 
these solutions are required to satisfy
the same constraints as in Definition 
\ref{def:forward-backward=MFG:solution}.

Throughout, we 
fix $(x_0,\mu) \in {\mathbb R}^d \times {\mathcal P}({\mathbb T}^d)$
and $(x_0',\mu') \in  {\mathbb R}^d \times {\mathcal P}({\mathbb T}^d)$,
and we consider the system satisfied by the following derivatives, for $t \in [0,T]$:
\begin{equation*} 
\begin{split}
&\bigl( \delta X_t^0,\delta Y_t^0,\delta Z^0 \bigr) = \frac{\ud}{\ud \varepsilon}_{\vert \varepsilon =0} 
\Bigl(
X_t^{0,x_0+\varepsilon (x_0'-x_0),\mu + \varepsilon (\mu'- \mu)},
Y_t^{0,x_0+\varepsilon (x_0'-x_0),\mu + \varepsilon (\mu'- \mu)},
Z_t^{0,x_0+\varepsilon (x_0'-x_0),\mu + \varepsilon (\mu'- \mu)}\Bigr) 
\\
&\bigl( \delta \mu_t ,\delta u_t, \delta v_t^0 \bigr) = \frac{\ud}{\ud \varepsilon}_{\vert \varepsilon =0} 
\Bigl( \mu_t^{x_0+\varepsilon (x_0'-x_0),\mu + \varepsilon (\mu'- \mu)}, 
u_t^{x_0+\varepsilon (x_0'-x_0),\mu + \varepsilon (\mu'- \mu)},
v_t^{0,x_0+\varepsilon (x_0'-x_0),\mu + \varepsilon (\mu'- \mu)} \Bigr) .
\end{split}
\end{equation*}
The sense given to the above derivatives will be clarified below (see in particular the proof of 
Proposition 
\ref{prop:strong:existence:uniqueness:passage2} in Subsection \ref{subse:4.7}). 
What matters now is the form of the system that is satisfied by 
the two triples $(\delta {\boldsymbol X}^0,\delta {\boldsymbol Y}^0,\delta {\boldsymbol Z}^0)$ 
and $(\delta {\boldsymbol \mu},\delta {\boldsymbol u},\delta {\boldsymbol v}^0)$, which we call 
`linearized' processes. 
Just say that the first triple takes values in 
${\mathbb R}^d \times {\mathbb R} \times {\mathbb R}^d$
whilst the second one takes values in a functional 
space, namely ${\boldsymbol \nu}$ is Schwartz-distributional valued and 
${\boldsymbol u}$ and ${\boldsymbol v}^0$ are functional valued.
Throughout, we use the following notation. For a generic (possibly random and say real-valued) function $\Phi$ of (some of) the inputs $(X_t^{0},Z_t^{0},\mu_t)$, 
we denote by 
$\delta[\Phi(X_t^0 ,Z_t^0,\mu_t)]$ the term 
\begin{equation*} 
\delta[\Phi(X_t^{0},Z_t^{0},\mu_t)]
:= \nabla_x \Phi(X_t^{0},Z_t^{0},\mu_t) \cdot \delta X_t^0
+ \nabla_z 
\Phi(X_t^{0},Z_t^{0},\mu_t) \cdot \delta Z_t^0
+ ( \delta_\mu \Phi (X_t^{0},Z_t^{0},\mu_t)(\cdot), \delta \mu_t).
\end{equation*} 
Then, 
using the relationship (see \eqref{eq:Hamiltonians:Fenchel}) 
\begin{equation*} 
\begin{split}
& - L^0(X_t^0,- \nabla_p H^0(X_t^0,Z_t^0)) 
 + 
Z_t^0 \cdot \bigl( \nabla_p H^0(X_t^0,Z_t^0) - \nabla_p H^0(X_t^0,\nabla_{x_0} w^0(t,X_t^0)) \bigr)
\\
&= H^0(X_t^0,Z_t^0)  - \bigl( Z_t^0 - \nabla_{x_0} w^0(t,X_t^0)
\bigr) \cdot \nabla_p H^0(X_t^0,\nabla_{x_0} w^0(t,X_t^0))
\\
&\hspace{15pt} - \nabla_{x_0} w^0(t,X_t^0) \cdot \nabla_p H^0(X_t^0,\nabla_{x_0} w^0(t,X_t^0))
\\
&=  \bigl( H^0(X_t^0,Z_t^0)
- H^0(X_t^0, \nabla_{x_0} w^0(t,X_t^0))
\bigr)  - \bigl( Z_t^0 - \nabla_{x_0} w^0(t,X_t^0)
\bigr) \cdot \nabla_p H^0(X_t^0,\nabla_{x_0} w^0(t,X_t^0))
\\
&\hspace{15pt} - L^0(X_t^0,- \nabla_p H^0(X_t^0,\nabla_{x_0} w^0(t,X_t^0))),
\end{split}
\end{equation*} 
and noticing that 
\begin{equation*} 
\begin{split}
&\delta \bigl[ - H^0(X_t^0, \nabla_{x_0} w^0(t,X_t^0)) 
+
 \nabla_{x_0} w^0(t,X_t^0)
 \cdot \nabla_p H^0(X_t^0,\nabla_{x_0} w^0(t,X_t^0))
 \bigr]
 \\
 &=- \nabla_{x_0}
 H^0(X_t^0, \nabla_{x_0} w^0(t,X_t^0))
 \cdot \delta X_t^0 +   \nabla_{x_0} w^0(t,X_t^0)
 \cdot 
 \delta \bigl[ 
 \nabla_p H^0(X_t^0,\nabla_{x_0} w^0(t,X_t^0)) \bigr],
 \end{split}
\end{equation*}
we obtained, as linearized system for the major player, 
\begin{align} 
&\ud \delta X_t^0 = - \delta \bigl[ \nabla_p H^0(X_t^0,\nabla_{x_0} w^0(t,X_t^0)) \bigr] \ud t,
\nonumber 
\\
&\ud \delta Y_t^0 = 
 - \delta \bigl[f_t^0(X_t^0,\mu_t)\bigr]  \ud t
\nonumber
\\
&\hspace{5pt}
+
\Bigl[ -\bigl( Z_t^0 - \nabla_{x_0} w^0(t,X_t^0) \bigr) \cdot \delta \bigl[ \nabla_p H^0(X_t^0,\nabla_{x_0} w^0(t,X_t^0))\bigr]
-
\delta \bigl[L^0(X_t^0,-\nabla_p H^0(X_t^0,\nabla_{x_0}w^0(t,X_t^0)))\bigr]
\Bigr] \ud t
\nonumber
\\
&\hspace{5pt} +  \bigl( \nabla_{x_0} H^0(X_t^0,Z_t^0) - \nabla_{x_0} H^0(X_t^0,\nabla_{x_0} w^0(t,X_t^0))\bigr) \cdot \delta X_t^0
  \ud t
  \label{eq:major:diff:1}
\\
&\hspace{5pt} + \bigl( \nabla_p H^0(X_t^0,Z_t^0) - \nabla_p H^0(X_t^0,\nabla_{x_0} w^0(t,X_t^0))
\bigr) \cdot \delta Z_t^0   \ud t
\nonumber
\\
&\hspace{5pt}+ 
\sigma_0 \delta Z_t^0 \cdot \ud   B_t^0, \quad t \in [0,T], 
\nonumber
\\
 &\delta Y_T^0 = 
 \delta \bigl[ g^0(X_T^0,\mu_T) \bigr] \nonumber,
\end{align} 
with the initial condition $\delta X_0=x_0'-x_0$  and 
where $(X^0_t,Y^0_t,Z_t^0)$ is a shorter notation for $(X_t^{0,x_0,\mu},Y_t^{0,x_0,\mu},Z_t^{0,x_0,\mu})$. 

Similarly, 
the linearized system of the minor player is 
\begin{align}
&\partial_t \delta \mu_t - \tfrac12 \Delta_x \delta \mu_t - {\rm div}_x \Bigl(  
\nabla_p H(\cdot ,\nabla_x u_t(\cdot))
 \delta \mu_t \Bigr) 
-
{\rm div}_x \Bigl(  
\delta \bigl[
\nabla_p H(\cdot ,\nabla_x u_t(\cdot))
\bigr] \mu_t
\Bigr) 
=0, \nonumber
\\
&\ud_t \delta u_t(x) = \Bigl( -  \tfrac12 \Delta_x \delta u_t(x) +   \nabla_p H(x,\nabla_x u_t(x)) \cdot \nabla_x \delta u_t(x) 
- \delta \bigl[ f_t(X_t^0,x,\mu_t) \bigr]
\Bigr) \ud t \nonumber
\\
&
\hspace{15pt}   
+ \Bigl( \nabla_p H^0(X_t^0,Z_t^0) - \nabla_p H^0(X_t^0,\nabla_{x_0} w^0(t,X_t^0))  \Bigr)\cdot \delta v_t^0(x) 
\ud t \label{eq:minor:diff:2}
\\
&
\hspace{15pt}   
+ \delta \Bigl[ \Bigl( \nabla_p H^0(X_t^0,Z_t^0) - \nabla_p H^0(X_t^0,\nabla_{x_0} w^0(t,X_t^0))  \Bigr)
\Bigr]
\cdot  v_t^0(x) 
\ud t 
+ \sigma_0 \delta v_t^0(x) \cdot \ud   B_t^0, \quad (t,x)  \in [0,T] \times 
{\mathbb T}^d,  \nonumber
\\
&\delta u_T(x) = \delta \bigl[ g(X_T^0,x,\mu_T) \bigr], \quad x \in {\mathbb T}^d,  \nonumber
\end{align}
with the initial condition $\delta \mu_0=\mu'-\mu$ 
and
 with the same convention as before that, for 
a (say smooth real-valued) function $\Phi(\nabla_x u_t(x))$ of $\nabla_x u_t(x)$,
\begin{equation*} 
\delta \bigl[ \Phi(\nabla_x u_t(x)) \bigr] = \nabla \Phi(\nabla_x u_t(x)) \cdot \nabla_x \delta u_t(x).
\end{equation*} 

\label{def:linear forward-backward=MFG:solution}
Solutions to the two systems \eqref{eq:major:diff:1}--\eqref{eq:minor:diff:2} are understood in the following sense: 

\begin{definition} 
\label{def:4.8:4.9}
Let $(({\boldsymbol X}^0,{\boldsymbol Y}^0,{\boldsymbol Z}^0), ({\boldsymbol\mu},{\boldsymbol u},{\boldsymbol v^0}))$ be a solution to \eqref{eq:major:FBG:1}--\eqref{eq:minor:FBG:2} in the sense of Definition \ref{def:forward-backward=MFG:solution}, for an initial condition $(x_0,\mu) \in\mathbb R^d\times\mathcal{P}(\mathbb T^d)$.
Then, 
for another 
$(x_0',\mu')\in\mathbb R^d\times\mathcal{P}(\mathbb T^d)$, 
\begin{enumerate}
\item 
Given the initial condition $x_0'-x_0$
and  an ${\mathbb F}^0$-adapted process 
${\delta\boldsymbol  \mu} = (\delta\mu_t)_{0 \le t \le T}$ with 
continuous trajectories in ${\mathcal C}^{-\lfloor \curss \rfloor+(d/2+1)}({\mathbb T}^d)$, we call solution 
to 
\eqref{eq:major:diff:1}
any ${\mathbb F}^0$-progressively measurable process 
$({\delta\boldsymbol X}^0,{\delta\boldsymbol Y}^0,{\delta\boldsymbol Z}^0)$ with values in ${\mathbb R}^d \times {\mathbb R}
\times {\mathbb R}^d$, such that: (i)
$\delta{\boldsymbol X}^0$ and $\delta{\boldsymbol Y}^0$ 
have continuous trajectories; (ii) $\sup_{0\leq t\leq T}|\delta X_t^0|\in L^\infty(\Omega^0,\mathbb P^0)$, $\sup_{0\leq t\leq T}|\delta Y_t^0|\in L^\infty(\Omega^0,\mathbb P^0)$; (iii) $\sup_{\tau} \| {\mathbb E}^0[ \int_{[\tau,T]} \vert \delta Z_t^0 \vert^2 \ud t \vert {\mathcal F}_\tau^0] \|_{L^\infty(\Omega^0,{\mathbb P}^0)} < \infty$, the supremum being 
taken over all stopping times $\tau$;
(iv) the system 
\eqref{eq:major:diff:1} is satisfied ${\mathbb P}^0$-almost surely.
\item Given the initial condition $\mu'-\mu$,  an ${\mathbb F}^0$-adapted process 
${\delta \boldsymbol X}^0 = (\delta X^0_t)_{0 \le t \le T}$ with $\sup_{0\leq t\leq T}|\delta X_t^0|\in L^\infty(\Omega^0,\mathbb P^0)$ and continuous trajectories in $\mathbb R^d$, and an ${\mathbb F}^0$-progressively measurable process ${\delta \boldsymbol Z}^0=  (\delta Z^0_t)_{0 \le t \le T}$ with $\sup_{\tau} \| {\mathbb E}^0[ \int_{[\tau,T]} \vert \delta Z_t^0 \vert^2 \ud t \vert {\mathcal F}_\tau^0] \|_{L^\infty(\Omega^0,{\mathbb P}^0)} < \infty$, we call solution 
to 
\eqref{eq:minor:diff:2}
any \textit{Bochner} ${\mathbb F}^0$-progressively measurable process 
$({\delta\boldsymbol \mu},{\delta\boldsymbol u},{\delta\boldsymbol v}^0)$ with values in 
${\mathcal C}^{-1}({\mathbb T}^d) \times {\mathcal C}^{3+\epsilon}({\mathbb T}^d) \times {\mathcal C}^{3+\epsilon}({\mathbb T}^d)$ for some 
$\epsilon >0$, such that: (i)
$({\delta\boldsymbol \mu},{\delta\boldsymbol u})$ has continuous trajectories in 
${\mathcal C}^{-1}({\mathbb T}^d) \times {\mathcal C}^{3+\epsilon}({\mathbb T}^d)$; 
(ii) $\sup_{0\leq t\leq T}\|\delta u_t\|_{3+\epsilon}\in L^\infty(\Omega^0,{\mathbb P}^0)$, $\sup_{0\leq t\leq T}\|\delta \mu_t\|_{-1}\in L^\infty(\Omega^0,{\mathbb P}^0)$; 
(iii) $\mathbb E^0[\int_0^T\|\delta v_t^0\|^2_{3+\epsilon}\ud t ]<\infty$; (iv) the forward equation in \eqref{eq:minor:diff:2} is satisfied ${\mathbb P}^0$-almost surely in the weak sense; (v) the backward equation in \eqref{eq:minor:diff:2} is satisfied ${\mathbb P}^0$-almost surely in the classical sense.

\item For a given initial condition $( x_0'-x_0, \mu'-\mu)$,
we call a solution to the coupled systems 
\eqref{eq:major:diff:1}--\eqref{eq:minor:diff:2} a pair $(({\delta\boldsymbol X}^0,{\delta\boldsymbol Y}^0,{\delta\boldsymbol Z}^0),({\delta\boldsymbol \mu},{\delta\boldsymbol u},{\delta\boldsymbol v}^0))$
satisfying items 1 and 2 right above. 
\end{enumerate}
\end{definition} 
Regarding (iv) in item 2, $\delta {\boldsymbol \mu}$ is said to satisfy \eqref{eq:minor:diff:2} ${\mathbb P}^0$-almost surely in the weak sense if, 
 ${\mathbb P}^0$, for any test function function $\varphi : [0,T] \times {\mathbb T}^d \rightarrow {\mathbb R}$ in the (separable) space of functions that are once differentiable in 
$t$ and three times differentiable in space, with jointly continuous derivatives, it holds, for any 
$t \in [0,T]$, 
\begin{equation}
\label{eq:weak:sense:deltamu} 
\begin{split}
&\bigl( \varphi(t,\cdot) ,  \delta \mu_t
\bigr) 
-\bigl( \varphi(0,\cdot) ,  \delta \mu_0
\bigr) 
= \int_0^t \bigl( \partial_t \varphi(s,\cdot) ,  \delta \mu_s
\bigr) 
\ud s
+ \tfrac12 \int_0^t \bigl( \Delta_x \varphi(s,\cdot), \delta \mu_s \bigr)  \ud s 
\\
&\hspace{15pt} 
+
\int_0^t \Bigl( \nabla_x \varphi(s,\cdot) 
\cdot  
\nabla_p H(\cdot ,\nabla_x u_s(\cdot)),
 \delta \mu_s \Bigr) \ud s 
+
\int_0^t 
\Bigl(   \nabla_x \varphi(s,\cdot)   \cdot 
\delta \bigl[
\nabla_p H(\cdot ,\nabla_x u_s(\cdot))
\bigr] , \mu_s
\Bigr) \ud s.
\end{split}
\end{equation}


Regarding the solvability of the system 
\eqref{eq:major:diff:1}--\eqref{eq:minor:diff:2}, the following statement is proven in Subsection 
\ref{subse:4.7}:

\begin{proposition}
\label{prop:strong:existence:uniqueness:passage2}
Let Assumption \hyp{B} and the assumptions of Proposition \ref{prop:strong:existence:uniqueness:passage} hold true (for any initial condition of 
the system 
\eqref{eq:major:FB:1}--\eqref{eq:minor:FB:2}). Assume also that
there exist two positive reals $\alpha$ and $L$
such that, for all $t \in [0,T]$, 
${\mathcal U}^0(t,\cdot,\cdot)\in {\mathscr D}^0(L,1)$
and  ${\mathcal U}(t,\cdot,\cdot,\cdot)\in {\mathscr D}(L,1,3+\alpha)$, and 
all 
$(t,x_0,\mu) \in [0,T] \times \mathbb R^d\times\mathcal{P}(\mathbb T^d)$,
\begin{equation}\label{eq:U:priori:contintime}
\begin{split}
&\lim_{h\to 0}\Bigl(\big\|\nabla_{x_0} {\mathcal U}(t+h,x_0,\cdot,\mu)-\nabla_{x_0}{\mathcal U}(t,x_0,\cdot,\mu)\big\|_{3+\alpha}
\\
&+\sup_{l\in\{0,1 \}} \sup_{y \in {\mathbb T}^d} \big\|\nabla_y^l\delta_\mu {\mathcal U}(t+h,x_0,\cdot,\mu,\cdot)-\nabla_y^l\delta_\mu {\mathcal U}(t,x_0,\cdot,\mu,y)\big\|_{3+\alpha} \Bigr)=0.
\end{split}
\end{equation}

Let $(x_0,\mu),(x_0',\mu')\in\mathbb R^d\times\mathcal{P}(\mathbb T^d)$.
Then, on any arbitrary 
probabilistic set-up $(\Omega^0,{\mathcal F}^0,{\mathbb F}^0,{\mathbb P}^0)$,  and for the given initial condition $(x_0'-x_0, \mu'-\mu)$, 
the system 
\eqref{eq:major:diff:1}--\eqref{eq:minor:diff:2} 
has a (hence strong) solution $(({\delta\boldsymbol X}^0,{\delta\boldsymbol Y}^0,{\delta\boldsymbol Z}^0),({\delta\boldsymbol \mu},{\delta\boldsymbol u},{\delta\boldsymbol v}^0))$
satisfying the requirements of 
Definition 
\ref{def:linear forward-backward=MFG:solution} and, ${\mathbb P}^0$ almost surely,  
\begin{equation} 
\label{eq:representation:weak:linear:solution}
\begin{split} 
&\delta Y_t^0 =\nabla_{x_0} {\mathcal U}^0(t,X_t^0,\mu_t)\cdot\delta X_t^0+\left(\delta_{\mu}\mathcal{U}^0(t,X_t^0,\mu_t),\delta\mu_t\right), \quad t \in [0,T], 
\\
&\delta u_t(x) = \nabla_{x_0}{\mathcal U}(t,X_t^0,x,\mu_t)\cdot \delta X_t^0+\left(\delta_{\mu}\mathcal{U}(t,X_t^0,x,\mu_t),\delta\mu_t\right), \quad (t,x) \in [0,T] \times 
{\mathbb T}^d. 
\end{split} 
\end{equation} 
\end{proposition}

{\bf In the following Subsections 
\ref{subse:4:3},
\ref{subse:4:4}
and
\ref{subse:4:5},  the assumptions of Propositions 
\ref{prop:strong:existence:uniqueness:passage}
and
\ref{prop:strong:existence:uniqueness:passage2}
are assumed to hold true (in addition to Assumption \hyp{B})}. Namely,
there exist 
two positive reals $\kappa,\alpha>0$ and 
two mappings 
${\mathcal U}^0 : [0,T] \times {\mathbb R}^d \times {\mathcal P}({\mathbb T}^d)
\rightarrow {\mathbb R}$ and 
${\mathcal U} : [0,T] \times {\mathbb R}^d \times {\mathbb T}^d 
\times {\mathcal P}({\mathbb T}^d) \rightarrow
{\mathbb R}$
in ${\mathscr D}^0(L,3+\alpha)$ and ${\mathscr D}(L,1,3+\alpha)$ respectively, satisfying the continuity 
property 
\eqref{eq:U:priori:contintime}, 
 such that 
the (hence unique) solution to 
\eqref{eq:major:FB:1}--\eqref{eq:minor:FB:2}
satisfies 
the representation formula
\eqref{eq:representation:weak:solution}.

\subsection{Tilting the linearized system}
\label{subse:4:3}

The objective is to provide (\textit{a priori}) estimates, for the linearized systems \eqref{eq:major:diff:1}--\eqref{eq:minor:diff:2}, that are independent of $T$. This is the key point in the 
study of the solvability of the (double) forward-backward system 
\eqref{eq:major:FBG:1}--\eqref{eq:minor:FBG:2}. Our strategy relies on a new change of measure which we explain now.  Indeed, we know (under the standing assumption) that the solution satisfies the analogue of 
\eqref{eq:Novikov} but under ${\mathbb P}^0$. Therefore, we can apply
a new Girsanov transformation (see \cite[Theorem 2.3]{Kazamaki}). Letting 
\begin{equation}\label{eq:Girsanov:bar:1}
\begin{split}
\overline{\mathcal E}_t &:= {\mathscr E}_t \biggl( - \sigma_0^{-1} \int_0^\cdot 
 \Bigl( \nabla_p H^0(X_r^0,Z_r^0) - \nabla_p H^0(X_r^0,\nabla_{x_0} w^0(r,X_r^0))  \Bigr)\cdot 
\ud B_r^0 
\biggr),
\quad t \in [0,T],  
 \end{split} 
\end{equation} 
we have that $(\overline{\mathcal E}_t)_{0 \le t \le T}$ is an ${\mathbb F}^0$-martingale (under 
${\mathbb P}^0$) and we can define the probability measure 
\begin{equation}\label{eq:Girsanov:bar:2} 
\overline{\mathbb P}^0 := \overline{\mathcal E}_T \cdot {\mathbb P}^0. 
\end{equation} 
Then, following the proof of Lemma 
\ref{lem:Girsanov:major},
and then proceeding as in the derivation of \eqref{eq:Novikov}, we deduce 
\begin{lemma}
Under Assumption \hyp{B}, 
there exist constants $\overline \gamma,\overline \gamma_0>0$ and $\overline C, \overline C_0 \geq 0$, only depending on 
the parameters in \hyp{B} expect $(\sigma_0,T)$, such that, for any ${\mathbb F}^0$-stopping time $\tau$ with values in $[0,T]$,   
\begin{equation}
\label{eq:Novikov:2} 
\begin{split}  
&\overline{\mathbb E}^0 \biggl[ \exp \biggl( \overline \gamma  \sigma_0^2 \int_\tau^T 
\bigl\|  \bar v_r^0  \bigr\|^2_{\lfloor \curss \rfloor -(d/2+1)}
 \ud r
\biggr) \, \vert \, {\mathcal F}_\tau^0 \biggr] \leq \overline C,\\
&\overline{\mathbb E}^0 \biggl[ \exp \biggl( \overline \gamma_0  \sigma_0^2 \int_\tau^T \vert Z_r^0 - \nabla_{x_0} w^0(r,X_r^0) \vert^2 \ud r
\biggr) \, \vert \, {\mathcal F}_\tau^0 \biggr] \leq \overline C_0.
\end{split} 
\end{equation}
\end{lemma}
Obviously, under the new measure 
$\overline{\mathbb P}^0$, 
the process 
\begin{equation*} 
\overline B_t^0 := B_t^0 + \sigma_0^{-1} \int_0^t 
 \Bigl( \nabla_p H^0(X_r^0,Z_r^0) - \nabla_p H^0(X_r^0,\nabla_{x_0} w^0(r,X_r^0))  \Bigr)
 \ud r, \quad t \in [0,T], 
\end{equation*} 
is an ${\mathbb F}^0$-Brownian motion and 
the forward-backward system \eqref{eq:major:diff:1}--\eqref{eq:minor:diff:2} writes
\begin{align} 
&\ud \delta X_t^0 = - \delta \bigl[ \nabla_p H^0(X_t^0,\nabla_{x_0} w^0(t,X_t^0)) \bigr] \ud t, \nonumber
\\
&\ud \delta Y_t^0 = 
 - \delta \bigl[f_t^0(X_t^0,\mu_t)\bigr]   \ud t \nonumber
\\
&\hspace{5pt}
+
\Bigl[ -\bigl( Z_t^0 - \nabla_{x_0} w^0(t,X_t^0) \bigr) \cdot \delta \bigl[ \nabla_p H^0(X_t^0,\nabla_{x_0} w^0(t,X_t^0))\bigr]
- 
\delta \bigl[L^0(X_t^0,-\nabla_p H^0(X_t^0,\nabla_{x_0}w^0(t,X_t^0)))\bigr]
\Bigr] \ud t \nonumber
\\
&\hspace{5pt} +  \bigl( \nabla_{x_0} H^0(X_t^0,Z_t^0) - \nabla_{x_0} H^0(X_t^0,\nabla_{x_0} w^0(t,X_t^0))\bigr) \cdot \delta X_t^0
  \ud t + 
\sigma_0 \delta Z_t^0 \cdot \ud   \overline B_t^0, \quad t \in [0,T], \label{eq:major:diff:g:1} 
\\
&\delta Y_T^0=\nabla_{x_0}g^0(X_T^0,\mu_T)\delta X_T^0+(\delta_\mu g^0(X_T^0,\mu_T),\delta \mu_T),
\nonumber
\end{align} 
and
\begin{align}
&\partial_t \delta \mu_t - \tfrac12 \Delta_x \delta \mu_t - {\rm div}_x \Bigl(  
\nabla_p H(\cdot ,\nabla_x u_t(\cdot))
 \delta \mu_t \Bigr) 
-
{\rm div}_x \Bigl(  
\delta \bigl[
\nabla_p H( \cdot ,\nabla_x u_t(\cdot))
\bigr] \mu_t
\Bigr) 
=0, \nonumber
\\
&\ud_t \delta u_t(x) = \Bigl( -  \tfrac12 \Delta_x \delta u_t(x) +   \nabla_p H(x,\nabla_x u_t(x)) \cdot \nabla_x \delta u_t(x) 
- \delta \bigl[ f_t(X_t^0,x,\mu_t) \bigr]
\Bigr) \ud t
\label{eq:minor:diff:g:2}
\\
&
\hspace{15pt}   
+ \delta \Bigl[ \Bigl( \nabla_p H^0(X_t^0,Z_t^0) - \nabla_p H^0(X_t^0,\nabla_{x_0} w^0(t,X_t^0))  \Bigr)
\Bigr]
\cdot  v_t^0(x) 
\ud t 
+ \sigma_0 \delta v_t^0(x) \cdot \ud \overline B_t^0,\quad (t,x)\in [0,T] \times {\mathbb T}^d,
\nonumber
\\
&\delta u_T(x)=\nabla_{x_0}g(X_T^0,x,\mu_T)\delta X_T^0+(\delta_\mu g(X_T^0,x,\mu_T),\delta \mu_T), \nonumber
\end{align}
where $\left(({\boldsymbol X}^0,{\boldsymbol Y}^0,{\boldsymbol Z}^0), ({\boldsymbol\mu},{\boldsymbol u},{\boldsymbol v^0})\right)$ solves
(under the mesure $\overline{\mathbb P}^0$)
\begin{equation} 
\label{eq:major:FB:bar:1}
\begin{split} 
&\ud X_t^0 = - \nabla_p H^0 \bigl( X_t^0 , Z_t^0\bigr)  \ud t + \sigma_0 \ud  \overline B_t^0, \quad t\in [0,T],
\\
&\ud Y_t^0 =   - \Bigl( f_t^0(X_t^0,\mu_t) + L^0\bigl(X_t^0,- \nabla_p H^0\bigl( X_t^0, Z_t^0\bigr) \bigr)
\Bigr) \ud t + 
\sigma_0 Z_t^0 \cdot \ud  \overline B_t^0, \quad t \in [0,T], 
\\
&X_0^0=x_0, \quad Y_T^0 = g^0(X_T^0,\mu_T).
\end{split}
\end{equation} 
and
\begin{equation}
\label{eq:minor:FB:bar:2}
\begin{split} 
&\partial_t \mu_t - \tfrac12 \Delta_x \mu_t - {\rm div}_x \Bigl( \nabla_p H \bigl(\cdot, \nabla_x u_t\bigr) \mu_t \Bigr) =0, 
\quad  \ (t,x)\in [0,T] \times {\mathbb T}^d, 
\\
&\ud_t u_t(x) =\Bigl(  -  \tfrac12 \Delta_x u_t(x) + H\bigl(x,\nabla_x u_t(x) \bigr)  - f_t(X_t^0,x,\mu_t)  \Bigr) \ud t
+ \sigma_0 v_t^0(x) \cdot \ud  \overline B_t^0, \quad (t,x) \in [0,T] \times {\mathbb T}^d, 
\\
&\mu_0=\mu,\quad u_T(x) = g(X_T^0,x,\mu_T), \quad x \in {\mathbb T}^d.  
\end{split}
\end{equation}

\subsection{A priori estimates for the linearized systems}
\label{subse:4:4}
We now state several standard \textit{a priori}  estimates for the linearized systems
\eqref{eq:major:diff:g:1}
and
\eqref{eq:minor:diff:g:2}. All of them are borrowed from the 
book \cite{CardaliaguetDelarueLasryLions}. 
As a preliminary observation, we notice
from the regularity of $w^0$
 that 
\begin{lemma}
\label{lem:forward:major:estimate}
Under Assumption \hyp{B}, 
there exists a constant $C_T$, only depending on the parameters in 
\hyp{B}, such that 
\begin{equation*} 
\sup_{0 \le t \le T} 
\vert \delta X_t^0 \vert \leq C_T \vert x_0-x_0' \vert.
\end{equation*}
\end{lemma}
\begin{proof}
This is a straightforward consequence of Gronwall's lemma. 
\end{proof} 

We continue with the analysis of the backward equation.
\begin{lemma} 
\label{lem:backward:estimate}
Under Assumption  \hyp{B}, 
there exist two constants $C_T$ and $C$, 
only depending on the parameters in 
 \hyp{B}
but
with $C$ being independent of $\sigma_0$ and $T$, such that, with probability 1, for any 
$t \in [0,T]$, and for $\cursr \in (0,\lfloor \curss \rfloor -(d/2+1)] \setminus {\mathbb N}$, 
\begin{equation*} 
\| \delta u_{t} \|_{\cursr} \leq  
C_T \vert x_0-x_0' \vert + 
 C \overline{\mathbb E}^0 \biggl[ \| \delta \mu_T \|_{-\cursr} +
\int_{t}^T 
\Bigl( 
  \| \delta \mu_r \|_{-\cursr} 
  +
\vert  \delta Z_r^0 \vert \,   \|  \bar v_r^0 \|_{\cursr} \Bigr) \,  \ud r \, \vert \, {\mathcal F}_{t}^0 \biggr],
\end{equation*}
 where
 we recall the notation  
 $\bar v_t^0 = v_t^0 - \int_{{\mathbb T}^d} v_t^0(x) \ud x$, see 
 \eqref{eq:barv0}. 
\end{lemma}
\begin{proof} 
The strategy is inspired from a  duality argument developed in the proof of Lemma 4.3.2
in \cite{CardaliaguetDelarueLasryLions}. It relies on Lemma 
\ref{le:appendix:2} in Appendix (which we already used in the previous section, 
see 
\eqref{eq:b:Hu-Htildeu}
and
\eqref{eq:duality:1st:time}).
For $t_0 \in [0,T]$ and for $q$ a smooth (deterministic) function  
with $\| q \|_{-\cursr} \leq 1$, we consider the conservation equation 
\begin{equation*} 
\begin{split}
&\partial_t q_t - \tfrac12 \Delta_x  q_t - \textrm{\rm div}_x \bigl( \nabla_p H(\cdot,\nabla_x u_t(\cdot))  q_t \bigr) = 0,
\quad t \in [t_0,t]; \quad q_{t_0}=q,
\end{split}
\end{equation*}
we compute 
\begin{equation*} 
\begin{split} 
\ud (\delta u_t,q_t) &= 
-\Bigl( \delta \bigl[ f_t(X_t^0,x,\mu_t) \bigr], q_t \Bigr) \ud t
+  \delta \Bigl[ \Bigl( \nabla_p H^0(X_t^0,Z_t^0) - \nabla_p H^0(X_t^0,\nabla_{x_0} w^0(t,X_t^0))  \Bigr)
\Bigr]
 \cdot  (q_t, v_t^0 ) \ud t 
\\
&\hspace{15pt} + \sigma_0 \bigl( q_t, \delta v_t^0 \bigr) \cdot \ud \overline{B}_t^0, \quad t\in [t_0,T]. 
\end{split}
\end{equation*}
And then,
using the regularity properties of $g$ and $f$
together with Lemmas 
\ref{lem:forward:major:estimate}
and 
\ref{le:appendix:2}, 
we obtain
\begin{equation*} 
\begin{split} 
(\delta u_{t_0},q) &=\overline{\mathbb E}^0 \biggl[ (\delta u_T,q_T) 
+\int_{t_0} ^T 
\Bigl( \delta \bigl[ f_t(X_t^0,x,\mu_t) \bigr], q_t \Bigr) \ud t
\\
&\hspace{15pt} - \int_{t_0}^T 
 \delta \Bigl[ \Bigl( \nabla_p H^0(X_t^0,Z_t^0) - \nabla_p H^0(X_t^0,\nabla_{x_0} w^0(t,X_t^0))  \Bigr)
\Bigr] \cdot  (q_t, v_t^0 ) \, \ud t 
 \, \vert \, {\mathcal F}_{t_0}^0 \biggr]
 \\
 &\leq C_T \vert x_0-x_0' \vert + 
 C  \overline{\mathbb E}^0 \biggl[ \| \delta \mu_T \|_{-\cursr} +
\int_{t_0}^T 
\Bigl(  \| \delta \mu_t \|_{-\cursr} + 
\bigl( 
\vert  \delta X_t^0 \vert
+
\vert  \delta Z_t^0 \vert \bigr) \,   \| \bar v_t^0 \|_{\cursr} \Bigr) \,  \ud t \, \vert \, {\mathcal F}_{t_0}^0 \biggr]. 
\end{split}
\end{equation*}
By \eqref{eq:Novikov:2}
and Lemma \ref{lem:forward:major:estimate}
again, 
the above bound can be rewritten 
(for a new value of the constant $C_T$)
\begin{equation*} 
\begin{split} 
(\delta u_{t_0},q) &\leq C_T \vert x_0-x_0' \vert + 
 C  \overline{\mathbb E}^0 \biggl[ \| \delta \mu_T \|_{-\cursr} +
\int_{t_0}^T 
\Bigl(  \| \delta \mu_t \|_{-\cursr} + 
\vert  \delta Z_t^0 \vert \,   \| \bar v_t^0 \|_{\cursr} \Bigr) \,  \ud t \, \vert \, {\mathcal F}_{t_0}^0 \biggr]. 
\end{split}
\end{equation*}
For $l \in \{0,\cdots,\lfloor \cursr \rfloor\}$, $x$ and $h$ two elements of ${\mathbb R}^d$ and $\rho$ a smooth density on 
${\mathbb R}^d$, we 
observe that, for
$q(\cdot) = (-1)^l  \partial^l \rho(\cdot-x)$
(with $\partial^l \rho$ denoting the derivative of $\rho$ along $l$ arbitrary directions of ${\mathbb R}^d$ with possible repetitions)
and
 any $\varphi \in {\mathcal C}^\cursr({\mathbb T}^d)$ satisfying 
$\| \varphi \|_{\cursr} \leq 1$,
\begin{equation*} 
\bigl\vert (\varphi,q) \bigr\vert \leq \Bigl\vert \nabla^l \varphi * \rho(x)  
\Bigr\vert \leq 1.
\end{equation*}
Similarly, for 
$q(\cdot) = (-1)^{\lfloor \cursr  \rfloor} \vert h \vert^{-\cursr + \lfloor \cursr \rfloor}  [\partial^{\lfloor \cursr
\rfloor} \rho(\cdot-x) -  \partial^{\lfloor \cursr \rfloor} \rho(\cdot-(x+h)]$,
\begin{equation*} 
\bigl\vert (\varphi,q) \bigr\vert \leq  \vert h \vert^{-\cursr + \lfloor \cursr \rfloor}   \Bigl\vert \nabla^{\lfloor \cursr
\rfloor} \varphi * \rho(x)  
- \nabla^{\lfloor \cursr \rfloor} \varphi * \rho(x+h)  
\Bigr\vert \leq 1.
\end{equation*}
And then, denoting by $(\rho_n)_{n \geq 1}$ a sequence of mollifiers on 
${\mathbb R}^d$, we deduce that, $\overline{\mathbb P}^0$-almost surely, for 
any $l \in \{0,\cdots,\lfloor \cursr \rfloor\}$, any $n \geq 1$ and 
any $x,h \in {\mathbb Q}^d$,
\begin{equation*} 
\Bigl\vert \nabla^l \Bigl( \delta u_{t_0} * \rho_n \Bigr) (x) \Bigr\vert \leq C_T \vert x_0-x_0' \vert + 
 C  \overline{\mathbb E}^0 \biggl[ \| \delta \mu_T \|_{-\cursr} +
\int_{t_0}^T 
\Bigl(  \| \delta \mu_t \|_{-\cursr} + 
\vert  \delta Z_t^0 \vert \,   \| \bar v_t^0 \|_{\cursr} \Bigr) \,  \ud t \, \vert \, {\mathcal F}_{t_0}^0 \biggr],
\end{equation*} 
and
\begin{equation*} 
\begin{split}
 &\vert h \vert^{-\cursr + \lfloor \cursr \rfloor}
\Bigl\vert \nabla^{ \lfloor \cursr \rfloor} \Bigl( \delta u_{t_0} * \rho_n \Bigr) (x) 
-
\nabla^{  \lfloor \cursr \rfloor} \Bigl( \delta u_{t_0} * \rho_n \Bigr) (x+h) 
\Bigr\vert 
\\
&\hspace{15pt} \leq C_T  \vert x_0-x_0' \vert + 
 C  \overline{\mathbb E}^0 \biggl[ \| \delta \mu_T \|_{-\cursr} +
\int_{t_0}^T 
\Bigl(  \| \delta \mu_t \|_{-\cursr} + 
\vert  \delta Z_t^0 \vert \,   \| \bar v_t^0 \|_{\cursr} \Bigr) \,  \ud t \, \vert \, {\mathcal F}_{t_0}^0 \biggr].
\end{split}
\end{equation*} 
Obviously, 
this holds true
almost surely, for 
any $l \in \{0,\cdots,\lfloor \cursr \rfloor\}$, any $n \geq 1$ and 
any $x,h \in {\mathbb R}^d$. And since 
$\delta u_{t_0}$ is already known to be in 
${\mathcal C}^{3+\epsilon}({\mathbb T}^d)$ (see 
Proposition 
\ref{prop:strong:existence:uniqueness:passage2}), we deduce that it belongs to 
${\mathcal C}^{\cursr}({\mathbb T}^d)$
with 
\begin{equation*} 
\| \delta u_{t_0} \|_{\cursr} \leq  
C_T \vert x_0-x_0' \vert + 
 C \overline{\mathbb E}^0 \biggl[ \| \delta \mu_T \|_{-\cursr} +
\int_{t_0}^T 
\Bigl( 
  \| \delta \mu_t \|_{-\cursr} 
  +
\vert  \delta Z_t^0 \vert  \,   \|  v_t^0 \|_{\cursr} \Bigr) \,  \ud t \, \vert \, {\mathcal F}_{t_0}^0 \biggr].
\end{equation*}
Here, we also recall that $\delta {\boldsymbol u}$ has continuous values in 
${\mathcal C}^{3+\epsilon}({\mathbb T}^d)$. We deduce that the above is true, almost surely, for any $t_0\in [0,T]$. 
\end{proof} 

\begin{lemma} 
\label{lem:forward:estimate}
Under Assumption 
\hyp{B}, there exist an exponent $\gamma >0$ and a constant $C$, 
only depending on the parameters in 
 \hyp{B}
except $\sigma_0$ and $T$, such that, with probability 1, for any 
$t \in [0,T]$, for any 
 $\cursr \in [1,\lfloor \curss \rfloor -(d/2+1)] \setminus {\mathbb N}$,  
\begin{equation*} 
\| \delta \mu_t \|_{-\cursr} \leq  
 C \biggl[ \| \delta \mu_0  \|_{-\cursr} + \int_0^t \exp(- \gamma (t-r))  \Bigl\| \nabla_x \delta u_r \, \mu_r \Bigr\|_{-\cursr+1} \ud r \biggr].
\end{equation*}
In particular, 
\begin{equation*} 
\| \delta \mu_t \|_{-\cursr}^2 \leq  
 C \biggl[ \| \delta \mu_0 \|_{-\cursr}^2 + \int_0^t 
 \exp(- \gamma (t-r))
  \bigl
  ( \vert \nabla_x \delta u_r \vert^2, \mu_r
 \bigr)  \ud r \biggr],
\end{equation*}
for a possibly new value of $C$. 
\end{lemma}

\begin{proof}
We fix $t_0 \in [0,T]$. 
We use again a duality argument by computing $\ud_t (\delta \mu_t, \varphi_t)$, where 
$(\varphi_t)_{0 \leq t \leq t_0}$ solves the 
(random) backward equation
\begin{equation*} 
\begin{split}
&\partial_t \varphi_t = -  \tfrac12 \Delta \varphi_t + \nabla_p H(\cdot,\nabla_x u_t(\cdot)) \cdot \nabla \varphi_t,
\quad t \in [0,t_0]; \quad \varphi_{t_0} =\phi,
\end{split}
\end{equation*}
where $\phi$ is a deterministic function in ${\mathcal C}^{\cursr}({\mathbb T}^d)$.
Observe that $(\varphi_t)_{0 \leq t \leq t_0}$ may be anticipative. 
By  Lemma 
\ref{le:appendix:1} together with the fact that 
$\nabla_x {\boldsymbol u}$ takes values in 
${\mathcal C}^{3+\epsilon}({\mathbb T}^d)$
for some $\epsilon >0$, we deduce that 
$(\varphi_t)_{0 \le t \le t_0}$ takes values in 
${\mathcal C}^{3+\epsilon}({\mathbb T}^d)$. 
In particular, 
$(\partial_t \varphi_t)_{0 \le t \le t_0}$ takes values in 
${\mathcal C}^{1+\epsilon}({\mathbb T}^d)$. Because 
$\delta {\boldsymbol \mu}$ takes values in 
${\mathcal C}^{-1}({\mathbb T}^d)$, this makes it possible to expand the duality product 
$(\delta \mu_t,\varphi_t)_{0 \le t \le t_0}$. We obtain
\begin{equation*} 
\begin{split} 
\ud_t( \delta \mu_t, \varphi_t)
&= - \Bigl( \nabla_x \varphi_t,
\delta \bigl[
\nabla_p H(\cdot ,\nabla_x u_t(\cdot))
\bigr] \mu_t \Bigr) \ud t.
\end{split}
\end{equation*} 
Then, 
\begin{equation*}
\begin{split} 
( \delta \mu_{t_0}, \phi)
&\leq 
( \delta \mu_{0}, \varphi_0)
+
C \int_0^{t_0}  
\bigl\| \nabla_x \varphi_t\|_{\cursr-1} \bigl\| \delta \bigl[
\nabla_p H(\cdot ,\nabla_x u_t(\cdot))
\bigr] \mu_t  \bigr\|_{-(\cursr-1)} \ud t
\\
&\leq 
C \| \phi \|_\cursr
\biggl[ \| \delta\mu_0 \|_{-\cursr} 
+
\int_0^{t_0} 
\exp(- \gamma (t_0-t))
\bigl\| \nabla_x \delta u_t \mu_t \bigr\|_{-(\cursr-1)} \ud t \biggr],
\end{split}
\end{equation*} 
where, to get the last line, we used the identity 
$\delta \bigl[
\nabla_p H(\cdot ,\nabla_x u_t(\cdot))
\bigr] =\nabla^2_{pp}H(\cdot ,\nabla_x u_t(\cdot)) \nabla_x \delta u_t$ together with the 
 bound $\|\nabla^2_{pp}H(\cdot ,\nabla_x u_t(\cdot)) \|_{\cursr-1} \leq C$, 
which follows from 
Proposition 
\ref{prop:minor:higher}.  
Taking the supremum over $\phi$ in the unit ball of ${\mathcal C}^{\cursr}({\mathbb T}^d)$, we complete
the proof. 
\end{proof}

\subsection{A key functional} 
\label{subse:4:5} 

We now introduce a key object. For a parameter $A>0$, we call $(e_t)_{0 \le t \le T}$ the solution to the backward SDE
\begin{equation} 
\label{eq:e:BSDE}
\begin{split}
&\ud_t e_t = \frac{\sigma_0^2}A e_t 
  \|  \bar v_t^0  \|_{\lfloor \curss \rfloor -(d/2+1)}^2 \ud t + \ell_t \ud  \overline{B}_t^0, 
 \quad t \in [0,T]; \quad e_T = 1, 
 \end{split} 
 \end{equation} 
 where
 we recall the notation  
 $\bar v_t^0 = v_t^0 - \int_{{\mathbb T}^d} v_t^0(x) \ud x$, see 
 \eqref{eq:barv0}. 

 The solution is given by 
 the following lemma, which is a straightforward consequence of   
 \eqref{eq:Novikov:2}:
 \begin{lemma}\label{lem:et}
 Under Assumption 
 \hyp{B}, 
the BSDE 
 \eqref{eq:e:BSDE} has a unique solution, which is 
 given by 
 \begin{equation*} 
 e_t = 
 \overline{\mathbb E}^0 \biggl[ \exp \biggl(- \frac{\sigma_0^2}A \int_t^T 
 \| \bar v_r^0  \|_{\lfloor \curss \rfloor -(d/2+1)}^2 \ud r  \biggr) 
 \, \vert \, {\mathcal F}_t^0 \biggr], \quad t \in [0,T]. 
 \end{equation*} 
In particular,  there exists a constant $c \in (0,1)$, only depending on the parameters in 
\hyp{B}
except   $\sigma_0$ and $T$, such that, 
for $A$ large enough (independently of $\sigma_0$ and $T$),
with probability 1, for all $t \in [0,T]$, 
\begin{equation}
\label{eq:et:c-1}
c < e_t \leq 1.
\end{equation}
 \end{lemma}
 
 \begin{proof}
For any solution $(e_t)_{0 \le t \le T}$ as in 
\eqref{eq:e:BSDE}, 
\begin{equation*}
\begin{split} 
&\ud_t \biggl[ e_t  \exp \biggl( - \frac{\sigma_0^2}A \int_0^t 
 \| \bar v_r^0 \|_{\lfloor \curss \rfloor -(d/2+1)}^2 \ud r  \biggr) 
\biggr]
 =e_t  \exp \biggl( - \frac{\sigma_0^2}A \int_0^t 
  \| \bar v_r^0 \|_{\lfloor \curss \rfloor -(d/2+1)}^2 \ud r  \biggr) \ell_t \ud \overline{B}_t^0.
 \end{split}
 \end{equation*} 
 Under the standard conditions that 
 \begin{equation*} 
 \overline{\mathbb E}^0 
 \biggl[ 
 \sup_{0 \le t \le T} \vert 
 e_t
 \vert^2 
 + \int_0^T \vert \ell_t \vert^2 
 \ud t \biggr]
 < \infty,
 \end{equation*} 
 we see that the right-hand side in the latter expansion yields a martingale. We deduce that 
 \begin{equation*} 
 e_t =  \overline{\mathbb E}^0 \biggl[ \exp \biggl(- \frac{\sigma_0^2}A \int_t^T 
  \|  \bar v_r^0 \|_{\lfloor \curss \rfloor -(d/2+1)}^2 \ud r  \biggr) 
 \, \vert \, {\mathcal F}_t^0 \biggr], \quad t \in [0,T], 
 \end{equation*} 
 and there is no difficulty for proving that the above right-hand side induces a solution. 
 In order to prove the estimate \eqref{eq:et:c-1} (which is the key point in the statement), we recall that 
 for any positive random variable $\xi$ with positive values
 \begin{equation*} 
\overline{\mathbb E}^0 \bigl[ \xi \, \vert \, {\mathcal F}_t^0 \bigr] \geq 
\overline{\mathbb E}^0 \bigl[ \xi^{-1} \, \vert \, {\mathcal F}_t^0 \bigr]^{-1}.
 \end{equation*}
 Choosing 
 \begin{equation*} 
 \xi = 
\exp \biggl(- \frac{\sigma_0^2}A \int_t^T 
 \| \bar{v}_r^0 \|_{\lfloor \curss \rfloor -(d/2+1)}^2 \ud r  \biggr),
 \end{equation*} 
 we complete the proof of the lower bound by using  
\eqref{eq:Novikov:2}. Importantly, we notice that the lower bound $c$ does not depend on $A$
  (as soon as the latter one satisfies $1/A \leq  \bar\gamma$). The upper bound is obvious. 
 \end{proof} 
 
Here is the way we use the process $(e_t)_{0 \le t \le T}$. 
To clarify the dependence on the various parameters in   
\hyp{B}, we let 
$\kappa_{H^0}:=\|\nabla^2_{pp}H^0\|_{\infty}/2$ and 
we call $\lambda_H$ the largest real such that $\nabla^2_{pp}H\geq \lambda_H I_{d}$.
We also introduce 
\begin{equation*}
\theta_t := \sup_{x_0 \in {\mathbb R}^d}
\sup_{\mu \in {\mathcal P}({\mathbb T}^d)}
\| \delta_\mu f_t^0(x_0,\mu)(\cdot) \|_{\cursr} + \min \bigl( 1, \frac1{T} \bigr), 
\quad
\Theta_t = \int_0^t \theta_s \ud s,  \quad t \in [0,T].
\end{equation*}
By assumption (see \hyp{B3}), $\Theta_T = \int_0^T \theta_t \ud t$ is bounded by a constant independent of $T$ (and of course of 
$\sigma_0$). And by construction, $\theta_t^{-1} \leq \max(1,T)$. 
Then, 
for three (positive) parameters $\varepsilon_1$, $\varepsilon_2$ and $\varrho$, we next compute
 \begin{equation*} 
 \ud_t \biggl[ \varepsilon_1 
 \exp ( \Theta_t ) 
 \vert \delta Y_t^0 \vert^2 - ( \delta \mu_t,\delta u_t) + 
\varepsilon_2 e_t \biggl( \| \delta \mu_0  \|_{-\cursr}^2 + 
 \int_0^t \Bigl( \vert \nabla_x \delta u_r  \vert^2, \mu_r
 \Bigr) \ud r
 \biggr)\biggr],
 \end{equation*} 
 for a real $\cursr \in [1,\lfloor \curss \rfloor - (d/2+1)] \setminus {\mathbb N}$, 
 where we recall $\delta\mu_0=\mu-\mu'$.

We start with the expansion of 
$\ud_t  [ \varepsilon_1  \exp ( \Theta_t )  \vert \delta Y_t^0 \vert^2]$. 
Back to 
\eqref{eq:major:diff:g:1}, we write
\begin{equation*} 
\begin{split}
\ud_t \bigl[ \varepsilon_1  \exp ( \Theta_t )  \vert \delta Y_t^0 \vert^2 \bigr] &= 
\varepsilon_1  \exp (\Theta_t )  \biggl[ 
  \sigma_0^2 \vert \delta Z_t^0 \vert^2 
  + \theta_t  \vert \delta Y_t^0 \vert^2
\\
&\hspace{-30pt} +   \delta Y_t^0   \Bigl[ \Bigl(\mathcal{O}_{{\mathbb R}^d }(1)+\mathcal{O}_{{\mathbb R}^d}(\vert Z_t^0 - \nabla_{x_0} w^0(t,X_t^0) \vert )
\Bigr) \cdot \delta X_t^0 \Bigr] +  \delta Y_t^0 \Big(\mathcal{O}_{{\mathcal C}^{\cursr}({\mathbb T}^d)}(\theta_t),\delta\mu_t\Big)\biggr]\ud t + \ud n_t,
\end{split}
\end{equation*}
where $(n_t)_{t \geq 0}$ is a generic martingale term, whose precise value may vary from line to line. 
For a normed space $(E,\| \cdot \|)$ and for a possibly random real $\Delta\geq 0$, 
${\mathcal O}_{E}(\Delta)$ stands for an $E$-valued term satisfying 
$\| {\mathcal O}_E( \Delta ) \| \leq C \vert \Delta \vert$, for a constant $C$ that only depends on the 
various  parameters in   \hyp{B} except 
$\sigma_0$ and $T$. 

Next, we compute $\ud_t ( \delta \mu_t,\delta u_t)$. 
By duality, we obtain 
from 
\eqref{eq:minor:diff:g:2}:
\begin{equation*} 
\begin{split} 
&\ud_t  ( \delta \mu_t,\delta u_t)
\\
&= - 
\Bigl(  
\delta [ \nabla_x u_t] , 
\delta \bigl[
\nabla_p H( \cdot ,\nabla_x u_t(\cdot))
\bigr] \mu_t
\Bigr) \ud t
- 
\Bigl( \delta \mu_t, \delta \bigl[ f_t(X_t^0,x,\mu_t) \bigr] \Bigr) 
\ud t 
\\
&\hspace{15pt}   
+ \delta \Bigl[ \Bigl( \nabla_p H^0(X_t^0,Z_t^0) - \nabla_p H^0(X_t^0,\nabla_{x_0} w^0(t,X_t^0))  \Bigr)
\Bigr]
\cdot \bigl(   v_t^0,\delta \mu_t \bigr)  
\ud t 
+ \ud n_t\\
&= \bigg[-\Big((\nabla_x \delta u_t)^{\top}\nabla_{pp}^2H(\cdot,\nabla_xu_t(\cdot))\nabla_x\delta u_t,\mu_t\Big)-\Big((\delta_\mu f_t(X_t^0,\cdot,\mu_t,\cdot),\delta \mu_t(\cdot)),\delta \mu_t(\cdot)\Big)\\
& \hspace{15pt}   +
{\mathcal O}_{\mathbb R} \bigl( \vert \delta X_t^0 \vert \| \delta \mu_t \|_{-\cursr} 
\bigr) 
+
\mathcal{O}_{{\mathbb R}^d} \big( \vert \delta X_t^0 \vert  \bigr)\cdot 
(v_t^0,\delta \mu_t)
+
\bigl( \nabla^2_{pp} H^0(X_t^0,Z_t^0)  
\delta Z_t^0
\bigr) 
\cdot
(v_t^0,\delta \mu_t)
\bigg]\ud t+\ud n_t.
\end{split} 
\end{equation*} 
Above, we notice that 
$(v_t^0,\delta \mu_t)=
(\bar v_t^0,\delta \mu_t)$ because
the duality bracket 
$(\delta \mu_t,{\boldsymbol 1})$
(with ${\boldsymbol 1}$ standing for the constant function, equal to 1) 
is equal to $0$.  

Lastly, we handle 
$
\ud_t [ \varepsilon_2 e_t  ( \| \delta\mu_0  \|_{-\cursr}^2 + 
 \int_0^t ( \nabla_x \delta u_r \vert^2,  \mu_r  
 ) \ud r
)]$. We have, for all $t \in [0,T]$,  
\begin{equation*} 
\begin{split} 
&\ud_t \biggl[ \varepsilon_2 e_t  \biggl( \| \delta\mu_0  \|_{-\cursr}^2 + 
 \int_0^t \bigl(\bigl\vert \nabla_x \delta u_r \bigr\vert^2,  \mu_r  
 \bigr) \ud r
\biggr)\biggr]
\\
&= \bigg[\varepsilon_2 e_t \bigl(\bigl\vert \nabla_x \delta u_t \bigr\vert^2,  \mu_t 
 \bigr) 
 +
 \frac{\varepsilon_2  \sigma_0^2}A 
 e_t 
   \|  \bar v_t^0 \|_{\lfloor \curss \rfloor-(d/2+1)}^2 
 \biggl( \|\delta \mu_0 \|_{-\cursr}^2 + 
 \int_0^t \bigl(\bigl\vert \nabla_x \delta u_r \bigr\vert^2,  \mu_r  
 \bigr)dr
\biggr)\bigg]  \ud t + \ud n_t. 
\end{split}
\end{equation*} 
%
%

By combining the last three displays, we get 
 \begin{equation*}
 \begin{split}
 &\ud_t \biggl[ \varepsilon_1 
 \exp( \Theta_t ) 
 \vert \delta Y_t^0 \vert^2 - ( \delta \mu_t,\delta u_t) + 
\varepsilon_2 e_t \biggl( \| \delta\mu_0 \|_{-\cursr}^2 + 
 \int_0^t \bigl(\bigl\vert \nabla_x \delta u_r \bigr\vert^2,  \mu_r  
 \bigr) \ud r
 \biggr)\biggr]
  \\
  &= \varepsilon_1   \exp( \Theta_t )  \biggl[   \sigma_0^2 \vert \delta Z_t^0 \vert^2   +
   \theta_t
 \vert \delta Y_t^0 \vert^2\biggr]\ud t 
  \\
  &\hspace{15pt} 
    + 
     \varepsilon_1   \exp( \Theta_t )
     \biggl[ 
     \delta Y_t^0   \Bigl[ \Bigl(\mathcal{O}_{{\mathbb R}^d }(1)+\mathcal{O}_{{\mathbb R}^d}(\vert Z_t^0 - \nabla_{x_0} w^0(t,X_t^0) \vert )
\Bigr) \cdot \delta X_t^0 \Bigr]  +  \delta Y_t^0\Big(\mathcal{O}_{{\mathcal C}^{\cursr}({\mathbb T}^d)}(\theta_t),\delta\mu_t\Big)\biggr]\ud t 
\\
&\hspace{15pt}   +\biggl[\Big((\nabla_x \delta u_t)^{\top}\nabla_{pp}^2H(\cdot,\nabla_xu_t(\cdot))\nabla_x\delta u_t,\mu_t\Big)+\Big((\delta_\mu f_t(X_t^0,\cdot,\mu_t,\cdot),\delta \mu_t(\cdot)),\delta \mu_t(\cdot)\Big)
\biggr] \ud t
\\
& \hspace{15pt}   +
\biggl[ {\mathcal O}_{\mathbb R}\bigl( \vert \delta X_t^0 \vert \| \delta \mu_t \|_{-\cursr} 
\bigr) 
+
\mathcal{O}_{{\mathbb R}^d} \big( \vert \delta X_t^0 \vert  \bigr) \cdot (\bar v_t^0,\delta \mu_t)   +
\bigl( \nabla^2_{pp} H^0(X_t^0,Z_t^0)  
\delta Z_t^0
\bigr) 
\cdot
(\bar v_t^0,\delta \mu_t)
\biggr]dt
 \\
 &\hspace{15pt}
 +\bigg[
 \varepsilon_2 e_t \bigl(\bigl\vert \nabla_x \delta u_t \bigr\vert^2,  \mu_t 
 \bigr) 
  +
 \frac{\varepsilon_2  \sigma_0^2}A 
 e_t 
   \|  \bar v_t^0 \|_{\lfloor \curss \rfloor -(d/2+1)}^2 
 \biggl( \| \delta\mu_0 \|_{-\cursr}^2 + 
 \int_0^t \bigl(\bigl\vert \nabla_x \delta u_r \bigr\vert^2,  \mu_r  
 \bigr)
\biggr) \bigg] \ud t + \ud n_t,
 \end{split} 
 \end{equation*} 
 where $(n_t)_{0 \le t \le T}$ stands for a generic martingale term.
 All the terms on the second line of the right-hand side are handled by means of Young's inequality, 
 in such a way that 
 the corresponding term 
 $\delta Y_t^0$ is `absorbed' by 
 $\theta_t \vert \delta Y_t^0 \vert^2$ (which is greater than $\min(1,T^{-1}) \vert \delta Y_t^0 \vert^2$) on the first line of the right-hand side. 
 The two terms on the third line of the right-hand side are non-negative: for the first one, this follows from 
 the convexity of $H$ in the variable $p$; for the second one, this follows from the monotonicity of 
 $f$ in the variables $(x,\mu)$. 
 The first two terms on the fourth line of the right-hand side are also handled by means of 
 Young's inequality: among the resulting two terms, one of them is $\vert \delta X_t^0 \vert^2$ multiplied by a possibly large constant. 
 The last term on the fourth line of the right-hand side is handled by a new application of 
 Young's inequality, with the term $\delta 
Z_t^0$ being now `absorbed' by  $\varepsilon_1 \sigma_0^2 \vert \delta Z_t^0 \vert^2$
on the first line. 

Then, using 
Lemma \ref{lem:et}
together with 
the fact that $\theta_t^{-1} \leq \max(1,T)$, $\theta_t \leq \kappa +1$ and $\Theta_T \leq C$, we observe that the above expansion is greater than 
\begin{equation*}
\begin{split} 
 &\ud_t \biggl[ \varepsilon_1 \exp( \Theta_t)  \vert \delta Y_t^0 \vert^2 - ( \delta \mu_t,\delta u_t) + 
\varepsilon_2 e_t \biggl( \| \delta\mu_0\|_{-\cursr}^2 + 
 \int_0^t \bigl(\bigl\vert \nabla_x \delta u_r \bigr\vert^2, \mu_r 
 \bigr) \ud r
 \biggr)\biggr]
\\
&\geq \biggl[  
\frac{\varepsilon_1 \sigma_0^2}{2} \vert \delta Z_t^0 \vert^2 
+ 
  \frac{c \varepsilon_2 \sigma_0^2}{2A} 
 \bigl\| \bar v_t^0 \bigr\|_{\lfloor \curss\rfloor-(d/2+1)}^2
   \biggl( \| \delta \mu_0 \|_{-\cursr}^2 + 
 \int_0^t \bigl(\bigl\vert \nabla_x \delta u_r \bigr\vert^2, \mu_r  
 \bigr) \ud r
 \biggr)
\\
&\hspace{15pt}
-  \frac{\kappa_{H^0}}{\varepsilon_1 \sigma_0^2}\| \delta \mu_t \|^2_{-\cursr}  \bigl\| \bar v_t^0 \bigr\|_{\lfloor \curss \rfloor -(d/2+1)}^2
+
\lambda_{H} \bigl( \vert \nabla_x \delta u_t \vert^2, \mu_t \bigr)
\\
&\hspace{15pt}
- C_{\varepsilon_1,\sigma_0,T} 
 |\delta X_t^0|^2
-
C_{\varepsilon_1,\sigma_0,T} 
|Z_t^0-\nabla_{x_0}w^0(X_t^0)|^2
 |\delta X_t^0|^2
- C \varepsilon_1 \| \delta \mu_t \|^2_{-\mathfrc{r}}
 \biggr]
  \ud t   
  + \ud n_t,
\end{split}
\end{equation*} 
for a constant $C$ only depending on the parameters in   \hyp{B} except $\sigma_0$ and $T$
and for a constant $C_{\varepsilon_1,\sigma_0,T}$ only depending on $\varepsilon_1$ and the parameters in   \hyp{B} (including $\sigma_0$ and $T$). 
We now recall the upper bound for $\| \delta \mu_t \|^2_{-\cursr} $ (see Lemma \ref{lem:forward:estimate}). 
Allowing the value of $C$ to vary from line to line, we obtain 
\begin{equation}
\label{eq:commonhelp}
\begin{split} 
 &\ud_t \biggl[ \varepsilon_1 \exp( \theta_t)  \vert \delta Y_t^0 \vert^2 - ( \delta \mu_t,\delta u_t) + 
\varepsilon_2 e_t \biggl( \| \delta\mu_0\|_{-\cursr}^2 + 
 \int_0^t \bigl(\bigl\vert \nabla_x \delta u_r \bigr\vert^2, \mu_r 
 \bigr) \ud r
 \biggr)\biggr]
\\
&\geq \biggl[  
\frac{\varepsilon_1 \sigma_0^2}{2} \vert \delta Z_t^0 \vert^2 
+
\lambda_H \bigl( \vert \nabla_x \delta u_t \vert^2, \mu_t \bigr)
\\
&\hspace{10pt}
+ 
\biggl\{ 
\Bigl(   \frac{c \varepsilon_2 \sigma_0^2}{2A} 
- 
 \frac{ C \kappa_{H^0}}{\varepsilon_1 \sigma_0^2}
\Bigr) 
 \bigl\| \bar v_t^0 \bigr\|_{\lfloor \curss\rfloor-(d/2+1)}^2
 - 
 C \varepsilon_1 
 \biggr\}
   \biggl( \| \delta \mu_0 \|_{-\cursr}^2 + 
 \int_0^t 
\exp(- \gamma (t-r))
 \bigl(\bigl\vert \nabla_x \delta u_r \bigr\vert^2, \mu_r  
 \bigr) \ud r
 \biggr)
\\
&\hspace{10pt}
- C_{\varepsilon_1,\sigma_0,T} 
 |\delta X_t^0|^2
-
C_{\varepsilon_1,\sigma_0,T} 
|Z_t^0-\nabla_{x_0}w^0(X_t^0)|^2
 |\delta X_t^0|^2
 \biggr]
  \ud t + \ud n_t,
\end{split}
\end{equation}

We deduce the following main inequality:
\begin{proposition}
\label{prop:4:7}
Under Assumption \hyp{B},
there exist a constant $C \geq 1$ only depending
on the parameters in \hyp{B} except $(\sigma_0,T)$ 
and a constant $C_{\varepsilon_1,\sigma_0,T}$ only depending on 
$\varepsilon_1$ and 
the parameters in   \hyp{B} (including $(\sigma_0,T)$) 
and non-decreasing with $T$, 
such that, 
with the two notations
$\kappa_{H^0}:=\|\nabla^2_{pp}H^0\|_{\infty}/2$ and 
$\lambda_H := \sup \{ \theta : \nabla^2_{pp}H\geq \theta I_{d}\}$, for 
$A$ and 
 $c$ as in Lemma 
\ref{lem:et}, 
and
under the following two conditions:
\begin{enumerate}[i.]
\item  $\displaystyle \varepsilon_1 + \varepsilon_2 <\frac{\lambda_H}{C}$, 
\item 
$\displaystyle  \frac{c \varepsilon_2\sigma_0^2}{2A} >  \frac{C\kappa_{H^0}}{\varepsilon_1 \sigma_0^2}$, 
\end{enumerate}
it holds, for any 
 $\cursr \in [1,\lfloor \curss \rfloor -(d/2+1)] \setminus {\mathbb N}$,  
\begin{equation} \label{eq:globalLip}
\begin{split}
| \delta Y_0^0 \vert^2 + \| \delta u_0\|_{\cursr}^2   
\leq C_{\varepsilon_1,\sigma_0,T} \Bigl(  \vert  x_0 - x_0' \vert^2 + \| \mu-\mu' \|_{-\cursr}^2 \Bigr).
\end{split} 
\end{equation} 
\end{proposition}
It is worth pointing out that the constant $C$ in the statement may depend on
$\lambda_H$ and $\kappa_{H^0}$ themselves. This makes the two conditions 
 \textit{i} and \textit{ii}
 less explicit than might be assumed at first sight. However, we think this formulation is useful for the proof.

\begin{proof} 
Integrating \eqref{eq:commonhelp} from $0$ to $T$, taking expectation and using
condition \textit{ii} in the statement, we obtain
(for a constant $C$ only depending on the parameters in   \hyp{B} except $(\sigma_0,T)$)
\begin{equation}\label{eq:commonhelp3}
\begin{split}
&\overline{\mathbb E}^0\bigg[ \varepsilon_1 \exp( \Theta_T)  \vert \delta Y_T^0 \vert^2 - ( \delta \mu_T,\delta u_T) + 
\varepsilon_2 e_T \biggl( \| \delta\mu_0\|_{-\cursr}^2 + 
 \int_0^T \bigl(\bigl\vert \nabla_x \delta u_r \bigr\vert^2, \mu_r 
 \bigr) \ud r
 \biggr)\bigg]
\\
&\geq  \overline{\mathbb E}^0\bigg\{
 \varepsilon_1   \vert \delta Y_0^0 \vert^2 - ( \delta \mu_0,\delta u_0) + 
\varepsilon_2 e_0  \| \delta\mu_0\|_{-\cursr}^2 
\\
  &\hspace{15pt}
+ \int_0^T\Bigl[ \frac{ \varepsilon_1 \sigma_0^2 }{2}\vert \delta Z_t^0 \vert^2 + \lambda_H \bigl( \vert \nabla_x \delta u_t \vert^2, \mu_t \bigr)
\Bigr] \ud t 
-  C \varepsilon_1 \int_0^T \biggl[ \bigl( \vert \nabla_x \delta u_t \vert^2, \mu_t \bigr)
\int_t^T \exp(-\gamma(r-t)) \ud r 
\biggr] \ud t 
\\
  &\hspace{15pt} -
   C_{\varepsilon_1,\sigma_0,T} \int_0^T
  \Bigl[
 |\delta X_t^0|^2
+  
|Z_t^0-\nabla_{x_0}w^0(X_t^0)|^2 
 |\delta X_t^0|^2
  \Bigr]
  \ud t -C_{\varepsilon_1,\sigma_0,T}\|\delta \mu_0\|_{-\cursr}^2\biggr\}.
\end{split}
\end{equation}
Applying \eqref{eq:Novikov:2} and Lemma \ref{lem:forward:major:estimate}, we get
that the whole term on the last line is less $C_{\varepsilon_1,\sigma_0,T} \vert x_0-x_0' \vert^2$, for a constant 
$C_{\varepsilon_1,\sigma_0,T}$ depending only on the parameters in  \hyp{B} 
(including $\sigma_0$ and $T$) 
and on $\varepsilon_1$.
Moreover, 
by
the monotonicity condition \hyp{A3} and 
Lemma 
\ref{lem:forward:estimate}, 
we have
(for possibly new values of $C_{\varepsilon_1,\sigma_0,T}$ and $C$)
\begin{equation} \label{eq:commonhelp6}
\begin{split}
&\overline{\mathbb E}^0\Big[( \delta \mu_T,\delta u_T) 
-
\varepsilon_1 \exp(\Theta_T) \vert \delta Y_T^0 \vert^2
\Big]
\\
&=\overline{\mathbb E}^0\Big[\delta X_T^0\cdot (\nabla_{x_0}g(X_T^0,\cdot,\mu_T),\delta\mu_T)+\Big((\delta_\mu g(X_T^0,\cdot,\mu_T,\cdot),\delta \mu_T(\cdot)),\delta \mu_T(\cdot)\Big)
\\
&\hspace{15pt} -\varepsilon_1 \exp(\Theta_T)  \Big(\nabla_{x_0}g^0(X_T^0,\mu_T)\cdot \delta X_T^0+(\delta_\mu g^0(X_T^0,\mu_T),\delta\mu_T)\Big)^2\Big]
\\
&\geq -C_{\varepsilon_1,\sigma_0,T}
\overline{\mathbb E}^0
\bigl[ 
\vert 
\delta X_T^0
\vert^2
\bigr]
 -    C \varepsilon_1     \exp(\Theta_T)
\| \delta \mu_T \|^2_{-\cursr} 
 \\
 &\geq -C_{\varepsilon_1,\sigma_0,T} \vert x_0-x_0' \vert^2 -   C \varepsilon_1   \biggl( \|  \delta \mu_0 \|_{-\cursr}^2 + \overline{\mathbb E}^0\bigg[
 \int_0^T \bigl(\bigl\vert \nabla_x \delta u_r \bigr\vert^2, \mu_r  
 \bigr) \ud r\bigg]
 \biggr).
\end{split}
\end{equation} 

Therefore, by adding \eqref{eq:commonhelp3} and 
 \eqref{eq:commonhelp6}, 
and by dropping the positive term $\varepsilon_2e_0\|  \delta \mu_0\|_{-\cursr}^2$, we obtain
(for a new value of $C$) 
\begin{equation} \label{eq:commonhelp2}
\begin{split}
&\overline{\mathbb E}^0\bigg[
\bigl( \lambda_H
- C(\varepsilon_1+ \varepsilon_2)
\bigr) 
\int_0^T\bigl(\bigl|\nabla_x\delta u_r \bigr|^2,\mu_r\bigr)dr+\frac{\varepsilon_1\sigma_0^2}{2}\int_0^T|\delta Z_r^0|^2dr\bigg]\\
&\leq \biggl[ C_{\varepsilon_1,\sigma_0,T} (\vert x_0-x_0' \vert^2 +\|   \delta \mu_0\|_{-\cursr}^2)-\varepsilon_1|\delta Y_0^0|^2+(\delta u_0,   \delta \mu_0)\biggr].
\end{split} 
\end{equation} 
Next, we recall 
 Lemmas 
\ref{lem:backward:estimate} 
and
\ref{lem:forward:estimate}, 
and use \eqref{eq:Novikov:2} to derive
that 
\begin{equation*} 
\begin{split}
\| \delta u_0 \|_{\cursr} &\leq  
 C_{T}|x_0-x_0'|+C \overline{\mathbb E}^0 \biggl[ \| \delta \mu_T \|_{-\cursr} +
\int_0^T 
\Bigl( 
\|\delta \mu_r\|_{-\cursr}+|\delta Z_r^0|\|\bar v_r^0\|_{\lfloor \curss \rfloor -(d/2+1)} \Bigr)  \ud r \biggr]
\\
&\leq C_{T}\bigl(|x_0-x_0'|+\|   \delta \mu_0\|_{-\cursr}\bigr)+ C\overline{\mathbb E}^0\biggl[ \int_0^T\bigl(\bigl| \nabla_x\delta u_r \bigr|^2,\mu_r\bigr)\ud r\biggr]^{1/2} +C \overline{\mathbb E}^0\biggl[\int_0^T|\delta Z_r^0|^2 \ud r\biggr]^{1/2},
\end{split}
\end{equation*}
Squaring it, and applying assumption \textit{i} in the statement
together with \eqref{eq:commonhelp2}, 
we obtain 
\begin{equation*} \begin{split}
\| \delta u_0 \|_{\cursr}^2
&\leq C_{T} \big( \vert x_0 - x_0' \vert^2  
+ 
\|   \delta \mu_0 \|_{-\cursr}^2\big)+C\overline{\mathbb E}^0\bigg[\int_0^T
\bigl(\bigl|\nabla_x\delta u_r \bigr|^2,\mu_r\bigr)\ud r\bigg]+C\overline{\mathbb E}^0\biggl[\int_0^T|\delta Z_r^0|^2 \ud r \biggr]
\\
& \leq  
  C_{\varepsilon_1,\sigma_0,T} \Bigl( \vert x_0 - x_0' \vert^2  
+ 
\|    \delta \mu_0\|_{-\cursr}^2 +
\bigl\vert 
(\delta u_0,\delta\mu_0) \bigr\vert \Bigr) - \frac{1}{C_{\varepsilon_1,\sigma_0,T}} |\delta Y_0^0|^2,
\end{split}
\end{equation*} 
where we assumed without any loss of generality that 
$C_{\varepsilon_1,\sigma_0,T} \geq 1$. 
Rearranging it
(and handling the duality bracket with a Young's inequality), we complete the proof.
\end{proof} 


\subsection{Application to existence and uniqueness}
\label{subse:4.6}

{\bf In this subsection, we NO longer take for granted the assumptions of Propositions 
\ref{prop:strong:existence:uniqueness:passage}
and
\ref{prop:strong:existence:uniqueness:passage2}.}
\vskip 4pt

The main objective is to prove 
the following statement, which subsumes 
Theorem
\ref{thm:1}:

\begin{theorem}
\label{thm:4.10}
Under Assumption \hyp{B}, 
there exist a constant $C \geq 1$, only depending
on the parameters in   \hyp{B} except $(\sigma_0,T)$, 
and a constant $C_{\varepsilon_1,\sigma_0,T}$, only depending on 
$\varepsilon_1$ and 
the parameters in   \hyp{B} (including $(\sigma_0,T)$) 
and non-decreasing with $T$, 
such that, with the two notations
$\kappa_{H^0}:=\|\nabla^2_{pp}H^0\|_{\infty}/2$ and 
$\lambda_H := \sup \{ \theta : \nabla^2_{pp}H\geq \theta I_{d}\}$, for 
$A$  
and $c$ as in Lemma 
\ref{lem:et}, 
and
 under the following two conditions:
\begin{enumerate}[i.]
\item  $\displaystyle \varepsilon_1 + \varepsilon_2 <\frac{\lambda_{H}}{C}$, 
\item 
$\displaystyle  \frac{c \varepsilon_2\sigma_0^2}{2A} >  \frac{C\kappa_{H^0}}{\varepsilon_1 \sigma_0^2}$, 
\end{enumerate}
the following holds true: 
\begin{enumerate}[a.]
\item  for all $(t,x_0,\mu)\in [0,T]\times\mathbb R^d\times\mathcal{P}(\mathbb T^d)$, the system \eqref{eq:major:FB:1}--\eqref{eq:minor:FB:2} has a unique solution;
\item
there exist 
two continuous mappings 
${\mathcal U}^0 : [0,T] \times {\mathbb R}^d \times {\mathcal P}({\mathbb T}^d)
\rightarrow {\mathbb R}$ and 
${\mathcal U} : [0,T] \times {\mathbb R}^d \times {\mathbb T}^d 
\times {\mathcal P}({\mathbb T}^d) \rightarrow
{\mathbb R}$
such that, for any 
$\cursr \in [1,\lfloor \curss \rfloor - (d/2+1)] \setminus {\mathbb N}$, 
${\mathcal U}^0(t,\cdot,\cdot)$
and 
${\mathcal U}(t,\cdot,\cdot,\cdot)$
belong to ${\mathscr C}^0(C_{\varepsilon_1,\sigma_0,T},\cursr)$ and 
${\mathscr C}(C_{\varepsilon_1,\sigma_0,T},\cursr,\cursr)$ 
for all 
$t \in [0,T]$
and the representation formula 
\eqref{eq:representation:weak:solution} holds true;
\item for a certain $\alpha >0$, 
the continuity condition 
\eqref{eq:U:priori:contintime}
is satisfied and ${\mathcal U}^0(t,\cdot,\cdot)$
and 
${\mathcal U}(t,\cdot,\cdot,\cdot)$
belong to ${\mathscr D}^0(C_{\varepsilon_1,\sigma_0,T},1)$
and 
${\mathscr D}(C_{\varepsilon_1,\sigma_0,T},1,3+\alpha)$ respectively. 
\end{enumerate}
\end{theorem}

\begin{proof}
The proof relies on a standard induction argument, used first in \cite{Delarue02}. 
\vskip 4pt

\noindent  \textit{First Step.} 
We prove in Section \ref{se:5}, see Proposition 
\ref{prop:H:approx}
with $R_0 := 
2 \max(C_{\varepsilon_1,\sigma_0,T},\kappa)$ 
and ${\mathbb s} := \lfloor \curss \rfloor - (d/2+1) > 3$ (since 
$\curss > d/2+5$)
that, under 
Assumptions \hyp{A1}, \hyp{A2} and \hyp{A4} (but not 
\hyp{A3}),
 there exists a constant $\delta >0$ (corresponding to 
${\mathfrak C}$ in 
Proposition 
 \ref{prop:H:approx}), 
depending on the parameters in Assumption 
 \hyp{A} but only depending on $g^0$ and $g$ through 
 any pair of (strictly) positive reals $(\tilde L,\tilde \alpha)$ such that 
 $g^0 \in {\mathscr C}^0(\tilde L,3+\tilde \alpha)$ and 
 $g \in  {\mathscr C}(\tilde L,3+\tilde \alpha,3+\tilde \alpha)$
 (the roles of $\tilde L$ and $\tilde \alpha$ in 
 the statement of Proposition 
 \ref{prop:H:approx} are respectively played by 
 ${\mathfrak L}$ and ${\mathbb s} - \lfloor {\mathbb s} \rfloor$)
  and satisfying 
 the following variant of \textit{a}, \textit{b} and \textit{c}:\emph{
 \begin{enumerate}[a'.]
\item \emph{[see part \textit{iii} in Proposition 
 \ref{prop:H:approx} together with the definition of 
 ${\mathbb r}$ in Theorem \ref{thm:local:FB}]}  for all $(t_0,x_0,\mu)\in [T-\delta,T]\times\mathbb R^d\times\mathcal{P}(\mathbb T^d)$, the system \eqref{eq:major:FB:1}--\eqref{eq:minor:FB:2} has a unique solution when: 1.  
 the BMO borm of
$(\int_0^{t} {Z}_s^0 \cdot \ud B_s^0)_{0 \le t \le T}$ is required to be bounded by $R_0$; 2.
${\boldsymbol u}$ is regarded as a continuous process with values in 
${\mathcal C}^{3+\beta}({\mathbb T}^d)$, for a certain $\beta \in (0,\tilde \alpha)$, 
and 
the gradient 
$(\nabla_x u_t)_{0 \le t \le T}$ is also required to be bounded by $R_0$; 3. the 
stochastic integral $(\int_{t_0}^t v_s^0(\cdot) \cdot \ud B_s^0)_{t_0 \le t \le T}$ is just 
regarded as a martingale $(m_t)_{t_0 \le t\le T}$ with values in 
${\mathcal C}^{1+\beta}({\mathbb T}^d)$, for the same value of $\beta$;
\item
\emph{[see part \textit{i} in Proposition  \ref{prop:H:approx}
together with the representation formulas 
in \eqref{eq:representation:short:time} and Proposition 
\ref{prop:representation:Z0}]}
there exist two positive reals $L$ and $\alpha$ 
(corresponding to $C$ and $({\mathbb s}+\lfloor {\mathbb s} \rfloor)/2$
in Proposition  \ref{prop:H:approx}) and 
two continuous mappings 
${\mathcal U}^0 : [T-\delta,T] \times {\mathbb R}^d \times {\mathcal P}({\mathbb T}^d)
\rightarrow {\mathbb R}$ and 
${\mathcal U} : [T-\delta,T] \times {\mathbb R}^d \times {\mathbb T}^d 
\times {\mathcal P}({\mathbb T}^d) \rightarrow
{\mathbb R}$
such that 
${\mathcal U}^0(t,\cdot,\cdot)$
and 
${\mathcal U}(t,\cdot,\cdot,\cdot)$
belong to ${\mathscr D}^0(L,1)$ and 
${\mathscr D}(L,1,3+\alpha)$ 
for all 
$t \in [T-\delta,T]$,
and the representation formula 
\eqref{eq:representation:weak:solution} (restricted to the interval $[T-\delta,T]$ and without the representation of $v^0$ which has not been defined yet) holds true;
\item \emph{[see part \textit{ii} in Proposition  \ref{prop:H:approx}]}
the continuity condition 
\eqref{eq:U:priori:contintime}
is satisfied. 
\end{enumerate}}

We now fix $(t_0,x_0,\mu) \in [T-\delta,T] \times {\mathbb R}^d \times 
{\mathcal P}({\mathbb T}^d)$. By Lemma 
\ref{lem:minor:regularity}
(which one can invoke with ${\boldsymbol \mu}$ given by 
item \textit{a}' above),
the pair $({\boldsymbol u},{\boldsymbol m})$ takes values in 
${\mathcal C}^{\curss}({\mathbb T}^d)
\times {\mathcal C}^{\curss-2}({\mathbb T}^d)$. 
We can also apply 
Lemma \ref{lem:reg:v0}: it
permits to represent the martingale as a stochastic integral (with 
${\boldsymbol v}^0$ as integrand). 
The combination of the latter two lemmas says that 
the solution given above fits
the requirements of Definition
\ref{def:forward-backward=MFG:solution}.
By Proposition \ref{prop:minor:higher} and Lemma 
\ref{lem:BMO:major}, we get bounds 
for $\vert Y_{t_0}^0 \vert$ and $\| u_{t_0}\|_{\curss}$ 
from which we deduce that, for a possibly new value of 
$C_{\varepsilon_1,\sigma_0,T}$, 
$\vert {\mathcal U}^0(t_0,x_0,\mu) \vert$  
and 
$\| {\mathcal U}(t_0,x_0,\cdot,\mu)\|_{\curss}$
are 
 bounded by 
$C_{\varepsilon_1,\sigma_0,T}$. By 
Lemma \ref{lema:representation:Z0:gen} (applied
to 
$({\boldsymbol U}^0,{\boldsymbol V}^0)
= (u_t(x),v_t(x))_{t_0  \leq t \le T}$
and 
${\mathcal V}^0= {\mathcal U}(\cdot,\cdot,x,\cdot)$ for 
any fixed $x \in {\mathbb T}^d$), 
we get the representation formula 
\eqref{eq:representation:weak:solution}
for 
${\boldsymbol v}^0_t(x)$ (observing that 
${\boldsymbol v}^0_t$ and $\nabla_{x_0} {\mathcal U}$ are continuous in 
$x$, we get the formula \eqref{eq:representation:weak:solution} 
for almost every $(t,\omega^0) \in [0,T] \times \Omega^0$, for all $x \in {\mathbb T}^d$).
By Lemma \ref{lem:BMO:major} again, we also deduce that, whatever the solution
$({\boldsymbol X}^{0,\prime},{\boldsymbol Y}^{0,\prime} ,{\boldsymbol Z}^{0,\prime},
{\boldsymbol \mu}',{\boldsymbol u}',{\boldsymbol m}')$
to 
\eqref{eq:major:FB:1}--\eqref{eq:minor:FB:2}, 
the BMO norm of $(\int_0^t Z_s^{0,\prime} \cdot \ud B_s^0)_{0 \le t \le T}$ is bounded by 
$R_0$ (again, this may require to modify $C_{\varepsilon_1,\sigma_0,T}$), which proves from \textit{a}' 
that $({\boldsymbol X}^{0},{\boldsymbol Y}^{0} ,{\boldsymbol Z}^{0},
{\boldsymbol \mu},{\boldsymbol u},{\boldsymbol m})$ is the unique solution to 
\eqref{eq:major:FB:1}--\eqref{eq:minor:FB:2}. 
Furthermore, the assumptions of Propositions 
\ref{prop:strong:existence:uniqueness:passage}
and
\ref{prop:strong:existence:uniqueness:passage2} 
are satisfied over the interval $[T-\delta,T]$. 
In turn, we can invoke 
Proposition 
\ref{prop:4:7}
to deduce that 
\eqref{eq:globalLip}
(at least, the analogue of it, but on the interval $[T-\delta,T]$) holds true. 
We explain below how to use this bound. 

We recall that the initial conditions to \eqref{eq:major:diff:1}--\eqref{eq:minor:diff:2} (which are now initialized from $t_0$) 
are $\delta X_0=x_0'-x_0$ and $\delta \mu_0 = \mu' - \mu$. And then, by the representation formula 
\eqref{eq:representation:weak:linear:solution}, we have, for 
$\cursr \in [3, \lfloor \curss \rfloor -(d/2+1)]\setminus {\mathbb N}$ (which is possible because $\curss > d/2+5$ and then $\lfloor \curss \rfloor > d/2+4 = (d/2+1)+3$), 
\begin{equation} 
\label{eq:14:12:5}
\begin{split}
&\bigl\vert \nabla_{x_0} {\mathcal U}^0(t_0,x_0,\mu)\cdot (x_0'-x_0) +
\left(\delta_{\mu}\mathcal{U}^0(t_0,x_0,\mu),\mu' - \mu\right)
\bigr\vert 
\leq C_{\varepsilon_1,\sigma_0,T} \Bigl(  \vert  x_0' - x_0 \vert + \| \mu'-\mu \|_{-\cursr} \Bigr),
\\
&\Bigl\| \nabla_{x_0}{\mathcal U}(t_0,x_0,\cdot,\mu)\cdot (x_0'-x_0) +
\left(\delta_{\mu}\mathcal{U}(t_0,x_0,\cdot,\mu),\mu' - \mu\right)
\Bigr\|_{\cursr} 
\leq C_{\varepsilon_1,\sigma_0,T} \Bigl(  \vert  x_0' - x_0 \vert + \| \mu'-\mu \|_{-\cursr} \Bigr).
\end{split} 
\end{equation} 
Next, 
we focus on $\delta_\mu {\mathcal U}$ (which is the most difficult term). 
Choosing $x_0'=x_0$, 
we rewrite the last inequality as
\begin{equation}
\label{eq:coroll:4:8:1}
\biggl\| \int_{{\mathbb T}^d} \delta_\mu {\mathcal U}(t_0,x_0,\mu,\cdot,y)
\ud \bigl( \mu'- \mu\bigr)(y) 
\biggr\|_{\cursr} 
\leq 
C_{\varepsilon_1,\sigma_0,T}  \| \mu'-\mu \|_{-\cursr}. 
\end{equation} 
Assume for a while that for some constant $\gamma>0$, $\mu$ is bounded from below by 
$\gamma \textrm{\rm Leb}_{{\mathbb T}^d}$ (which we denote 
$\mu \geq \gamma \textrm{\rm Leb}_{{\mathbb T}^d}$). 
For a given smooth even density $\rho$
and for $l \in \{1,\cdots,\lfloor \cursr \rfloor\}$, consider 
the $l$th partial derivative 
$\partial^l \rho$ of $\rho$ along $l$ (arbitrary) directions of ${\mathbb R}^d$ (with possible repetitions). Then,
for $\epsilon \in {\mathbb R}$ small enough and $y_0 \in {\mathbb T}^d$ fixed, 
$\mu + \epsilon \partial^l \rho(y_0-\cdot) \cdot \textrm{\rm Leb}_{{\mathbb T}^d}$ is a probability measure and  the above yields:  
\begin{equation}
\label{eq:14:12:24:1}
\biggl\|  \nabla^l_{y_0} \biggl( \int_{{\mathbb R}^d} \delta_\mu {\mathcal U}(t_0,x_0,\mu,\cdot,y)
\rho(y_0-y) \ud y \biggr)
\biggr\|_{\cursr} 
\leq 
C_{\varepsilon_1,\sigma_0,T}  \|  \nabla^l \rho \cdot \textrm{\rm Leb}_{{\mathbb T}^d}  \|_{-\cursr} \leq C_{\varepsilon_1,\sigma_0,T}. 
\end{equation} 
The case $l=0$ can be easily included in the left-hand side by returning back to 
\eqref{eq:coroll:4:8:1}, choosing therein $\mu' = \epsilon \delta_{y_0} + (1-\epsilon) \mu$ and then invoking 
\eqref{eq:convention:deltam}.

Denoting the integral in the right-hand side of 
\eqref{eq:14:12:24:1} by $[\delta_\mu {\mathcal U}(t_0,x_0,\mu,\cdot,\cdot) * \rho](y_0)$ (with the convolution acting implicitly on the last argument), 
we now address the H\"older regularity of $\nabla_{y}^{\lfloor \cursr \rfloor} [\delta_\mu {\mathcal U}(t_0,x_0,\mu,\cdot,\cdot) * \rho]$ in the variable 
$y$ by a similar argument. For a vector $z \in {\mathbb R}^d$, we apply 
\eqref{eq:coroll:4:8:1} 
to $\mu$ (still satisfying $\mu \geq \gamma \textrm{\rm Leb}_{{\mathbb T}^d}$) and  $\mu' = \mu + \epsilon [ \partial^l \rho ( \cdot - (y_0+ z))- \partial^l \rho(\cdot - y_0)]$
(for $\epsilon$ small enough and with $l=\lfloor \cursr \rfloor$). We obtain 
\begin{equation}
\label{eq:14:12:24:2} 
\begin{split}
&\biggl\|  \nabla^{\lfloor \cursr \rfloor}_{y_0} \biggl(  \int_{{\mathbb R}^d} \Bigl[  \delta_\mu {\mathcal U}(t_0,x_0,\mu,\cdot,y_0+z-y)
- 
 \delta_\mu {\mathcal U}(t_0,x_0,\mu,\cdot,y_0-y)
\Bigr] \rho (y) \ud y 
\biggr)
\biggr\|_{\cursr} 
\\
&\leq 
 C_{\varepsilon_1,\sigma_0,T}
\Bigl\|  \nabla^{\lfloor \cursr \rfloor} \bigl[  \rho(\cdot + z) - \rho \bigr] \cdot \textrm{\rm Leb}_{{\mathbb T}^d}  \Bigr\|_{-\cursr}
\leq  C_{\varepsilon_1,\sigma_0,T} \vert z \vert^{\cursr - \lfloor \cursr \rfloor}. 
\end{split}
\end{equation} 
The two inequalities \eqref{eq:14:12:24:1}
and 
\eqref{eq:14:12:24:2}
are true for any fixed (smooth) density $\rho$.
Considering a sequence of mollifiers $(\rho_n)_{n \geq 1}$, 
this 
 shows that, for any $l_x,l_y \in \{0,\cdots,\lfloor r \rfloor\}$, 
the functions 
$(\nabla^{l_x}_x \nabla^{l_y}_y [\delta_\mu {\mathcal U}(t_0,x_0,\mu,\cdot,\cdot) * \rho_n])_{n \geq 1}$
are relatively compact 
for the topology of uniform convergence on ${\mathbb T}^d \times {\mathbb T}^d$. 
Since ${\mathcal U}(t,\cdot,\cdot,\cdot)$ belongs to 
${\mathscr D}(L,1,3+\alpha)$ 
for all $t \in [T-\delta,T]$, we 
already know that 
$\delta_\mu {\mathcal U}(t_0,x_0,\mu,\cdot,\cdot) * \rho_n$
converges to 
$\delta_\mu {\mathcal U}(t_0,x_0,\mu,\cdot,\cdot)$
for the topology of uniform convergence on ${\mathbb T}^d \times {\mathbb T}^d$. 
We deduce that, 
for any $l_x,l_y \in \{0,\cdots,\lfloor r \rfloor\}$, 
the derivatives
$\nabla^{l_x}_x \nabla^{l_y}_y \delta_\mu {\mathcal U}(t_0,x_0,\mu,\cdot,\cdot)$ exist and 
satisfy 
\begin{equation}
\label{eq:14:12:4}
\begin{split}
&\sup_{l=0,\cdots,\lfloor \cursr \rfloor}
\Bigl\|  \nabla^l_{y_0}  \delta_\mu {\mathcal U}(t_0,x_0,\mu,\cdot,y)
\Bigr\|_{\cursr} 
 \leq C_{\varepsilon_1,\sigma_0,T}, 
\\
&\Bigl\|    \nabla^{\lfloor \cursr \rfloor}_{y_0}   \delta_\mu {\mathcal U}(t_0,x_0,\mu,\cdot,y_0+z)
- 
 \nabla^{\lfloor \cursr \rfloor}_{y_0} \delta_\mu {\mathcal U}(t_0,x_0,\mu,\cdot,y_0)
\Bigr)
\Bigr\|_{\cursr} 
 \leq 
C_{\varepsilon_1,\sigma_0,T} \vert z \vert^{\cursr - \lfloor \cursr \rfloor}. 
\end{split}
\end{equation} 
The above holds true under the condition 
$\mu \geq \gamma {\rm Leb}_{{\mathbb T}^d}$ for some $\gamma >0$.  
Also,
thanks to the continuity of 
$\delta_\mu {\mathcal U}$ in $\mu$ (given again by the fact that 
${\mathcal U}(t,\cdot,\cdot,\cdot)$ belongs to 
${\mathscr D}(L,1,3+\alpha)$ 
for all $t \in [T-\delta,T]$),
 we can prove that \eqref{eq:14:12:4} is true for any $\mu \in {\mathcal P}({\mathbb T}^d)$
by a new approximation argument. 
Indeed, for any $\mu \in {\mathcal P}({\mathbb T}^d)$ and $\epsilon \in (0,1)$, we can consider 
$(1-\epsilon) \mu + \mu {\rm Leb}_{{\mathbb T}^d}$ as new probability measure:
\eqref{eq:14:12:4} is true for all $\epsilon \in (0,1)$, and we can argue by a new compactness argument. Together with 
\eqref{eq:14:12:5}, this shows that, for any $t_0 \in [T-\delta,T]$, 
the function 
${\mathcal U}(t_0,\cdot,\cdot,\cdot)$ satisfies the condition 
${\mathscr C}(C_{\varepsilon_1,\sigma_0,T},\cursr,\cursr)$
By the same argument, ${\mathcal U}^0(t_0,\cdot,\cdot)$ satisfies the condition 
${\mathscr C}^0(C_{\varepsilon_1,\sigma_0,T},\cursr)$. 
\vskip 4pt

\noindent  {\it Second Step.} 
The next step is to iterate, considering now the problem 
\eqref{eq:major:FB:1}--\eqref{eq:minor:FB:2}
with ${\mathcal U}^0(T-\delta,\cdot,\cdot)$ and 
${\mathcal U}(T-\delta,\cdot,\cdot,\cdot)$ as new terminal conditions 
at time $T-\delta$ (in place of $g^0$ and $g$ at time $T$). 
By the first step (which does not require the terminal condition $g$ to satisfy \hyp{A3}) and thanks to the lower bound $\cursr > 3$ (without any loss of generality, we can assume 
$\cursr  > 3 + \tilde{\alpha}$ for the same $\tilde{\alpha}$ as the one used in the first step), 
we can find a new constant $\delta'>0$ only depending on the parameters in 
  \hyp{B} and on the constant $C_{\varepsilon_1,\sigma_0,T}$ 
(which also depends on the parameters in \hyp{B}) such that 
the items \textit{a}', \textit{b}' and \textit{c}' above are satisfied but on 
$[T-(\delta+\delta'),T-\delta]$ in place of $[T-\delta,T]$.
This makes it possible to extend the fields 
${\mathcal U}^0$ and ${\mathcal U}$ to the interval $[T-(\delta+\delta'),T]$. 
By combining the regularity properties given by item \textit{b}'
 on the intervals $[T-(\delta+\delta'),T-\delta]$ and $[T-\delta,T]$ respectively, we deduce that 
 \textit{b}' holds true on the entire $[T-(\delta+\delta'),T]$. Similarly, 
 we can easily check that \textit{c}' is satisfied on the entire 
$[T-(\delta+\delta'),T]$.

Now, we observe that any solution given by \textit{a}' on the interval $[T-(\delta+\delta'),T-\delta]$ can be extended 
to the interval $[T-\delta,T]$ by solving the problem addressed in the first step but restarting from 
the random initial condition $(X_{T-\delta}^0,\mu_{T-\delta})$. Solvability of the problem 
\eqref{eq:major:FB:1}--\eqref{eq:minor:FB:2} with a random initial condition is not an issue. The analysis of forward-backward systems with random initial conditions has been well documented in other examples: see for instance
\cite{Delarue02} for forward-backward stochastic differential equations in finite dimension and 
\cite[Subsection 5.1]{CardaliaguetDelarueLasryLions} for a similar analysis in the mean field case. Here, one can use the collection of regular conditional probabilities
$({\mathbb P}^0_{T-\delta}(\omega^0,\cdot))_{\omega^0 \in \Omega^0}$ of 
${\mathbb P}^0$ given the $\sigma$-field generated by $(B_s^0)_{T-(\delta+\delta') \leq s \leq T-\delta}$
 in order to reduce the system (initialized from a random initial 
condition) to a problem with a deterministic initial condition. Namely, 
for an initial condition $(X_{T-\delta}^0,\mu_{T-\delta})$ as above, one can solve the pair of two forward equations:
\begin{equation*} 
\begin{split}
&\ud X_t^0 = - \nabla_p H^0 \bigl( X_t^0, \nabla_{x_0} {\mathcal U}^0(t,X_t^0,\mu_t^0) \bigr) 
\ud t + \sigma_0 \ud B_t^0,
\\
&\partial_t \mu_t - \tfrac12 \Delta_x \mu_t - {\rm div}_x 
\Bigl( \nabla_p H \bigl( \cdot, \nabla_x {\mathcal U}(t,X_t^0,\cdot,\mu_t^0) \bigr) \mu_t
\Bigr) =0, \quad t \in [T-\delta,T].
\end{split}
\end{equation*} 
Since ${\mathcal U}^0 \in {\mathscr D}^0(L,3+\alpha)$ and 
${\mathcal U} \in {\mathscr D}(L,1,3+\alpha)$, existence and uniqueness are standard (it suffices to solve a conditional McKean-Vlasov equation, see for instance 
\cite[Chapter 2]{CarmonaDelarue_book_II}). 
Then, defining $(Y_t^0,Z_t^0)_{T-\delta \leq t \leq T}$ and 
$(u_t,v_t^0)_{T-\delta \le t \le T}$ 
as suggested by 
\eqref{eq:representation:weak:solution}, namely letting
\begin{equation*} 
\begin{split} 
&Y_t^0 := {\mathcal U}^0(t,X_t^0,\mu_t),
\
Z_t^0 := \nabla_{x_0} {\mathcal U}^0(t,X_t^0,\mu_t), \quad t \in [T-\delta,T], 
\\
&u_t(x) := {\mathcal U}(t,X_t^0,x,\mu_t),
\
v_t^0(x) := \nabla_{x_0} {\mathcal U}(t,X_t^0,x,\mu_t), 
 \quad (t,x) \in [T-\delta,T] \times 
{\mathbb T}^d,
\end{split} 
\end{equation*} 
we know from Proposition 
\ref{prop:strong:existence:uniqueness:passage} that, $\omega^0$-${\mathbb P}^0$-almost surely, under the 
disintegrated probability ${\mathbb P}^0_{T-\delta}(\omega^0,\cdot)$, 
the two triplets $(X^0_t,Y_t^0,Z^0_t)_{T-\delta \le t \le T}$ and 
$(\mu_t,u_t,v_t^0)_{T-\delta \le t \le T}$ that have been hence defined solve the 
two forward-backward equations in 
\eqref{eq:major:FB:1}--\eqref{eq:minor:FB:2} 
(with the realization 
$(X_{T-\delta}^0(\omega^0),\mu_{T-\delta}(\omega^0))$ as deterministic initial condition). 
And then, 
by construction of the regular conditional probability, 
they solve the systems 
\eqref{eq:major:FB:1}--\eqref{eq:minor:FB:2}
under the original probability
${\mathbb P}^0$, with 
$(X_{T-\delta}^0,\mu_{T-\delta})$
as (now random) initial condition. 

By combining the solution 
originally given by \textit{a}' on $[T-(\delta+\delta'),T-\delta]$ with 
its extension to $[T-\delta,T]$, we get a solution to 
\eqref{eq:major:FB:1}--\eqref{eq:minor:FB:2} in the sense of \textit{a}', starting from $(t_0,x_0,\mu)$ with $t_0 \in [T-(\delta+\delta'),T]$. 
This solution is necessarily unique. Indeed, if  
$(\widetilde X^0_t,\widetilde Y_t^0,\widetilde Z^0_t)_{T-(\delta+\delta') \le t \le T}$ and 
$(\widetilde \mu_t,\widetilde u_t,\widetilde v_t^0)_{T-(\delta+\delta') \le t \le T}$
form another solution with the same initial condition, then it satisfies the system 
\eqref{eq:major:FB:1}--\eqref{eq:minor:FB:2} on $[T-\delta,T]$ with 
$(\widetilde X^0_{T-\delta},\widetilde \mu_{T-\delta})$ as (random) initial condition. 
In particular, for ${\mathbb P}^0$-almost every $\omega^0$, it satisfies the system 
\eqref{eq:major:FB:1}--\eqref{eq:minor:FB:2} on $[T-\delta,T]$ under 
${\mathbb P}_{T-\delta}^0(\omega^0,\cdot)$,
with the realization 
$(\widetilde X^0_{T-\delta}(\omega^0),\widetilde \mu_{T-\delta}(\omega^0))$ as deterministic initial condition. 
And, then, by the uniqueness property on $[T-\delta,T]$, 
this new solution can be represented 
as in 
\eqref{eq:representation:weak:solution}
on the interval $[T-\delta,T]$. In particular, 
$\widetilde{Y}_{T-\delta}^0 = {\mathcal U}^0(T-\delta,\widetilde X_{T-\delta}^0,\widetilde \mu_{T-\delta})$
and
$\widetilde{u}_{T-\delta} = {\mathcal U}(T-\delta,\widetilde X_{T-\delta}^0,\cdot,\widetilde \mu_{T-\delta})$. 
This says that the triplets 
$(\widetilde X^0_t,\widetilde Y_t^0,\widetilde Z^0_t)_{T-(\delta+\delta') \le t \le T-\delta}$ and 
$(\widetilde \mu_t,\widetilde u_t,\widetilde v_t^0)_{T-(\delta+\delta') \le t \le T-\delta}$
satisfy the same equations as 
$(X^0_t,Y_t^0,Z^0_t)_{T-(\delta+\delta') \le t \le T-\delta}$ and 
$(\mu_t,u_t,v_t^0)_{T-(\delta+\delta') \le t \le T-\delta}$
on $[T-(\delta+\delta'),T-\delta]$. By (short time) uniqueness on $[T-(\delta+\delta'),T-\delta]$, the latter two processes coincide. 
Then, they restart from the same states at time $T-\delta$ and eventually coincide on the entire $[T-(\delta+\delta'),T]$. 

In the end, this proves that the three items \textit{a}, \textit{b} and \textit{c} listed in the statement now hold true but on the wider interval $[T-(\delta+\delta'),T]$. 
\vskip 4pt

\noindent  \textit{Third Step.} The proof now follows the scheme introduced in \cite{Delarue02}. We can restart from 
the first step, but replacing $[T-\delta,T]$ by $[T-(\delta+\delta'),T]$. We get that, for any 
$t_0 \in [T-(\delta+\delta'),T]$, 
${\mathcal U}^0(t_0,\cdot,\cdot)$
and
${\mathcal U}(t_0,\cdot,\cdot,\cdot)$ satisfy the conditions
 ${\mathscr C}^0(C_{\varepsilon_1,\sigma_0,T},\cursr)$ 
 and
${\mathscr C}(C_{\varepsilon_1,\sigma_0,T},\cursr,\cursr)$
respectively. And then, we can repeat the second step and extend the
maps ${\mathcal U}^0$ and ${\mathcal U}$ to 
$ [T-(\delta+2\delta'),T-(\delta+\delta')]$ for the same $\delta'$ as in the second step. 
By iterating the argument, we can cover the entire interval to $[0,T]$. Details are left to the reader. 
\end{proof} 

We are now in position to prove 
Theorem 
\ref{thm:2}:

\begin{theorem}
\label{thm:4.10:b}
In addition to Assumption  \hyp{B}, 
assume that 
\begin{enumerate}[i.]
\item  the coefficient $(t,x_0,\mu) \mapsto 
f_t^0(x_0,\mu)$ is H\"older continuous in time, uniformly in 
$(x_0,\mu)$; the coefficient $(t,x_0,x,\mu) \mapsto 
f_t(x_0,x,\mu)$ is H\"older continuous in time, 
uniformly in $(x_0,x,\mu)$;
\item $x_0 \mapsto g^0(x_0,\mu)$ has 
H\"older continuous second-order derivatives,
uniformly in $\mu$; $x_0 \mapsto g(x_0,x,\mu)$
has H\"older continuous second-order derivatives,  
uniformly in $(x,\mu)$. 
 \end{enumerate}
Then, 
under the same conditions (i) and (ii) as in the statement of 
Theorem 
\ref{thm:4.10}, 
the functions ${\mathcal U}^0$ and ${\mathcal U}$ are the unique solutions
to the
master equation
\eqref{eq:major:1}--\eqref{eq:minor:1}
in the class of pairs 
$(V^0,V)$ such that
\begin{enumerate}
\item $(t,x_0,\mu) \mapsto (\partial_t V^0(t,x_0,\mu),  \nabla_{x_0} V^0(t,x_0,\mu), \nabla^2_{x_0} V^0(t,x_0,\mu))$ is continuous on $[0,T] \times {\mathbb R}^d \times {\mathcal P}({\mathbb T}^d)$; 
$(t,x_0,\mu,y) \mapsto (\partial_\mu V^0(t,x_0,\mu,y), \nabla_y \partial_\mu V^0(t,x_0,\mu,y))$ is continuous on $[0,T] \times {\mathbb R}^d \times {\mathcal P}({\mathbb T}^d) \times {\mathbb T}^d$; 
\item $(t,x_0,x,\mu) \mapsto (\partial_t V(t,x_0,x,\mu),  \nabla_{x_0} V(t,x_0,x,\mu), \nabla^2_{x_0} V(t,x_0,x,\mu), 
 \nabla_{x} V(t,x_0,x,\mu), \nabla^2_{x} V(t,x_0,x,\mu))$ is continuous on $[0,T] \times {\mathbb R}^d \times {\mathbb T}^d \times {\mathcal P}({\mathbb T}^d)$; 
$(t,x_0,x,\mu,y) \mapsto (\partial_\mu V(t,x_0,x,\mu,y), \nabla_y \partial_\mu V(t,x_0,x,\mu,y))$ is continuous on $[0,T] \times {\mathbb R}^d \times {\mathbb T}^d \times {\mathcal P}({\mathbb T}^d) \times {\mathbb T}^d$, 
\item $(t,x_0,\mu) \mapsto \nabla_{x_0} V^0(t,x_0,\mu)$ and $(t,x_0,x,\mu) \mapsto(\nabla_{x_0} V(t,x_0,x,\mu), \nabla_x V(t,x_0,x,\mu))$ are Lipschitz 
continuous with respect to $(x_0,\mu)$ and $(x_0,x,\mu)$ respectively 
(using the distance ${\mathbb W}_1$ to handle the argument $\mu$), uniformly in $t$;
\item  
 $(t,x_0,\mu) \mapsto (V^0(t,x_0,\mu),\nabla_{x_0} V^0(t,x_0,\mu))$
 and
$(t,x_0,\mu) \mapsto( \| V(t,x_0,\cdot,\mu)\|_{\cursr}, \| \nabla_{x_0} V(t,x_0,\cdot,\mu)\|_{\cursr})$ are globally
bounded, for any $\cursr \in [1,\lfloor \curss \rfloor -(d/2+1)] \setminus {\mathbb N}$. \end{enumerate}
In particular,  
${\mathcal U}^0$ and ${\mathcal U}$ have 
first-order derivative in $t$ and 
second-order derivatives in $x_0$, which are jointly continuous with respect to all their arguments.\end{theorem}

\begin{proof}
The proof relies on 
the iteration principle introduced in the proof of Theorem 
\ref{thm:4.10}. We use in particular the same notation as therein. 

On the interval $[T-\delta,T]$, the result follows from 
Corollary
\ref{corollary:U0:PDE} and Proposition \ref{prop:H:approx} (item iv). 
It says in particular that, at time $T-\delta$, the function 
$x_0 \mapsto {\mathcal U}^0(T-\delta,x_0,\mu)$ has 
H\"older continuous second-order derivatives,
uniformly in $\mu$, and the function $x_0 \mapsto {\mathcal U}(T-\delta,x_0,x,\mu)$
has H\"older continuous second-order derivatives, 
uniformly in $(x,\mu)$. This makes it possible to reapply 
Corollary
\ref{corollary:U0:PDE}
and Proposition \ref{prop:H:approx} (item iv)
but on the interval
$[T-(\delta+\delta'),T-\delta]$. By continuity properties of 
${\mathcal U}^0$ and ${\mathcal U}$ (and of their derivatives), we get a solution 
on $[T-(\delta+\delta'),T]$. Again,
Corollary
\ref{corollary:U0:PDE} and Proposition \ref{prop:H:approx} (item iv) say that the functions 
$x_0 \mapsto {\mathcal U}^0(T-(\delta+\delta'),x_0,\mu)$ has 
H\"older continuous second-order derivatives,
uniformly in $\mu$, and the function $x_0 \mapsto {\mathcal U}(T-(\delta+\delta'),x_0,x,\mu)$
has H\"older continuous second-order derivatives, 
uniformly in $(x,\mu)$. By iterating, we get that ${\mathcal U}^0$ and 
${\mathcal U}$ solve the two master equations on 
the entire $[0,T]$. 

We now turn to 
uniqueness. 
By 
Proposition 
\ref{prop:verif}, we know that any solution 
$(V^0,V)$ to the master equation, as given in the statement, induces a solution to the 
forward-backward system 
\eqref{eq:major:FB:1}--\eqref{eq:minor:FB:2} (for a given initial condition $(x_0,\mu) \in {\mathbb R}^d \times {\mathcal P}({\mathbb T}^d)$). The  bounds on 
$(V^0,V)$ required in item (4) of the statement say that the resulting solution 
$({\boldsymbol u},{\boldsymbol v}^0)$ 
to 
\eqref{eq:minor:FB:2}
takes values in ${\mathcal C}^\mathfrc{r}({\mathbb T}^d) \times 
{\mathcal C}^{\cursr}({\mathbb T}^d)$, for any 
$\cursr \in [1,\lfloor \mathfrc{s} \rfloor -(d/2+1)] \setminus {\mathbb N}$.
Then, 
Lemmas
\ref{lem:minor:regularity}
and
\ref{lem:reg:v0} make it possible to gain extra regularity and  guarantee 
that (ii) and (iii) in item 2 of Definition 
\ref{def:forward-backward=MFG:solution}
are satisfied. 
This shows that, for any fixed
initial condition $(x_0,\mu) \in {\mathbb R}^d \times {\mathcal P}({\mathbb T}^d)$, 
 $(V^0,V)$ induces a solution to \eqref{eq:major:FB:1}--\eqref{eq:minor:FB:2} in 
the sense of Definition 
\ref{def:forward-backward=MFG:solution}. Uniqueness 
to \eqref{eq:major:FB:1}--\eqref{eq:minor:FB:2}, as given by 
Theorem \ref{thm:4.10} under the initial condition 
 $(x_0,\mu)$, together with the representation formula 
 \eqref{eq:representation:weak:solution}
 imply that 
 $(V^0,V)$ is necessarily equal to 
 $({\mathcal U}^0,{\mathcal U})$.
\end{proof}

\subsection{Proof of Proposition \ref{prop:strong:existence:uniqueness:passage2}}
\label{subse:4.7}

\begin{proof}
Recall that $(({\boldsymbol X}^0,{\boldsymbol Y}^0,{\boldsymbol Z}^0), ({\boldsymbol\mu},{\boldsymbol u},{\boldsymbol v^0}))$ is the unique solution of \eqref{eq:major:FBG:1}--\eqref{eq:minor:FBG:2} and has the representation \eqref{eq:representation:weak:solution}. For any $\varepsilon>0$, let $(({\boldsymbol X}^{0,\varepsilon},{\boldsymbol Y}^{0,\varepsilon},{\boldsymbol Z}^{0,\varepsilon}), ({\boldsymbol\mu}^{\varepsilon},{\boldsymbol u}^{\varepsilon},{\boldsymbol v^{0,\varepsilon}}))$ be the unique solution of \eqref{eq:major:FBG:1}--\eqref{eq:minor:FBG:2} corresponding to the initial data $(x_0+\varepsilon  (x_0'-x_0),\mu+\varepsilon (\mu'-\mu))$. By Proposition \ref{prop:strong:existence:uniqueness:passage}, we have, ${\mathbb P}^0$-almost surely, 
 \begin{equation} 
 \label{eq:proof:prop:4.2:representation:solutions}
\begin{split} 
&Y_t^{0,\varepsilon} = {\mathcal U}^0(t,X_t^{0,\varepsilon},\mu_t^{\varepsilon}),
\
Z_t^{0,\varepsilon} = \nabla_{x_0} {\mathcal U}^0(t,X_t^{0,\varepsilon},\mu_t^{\varepsilon}), \quad t \in [0,T], 
\\
&u_t^{\varepsilon}(x) = {\mathcal U}(t,X_t^{0,\varepsilon},x,\mu_t^{\varepsilon}),
\
v_t^{0,\varepsilon}(x) = \nabla_{x_0} {\mathcal U}(t,X_t^{0,\varepsilon},x,\mu_t^{\varepsilon}), 
 \quad (t,x) \in [0,T] \times 
{\mathbb T}^d. 
\end{split} 
\end{equation} 
We let
\begin{equation*}
\begin{split}
&\left((\triangle{\boldsymbol X}^{0,\varepsilon},\triangle{\boldsymbol Y}^{0,\varepsilon},\triangle{\boldsymbol Z}^{0,\varepsilon}), (\triangle{\boldsymbol\mu}^{\varepsilon},\triangle{\boldsymbol u}^{\varepsilon},\triangle{\boldsymbol v^{0,\varepsilon}})\right)
\\
&:=\left( \Bigl(\frac{{\boldsymbol X}^{0,\varepsilon}-{\boldsymbol X}^{0}}{\varepsilon},\frac{{\boldsymbol Y}^{0,\varepsilon}-{\boldsymbol Y}^{0}}{\varepsilon},\frac{{\boldsymbol Z}^{0,\varepsilon}-{\boldsymbol Z}^{0}}{\varepsilon} \Bigr), \Bigl(\frac{{\boldsymbol\mu}^{\varepsilon}-{\boldsymbol\mu}}{\varepsilon},\frac{{\boldsymbol u}^{\varepsilon}-{\boldsymbol u}}{\varepsilon},\frac{{\boldsymbol v^{0,\varepsilon}-\boldsymbol v^{0}}}{\varepsilon} \Bigr)\right).
\end{split}
\end{equation*}

\noindent {\it First Step.}  Let us first solve the following forward system
\begin{equation}
\label{eq:forward:forward:diff:1}
\begin{split}
&\ud \delta X_t^0 = - \delta \bigl[ \nabla_p H^0(X_t^0,\nabla_{x_0} w^0(t,X_t^0)) \bigr] \ud t,
\\
&\partial_t \delta \mu_t - \tfrac12 \Delta_x \delta \mu_t - {\rm div}_x \Bigl(  
\nabla_p H(\cdot ,\nabla_x u_t(\cdot))
 \delta \mu_t \Bigr) 
-
{\rm div}_x \Bigl(  
\delta \bigl[
\nabla_p H(\cdot ,\nabla_x \mathcal{U}(t,X_t^0,\cdot,\mu_t))
\bigr] \mu_t
\Bigr) 
=0,  
\end{split}
\end{equation}
for $t \in [0,T]$, 
with the two initial conditions 
$\delta X_0^0=x_0'-x_0$ and 
$\delta \mu_0= \mu'-\mu$. Above, the second equation is formally obtained from the forward equation in \eqref{eq:minor:diff:2} by replacing the term ${\rm div}_x(
\delta [
\nabla_p H(\cdot ,\nabla_x u_t(\cdot))
] \mu_t
)$ by ${\rm div}_x (
\delta [
\nabla_p H(\cdot ,\nabla_x \mathcal{U}(t, X_t^0, \cdot, \mu_t))
] \mu_t
)$.
Since $\nabla_{x_0} w^0$ is $x_0$-continuously differentiable with a bounded derivative 
(see \eqref{eq:auxiliary:HJB}), we easily deduce from the theory of ODEs that the first equation in \eqref{eq:forward:forward:diff:1} admits a unique solution ${\delta\boldsymbol X^0}$ in the sense of item 1 in Definition \ref{def:linear forward-backward=MFG:solution}. It satisfies $\sup_{t\in [0,T]}|\delta X_t^0|\leq C_T| x_0'-x_0|$. Moreover, 
${\mathbb P}^0$-almost surely, 
\begin{equation}\label{eq:stability:X0}
\sup_{t\in [0,T]}|\Delta X_t^{0,\varepsilon}|\leq C_T| x_0'-x_0|,\quad\lim_{\varepsilon\to 0}\sup_{t\in [0,T]}\left|X_t^{0,\varepsilon}-X_t^0\right|=0 \quad\text{and}\quad \lim_{\varepsilon\to 0}\sup_{t\in [0,T]}\left|\Delta X_t^{0,\varepsilon}-\delta X_t^0\right|=0.
\end{equation} 

Now let us turn to the second equation in \eqref{eq:forward:forward:diff:1}, which we prove to be solvable by linearizing the 
forward equation in \eqref{eq:minor:FBG:2}. To do so, we represent the solutions by means on an SDE, very much in the spirit of Lemma
\ref{lem:alpha:weak:solution}. 
Let $(\Omega,{\mathcal F},{\mathbb F},{{\mathbb P}})$
be a filtered probability space (different from $(\Omega^0,{\mathcal F}^0,{\mathbb F}^0,{{\mathbb P}}^0)$)
equipped with an ${\mathbb R}^d$-valued ${\mathbb F}$-Brownian $(B_t)_{0 \le t \le T}$ 
and two ${\mathbb R}^d$-valued random variables $\xi$ and $\xi^{\varepsilon}$
satisfying 
${\mathbb P} \circ \xi^{-1}=\mu_0=\mu$, 
 ${\mathbb P} \circ (\xi^{\varepsilon})^{-1}=\mu_0^\varepsilon=\mu+\varepsilon(\mu'-\mu)$ and 
 $\mathbb E[|\xi^\varepsilon-\xi|]=\mathbb{W}_1(\mu_0^\varepsilon,\mu_0)$. Let also $(X_t^\varepsilon)_{0 \le t \le T},(X_t)_{0 \le t \le T}$ be the solutions to
\begin{align} 
& X_t^{\varepsilon}=\xi^\varepsilon-\int_0^t\nabla_pH\left(X_s^{\varepsilon},\nabla_x \mathcal{U}(s,X_s^{0,\varepsilon},X_s^{\varepsilon},\mu_s^\varepsilon)\right)\ud s+B_t, \quad t \in [0,T], 
\label{eq:Xtvarepsilon:nablapH}
\\
& X_t=\xi-\int_0^t\nabla_pH\left(X_s,\nabla_x \mathcal{U}(s,X_s^{0},X_s,\mu_s)\right)\ud s+B_t, \quad t \in [0,T].
\label{eq:Xt0:nablapH}
\end{align} 
Then ${\mathbb P} \circ (X_t^{\varepsilon})^{-1} = \mu_t^{\varepsilon}$ and ${\mathbb P} \circ X_t^{-1} = \mu_t$ for any $t\in [0,T]$, ${\mathbb P}^0$-almost surely.
 We have
\begin{equation*}
\mathbb E\left[|X_t^{\varepsilon}-X_t|\right]\leq\mathbb E\left[|\xi^\varepsilon-\xi \vert \right]+\mathbb E\bigg[\int_0^t \Bigl|\nabla_pH\left(X_s^{\varepsilon},\nabla_x \mathcal{U}(s,X_s^{0,\varepsilon},X_s^{\varepsilon},\mu_s^\varepsilon)\right)
-\nabla_pH\left(X_s,\nabla_x \mathcal{U}(s,X_s^{0},X_s,\mu_s)\right)\Bigr|\ud s\bigg].
\end{equation*}
Since ${\mathcal U}(t,\cdot,\cdot,\cdot)$ belongs to ${\mathscr C}(\kappa,1,2)$ for every $t \in [0,T]$ and for some 
$\kappa >0$, the gradient $\nabla_x {\mathcal U}$ is Lipschitz in the variable $(x_0,x,\mu)$ (with 
${\mathcal P}({\mathbb T}^d)$ being equipped with 
${\mathbb W}_1$), uniformly in $t \in [0,T]$. 
By \eqref{eq:stability:X0} and Gronwall's lemma, we  deduce that, ${\mathbb P}^0$-almost surely,
\begin{equation}
\label{eq:stability:muuu}
\lim_{\varepsilon\to 0}\sup_{t\in [0,T]}\mathbb{W}_1(\mu_t^{\varepsilon},\mu_t)=0.
\end{equation} 
In fact, we claim that the result also holds in total variation. 
When $\mu_0^\varepsilon=\mu$ (i.e., $\mu'=\mu$), 
this follows from 
\cite[(3.1) and Proof of Theorem 2.4]{Lacker_ECP}
(which rely on Pinsker's inequality and Girsanov theorem) 
together with 
\eqref{eq:stability:X0}
and 
\eqref{eq:stability:muuu}. 
When $\mu'\not =\mu$, one must introduce 
the solution $\hat{\boldsymbol X}^\varepsilon$ 
to the SDE 
\eqref{eq:Xtvarepsilon:nablapH} when initialized from 
$\xi$ at time $0$. 
Then, 
\cite[(3.1) and Proof of Theorem 2.4]{Lacker_ECP}
say that, 
${\mathbb P}^0$-almost surely,
\begin{equation}
\label{eq:stability:muuu:TV}
\lim_{\varepsilon\to 0}\sup_{t\in [0,T]}d_{\rm TV}\Bigl( {\mathbb P} \circ (\hat{X}_t^{\varepsilon})^{-1} ,\mu_t\Bigr)=0.
\end{equation} 
Also, since $d_{\rm TV}(\mu_0^{\varepsilon},\mu_0) = 
d_{\rm TV}((1-\varepsilon) \mu + \varepsilon \mu',\mu) \leq 4 \varepsilon$, we can now choose 
$\xi^{\varepsilon}$ in 
\eqref{eq:Xtvarepsilon:nablapH} such that ${\mathbb P}(\{ \xi = \xi^\varepsilon\})\geq 1- 4 \varepsilon$. Then, it is easy to see that 
${\mathbb P}(\{ \hat{X}_t^{\varepsilon} = X_t^\varepsilon \})\geq 1- 4 \varepsilon$, for all 
$t \in [0,T]$. Since ${\mathbb P} \circ (\xi^{\varepsilon})^{-1}$ is still required to be equal to 
$\mu_0^{\varepsilon}$, this does not change the flow
$({\mathbb P} \circ (X_t^{\varepsilon})^{-1})_{0 \le t \le T}$, 
which is still equal to 
$(\mu_t^{\varepsilon})_{0 \le t \le T}$. We deduce that, ${\mathbb P}^0$-almost surely,
\begin{equation}
\label{eq:stability:muuu:TV}
\lim_{\varepsilon\to 0}\sup_{t\in [0,T]}d_{\rm TV}(\mu_t^{\varepsilon},\mu_t)=0.
\end{equation} 
We note that $\triangle{\boldsymbol\mu}^{\varepsilon}
= (\triangle \mu^{\varepsilon}_t)_{0 \le t \le T}$ satisfies
(in a weak sense similar to 
\eqref{eq:weak:sense:deltamu}) 
\begin{equation}
\label{eq:equation:triangle:mu_tvarepsilon}
\begin{split}
&\partial_t \left(\triangle{\mu}_t^{\varepsilon} \right)- \tfrac12 \Delta_x \left(\triangle{\mu}_t^{\varepsilon}\right) 
\\
&\hspace{15pt} - {\rm div}_x \bigl(  
\nabla_p H(\cdot ,\nabla_x u_t(\cdot))
\triangle{\mu}_t^{\varepsilon}\bigr) -
{\rm div}_x \Bigl(  \mu_t^{\varepsilon}\Bigl[\Gamma_t^{1,\varepsilon}(\cdot)\triangle X_t^{0,\varepsilon} +\Big(\Gamma_t^{2,\varepsilon}(\cdot,\cdot),\triangle{\mu}_t^{\varepsilon}\Big)\Bigr]\Bigr) 
=0,
\end{split}
\end{equation} 
where 
\begin{equation*}
\begin{split}
&\Gamma_t^{1,\varepsilon}(x):=\left(\int_0^1\nabla_{pp}^2 H\Bigl(x,
\theta\nabla_x 
\mathcal{U}(t,X_t^{0,\varepsilon},x,\mu_t^{\varepsilon})
+(1-\theta)\nabla_x  
\mathcal{U}(t,X_t^{0},x,\mu_t)
\Bigr)\ud\theta\right)\\
&\qquad\qquad\cdot\left(\int_{0}^1\nabla_{xx_0}^2 \mathcal{U}\Bigl(t,\theta X_t^{0,\varepsilon}+(1-\theta)X_t^0,x,\theta\mu_t^{\varepsilon}+(1-\theta)\mu_t\Bigr)\ud\theta\right),\\
&\Gamma_t^{2,\varepsilon}(x,y):=\left(\int_0^1\nabla_{pp}^2 H\Bigl(x,
\theta\nabla_x 
\mathcal{U}(t,X_t^{0,\varepsilon},x,\mu_t^{\varepsilon})
+(1-\theta)\nabla_x  
\mathcal{U}(t,X_t^{0},x,\mu_t) \Bigr)
d\theta\right)\\
&\qquad\qquad\quad\,\,\,\cdot\left( \int_{0}^1\nabla_{x}\delta_\mu\mathcal{U}\Bigl(t,\theta X_t^{0,\varepsilon}+(1-\theta)X_t^0,x,\theta\mu_t^{\varepsilon}+(1-\theta)\mu_t,y\Bigr)\ud\theta\right),
\end{split}
\end{equation*}
the dots right above being understood as matricial products and 
the duality bracket between 
$\triangle{\mu}_t^{\varepsilon}$ and 
$\Gamma_t^{2,\varepsilon}(\cdot,\cdot)$ in 
\eqref{eq:equation:triangle:mu_tvarepsilon} being acting on the variable $y$ of the latter. 
Following the proof of Lemma \ref{lem:forward:estimate} and using the fact that ${\mathcal U}(t,\cdot,\cdot,\cdot)\in {\mathscr C}(\kappa, 1 ,2)$, we get, ${\mathbb P}^0$-almost surely, for any $t \in [0,T]$, 
\begin{equation*} 
\begin{split}
\| \triangle \mu_t^{\varepsilon} \|_{-1}
&\leq 
 C \biggl[ \|  \mu'-\mu\|_{-1} + \int_0^t 
\Bigl[ \bigl\| \Gamma_s^{1,\varepsilon}(\cdot) \bigr\|_{L^\infty} 
\bigl\vert \triangle X_s^{0,\varepsilon} \bigr\vert +
\Bigl\| 
\Big(\Gamma_s^{2,\varepsilon}(\cdot,\cdot),\triangle{\mu}_s^{\varepsilon}\Big)
\Bigr\|_{L^\infty} 
\Bigr] \ud s \biggr]
\\
&\leq  
 C \biggl[ \|  \mu'-\mu\|_{-1} + \int_0^t 
 \Bigl( 
|\Delta X_s^{0,\varepsilon}|+\|\Delta \mu_s^{\varepsilon}\|_{-1} \Bigr)  \ud s \biggr].
\end{split}
\end{equation*}
By 
\eqref{eq:stability:X0}
and 
Gronwall's lemma, we obtain, with probability 1 under ${\mathbb P}^0$,
\begin{equation}
\label{eq:stability:muuu:2}
\sup_{t\in [0,T]}\| \triangle \mu_t^{\varepsilon} \|_{-1} \leq  
 C  \biggl[|x_0'-x_0|+ \|\mu'-\mu\|_{-1} \biggr].
\end{equation}
Now, for any sequence of reals $(\varepsilon_k)_{k\geq 1}$ converging to $0$, we let 
  \begin{equation*}
  \begin{split}
  &\triangle^2{\mu}_t^{k',k}:=\triangle{\mu}_t^{\varepsilon_{k'}}-\triangle{\mu}_t^{\varepsilon_k},\quad \triangle^2{X}_t^{0,k',k}:=\triangle{X}_t^{0,\varepsilon_{k'}}-\triangle X_t^{0,\varepsilon_k},\\
  &\triangle{\Gamma_t^{1,k',k}}:=\Gamma_t^{1,\varepsilon_{k'}}-\Gamma_t^{1,\varepsilon_k},\,\,\,\quad
  \triangle{\Gamma_t^{2,k',k}}:=\Gamma_t^{2,\varepsilon_{k'}}-\Gamma_t^{2,\varepsilon_{k}}, \quad t \in [0,T]. 
  \end{split}
  \end{equation*}
Then, we can verify that  $\triangle^2{\boldsymbol\mu}^{k',k}=
(\triangle^2{\mu}_t^{k',k})_{0 \le t \le T}$ satisfies
\begin{align*}
&\partial_t \left(\triangle^2{\mu}_t^{k',k} \right)- \tfrac12 \Delta_x \left(\triangle^2{\mu}_t^{k',k} \right) - {\rm div}_x \left(  
\nabla_p H(\cdot ,\nabla_x u_t(\cdot))
\triangle^2{\mu}_t^{k',k} \right) \\
&\hspace{15pt}  
-{\rm div}_x \Bigl(  \mu_t^{\varepsilon_k} \Bigl[\Gamma_t^{1,\varepsilon_k}(\cdot)\triangle^2 X_t^{0,k',k} +\Big(\Gamma_t^{2,\varepsilon_k}(\cdot,\cdot),\triangle^2{\mu}_t^{k',k}\Big)\Bigr]\Bigr)\\
&\hspace{15pt}  
 -{\rm div}_x \Bigl(  \mu_t^{\varepsilon_k}  \Bigl[\triangle\Gamma_t^{1,k',k}(\cdot)\triangle X_t^{0,\varepsilon_{k'}} +\Big(\triangle\Gamma_t^{2,k',k}(\cdot,\cdot),\triangle{\mu}_t^{\varepsilon_{k'}}\Big)\Bigr]\Bigr) 
 \\
&\hspace{15pt}  
 -{\rm div}_x \Bigl( \bigl(  \mu_t^{\varepsilon_{k'}} -  \mu_t^{\varepsilon_k} 
 \bigr) \Bigl[ \Gamma_t^{1,\varepsilon_{k'}}(\cdot)\triangle X_t^{0,\varepsilon_{k'}} +\Big( \Gamma_t^{2,\varepsilon_{k'}}(\cdot,\cdot),\triangle{\mu}_t^{\varepsilon_{k'}}\Big)\Bigr]\Bigr) 
=0, \quad t \in [0,T].
\nonumber
\end{align*} 
Again, following the proof of Lemma \ref{lem:forward:estimate}, we can show that, ${\mathbb P}^0$-almost surely, 
for all $t \in [0,T]$, 
\begin{equation*} 
\begin{split}
\| \triangle^2 \mu_t^{k',k} \|_{-1}
&\leq 
C \int_0^t 
\biggl[ \bigl\| \Gamma_s^{1,\varepsilon_k}(\cdot) \bigr\|_{L^\infty} 
\bigl\vert \triangle^2 X_t^{0,k',k} \bigr\vert +
\bigl\| \triangle \Gamma_s^{1,k',k}(\cdot) \bigr\|_{L^\infty} 
\bigl\vert \triangle X_t^{0,\varepsilon_{k'}} \bigr\vert 
\\
&\hspace{30pt} +
\Bigl\| 
\Big(\Gamma_s^{2,\varepsilon_k}(\cdot,\cdot),\triangle^2 {\mu}_s^{k',k}\Big)
\Bigr\|_{L^\infty} 
+
\Bigl\| 
\Big(\triangle \Gamma_s^{2,k',k}(\cdot,\cdot),\triangle{\mu}_s^{\varepsilon_{k'}}\Big)
\Bigr\|_{L^\infty} 
\\
&\hspace{30pt} + {d}_{\rm TV} \bigl(\mu_s^{\varepsilon_{k'}},\mu_s^{\varepsilon_k}\bigr) 
\Bigl( \bigl\| \Gamma_s^{1,\varepsilon_{k'}}(\cdot) \bigr\|_{L^\infty} 
\bigl\vert \triangle X_s^{0,\varepsilon_{k'}} \bigr\vert +
\Bigl\|
\Bigl(  \Gamma_s^{2,\varepsilon_{k'}}(\cdot,\cdot), 
  \triangle \mu_s^{\varepsilon_{k'}} 
\Bigr)
\Bigr\|_{L^\infty}
\biggr] \ud s.
\end{split}
\end{equation*}
Using the fact that 
${\mathcal U}(t,\cdot,\cdot)$ belongs to ${\mathscr D}(\kappa,1,2)$ for all 
$t \in [0,T]$, we deduce from 
\eqref{eq:stability:X0}
and
\eqref{eq:stability:muuu}
that, ${\mathbb P}^0$-almost surely, 
\begin{equation*} 
\lim_{n \rightarrow \infty} \sup_{k,k' \geq n} 
\sup_{s \in [0,T]} 
\Bigl( 
\bigl\| \triangle \Gamma_s^{1,k',k}(\cdot) \bigr\|_{L^\infty} 
+
\sup_{x \in {\mathbb T}^d} 
\bigl\|
\triangle \Gamma_s^{2,k',k}(x,\cdot)
\bigr\|_1 
\Bigr) 
= 0. 
\end{equation*} 
Moreover, by \eqref{eq:stability:X0} again, ${\mathbb P}^0$-almost surely, 
\begin{equation*}
\lim_{n \rightarrow \infty} \sup_{k,k' \geq n} 
\sup_{s \in [0,T]}  
\bigl\vert \triangle^2 X_s^{0,k',k} \bigr\vert
= 0. 
\end{equation*} 
By 
\eqref{eq:stability:muuu:TV}
and
\eqref{eq:stability:muuu:2},  we deduce that, ${\mathbb P}^0$-almost surely, there exists a positive sequence $(\delta_n)_{n \geq 1}$, 
converging to $0$, such that, for all $k,k' \geq n$
and for all $t \in [0,T]$, 
\begin{equation*} 
\begin{split}
\| \triangle^2 \mu_t^{k',k} \|_{-1}
&\leq \delta_n
+
C \int_0^t 
\| \triangle^2 \mu_s^{k',k} \|_{-1}
\ud s.
\end{split}
\end{equation*}
And then, by Gronwall's lemma, 
${\mathbb P}^0$-almost surely, 
$(\triangle {\boldsymbol \mu}^{\varepsilon_k})_{k \geq 1}$ is a Cauchy sequence in ${\mathcal C}([0,T],\mathcal{C}^{-1}(\mathbb T^d))$. Therefore, there exists a 
continuous 
 $\mathbb F^0$-progressively measurable process $\delta {\boldsymbol\mu}=(\delta \mu_t)_{0 \le t \le T}$ with values in $\mathcal{C}^{-1}(\mathbb T^d)$ such that, 
 ${\mathbb P}^0$-almost surely, 
\begin{equation*}
 \lim_{k \rightarrow \infty} 
 \sup_{t \in [0,T]} 
 \bigl\| 
 \triangle \mu_t^k - \delta \mu_t \bigr\|_{-1} = 0. 
\end{equation*}
Passing to the limit in 
\eqref{eq:equation:triangle:mu_tvarepsilon}, we deduce that 
$\delta {\boldsymbol \mu}$ satisfies 
(in the same sense as \eqref{eq:weak:sense:deltamu})
\begin{equation}
\label{eq:equation:triangle:mu_tvarepsilon:U}
\begin{split}
&\partial_t \left(\delta \mu_t \right)- \tfrac12 \Delta_x \left( \delta \mu_t\right) 
- {\rm div}_x \bigl(  
\nabla_p H(\cdot ,\nabla_x u_t(\cdot))
\delta \mu_t \bigr)
\\
&\hspace{15pt}  -
{\rm div}_x \Bigl(  \mu_t \Bigl[
\nabla_{pp}^2 H(\cdot ,\nabla_x u_t(\cdot))
\nabla^2_{xx_0} 
{\mathcal U}(t,X_t^0,x,\mu_t)
 \delta X_t^0 + 
 \Big(\nabla_{pp}^2 H(\cdot ,\nabla_x u_t(\cdot)) \nabla_x \delta_\mu {\mathcal U}(t,\cdot,\cdot),\delta {\mu}_t \Big)\Bigr]\Bigr) 
\\
&=0, \quad t \in [0,T]. 
\end{split}
\end{equation} 
Moreover, passing to the limit in 
\eqref{eq:stability:muuu:2}, we obtain 
 \begin{equation}\label{eq:stability:deltamu}
\sup_{t\in [0,T]}\| \delta \mu_t \|_{-1} \leq  
 C \biggl[|x_0'-x_0|+ \|\mu'-\mu\|_{-1} \biggr].
\end{equation}
This proves that $\delta {\boldsymbol\mu}$ solves the second equation in \eqref{eq:forward:forward:diff:1} in the sense of item 2 in Definition \ref{def:linear forward-backward=MFG:solution}. 
\vskip 4pt

\noindent  {\it Second Step.} We now address  the linearized backward equations
\begin{align}
&\ud \delta Y_t^0 = 
 - \delta \bigl[f_t^0(X_t^0,\mu_t)\bigr]  \ud t
 \nonumber
\\
&\hspace{15pt}
+
\Bigl[ -\bigl( Z_t^0 - \nabla_{x_0} w^0(t,X_t^0) \bigr) \cdot \delta \bigl[ \nabla_p H^0(X_t^0,\nabla_{x_0} w^0(t,X_t^0))\bigr]
-
\delta \bigl[L^0(X_t^0,-\nabla_p H^0(X_t^0,\nabla_{x_0}w^0(t,X_t^0))\bigr]
\Bigr] \ud t
\nonumber
\\
&\hspace{15pt} +  \bigl( \nabla_{x_0} H^0(X_t^0,Z_t^0) - \nabla_{x_0} H^0(X_t^0,\nabla_{x_0} w^0(t,X_t^0))\bigr) \cdot \delta X_t^0
  \ud t
  \nonumber
  \\
  &\hspace{15pt} + \bigl( \nabla_p H^0(X_t^0,Z_t^0) - \nabla_p H^0(X_t^0,\nabla_{x_0} w^0(t,X_t^0))
\bigr) \cdot \delta Z_t^0   \ud t
\label{eq:backward:backward:diff:2:Y0}
\\
&\hspace{15pt}+ 
\sigma_0 \delta Z_t^0 \cdot \ud   B_t^0, \quad t \in [0,T], 
\nonumber
\\
 &\delta Y_T^0 = 
 \delta \bigl[ g^0(X_T^0,\mu_T) \bigr],
 \nonumber 
\end{align}
and
\begin{equation}
\label{eq:backward:backward:diff:2:ut}
\begin{split}
&\ud_t \delta u_t(x) = \Bigl( -  \tfrac12 \Delta_x \delta u_t(x) +   \nabla_p H(x,\nabla_x u_t(x)) \cdot \nabla_x \delta u_t(x) 
- \delta \bigl[ f_t(X_t^0,x,\mu_t) \bigr]
\Bigr) \ud t,
\\
&
\hspace{15pt}   
+ \Bigl( \nabla_p H^0(X_t^0,Z_t^0) - \nabla_p H^0(X_t^0,\nabla_{x_0} w^0(t,X_t^0))  \Bigr)\cdot \delta v_t^0(x) 
\ud t 
\\
&
\hspace{15pt}   
+ \delta \Bigl[ \Bigl( \nabla_p H^0(X_t^0,Z_t^0) - \nabla_p H^0(X_t^0,\nabla_{x_0} w^0(t,X_t^0))  \Bigr)
\Bigr]
\cdot  v_t^0(x) 
\ud t 
\\
&\hspace{15pt} 
+ \sigma_0 \delta v_t^0(x) \cdot \ud   B_t^0, \quad (t,x) \in [0,T] \times {\mathbb T}^d, 
\\
&\delta u_T(x) = 
 \delta \bigl[ g(X_T^0,x,\mu_T) \bigr], \quad x \in {\mathbb T}^d. 
\end{split}
\end{equation}
We first address 
\eqref{eq:backward:backward:diff:2:Y0}. 
Recalling that $Z_t^0=\nabla_{x_0}\mathcal{U}^0(t,X_t^0,\mu_t)$ is bounded, it follows from the standard BSDE theory that \eqref{eq:backward:backward:diff:2:Y0} admits a solution $({\delta\boldsymbol Y}^0,{\delta\boldsymbol Z}^0)\in {\mathscr S}^2(\mathbb R)\times{\mathscr H}^2(\mathbb R^d)$. Moreover, 
following 
\cite{PardouxPeng92}, 
we have
\begin{equation}
\begin{split}
&\lim_{\varepsilon\to 0}\mathbb E^0\biggl[\sup_{t\in [0,T]}  |\triangle Y_t^{0,\varepsilon}-\delta Y_t^0|^{2}+
  \int_0^T|\triangle Z_t^{0,\varepsilon}-\delta Z_t^0|^2\ud t 
 \biggr] =0,
\end{split}
\end{equation}
which, together with the first step and the fact that 
${\mathcal U}^0(t,\cdot,\cdot) \in {\mathscr D}^0(\kappa,1)$ for all 
$t \in [0,T]$, implies
\begin{equation*}
\delta Y_t^0 =\nabla_{x_0} {\mathcal U}^0(t,X_t^0,\mu_t)\cdot\delta X_t^0+\left(\delta_{\mu}\mathcal{U}^0(t,X_t^0,\mu_t,\cdot),\delta\mu_t\right), \quad t \in [0,T].
\end{equation*}
Because ${\mathcal U}^0(t,\cdot,\cdot)\in {\mathscr C}^0(\kappa,1)$,
for all $t \in [0,T]$, and $|{\delta\boldsymbol X^0}|$ and $\| \delta \boldsymbol \mu \|_{-1}$ are bounded, we have
\begin{equation*}
\Bigl\|\sup_{0\leq t\leq T}|\delta Y_t^0|\Big\|_{L^\infty(\Omega^0,\mathbb P^0)}<\infty.
\end{equation*}
Then we can follow the proof of Lemma \ref{lem:BMO:major} to deduce that
\begin{equation*}
 \sup_{\tau}\Big \| {\mathbb E}^0\Big[ \int_{[\tau,T]} \vert \delta Z_t^0 \vert^2 \ud t \vert {\mathcal F}_\tau^0\Big] \Big\|_{L^\infty(\Omega^0,{\mathbb P}^0)} < \infty,
\end{equation*}
where the supremum is taken over all stopping times $\tau$. This justifies that $({\delta\boldsymbol Y}^0,{\delta\boldsymbol Z}^0)$ solves 
\eqref{eq:backward:backward:diff:2:Y0}
 in the sense of item 1 in Definition \ref{def:linear forward-backward=MFG:solution}.

The well-posedness of the  equation \eqref{eq:backward:backward:diff:2:ut} is obtained via \cite[Theorem 2.2]{Du:Meng}, with solutions 
in the set of 
progressively-measurable processes 
$(\delta {\boldsymbol u},\delta {\boldsymbol v}^0)$ with values in 
(the separable space) 
${\mathcal H}^1({\mathbb T}^d) \times [{\mathcal H}^1({\mathbb T}^d)]^d$,
such that $\delta {\boldsymbol u}$ has continuous trajectories and 
for which the following bound holds true:
\begin{equation*}
\mathbb E^0\left[\sup_{0\leq t\leq T}\|\delta u_t\|_{ 1, 2}^2+\int_0^T\|\delta v^0_t\|_{1, 2}^2\ud t\right]<\infty.
\end{equation*}
Note that
\begin{equation*}
\begin{split} 
&\ud_t\left(\triangle u^\varepsilon_t\right) = \Bigl[ -  \tfrac12 \Delta_x\left(\triangle u^\varepsilon_t\right) +   \Upsilon_t^\varepsilon(x) \cdot \nabla_x \triangle u^\varepsilon_t
- \nabla_xf_t^\varepsilon(x)\cdot\triangle X_t^{0,\varepsilon}-\left(\delta_\mu f_t^\varepsilon(x,\cdot) ,\triangle \mu^{\varepsilon}_t\right)
\Bigr] \ud t
\\
&
\hspace{15pt}   
+ \Bigl( \nabla_p H^0(X_t^0,Z_t^0) - \nabla_p H^0(X_t^0,\nabla_{x_0} w^0(t,X_t^0))  \Bigr)\cdot \triangle v_t^{0,\varepsilon}\ud t 
\\
&
\hspace{15pt}   
+ \Bigl[(j_t^{1,\varepsilon}-j_t^{3,\varepsilon})\triangle X_t^{0,\varepsilon}+j_t^{2,\varepsilon}\triangle Z_t^{0,\varepsilon}\Bigr]
\cdot  v_t^{0,\varepsilon}(x) 
\ud t +\sigma_0\triangle v_t^{0,\varepsilon} \cdot \ud   B_t^0,
\quad (t,x) \in [0,T] \times {\mathbb T}^d, 
\\
&\triangle u_T^{\varepsilon}(x) = \nabla_x g^\varepsilon(x)\cdot\triangle X_T^{0,\varepsilon} + \left(\delta_\mu g^\varepsilon(x,\cdot) ,\triangle \mu^{\varepsilon}_T\right),
\end{split}
\end{equation*}
where 
\begin{equation*}
\begin{split}
&\Upsilon_t^\varepsilon(x):=\int_0^1\nabla_{p}H\bigl(x,\theta\nabla_x u_t^{\varepsilon}(x)+(1-\theta)\nabla_x u_t(x)\bigr)\ud\theta,
\\
&\nabla_xf_t^\varepsilon(x):=\int_0^1\nabla_{x}f_t\bigl(\theta X_t^{0,\varepsilon}+(1-\theta)X_t^0,x,\theta \mu_t^{\varepsilon}+(1-\theta) \mu_t\bigr)\ud\theta,
\quad \text{and similarly for } \nabla_x g^\varepsilon(x)
\\
&\delta_\mu f_t^\varepsilon(x,y):=\int_0^1\delta_{\mu}f_t\bigl(\theta X_t^{0,\varepsilon}+(1-\theta)X_t^0,x,\theta\mu_t^{\varepsilon}+(1-\theta)\mu_t,y\bigr)\ud\theta, 
\quad \text{and similarly for } \partial_\mu g^\varepsilon(x,y)
\\
&j_t^{1,\varepsilon}:=\int_0^1\nabla_{x_0p}^2H^0\bigl(\theta X_t^{0,\varepsilon}+(1-\theta)X_t^0,\theta Z_t^{0,\varepsilon}+(1-\theta)Z_t^0\bigr)\ud\theta,
\\
&j_t^{2,\varepsilon}:=\int_0^1\nabla_{pp}^2H^0\bigl(\theta X_t^{0,\varepsilon}+(1-\theta)X_t^0,\theta Z_t^{0,\varepsilon}+(1-\theta)Z_t^0\bigr)\ud\theta,
\\
&j_t^{3,\varepsilon}:=\int_0^1[\nabla_{x_0p}^2H^0+\nabla_{pp}^2H^0\nabla^2_{x_0x_0}w^0]\bigl(\theta X_t^{0,\varepsilon}+(1-\theta)X_t^0,\nabla_{x_0}w^0(t,\theta X_t^{0,\varepsilon}+(1-\theta)X_t^0)\bigr)\ud \theta.
\end{split}
\end{equation*}
Using the first step together with the representation formulas
\eqref{eq:proof:prop:4.2:representation:solutions}, there is no difficulty in proving 
that 
\begin{equation*} 
\begin{split}
&\lim_{\varepsilon \rightarrow 0} 
{\mathbb E}^0 \int_0^T \Bigl[ \bigl\| 
\bigl( \Upsilon_t^{\varepsilon} - \Upsilon_t^0 \bigr) 
\cdot 
\nabla_x \triangle u_t^{\varepsilon} \bigr\|_0^2 \Bigr] \ud t 
=0,
\\
&\lim_{\varepsilon \rightarrow 0} 
{\mathbb E}^0 \int_0^T \Bigl[ \bigl\| 
\nabla_x f_t^{\varepsilon} \cdot \triangle X_t^{0,\varepsilon} 
- 
\nabla_x f_t(X_t^0,\cdot,\mu_t) \cdot \delta X_t^{0} 
\bigr\|_0^2 \Bigr] \ud t 
=0,
\\
&\lim_{\varepsilon \rightarrow 0} 
{\mathbb E}^0 \int_0^T \Bigl[ \sup_{x \in {\mathbb T}^d} \bigl\vert  
\bigl( 
\delta_\mu f_t^{\varepsilon} , \triangle \mu_t^{\varepsilon} 
\bigr) 
- 
\bigl( \nabla_\mu f_t(X_t^0,x,\mu_t,\cdot) , \delta \mu_t 
\bigr) 
\bigr|^2 \Bigr] \ud t 
=0,
\\
&\lim_{\varepsilon \rightarrow 0} 
{\mathbb E}^0 \int_0^T 
\Bigl\| 
\Bigl[
\bigl( j_t^{1,\varepsilon} - j_t^{3,\varepsilon} \bigr) 
\triangle X_t^{0,\varepsilon} + 
j_t^{2,\varepsilon} \triangle Z_t^{0,\varepsilon}
\Bigr] \cdot v_t^{0,\varepsilon} 
- 
\Bigl[
\bigl( j_t^{1,0} - j_t^{3,0} \bigr) 
\delta X_t^{0} + 
j_t^{2,0} \delta Z_t^{0}
\Bigr] \cdot v_t^{0} 
\Bigr\|_0^2 \ud t = 0. 
\end{split} 
\end{equation*}
Proceeding in the same manner for the terminal condition, we can invoke \cite[Theorem 2.2]{Du:Meng} again to obtain
\begin{equation*}
\lim_{\varepsilon\to 0}\mathbb E^0\left[\sup_{0\leq t\leq T}\|\triangle u_t^{\varepsilon}-\delta u_t\|_{1,2}^2+\int_0^T\|\triangle v_t^{0,\varepsilon}-\delta v^0_t\|_{1,2}^2\ud t\right]=0.
\end{equation*}
Since ${\mathcal U}(t,\cdot,\cdot,\cdot)\in {\mathscr D}(\kappa,1,3+\alpha)$ for all 
$t \in [0,T]$, we can derive that 
$$\delta u_t(x)=\nabla_{x_0}\mathcal{U}(t,X_t^0,x,\mu_t)\cdot\delta X_t^0+\left(\delta_\mu\mathcal{U}(t,X_t^0,x,\mu_t,\cdot),\delta\mu_t\right),$$ for any $t \in [0,T]$ with probability 1 under ${\mathbb P}^0$,
which implies that $\sup_{0\leq t\leq T}\|\delta u_t\|_{3+\alpha}\in L^\infty(\Omega^0,{\mathbb P}^0)$.
Further applying \eqref{eq:U:priori:contintime}, we can verify that ${\delta\boldsymbol u}$ has continuous trajectories in $ {\mathcal C}^{3+\alpha}({\mathbb T}^d)$. 
Inserting the above identity in 
\eqref{eq:equation:triangle:mu_tvarepsilon:U}, we obtain
\eqref{eq:minor:diff:2}.

Moreover, 
from the representation formula
\eqref{eq:proof:prop:4.2:representation:solutions}
together with 
\eqref{eq:stability:X0}
and
\eqref{eq:stability:muuu:2}, and the fact that ${\mathcal U}(t,\cdot,\cdot,\cdot)$ belongs to 
${\mathscr D}(\kappa,1,3+\alpha)$, we deduce that 
\begin{equation*}
\| \triangle v^{0,\varepsilon}_t \|_{3+\alpha} \leq 
C \biggl[|x_0'-x_0|+ \|\mu'-\mu\|_{-1} \biggr],
\end{equation*}
which proves by passing to the limit that 
$\delta v_t^0$ takes values in a compact subset of 
${\mathcal C}^{3+\epsilon}({\mathbb T}^d)$, for any $\epsilon \in (0,\alpha)$. 
Therefore, $\delta {\boldsymbol v}^0$, regarded as a process taking values in 
${\mathcal C}^{3+\epsilon}({\mathbb T}^d)$, is progressively-measurable (in the 
\textit{Bochner} sense). 

In the end, we obtain $({\delta\boldsymbol u},{\delta\boldsymbol v}^0)$ solves the second equation of the backward system  \eqref{eq:minor:diff:2} in the sense of item 2 in Definition \ref{def:linear forward-backward=MFG:solution}. 
We finally get $(({\delta\boldsymbol X}^0,{\delta\boldsymbol Y}^0,{\delta\boldsymbol Z}^0),$ $({\delta\boldsymbol \mu},{\delta\boldsymbol u},{\delta\boldsymbol v}^0))$
is a solution to the system 
\eqref{eq:major:diff:1}--\eqref{eq:minor:diff:2} in the sense of
Definition 
\ref{def:linear forward-backward=MFG:solution}.
\end{proof}

\section{Short time analysis} 
\label{se:5}
In this section we study the local well-posedness of the
forward-backward system 
\eqref{eq:major:FB:1}--\eqref{eq:minor:FB:2}
(and then of the 
 master equations \eqref{eq:major:1}-\eqref{eq:minor:1}). 
One of the difficulties to do so is that the coefficients of 
\eqref{eq:major:FB:1}--\eqref{eq:minor:FB:2} are not Lipschitz continuous (due to the quadratic growth of the two Lagrangians $L^0$ and $L$ in the control variable $\alpha$). Conceptually, this does not cause any major trouble because
we know in the end (from the analysis carried out in the previous sections, see in particular  \eqref{eq:globalLip}) that 
$\vert Z_t^0 \vert$ and 
$\vert \nabla_x u_t \vert$ are bounded 
(by deterministic constants). 
However, these bounds cannot be directly used in the small time analysis.

To circumvent this difficulty, we introduce two auxiliary functions $\hat{L}^0 : {\mathbb R}^d \times {\mathbb R}^d \rightarrow {\mathbb R}$ and 
$\hat{H} : {\mathbb T}^d \times {\mathbb R}^d \rightarrow {\mathbb R}$ that are globally Lipschitz continuous. 
The function $(x_0,p) \mapsto \hat{L}^0(x_0,p)$ is understood as an approximation of the function 
$(x_0,p) \mapsto  L^0(x_0,-\nabla_p H^0(x_0,p))$ and the function $(x,p) \in {\mathbb T}^d \times 
{\mathbb R}^d \mapsto \hat{H}(x,p)$ is understood as an approximation of $H$. 

In a first time, we thus focus on the systems of two forward-backward equations (set on 
an arbitrary 
filtered probability space 
$(\Omega^0,{\mathcal F}^0,{\mathbb F}^0,{\mathbb P}^0)$
satisfying the usual conditions and 
equipped with a 
Brownian motion $(B_t^0)_{0 \leq t \leq T}$ with values in ${\mathbb R}^d$):
\begin{equation} 
\label{eq:major:FB:1:small}
\begin{split} 
&\ud X_t^0 = - \nabla_p H^0 \bigl( X_t^0 , Z_t^0\bigr)  \ud t + \sigma_0 \ud B_t^0, \quad t\in [0,T],
\\
&\ud Y_t^0 =   - \Bigl( f_t^0(X_t^0,\mu_t) + \hat{L}^0\bigl(X_t^0,  Z_t^0\bigr) 
\Bigr) \ud t + 
\sigma_0 Z_t^0 \cdot \ud B_t^0, \quad t \in [0,T], 
\\
&X_0^0=x_0, \quad Y_T^0 = g^0(X_T^0,\mu_T),
\end{split}
\end{equation} 
and
\begin{equation}
\label{eq:minor:FB:2:small}
\begin{split} 
&\partial_t \mu_t - \tfrac12 \Delta_x \mu_t - {\rm div}_x \Bigl( \nabla_p \hat{H} \bigl(\cdot, \nabla_x u_t\bigr) \mu_t \Bigr) =0, 
\quad  \ (t,x)\in [0,T] \times {\mathbb T}^d, 
\\
&\ud_t u_t(x) =\Bigl(  -  \tfrac12 \Delta_x u_t(x) + \hat{H}\bigl(x,\nabla_x u_t(x) \bigr)  - f_t(X_t^0,x,\mu_t)  \Bigr) \ud t
+  \ud m_t(x), \quad (t,x) \in [0,T] \times {\mathbb T}^d, 
\\
&\mu_0=\mu,\quad u_T(x) = g(X_T^0,x,\mu_T), \quad x \in {\mathbb T}^d.  
\end{split}
\end{equation}
This is only in the end of the section that we return back to the original 
equations \eqref{eq:major:FB:1}
and
\eqref{eq:minor:FB:2}, see Subsection \ref{subse:original:coeff}. 
Notice in particular that, differently from 
\eqref{eq:minor:FB:2}, the martingale part in the backward equation of 
\eqref{eq:minor:FB:2:small}
is not represented as 
a stochastic integral (as there is no need at this stage). 
\vskip 4pt

We thus introduce the following variant of Assumption \hyp{A}: 
\vskip 4pt

 \noindent \textbf{Assumption \^A}. 
 There exist four reals ${\mathfrak L}>0$, 
$\kappa>0$, 
 $\hat{\kappa}>0$ and ${\mathbb s} > 3$ (${\mathbb s} \not \in {\mathbb N}$) such that 
\vskip 4pt

\noindent {\bf (\^A1)}
The function $(f^0_t)_{0 \leq t \leq T}$ 
satisfies
${\mathscr D}^0(\kappa, \mathbb s )$
and 
the function $(f_t)_{0 \le t \le T}$ 
satisfies
${\mathscr D}(\kappa,{\mathbb s},{\mathbb s})$. 
\vskip 4pt

\noindent {\bf (\^A2)} The function $g^0$ 
satisfies
${\mathscr C}^0({\mathfrak L}, \mathbb s )$
and 
${\mathscr D}^0(\hat{\kappa}, \mathbb s )$,
and 
the function $g$ 
satisfies
${\mathscr C}({\mathfrak L},{\mathbb s},{\mathbb s})$
and
${\mathscr D}(\hat{\kappa},{\mathbb s},{\mathbb s})$. 

\vskip 4pt
\noindent {\bf (\^A3)} 
The functions 
$H^0$ and
$\hat{L}^0$  have continuous joint derivatives up to the order 2 in $x_0$ and $\lfloor {\mathbb s} \rfloor +1$ in 
 $p$, and 
 the function 
 $\hat{H}$ has continuous joint derivatives up to the order $\lfloor {\mathbb s} \rfloor +1$ in 
 $(x,p)$.  
 All the existing derivatives (of order greater than 1) of $\hat{L}^0$ and 
 $\hat{H}$ are bounded by $\kappa$
 and the quantity $ \hat{L}^0(x_0,p)$ is bounded by 
 $\kappa( 1+ \vert p \vert)$. 
 The quantity $H^0(x_0,p)$ is bounded by 
 $\kappa(1+\vert p\vert^2)$, the gradient $\nabla_{p} H^0(x_0,p)$ 
  is bounded by $\kappa( 1+ \vert p \vert)$.   
 All the other existing derivatives of $H^0$ are bounded by $\kappa$.
 
\vskip 4pt

\noindent {\bf (\^A4)} 
The function  $\hat{H}$
satisfies the coercivity condition 
$\nabla^2_{pp} \hat{H}(x,p) >0$ (in the sense of positive definite matrices), for all $(x,p)
\in {\mathbb T}^d \times {\mathbb R}^d$. 
For any $r>0$, we let 
\begin{equation*}
\zeta(r) := \inf_{(x,p) \in {\mathbb T}^d \times {\mathbb R}^d : \vert p \vert \leq r}
\inf_{\rho \in {\mathbb R}^d : \vert \rho \vert =1}
\rho \cdot \bigl( \nabla^2_{pp} \hat{H}(x,p) \rho \bigr) >0.
\end{equation*}

\subsection{Local well-posedness}

\begin{theorem}\label{thm:local:FB}
Under Assumption \hyp{\^A}, 
there exists a constant ${\mathfrak C}>0$, only depending on 
${\mathfrak L}$, $\kappa$, $\sigma_0$
and ${\mathbb s}$
such that for $T \leq {\mathfrak C}$, the forward-backward system \eqref{eq:major:FB:1:small}-\eqref{eq:minor:FB:2:small} admits a unique solution $({\boldsymbol X}^0, {\boldsymbol Y}^0, {\boldsymbol Z}^0,{\boldsymbol \mu}, {\boldsymbol u}, {\boldsymbol m})$   in the sense of Definition 
\ref{def:forward-backward=MFG:solution}, i.e., $({\boldsymbol X}^0, 
{\boldsymbol Y}^0, {\boldsymbol Z}^0)$ 
is required to be ${\mathbb F}^0$-progressively measurable, with the trajectories of 
${\boldsymbol X}^0$ and ${\boldsymbol Y}^0$ 
being continuous, such that
$\sup_{0 \le t \le T} \vert X_t^0 \vert \in L^2(\Omega^0,{\mathbb P}^0)$
and
 \begin{equation}\label{eq:local:BMO}
\Big\|\sup_{t\in [0,T]}|Y_t^0|\Big\|_{L^\infty(\Omega^0,\mathbb P^0)}+\sup_{\tau}\bigg\|\mathbb E^0\bigg[\int_\tau^T|Z_t^0|^2\ud t\,|\,\mathcal{F_\tau}^0\bigg]\bigg\|_{L^\infty(\Omega^0,\mathbb P^0)}<\infty,
\end{equation}
the second supremum being taken over ${\mathbb F}^0$-stopping times $\tau$; 
and $({\boldsymbol \mu},{\boldsymbol u},{\boldsymbol m})$ 
is required to 
be an ${\mathbb F}^0$-progressively measurable process 
with values in ${\mathcal P}({\mathbb T}^d) \times {\mathcal C}^\mathbb{s}({\mathbb T}^d) \times 
{\mathcal C}^{\mathbb{s}-2}({\mathbb T}^d)$, with continuous trajectories in 
${\mathcal P}({\mathbb T}^d) \times   {\mathcal C}^{{\mathbb r}}({\mathbb T}^d)
\times   {\mathcal C}^{{\mathbb r}-2}({\mathbb T}^d)$
for any $\mathbb{r} < \mathbb{s}$, 
such that 
\begin{equation}\label{eq:local:reg}
\Big\|\sup_{t\in [0,T]}\Bigl[\|u_t\|_{\mathbb s}\Bigr]\Big\|_{L^\infty(\Omega^0,\mathbb P^0)}
+
\Big\|\sup_{t\in [0,T]}\Bigl[\|m_t\|_{{\mathbb s}-2}\Bigr]\Big\|_{L^\infty(\Omega^0,\mathbb P^0)}
<\infty,
\end{equation}
and, for any $x \in {\mathbb T}^d$, $(m_t(x))_{0 \le t \le T}$ is an ${\mathbb F}^0$-martingale  satisfying 
$m_0(x) \equiv 0$. 

Moreover, for any pair of initial conditions $((x_0^i,\mu^i))_{i=1,2} \in [\mathbb R^d\times \mathcal{P}(\mathbb T^d)]^2$,  
the  corresponding pair of solutions 
 $(({\boldsymbol X}^{0,i}, {\boldsymbol Y}^{0,i}, {\boldsymbol Z}^{0,i},{\boldsymbol \mu}^i, {\boldsymbol u}^i, {\boldsymbol m}^{i}))_{i=1,2}$ to  \eqref{eq:major:FB:1:small}-\eqref{eq:minor:FB:2:small} satisfy, for any $p \in [1,8]$, 
 \begin{equation}\label{eq:local:diff-1}
\begin{split}
&\mathbb E^0\bigg[\sup_{t\in[0,T]}\Big[\|u^1_t-u_t^2\|_{\mathbb s}^{2p} +
 \|m^1_t-m_t^2\|_{{\mathbb s}-2}^{2p} + \mathbb W_1(\mu_t^1,\mu_t^2)^{2p}\bigg]
 \\
&+\mathbb E^0\bigg[\sup_{t\in [0,T]}\Bigl[ |X_t^{0,1}-X_t^{0,2}|^{2p}+
|Y_t^{0,1}-Y_t^{0,2}|^{2p} \Bigr] + \biggl( \int_0^T|Z^{0,1}_t-Z_t^{0,2}|^2\ud t \biggr)^{p} \bigg]\\
&\leq C\Big[|x_0^1-x_0^2|^{2p}+\mathbb W_1 (\mu^1,\mu^2)^{2p}\Big],
\end{split}
\end{equation}
for a constant $C$ only depending on 
$d$, ${\mathfrak L}$, $\kappa$, $\sigma_0$ and ${\mathbb s}$. For the same $C$, we can assume that the left-hand sides of \eqref{eq:local:BMO}
 and
\eqref{eq:local:reg} are bounded by $C$.
\end{theorem}
\begin{proof}
The proof relies on a suitable application of the Banach fixed point theorem. 
Throughout, $T$ is less than 1. 
{ \ }
\vskip 4pt

\noindent  {\it First Step.} 
Assuming that ${\mathbb F}^0$ is generated by ${\mathcal F}_0^0$ and ${\boldsymbol B}^0$, 
we first construct the mapping to which we will eventually apply the fixed point theorem. 
Inputs ${\boldsymbol X}^0$ are taken in the set 
${\mathscr S}^2({\mathbb R}^d)$ of ${\mathbb F}^0$-progressively measurable 
continuous processes with values in ${\mathbb R}^d$ such that 
$\sup_{t \in [0,T]} \vert X_t^0 \vert \in L^2(\Omega^0,{\mathbb P}^0)$. 
For $X^{0}\in {\mathscr S}^2(\mathbb R^d)$, we can follow the proof of \cite[Theorem 4.3.1]{CardaliaguetDelarueLasryLions} to obtain, for a small time horizon, a unique solution to
\begin{equation}
\label{eq:local:FB:tilde1}
\begin{split} 
&\partial_t \tilde\mu_t - \tfrac12 \Delta_x \tilde\mu_t - {\rm div}_x \Bigl( \nabla_p \hat H (\cdot, \nabla_x \tilde u_t) \tilde\mu_t \Bigr) =0, 
\\
&\ud_t \tilde u_t(x) =\Bigl(  -  \tfrac12 \Delta_x \tilde u_t(x) + \hat{H}(x,\nabla_x \tilde u_t(x) ) - f_t(X_t^0,x,\tilde \mu_t)  \Bigr) \ud t
+ \ud m_t(x), \quad (t,x) \in [0,T] \times {\mathbb T}^d, 
\\
&\tilde \mu_0=\mu,\quad \tilde u_T(x) = g(X_T^0,x,\tilde \mu_T), \quad x \in {\mathbb T}^d, 
\end{split}
\end{equation}
where $(\tilde{\boldsymbol \mu},\tilde{\boldsymbol u},\tilde{\boldsymbol m})$ 
is an ${\mathbb F}^0$-progressively measurable processes 
with values in ${\mathcal P}({\mathbb T}^d) \times {\mathcal C}^\mathbb{s}({\mathbb T}^d) \times 
{\mathcal C}^{\mathbb{s}-2}({\mathbb T}^d)$, with continuous trajectories in 
${\mathcal P}({\mathbb T}^d) \times   {\mathcal C}^{{\mathbb r}}({\mathbb T}^d)
\times   {\mathcal C}^{{\mathbb r}-2}({\mathbb T}^d)$
for any $\mathbb{r} < \mathbb{s}$,
such that 
\begin{equation}
\label{eq:u:bdd}
\sup_{t\in [0,T]}\left\{\|\tilde u_t\|_{\mathbb s}+ \|m_t \|_{{\mathbb s}-2}\right\}\in L^\infty(\Omega^0,\mathcal{F}^0,\mathbb P^0),
\end{equation}
and $(m_t(x))_{0 \le t \le T}$ is an ${\mathbb F}^0$-martingale starting from $0$ for any $x \in {\mathbb T}^d$. 

Actually, the framework is not (exactly) the same as in 
\cite[Theorem 4.3.1]{CardaliaguetDelarueLasryLions}. First, ${\boldsymbol X}^0$ plays the role of the common noise, even though it is not acting 
additively (as ${\boldsymbol B}^0$ does in \cite{CardaliaguetDelarueLasryLions}). 
Second (and this is the main difference), 
\cite[Theorem 4.3.1]{CardaliaguetDelarueLasryLions}
is proven under the Lasry-Lions monotonicity condition, which we have not 
assumed 
here. 
However, the proof 
of \cite[Theorem 4.3.1]{CardaliaguetDelarueLasryLions} (which relies on 
the method of continuation) remains 
conceptually  the same in our context. Essentially, the parameter 
$\varpi$ in  
\cite[Subsection 4.3.5]{CardaliaguetDelarueLasryLions}
is here taken as $0$ and the role of $\epsilon$ 
(which is assumed to be small in
\cite{CardaliaguetDelarueLasryLions})
is now played by $T$.
Although the correspondence between 
$\epsilon$ and $T$ is not immediate, the similarity between the two parameters 
can be well noticed 
in the proof of 
\cite[Lemma 4.3.9]{CardaliaguetDelarueLasryLions}: the Cauchy-Schwarz inequality that permits to pass from 
\cite[(4.30)]{CardaliaguetDelarueLasryLions}
to 
\cite[(4.31)]{CardaliaguetDelarueLasryLions}
makes an additional $\sqrt{T}$ appear in front of 
the parameter 
$\vartheta$ in 
\cite[(4.30)]{CardaliaguetDelarueLasryLions}. 
This additional $\sqrt{T}$ then plays the role of 
$\epsilon$ 
in the second step of 
 \cite[Subsection 4.3.5]{CardaliaguetDelarueLasryLions}. 
 With this analogy in mind, one notices from 
 $(c)$ in the 
 second step of 
 \cite[Subsection 4.3.5]{CardaliaguetDelarueLasryLions}
 that the threshold ${\mathfrak C}>0$ below which 
 $T$ must be taken only depends on 
 $g$ through ${\mathfrak L}$ (and not $\hat{\kappa}$). 
 
 To be complete, one mentions
 another subtlety: in \cite[Theorem 4.3.1]{CardaliaguetDelarueLasryLions}, 
continuities of 
$\tilde{\boldsymbol u}$
and 
$\tilde{\boldsymbol m}$ are stated in 
${\mathcal C}^{\lfloor {\mathbb s} \rfloor}({\mathbb T}^d)$
and
${\mathcal C}^{\lfloor {\mathbb s} \rfloor-2}({\mathbb T}^d)$
 respectively, but there is no difficulty in replacing $\lfloor \mathbb{s} \rfloor
$
by any $\mathbb{r} \in (\lfloor {\mathbb s}\rfloor,\mathbb{s})$
(for instance by combining the above bound with an interpolation argument). 
In this regard, we emphasize that \textit{Bochner} measurability
(with values in ${\mathcal C}^{{\mathbb s} }({\mathbb T}^d)$
and
${\mathcal C}^{  {\mathbb s}  -2}({\mathbb T}^d)$) can be obtained as 
explained in 
Remark
\ref{rem:def:2:9}, by noticing that, for $t <T$, $u_t$ and $m_t$ takes 
values in 
${\mathcal C}^{  {\mathbb s}' }({\mathbb T}^d)$
and
${\mathcal C}^{{\mathbb s}'-2}({\mathbb T}^d)$
for some ${\mathbb s}'
> {\mathbb s}$. This follows from the smoothing properties of the heat kernel: 
we refer to   \cite[(4.16)]{CardaliaguetDelarueLasryLions}
for the way it can be used here. 
Following Proposition \ref{prop:minor:higher} (and \cite[Proposition 4.3.8]{CardaliaguetDelarueLasryLions}), we notice that 
$\sup_{t\in [0,T]} \{\|\tilde u_t\|_{\mathbb s}+ \|m_t \|_{{\mathbb s}-2}\}$
can be bounded by constants only depending on the parameters
${\mathfrak L}$, $\kappa$, $\sigma_0$, ${\mathbb s}$ and $T$ 
(but not 
$\hat{\kappa}$) in \hyp{\^A}.

Next we define $(\tilde {\boldsymbol X}^0,\tilde {\boldsymbol Y}^0,\tilde {\boldsymbol Z}^0)$ as the strong solution to the following FBSDE system
(see \cite{Delarue02}):
\begin{equation} 
\label{eq:major:FB:tilde2}
\begin{split} 
&\ud \tilde X_t^0 = - \nabla_p H^0 ( \tilde X_t^0 , \tilde Z_t^0)  \ud t + \sigma_0 \ud B_t^0, \quad t\in [0,T],
\\
&\ud \tilde Y_t^0 =  - \Bigl( f_t^0(\tilde X_t^0,\tilde \mu_t) +\hat L^0(\tilde X_t^0, \tilde Z_t^0) 
\Bigr) \ud t + 
\sigma_0 \tilde Z_t^0 \cdot \ud B_t^0, \quad t \in [0,T], 
\\
&\tilde X_0^0=x_0, \quad \tilde Y_T^0 = g^0(\tilde X_T^0,\tilde \mu_T).
\end{split}
\end{equation} 
Setting $\mathfrak{T}({\boldsymbol X}^0):=\tilde {\boldsymbol X}^0$, we define a map $\mathfrak{T}: {\mathscr S}^2(\mathbb R^d)\rightarrow {\mathscr S}^2(\mathbb R^d)$.
 In the rest of the proof, we show that $\mathfrak{T}$ is a contraction mapping from $ {\mathscr S}^2(\mathbb R^d)$ into itself (for $T$ small enough). 
\vskip 4pt

\noindent  {\it Second Step.} Denote the solutions to \eqref{eq:local:FB:tilde1} with two different inputs ${\boldsymbol X}^{0,1}$ and ${\boldsymbol X}^{0,2}$ with $X^{0,1}_0= X^{0,2}_0=x_0$ by $(\tilde {\boldsymbol \mu}^1,\tilde {\boldsymbol u}^1,\tilde {\boldsymbol m}^{1})$ and $( \tilde {\boldsymbol \mu}^2, \tilde {\boldsymbol u}^2,\tilde  {\boldsymbol m}^{2})$, and the corresponding solutions to \eqref{eq:major:FB:tilde2} by $(\tilde {\boldsymbol X}^{0,1}, \tilde {\boldsymbol Y}^{0,1}, \tilde {\boldsymbol Z}^{0,1})$ and $(\tilde {\boldsymbol X}^{0,2}, \tilde {\boldsymbol Y}^{0,2}, \tilde {\boldsymbol Z}^{0,2})$.
By the standard duality method of mean field games (without using the monotonicity condition of $f_t$ and $g$, see for instance 
\cite[(4.29)]{CardaliaguetDelarueLasryLions}) and thanks 
to \eqref{eq:u:bdd} (which, together with \textbf{(\^{A}4)}, says that we are working on a subdomain of the space on which the Hessian matrix 
$\nabla^2_{pp} H$ is bounded from below by a strictly positive constant), we have
\begin{equation}\label{eq;local:diff0}
\mathbb E^0\biggl[\int_0^T\Bigl(|\nabla_x(\tilde u_t^1-\tilde u_t^2) |^2, \tilde \mu_t^1+\tilde \mu_t^2\Bigr)\ud t\biggr]\leq C\mathbb E^0\biggl[\sup_{t\in [0,T]}|X_t^{0,1}-X_t^{0,1}|^2+\sup_{t\in [0,T]}\mathbb W_1(\tilde\mu_t^1,\tilde\mu_t^2)^2\biggr].
\end{equation}
As the maps $x\mapsto \nabla_pH(x,\nabla_x \tilde u^1(t,x))$ and $p\mapsto \nabla_p H(x,p)$ are  Lipschitz continuous (with Lipschitz constants only depending on the parameters
${\mathfrak L}$, $\kappa$, $\sigma_0$, ${\mathbb s}$ and $T$, 
but not 
$\hat{\kappa}$), we use
\eqref{eq:u:bdd}
and
 \eqref{eq;local:diff0} to derive
 (the reader may compare with the derivation of \cite[(3.7)]{CardaliaguetDelarueLasryLions})
\begin{equation*}
\begin{split}
\mathbb E^0 \biggl[ \sup_{t\in [0,T]}\Bigl[\mathbb W_1(\tilde \mu_t^1,\tilde \mu_t^2)^2\Bigr]\biggr]&\leq C\mathbb E^0\bigg[\bigg(\int_0^T\Bigl(|\nabla_x(\tilde u_t^1-\tilde u_t^2) |^2, \tilde \mu_t^1+\tilde \mu_t^2\Bigr)^{\frac{1}{2}}\ud t\bigg)^2\bigg]\\
&\leq CT\mathbb E^0\biggl[\int_0^T\Bigl(|\nabla_x(\tilde u_t^1-\tilde u_t^2) |^2, \tilde \mu_t^1+\tilde \mu_t^2\Bigr)\ud t\biggr]\\
&\leq CT\mathbb E^0\biggl[\sup_{t\in [0,T]}|X_t^{0,1}-X_t^{0,1}|^2+\sup_{t\in [0,T]}\mathbb W_1(\tilde\mu_t^1,\tilde\mu_t^2)^2\biggr],
\end{split}
\end{equation*}
where $C$ 
only depends on
${\mathfrak L}$, $\kappa$, $\sigma_0$, ${\mathbb s}$. 
And then, for $CT \leq 1/2$, we obtain
\begin{equation}\label{eq:local:diff1}
\mathbb E^0 \biggl[ \sup_{t\in [0,T]}\Bigl[\mathbb W_1(\tilde \mu_t^1,\tilde \mu_t^2)^2\Bigr]\biggr]\leq 2CT\mathbb E^0\biggl[\sup_{t\in [0,T]}|X_t^{0,1}-X_t^{0,1}|^2\biggr].
\end{equation}
By standard short time FBSDE estimates (see \cite{Delarue02}), we can derive, still under the condition $CT \leq 1/2$
(but for a possibly new value of $C$ that only depends on
${\mathfrak L}$, $\kappa$, $\sigma_0$ and ${\mathbb s}$),
\begin{equation}\label{eq:local:diff2}
\mathbb E^0 \biggl[ \sup_{t\in [0,T]}\Bigl[|\tilde X_t^{0,1}-\tilde X_t^{0,2}|^2\biggr]\leq C\sup_{t\in [0,T]}\mathbb E^0\Bigl[\mathbb W_1(\tilde \mu_t^1,\tilde\mu_t^2)^2\Bigr]\leq CT\mathbb E^0\biggl[\sup_{t\in [0,T]}|X_t^{0,1}-X_t^{0,1}|^2\biggr].
\end{equation}
Since $C T \leq 1/2$, we get
\begin{equation*} 
\mathbb E^0\biggl[\sup_{t\in [0,T]}|\tilde X_t^{0,1}-\tilde X_t^{0,2}|^2\biggr]\leq \frac{1}{2} \mathbb E^0\biggl[\sup_{t\in [0,T]}|X_t^{0,1}-X_t^{0,2}|^2\biggr],
\end{equation*}
which implies $\mathfrak{T}$ is a contraction mapping from $ {\mathscr S}^2(\mathbb R^d)$ to $ {\mathscr S}^2(\mathbb R^d)$. Therefore, there exists a unique fixed point ${\boldsymbol X}^0\in  {\mathscr S}^2(\mathbb R^d)$ with $X_0^0=x_0$. For this choice of ${\boldsymbol X}^0$, the FBSPDE system 
\eqref{eq:local:FB:tilde1} admits a unique solution, which we denote $({\boldsymbol \mu},{\boldsymbol u},{\boldsymbol m})$.
Since ${\boldsymbol X}^0$ is the fixed point of $\mathfrak{T}$, the FBSDE system  \eqref{eq:major:FB:1:small} admits a unique strong solution $({\boldsymbol X}^0,{\boldsymbol Y}^0,{\boldsymbol Z}^0)$. In this way, we get existence and uniqueness of a solution to the pair 
\eqref{eq:major:FB:1:small}--\eqref{eq:minor:FB:2:small}.
The bound
\eqref{eq:local:reg}
follows from \eqref{eq:u:bdd} 
and 
\eqref{eq:local:BMO}
follows from the argument used in the proof 
of Lemma \ref{lem:BMO:major}.
\vskip 4pt

\noindent {\it Third Step.} We now prove the stability property, removing the assumption that 
${\mathbb F}^0$ is generated by 
${\mathcal F}_0^0$ and ${\boldsymbol B}^0$ 
(which will imply in particular uniqueness of the solution constructed above, but on the larger set-up 
$(\Omega^0,{\mathcal F}^0,{\mathbb F}^0,{\mathbb P}^0)$). 
For any initial conditions $(x_0^1,\mu^1),(x_0^2,\mu^2)\in \mathbb R^d\times \mathcal{P}(\mathbb T^d)$, let $({\boldsymbol X}^{0,1}, {\boldsymbol Y}^{0,1}, {\boldsymbol Z}^{0,1}, {\boldsymbol \mu}^1, {\boldsymbol u}^1, {\boldsymbol m}^{1})$ and $({\boldsymbol X}^{0,2}, {\boldsymbol Y}^{0,2}, {\boldsymbol Z}^{0,2}, {\boldsymbol \mu}^2, 
{\boldsymbol u}^2, {\boldsymbol m}^{2})$ be the solutions of the forward-backward system \eqref{eq:major:FB:1:small}-\eqref{eq:minor:FB:2:small}  with (respective) initial data $(x_0^1,\mu^1)$ and $(x_0^2,\mu^2)$. 

Note that 
\begin{equation}\label{eq:local:newdiff3}
\begin{split}
&\ud_t\Big(u_t^1(x)-u_t^2(x)\Big)\\
&=\bigg[-\tfrac{1}{2}\Delta_x\bigl(u_t^1(x)-u_t^2(x)\bigr)-\bigl(f_t(X_t^{0,1},x,\mu_t^{1})-f_t(X_t^{0,2},x,\mu_t^{2}) \bigr)
\\
&\hspace{15pt} + \biggl\{ 
\int_{0}^1\nabla_p \hat{H}\bigl(x,r\nabla_x u^1_t(x)+(1-r)\nabla_xu^2(x)\bigr)\ud r \biggr\} \cdot \nabla_x \bigl( u^1_t(x)- u^2_t(x) \bigr)\bigg]\ud t
  + \ud \bigl(m_t^1(x) - m_t^2(x)\bigr). 
\end{split}
\end{equation}
and
\begin{equation*}
u_T^1(x)-u_T^2(x)=g(X_T^{0,1},x,\mu_T^1)-g(X_T^{0,2},x,\mu_T^2).
\end{equation*}
Following \cite[(4.14)--(4.16)]{CardaliaguetDelarueLasryLions} and recalling $T \leq 1$, we have, for 
all $t \in [0,T]$, 
\begin{align*}
&\|u_{t}^1-u_t^2\|_{\mathbb s}
\nonumber
\\
&\leq C {\mathbb E}^0 \Big[  \sup_{r\in [0,T]} \bigl\|f_r(X_r^{0,1},\cdot,\mu_r^1)-f_r(X_r^{0,2},\cdot,\mu_r^2)\bigr\|_{\mathbb s-1}+ \bigl\|g(X_T^{0,1},\cdot,\mu_T^1)-g(X_T^{0,2},\cdot,\mu_T^2)\bigr\|_{\mathbb s}
\, \vert \, {\mathcal F}_t
\Big]\nonumber
\\
&\leq C {\mathbb E}^0 \Big[ \sup_{r \in [0,T]} |X_r^{0,1}-X_r^{0,2}|+ \sup_{r\in [0,T]}\mathbb W_1(\mu_r^1,\mu_r^2) \, \vert \, {\mathcal F}_t \Big],
\end{align*}
for a constant $C$ that only depends on 
${\mathfrak L}$, $\kappa$, $\sigma_0$ and ${\mathbb s}$. 

Taking 
power $2p$, for $p \in [1,8]$, 
and then expectation (under ${\mathbb P}^0$) on both sides, we obtain,
by means of Doob's inequality,
\begin{equation}
{\mathbb E}^0 \Big[ \sup_{t \in [0,T]} 
\|u_{t}^1-u_t^2\|_{\mathbb s}^{2p}
\Bigr]
\leq C {\mathbb E} \Big[ \sup_{r \in [0,T]} |X_r^{0,1}-X_r^{0,2}|^{2p}+ \sup_{r\in [0,T]}\mathbb W_1(\mu_r^1,\mu_r^2)^{2p}  \Big],
\label{eq:local:diff4}
\end{equation}
where we 	allowed the constant $C$ to vary from line to line 
as long as it only depends on 
${\mathfrak L}$, $\kappa$, $\sigma_0$, ${\mathbb s}$. 

Moreover, 
similar to the derivation of \eqref{eq:local:diff1} (but with a power $2p$ in the computations, which does not change the argument), we can obtain 
(from \eqref{eq:local:diff4})
\begin{equation*}
\begin{split}
\mathbb E^0\Big[\sup_{t\in [0,T]}\mathbb W_1(\mu_t^1,\mu_t^2)^{2p}\Big]&\leq C\bigg\{ \mathbb W_1(\mu^1,\mu^2)^{2p}+  
\mathbb E^0\biggl[ \biggl( \int_0^T\bigl(|\nabla_x( u_t^1- u_t^2) |^2,  \mu_t^1+ \mu_t^2\bigr)\ud t \bigg)^{2p} \biggr]
\biggr\}
\\
&\leq 
C\bigg\{\mathbb W_1^2(\mu^1,\mu^2)+T^{2p}
 \mathbb E^0\Big[
\sup_{t \in [0,T]}
\bigl\|\nabla_x(u^1_t-u^2_t)\bigr\|_{L^\infty}^{2p} \Bigr]
\biggr\}
\\
&\leq 
C\bigg\{\mathbb W_1^2(\mu^1,\mu^2)+T^{2p}
\mathbb E^0\Big[\sup_{t\in [0,T]}|X_t^{0,1}-X_t^{0,2}|^{2p}+\sup_{t\in [0,T]}\mathbb W_1(\mu_t^1,\mu_t^2)^{2p}\Big] \bigg\}.
\end{split}
\end{equation*}
For $CT \leq 1/2$ (and since $T \leq 1$), we obtain
\begin{equation}\label{eq:local:diff5}
\begin{split}
&\mathbb E^0\Big[\sup_{t\in [0,T]}\mathbb W_1(\mu_t^1,\mu_t^2)^{2p}\Big]
\leq  2C \bigg\{\mathbb W_1(\mu^1,\mu^2)^{2p}+ T  \mathbb E^0\Big[\sup_{t\in [0,T]}|X_t^{0,1}-X_t^{0,2}|^{2p}\Big]  \bigg\}.
\end{split}
\end{equation}
By   standard short time FBSDE estimates
(see again \cite{Delarue02}),  we obtain (still under a condition of the form $C T \leq 1/2$)
\begin{equation*}
\begin{split} 
&\mathbb E^0 \biggl[ \sup_{t\in [0,T]}\big[| X_t^{0,1}- X_t^{0,2}|^{2p} +  | Y_t^{0,1}- Y_t^{0,2}|^{2p}\big]+
\biggl( \int_0^T| Z_t^{0,1}- Z_t^{0,2}|^2 \ud t \biggr)^p \biggr]
\\
&\leq C\Bigl( |x_0^1-x_0^2|^{2p}+\mathbb E^0\Bigl[ \sup_{t\in [0,T]} \mathbb W_1(\mu_t^1,\mu_t^2)^{2p}\Bigr]\Bigr)
\\
&\leq \max( C, 2 C^2)  \Bigl( 
|x_0^1-x_0^2|^{2p}+
\mathbb W_1(\mu^1,\mu^2)^{2p}+ T  \mathbb E^0\Big[\sup_{t\in [0,T]}|X_t^{0,1}-X_t^{0,2}|^{2p}\Big] 
\Bigr). 
\end{split} 
\end{equation*}
If $\max( C, 2 C^2) T \leq 1/2$ (and $T \leq 1$), then
\begin{equation*} 
\begin{split}
&\mathbb E^0 \biggl[ \sup_{t\in [0,T]}\big[| X_t^{0,1}- X_t^{0,2}|^{2p} + | Y_t^{0,1}- Y_t^{0,2}|^{2p} \big]+
\biggl( \int_0^T| Z_t^{0,1}- Z_t^{0,2}|^2\ud t\biggr)^p \biggr]
\\
&\leq 2 \max( C, 2 C^2) \Big(|x_0^1-x_0^2 |^{2p}+ \mathbb W_1(\mu^1,\mu^2)^{2p}\Big).
\end{split}
\end{equation*}
Plugging the above inequality into \eqref{eq:local:diff4} and \eqref{eq:local:diff5}, we get (for a new value of $C$) 
\begin{equation*}
\mathbb E^0\Big[\sup_{t\in [0,T]}\|u_t^1-u_t^2\|_{\mathbb s}^{2p} + \sup_{t\in [0,T]}\mathbb W_1(\mu_t^1,\mu_t^2)^{2p}\Big]\leq C\Big[|x_0^1-x_0^2|^{2p} + \mathbb W_1(\mu^1,\mu^2)^{2p}\Big].
\end{equation*}
Recalling \eqref{eq:local:newdiff3} and inserting the above bound, we deduce that
\begin{equation*}
\mathbb E^0\Big[\sup_{t\in [0,T]}\|m_t^1-m_t^2\|_{{\mathbb s}-2}^{2p} \Bigr] 
\leq C\Big[|x_0^1-x_0^2|^{2p} + \mathbb W_1(\mu^1,\mu^2)^{2p}\Big].\\
\end{equation*}
In summary, we obtain \eqref{eq:local:diff-1}.
\end{proof}

\subsection{Linearized system} 


We now address the following general linearized forward-backward system, set on $[0,T]$:
\begin{align} 
&\ud \delta X_t^0 =  \Big[-p_t \delta X_t^0-q_t\delta Z_t^0+a_t\Big] \ud t, 
\nonumber
\\
&\ud \delta Y_t^0 =\Big[ l_t\cdot \delta X_t^0+o_t\cdot \delta Z_t^0
 -  \nabla_{x_0} f_t^0(X_t^0,\mu_t)\cdot \delta X_t^0 - \Big(\delta_\mu f_t^0(X_t^0,\mu_t),\delta \mu_t\Big)  +b_t\Big]
 \ud t + 
\sigma_0 \delta Z_t^0 \cdot \ud B_t^0,
\nonumber
\\
& \delta X_0^0= \triangle x_0, \quad \delta Y_T^0 = \nabla_{x_0}g^0(X_T^0,\mu_T)\cdot \delta X_T^0+\Big(\delta_\mu g^0(X_T^0,\mu_T),\delta \mu_T\Big)+c_T;
\label{eq:linear:local}
\\
&  \partial_t \delta \mu_t -\tfrac{1}{2}\Delta_{x}\delta \mu_t - {\rm div}_x\big(\Upsilon_t(x)\delta \mu_t +\mu_t \Gamma_t(x)\nabla_x\delta u_t(x)\big)-{\rm div}_x\big(d_t \big)=0,\quad 
{\rm on} \ \mathbb T^d,
\nonumber
\\
& \ud_t\delta u_t(x)=\Big[-\tfrac{1}{2}\Delta_{x}\delta u_t(x)+\Upsilon_t(x)\cdot\nabla_x \delta u_t(x)-\nabla_{x_0}f_t(X_t^0,x,\mu_t)\cdot\delta X_t^0-\Big(\delta_\mu f_t(X_t^0,x,\mu_t),\delta \mu_t\Big)+j_t(x)\Big]\ud t
\nonumber
\\
& \hspace{15pt} + \ud_t \delta m_t(x),\quad x \in \mathbb T^d,
\nonumber
\\
&\delta \mu_0=\triangle \mu,\quad \delta u_T(x)=\nabla_{x_0}g(X_T^0,x,\mu_T)\cdot \delta X_T^0+\Big(\delta_\mu g(X_T^0,x,\mu_T),\delta \mu_T\Big)+k_T(x)\quad x \in  \mathbb T^d,
\nonumber
\end{align} 
which should be interpreted as a
generalized version of the system satisfied by the derivative of the flow 
induced by the solution to 
\eqref{eq:major:FB:1:small}-\eqref{eq:minor:FB:2:small}. 
On top of Assumption \hyp{\^A}, 
the  coefficients 
driving 
\eqref{eq:linear:local}
are required to satisfy the set of conditions below:
\vskip 5pt

 \noindent \textbf{Assumption C.} 
For a real ${\mathbb r} \in ( \lfloor {\mathbb s} \rfloor, {\mathbb s})$ and for another real $\mathfrak{K}>1$, 
\begin{enumerate}[(i)]
\item The initial conditions $\triangle \mu$ and $\triangle x_0$ are deterministic, 
taking values in ${\mathcal C}^{- \mathbb{r} + 1}$ and 
${\mathbb R}^d$ respectively; 

\item The random variable $k_T$ is $\mathcal{F}_T^0$-measurable with values in 
$\mathcal C^{ {\mathbb s}}(\mathbb T^d)$, 
and satisfies
$\essup \|k_T\|_{{\mathbb s}}<+\infty$; 

\item The process $(j_t)_{0 \le t \le T}$ is $\mathbb F^0$-adapted with continuous paths 
from $[0,T]$ to $\mathcal C^{{\mathbb s}'}(\mathbb T^d)$, for any ${\mathbb s}' \in (0,{\mathbb s}-1)$,  
and satisfies the bound $\essup \sup_{t \in [0,T]} 
\|j_t\|_{{\mathbb s}-1} < \infty$; the process $(d_t)_{0 \le t \le T}$ is $\mathbb F^0$-adapted with 
continuous paths from $[0,T]$ to $[{\mathcal C}^{-\mathbb{r}'+2}(\mathbb T^d)]^d$, for any 
$\mathbb{r}'>\mathbb{r}$
and satisfies
$\essup \sup_{t\in [0,T]} \|d_t\|_{-\mathbb{r} +2}<+\infty$;

\item The process $(\Upsilon_t)_{0 \le t \le T}$ is 
$\mathbb F^0$-adapted with
continuous 
paths from $[0,T]$ to $\mathcal C^{{\mathbb s}'}(\mathbb T^d\,,\,\mathbb R^d)
\simeq [\mathcal C^{{\mathbb s}'}(\mathbb T^d)]^d$, for any ${\mathbb s}' \in (0,{\mathbb s}-1)$, and satisfies
$\essup \sup_{t\in [0,T]}\|\Upsilon_t\|_{{\mathbb s}-1}\leq \mathfrak{K}$;

\item The process $(\Gamma_t)_{0 \le t \le T}$ is $\mathbb F^0$-adapted with
continuous paths from  $[0,T]$  to ${\mathcal C}^1({\mathbb T}^d,{\mathbb R}^{d \times d}) \simeq [{\mathcal C}^1(\mathbb T^d)]^{d\times d}$; it satisfies
$\essup \sup_{t\in [0,T]}\|\Gamma_t\|_{1}\leq \mathfrak{K}$
and, ${\mathbb P}^0$-almost surely, 
\begin{equation*}
\forall (t,x) \in [0,T] \times {\mathbb T}^d, 
\quad  \mathfrak{K}^{-1}I_d\leq \Gamma_t(x)\leq \mathfrak{K} I_d;
\end{equation*}

\item The processes $(a_t)_{0 \le t \le T}$
and  $(b_t)_{0 \le t \le T}$ belong to 
${\mathscr H}^2(\Omega^0,{\mathcal F}^0,{\mathbb F}^0,{\mathbb P}^0;\mathbb R^d)$
and
${\mathscr H}^2(\Omega^0,{\mathcal F}^0,{\mathbb F}^0,{\mathbb P}^0;\mathbb R)$ respectively; 
the random variable $c_T$ is 
in $L^2 (\Omega^0,{\mathcal F}^0_T,{\mathbb P}^0;{\mathbb R})$;
the processes $(p_t)_{0 \le t \le T}$, $(q_t)_{0 \le t \le T}$, $(l_t)_{0 \le t \le T}$
and $(o_t)_{0 \le t \le T}$ are $\mathbb F^0$-adapted, 
with respective values in ${\mathbb R}^{d}$, ${\mathbb R}^{d \times d}$, 
${\mathbb R}^d$ and ${\mathbb R}^d$, and satisfies
$\essup\sup_{t\in [0,T]}[|p_t|+|q_t|+|l_t|+|o_t|]\leq \mathfrak{K}$.
\end{enumerate}

\noindent Within the above framework
and under the additional condition  that 
 $(d_t)_{0 <t \leq T}$ takes values in ${\mathcal C}^{-\mathbb{r}'+2}({\mathbb T}^d)$ 
 for
some $\mathbb{r}'<\mathbb{r}$, 
 we let, for any $p \geq 1$, 
\begin{equation}\label{eq: M}
\begin{split}
M_p&:= \vert \triangle x_0 \vert^{2p} 
+ 
\|
\triangle \mu 
\|_{-\mathbb{r} +1}^{2p}
\\
&\hspace{15pt} +
{\mathbb E}^0 \biggl[ \sup_{t \in [0,T]} \Bigl(     \|d_t\|^{2p}_{-\mathbb{r}+2}
+
 \|j_t\|^{2p}_{{\mathbb s}-1}
 \Bigr) + 
\biggl( 
\int_0^T \bigl( \vert a_t \vert + \vert b_t \vert \bigr) \ud t  
\biggr)^{2p} 
+
\vert c_T \vert^{2p} 
+
\| k_T \|_{{\mathbb s}}^{2p}
\biggr].
\end{split}
\end{equation}

\begin{theorem}\label{thm:local:linearFB}
Under Assumptions \hyp{\^A} and \hyp{C}, 
there exists a real ${\mathfrak C}>0$, only depending on 
$d$, $\kappa$, ${\mathfrak K}$, ${\mathfrak L}$, $\sigma_0$ and $(\mathbb{r},{\mathbb s})$, 
such that for $T \leq {\mathfrak C}$, 
the forward-backward system \eqref{eq:linear:local} admits a unique solution $(\delta {\boldsymbol \mu}, \delta {\boldsymbol u}, \delta {\boldsymbol m}, \delta {\boldsymbol X}^0, \delta {\boldsymbol Y}^0, \delta {\boldsymbol Z}^0)$, adapted with respect to the filtration $\mathbb F^0$ and with values in 
${\mathcal C}^{-\mathbb{r}+1}({\mathbb T}^d) \times {\mathcal C}^{{\mathbb s}}({\mathbb T}^d) 
\times {\mathcal C}^{{\mathbb s} -2}({\mathbb T}^d)\times\mathbb R^d\times \mathbb R\times\mathbb R^d$, 
 satisfying
\begin{enumerate}[(a)]
\item 
 $(\delta {\boldsymbol \mu},\delta {\boldsymbol u},\delta {\boldsymbol m})$
 has continuous trajectories in $\mathcal{C}^{-\mathbb{u}+1}(\mathbb T^d)\times \mathcal{C}^{\mathbb{u}}(\mathbb T^d)\times \mathcal{C}^{\mathbb{u}-2}(\mathbb T^d)$ for any $\mathbb{u} \in (\mathbb{r}, {\mathbb s})$, 
 and
\begin{equation}\label{eq:local:linear:diff-1}
\textrm{\rm essup}_{\omega \in \Omega^0}\sup_{t\in [0,T]}\Big(\|\delta u_t\|_{\mathbb s}+
\|\delta m_t \|_{{\mathbb s}-2}+\|\delta \mu_t\|_{-\mathbb{r}+1}\Big)<+\infty;
\end{equation}
\item for any $x \in {\mathbb T}^d$, $(\delta m_t(x))_{0 \le t \le T}$ is an ${\mathbb F}^0$-martingale; 
\item  $(\delta {\boldsymbol X^0},\delta{\boldsymbol Y^0},\delta {\boldsymbol Z^0})$ belongs to $ {\mathscr S}^2(\mathbb R^d)\times {\mathscr S}^2(\mathbb R)\times  {\mathscr H}^2(\mathbb R^d)$.
\end{enumerate}
Moreover, 
if
$(d_t)_{t\in (0,T]}$ takes values in $[{\mathcal C}^{-\mathbb{r}'+2}({\mathbb T}^d)]^d$ for
some $\mathbb{r}'<\mathbb{r}$, 
there exists a constant $C$, only depending on  $d$, $\kappa$, ${\mathfrak K}$, ${\mathfrak L}$, $\sigma_0$ and $(\mathbb{r},{\mathbb s})$, such that, for any 
$p \in [1,8]$,  
\begin{equation}\label{eq:local:linear:diff0}
\begin{split}
&\mathbb E^0\biggl[\sup_{t\in [0,T]} \Bigl( 
\|\delta \mu_t\|_{-\mathbb{r}+1}^{2p}
+
\|\delta u_t\|_{\mathbb s}^{2p}
+
\|\delta m_t\|_{{\mathbb s}-2}^{2p}
 \Bigr)  + \sup_{t\in [0,T]} \Bigl( |\delta X_t^0|^{2p}+ |\delta Y_t^0|^{2p} \Bigr) +
 \biggl( \int_0^T|\delta Z_t^0|^2\ud t
 \biggr)^p 
 \biggr]
 \\
 & \leq C_p M_p.
\end{split}
\end{equation}
\end{theorem}

\begin{remark}
\label{rem:thm:local:linearFB}
\begin{enumerate}
\item 
In the statement of 
Theorem 
\ref{thm:local:linearFB}, measurability and progressive-measurability are understood in the same \textit{Bochner} sense as 
in the third item of Remark \ref{rem:def:2:9}. In comparison, the 
novelty here is that 
the forward component of the FBSPDE is also impacted by issues of measurability. 
The difficulty is the same as in Definition \ref{def:forward-backward=MFG:solution}:
the dual space ${\mathcal C}^{-\mathbb{r}+1}({\mathbb T}^d)$ 
of 
${\mathcal C}^{\mathbb{r}-1}({\mathbb T}^d)$ 
is not separable. The remedy is the same: 
we can easily prove
(see for instance \cite[Lemma 3.3.1]{CardaliaguetDelarueLasryLions}) that, for $t \in (0,T]$, 
$\delta \mu_t$ is in fact in the subspace 
 ${\mathcal C}^{-\mathbb{r}'+1}({\mathbb T}^d)$, 
 for some $\mathbb{r}' < \mathbb{r}$, which is separable when 
 equipped  
 with $\| \cdot \|_{-\mathbb{r}+1}$. 
 
 Actually, some care is also needed in the formulation of 
 Assumption \hyp{C}. In item (ii) therein, measurability is also understood in \textit{Bochner} sense. 
 In item (iii), each $j_t$ is regarded as a random variable with values in ${\mathcal C}^{{\mathbb s}'}({\mathbb T}^d)$, 
 for any ${\mathbb s}' \in (0,{\mathbb s}-1)$. Since $j_t$ takes almost surely values in 
 ${\mathcal C}^{{\mathbb s}-1}({\mathbb T}^d)$, this says that 
 $j_t$ is in fact \textit{Bochner} measurable with values in ${\mathcal C}^{{\mathbb s}'}({\mathbb T}^d)$, for any 
 ${\mathbb s}' < {\mathbb s}-1$. The same argument applies to $(d_t)_{0 \le t \le T}$ and $(\Psi_t)_{0 \le t \le T}$ in items (iii) and (iv) in Assumption \hyp{C}.  
 \item In 
\eqref{eq:local:linear:diff-1}, $\sup_{t \in [0,T]} \| \delta \mu_t \|_{-\mathbb{r}+1}$
can be shown to be 
measurable by the same argument as above. 
For any $\varepsilon>0$, we have in fact a (deterministic) bound for $\sup_{\varepsilon \leq t \leq T} 
\| \delta \mu_t \|_{-\mathbb{r}'+1}$ for some $\mathbb{r}' <\mathbb{r}$, from which 
we deduce that $\delta {\boldsymbol \mu}$ has  continuous trajectories from $[\varepsilon,T]$ to ${\mathcal C}^{-\mathbb{r}+1}({\mathbb T}^d)$. 
This proves that 
$\sup_{t \in [\varepsilon,T]} \| \delta \mu_t \|_{-\mathbb{r}+1}$
is measurable. Letting $\varepsilon$ tend to $0$, we deduce that 
$\sup_{t \in (0,T]} \| \delta \mu_t \|_{-\mathbb{r}+1}$
and
then
$\sup_{t \in [0,T]} \| \delta \mu_t \|_{-\mathbb{r}+1}$
are measurable. 
That said, it must be clear that there is actually no need of measurability of the supremum to 
write down the result. Instead, we can just say  that there exists $C >0$ such that ${\mathbb P}^0(\{ \forall t \in [0,T], \ \| \delta \mu_t \|_{-\mathbb{r}+1} \leq C\})=1$, which is licit since ${\mathbb P}^0$ is complete. 
In particular, this is exactly how
the condition
$\essup \sup_{t\in [0,T]} \|d_t\|_{-\mathbb{r} +2}<+\infty$
in item (iii) of \hyp{C} 
should be understood: 
formally, we cannot prove that 
$\sup_{t\in [0,T]} \|d_t\|_{-\mathbb{r} +2}$
is measurable, but the condition still makes sense. (Notice that, in comparison, 
$\sup_{t\in [0,T]} \|j_t\|_{{\mathbb s}-1}$
-- also in item (iii) of Assumption \hyp{C}---
is a random variable, see Remark 
\ref{rem:def:2:9}.
 Here, the same argument does not hold 
for processes taking values in 
the dual space 
${\mathcal C}^{-\mathbb{r}+2}({\mathbb T}^d)$ 
because the norm 
$\| \cdot \|_{-\mathbb{r}+2}$ is 
not lower-semicontinuous 
on ${\mathcal C}^{-\mathbb{r}'+2}({\mathbb T}^d)$
for $\mathbb{r}' >\mathbb{r}$. 
)
\item The fact that 
$\sup_{t\in [0,T]} \|d_t\|_{-\mathbb{r} +2}$
may not be a random variable explains why we need another condition
to define $M$ in 
\eqref{eq: M} (in order to guarantee that the \textit{sup} is in fact measurable). Indeed, if $d_t$ takes values in $[{\mathcal C}^{-\mathbb{r}'+2}({\mathbb T}^d)]^d$
for some $\mathbb{r}'<\mathbb{r}$, then (for any $\omega^0 \in \Omega^0)$
\begin{equation*} 
\| d_t \|_{-\mathbb{r}+2} 
= \sup_{ g \in {\mathcal C}^{\infty}({\mathbb T}^d) : \| g \|_{\mathbb{r}-2} \leq 1} 
(d_t,g).
\end{equation*} 
And then, using the fact that the process ${\boldsymbol d}$ has continuous trajectories with values in 
$[{\mathcal C}^{-\mathbb{r}''+2}({\mathbb T}^d)]^d$ for any $\mathbb{r}''>\mathbb{r}$ (see item (ii) in 
Assumption \hyp{C}), we deduce have that, for any 
$g \in {\mathcal C}^{\infty}({\mathbb T}^d)$, 
\begin{equation*} 
(d_t,g) = \lim_{s \rightarrow t, s \in {\mathbb Q}} (d_s,g) \leq \liminf_{s \rightarrow t, s \in {\mathbb Q}} 
\| d_s \|_{-\mathbb{r}+2}, 
\end{equation*} 
and then 
\begin{equation*} 
\sup_{t \in [0,T]} \| d_t \|_{-\mathbb{r} +2} 
= 
\sup_{t \in [0,T] \cap {\mathbb Q}} \| d_t \|_{-\mathbb{r} +2}. 
\end{equation*}
It then remains to prove that 
each $d_t$ is a \textit{Bochner} random variable with values in 
$[{\mathcal C}^{-\mathbb{r}+2}({\mathbb T}^d)]^d$, but this follows from the fact 
that it is a
random variable with values in 
 $[{\mathcal C}^{-\mathbb{r}''+2}({\mathbb T}^d)]^d$
 for any $\mathbb{r}''>\mathbb{r}$ and that it takes values in 
 ${\mathcal C}^{-\mathbb{r}'+2}({\mathbb T}^d)$
 for some 
 $\mathbb{r}' < \mathbb{r}$. 
\end{enumerate}
\end{remark}

\begin{proof}
Small time unique solvability 
of  
\eqref{eq:linear:local}
is proven by means of a suitable contraction argument, 
which is explained in the first two steps below. 
In the third step, we establish 
\eqref{eq:local:linear:diff0}. Proceeding as in the proof of Theorem
\ref{thm:local:FB}, one can assume that 
${\mathbb F}^0$ is generated by 
${\mathcal F}_0^0$ and ${\boldsymbol B}^0$. 
\vskip 4pt

\noindent   {\it First Step.} 
The first of the two steps underpinning the contraction 
argument relies on the following idea. 
For a given input $\delta {\boldsymbol X}^{0}:=
(\delta X^0_t)_{0 \le t \le T} \in  {\mathscr S}^2(\mathbb R^d)$ satisfying $\delta X_0^0=\triangle x_0$, 
we want to solve
the FBSPDE, on $[t,T]$, 
\begin{align}
&\partial_t \delta\tilde\mu_t - \tfrac12 \Delta_x\delta  \tilde\mu_t - {\rm div}_x \Bigl( \Upsilon_t(x)\delta \tilde\mu_t  +\mu_t \Gamma_t(x)\nabla_x\delta \tilde u_t(x) \Bigr)-{\rm div}_x\big(d_t \big) =0, \quad {\rm on} \ {\mathbb T}^d, 
\nonumber
\\
&\ud_t \delta \tilde u_t(x) =\Bigl[  -  \tfrac12 \Delta_x \delta \tilde u_t(x) +\Upsilon_t(x)\cdot \nabla_x\delta \tilde u_t(x)-\nabla_{x_0} f_t(X_t^0,x,\tilde \mu_t)\cdot \delta X_t^0-\Big(\delta_\mu f_t(X_t^0,x,\mu_t),\delta \tilde \mu_t\Big)+j_t(x)  \Bigr] \ud t
\nonumber
\\
&+ \ud_t \delta \tilde m_t(x), \quad x \in  {\mathbb T}^d, 
\label{eq:local:linear:diff1}
\\
&\delta\tilde \mu_0=\triangle\mu,\quad \delta \tilde u_T(x)=\nabla_{x_0}g(X_T^0,x,\mu_T)\cdot\delta X_T^0+\Big(\delta_\mu g(X_T^0,x,\mu_T),\delta \tilde \mu_T\Big)+k_T(x), \quad  x \in \mathbb T^d, \nonumber
\end{align}
with paths in the space $\mathcal C^0([0,T]\,,\,\mathcal{C}^{-\mathbb{u}+1}(\mathbb T^d)\times\mathcal{C}^{\mathbb{u}}(\mathbb T^d)\times\mathcal{C}^{\mathbb{u}-2}(\mathbb T^d))$ and with
\begin{equation}
\label{eq:local:linear:diff1:l2}
{\mathbb E}^0 
\biggl[
\sup_{t\in [0,T]}\left\{\|\delta \tilde\mu_t\|_{-\mathbb{u}+1}^2 +\|\delta\tilde  u_t\|_\mathbb{u}^2 +\|\delta \tilde m_t \|^2_{\mathbb{u}-2}\right\}
\biggr] <+\infty.
\end{equation}
Basically, 
we would like to 
apply \cite[Theorem 4.4.2]{CardaliaguetDelarueLasryLions} in order to guarantee that there is a unique solution 
$(\delta\tilde {\boldsymbol \mu},\delta{\tilde {\boldsymbol u}},\delta{\tilde {\boldsymbol m}})$, adapted with respect to the filtration $\mathbb F^0$, to 
\eqref{eq:local:linear:diff1}--\eqref{eq:local:linear:diff1:l2}. 
Notice that if we can indeed apply  \cite[Theorem 4.4.2]{CardaliaguetDelarueLasryLions}, 
then 
\textit{Bochner} measurability 
with values in any $\mathcal{C}^{-\mathbb{u}+1}(\mathbb T^d)\times\mathcal{C}^{\mathbb{u}}(\mathbb T^d)\times\mathcal{C}^{\mathbb{u}-2}(\mathbb T^d)$ is obvious because 
\cite[Theorem 4.4.2]{CardaliaguetDelarueLasryLions} then implies that the solution is also 
in $\mathcal{C}^{-\mathbb{u}''+1}(\mathbb T^d)\times\mathcal{C}^{\mathbb{u}'}(\mathbb T^d)\times\mathcal{C}^{\mathbb{u}'-2}(\mathbb T^d)$ for $\mathbb{r} < \mathbb{u}'' < \mathbb{u} < \mathbb{u}' < \mathbb{s}$. 

Actually, the difficulty is that
the current setting 
does not exactly fit the assumption 
of 
\cite[Theorem 4.4.2]{CardaliaguetDelarueLasryLions}. 
One first difficulty is that 
\cite[Theorem 4.4.2]{CardaliaguetDelarueLasryLions}
is stated in the monotone framework. 
This is the same issue as the one mentioned in the proof of 
Theorem 
\ref{thm:local:FB}
and the remedy is very similar. Here is an overview of it. Because of the lack of monotonicity, 
there is an additional term in 
\cite[(4.44)]{CardaliaguetDelarueLasryLions}
that writes
(with the notations
from \cite{CardaliaguetDelarueLasryLions})
${\mathbb E}[ \sup_{t \in [0,T]} 
\| \tilde \rho_t - \tilde \rho_t' \|_{-(n+\alpha')}^2]$
(in our system of notations, $\tilde \rho_t$ 
corresponds to 
$\delta \tilde \mu_t$, $\tilde \rho_t'$
to another (forward) solution 
$\delta \tilde \mu_t'$ and 
$n+\alpha'$ to ${\mathbb u}-1$, for
${\mathbb u} \in (\mathbb{r},\mathbb{s})$). 
If $T$ is small enough, then 
Cauchy-Schwarz inequality yields an additional factor 
$\sqrt{T}$ in 
\cite[(4.47)]{CardaliaguetDelarueLasryLions}
so that
\cite[(4.49)]{CardaliaguetDelarueLasryLions}
remains true, with an additional factor 
$T$ in front of 
$\sup_{t \in [0,T]} \| \tilde z_t - \tilde z_t' \|_{n+1+\alpha}^2$
(again with the notations 
from 
\cite{CardaliaguetDelarueLasryLions}, 
$\tilde z_t$ being here understood as 
$\delta \tilde u_t$). This makes it possible to apply the same contraction 
argument as in the proof of 
\cite[Proposition 4.4.7]{CardaliaguetDelarueLasryLions}, the role 
of $\varepsilon$ (which is assumed to be small therein) being now played by $\sqrt{T}$.

Another difference is that the
process $\delta {\boldsymbol X}^0$ (which appears here both in the driver and in the boundary condition of the backward equation) is not bounded (in $\omega^0$). 
This creates another difficulty since the perturbations 
$\tilde f_t^0$ and $\tilde g_T^0$ in 
\cite[(4.37)]{CardaliaguetDelarueLasryLions} are assumed to be bounded. 
In order to construct a solution, we thus apply 
\cite[Theorem 4.4.2]{CardaliaguetDelarueLasryLions}
to an
approximation of 
\eqref{eq:local:linear:diff1}, in which 
$(\delta X_t^0)_{0 \le t \le T}$ is replaced by 
$(\varphi_R(\delta X^0_t))_{0 \leq t \leq T}$ for a
real $R>0$ and a function 
$\varphi_R : {\mathbb R}^d \rightarrow {\mathbb R}^d$ that is bounded by $2R$ and equal to the identity mapping (from 
${\mathbb R}^d$ into itself) on the $d$-dimensional ball of center $0$ and radius $R$. 
This makes it possible to combine 
\cite[Theorem 4.4.2]{CardaliaguetDelarueLasryLions}
and 
\cite[Proposition 4.4.5]{CardaliaguetDelarueLasryLions}: 
The first statement shows that there is a solution for each $R$ (and for a fixed initial condition), 
denoted 
$(\delta \tilde {\boldsymbol \mu}^R,\delta \tilde {\boldsymbol u}^R, \delta \tilde {\boldsymbol m}^{R})$, 
and the second one shows (by comparing the solutions for two different values of $R$) 
that the family of (hence constructed) solutions  
is Cauchy as $R$ tends to $\infty$. Here, the Cauchy property is understood in the same 
$L^2$ sense 
as in 
\eqref{eq:local:linear:diff1:l2}. The limit of the Cauchy sequence is a solution 
of 
\eqref{eq:local:linear:diff1}. Uniqueness is slightly more challenging. We take one solution 
to 
\eqref{eq:local:linear:diff1}
satisfying 
\eqref{eq:local:linear:diff1:l2}. We then introduce the stopping 
time 
$\tau_R :=
\inf \{ t \in  [0,T] : \vert \delta X_t^0 \vert + \| \delta \tilde u_t \|_{{\mathbb s}} \geq R\}$.  
Writing the system 
solved by 
$(\delta \tilde {\boldsymbol \mu}_{t \wedge \tau_R} ,\delta \tilde {\boldsymbol u}_{t \wedge \tau_R}, \delta \tilde {\boldsymbol m}_{t \wedge \tau_R})_{0 \le t \le T}$
as a forward-backward system on $[0,T]$, 
with 
\begin{equation*} 
\delta \tilde u_{\tau_R}(x) = 
\Bigl[ \nabla_{x_0}g(X_T^0,x,\mu_T)\cdot\delta X_T^0+\Big(\delta_\mu g(X_T^0,x,\mu_T),\delta \tilde \mu_T\Big)+k_T(x)
\Bigr] {\mathbf 1}_{\{ \tau_R = T\}} + 
\delta \tilde u_{\tau_R}(x)
{\mathbf 1}_{\{ \tau_R < T\} }, \quad x \in {\mathbb T}^d, 
\end{equation*} 
as terminal condition, 
we can compare 
$(\delta \tilde {\boldsymbol \mu}_{t \wedge \tau_R} ,\delta \tilde {\boldsymbol u}_{t \wedge \tau_R}, \delta \tilde {\boldsymbol v}^{0}_{t \wedge \tau_R})_{0 \le t \le T}$
 with 
$(\delta \tilde {\boldsymbol \mu}^R,\delta \tilde {\boldsymbol u}^R, \delta \tilde {\boldsymbol v}^{0,R})$ by means of
\cite[Proposition 4.4.5]{CardaliaguetDelarueLasryLions}. Thanks to the above decomposition, we observe that, on the event 
$\{ \tau_R=T\}$, 
 the terminal condition of the forward-backward system obtained by making the difference of the two solutions 
just writes  
$(\delta_\mu g(X_T^0,x,\mu_T),\delta \tilde \mu_T- \delta \tilde \mu_T^R)$
and
 is thus purely linear 
 (equivalently, the intercept is zero). 
 On the event $\{ \tau_R<T\}$, 
 the terminal condition writes $[\delta \tilde u_{\tau_R}(x)
 - \delta \tilde u_{T}^R(x)]
{\mathbf 1}_{\{ \tau_R < T\}}$, which tends, by 
a standard uniform integrability argument, to $0$ in $L^2$ as $R$ tends to $0$. 
By
\cite[Proposition 4.4.5]{CardaliaguetDelarueLasryLions}, 
this proves uniqueness. Also, this makes it possible to compare any two solutions satisfying 
  \eqref{eq:local:linear:diff1}
and \eqref{eq:local:linear:diff1:l2}
by means of
\cite[Proposition 4.4.5]{CardaliaguetDelarueLasryLions} again. 
All these claims hold true for $T \leq {\mathfrak C}$, 
where 
${\mathfrak C}>0$ only depends 
on
$d$, $\kappa$, ${\mathfrak K}$, ${\mathfrak L}$, $\sigma_0$ and $(\mathbb{r},{\mathbb s})$. In particular, 
${\mathfrak C}$ is independent of the initial condition. 

Next we define $(\delta\tilde {\boldsymbol X}^0,\delta\tilde {\boldsymbol Y}^0,\delta\tilde 
{\boldsymbol Z}^0)$ as the strong solution to the following FBSDE system (on $[0,T]$)
\begin{equation} 
\label{eq:local:linear:diff2}
\begin{split} 
&\ud \delta \tilde X_t^0 =  \Big[-p_t \delta \tilde X_t^0-q_t\delta \tilde Z_t^0+a_t\Big] \ud t,
\\
&\ud \delta \tilde Y_t^0 =\Big[l_t\cdot \delta \tilde X_t^0+o_t\delta \tilde Z_t^0
\\
& \hspace{45pt}  -  \nabla_{x_0} f_t^0(X_t^0,\mu_t) \cdot \delta \tilde X_t^0 - \Big(\delta_\mu f_t^0(X_t^0,\mu_t),\delta \tilde \mu_t\Big)  +b_t\Big]
 \ud t + 
\sigma_0 \delta \tilde Z_t^0 \cdot \ud B_t^0, \quad t \in [0,T], 
\\
& \delta \tilde X_0^0= \triangle x_0, \quad \delta \tilde Y_T^0 = \nabla_{x_0}g^0(X_T^0,\mu_T)\cdot \delta \tilde X_T^0+\Big(\delta_\mu g^0(X_T^0,\mu_T),\delta \tilde \mu_T\Big)+c_T.
\end{split}
\end{equation} 
Here, solutions are required to satisfy
\begin{equation}
\label{eq:local:linear:diff2:l2}  
{\mathbb E}^0 \biggl[ \sup_{0 \le t \le T} \bigl( \vert \delta \tilde X_t^0 \vert^2 + 
\vert \delta \tilde Y_t^0 \vert^2 \bigr) + \int_0^T \vert \delta \tilde Z_t^0 \vert^2 \ud t\biggr] < \infty.  
\end{equation} 
By \cite{Delarue02}, existence and uniqueness of a solution hold true for $T \leq {\mathfrak C}$, for
a possibly new value of the constant ${\mathfrak C}$. 
Next, we implicitly require $T$ to be less than ${\mathfrak C}$. 

Setting $\delta {\mathfrak T}(\delta X^0):=\delta \tilde X^0$, this defines a map $\delta {\mathfrak  T}: {\mathscr S}^2(\mathbb R^d)\rightarrow  {\mathscr S}^2(\mathbb R^d)$. 
\vskip 4pt

\noindent {\it Second Step.}
In this step, we show that, for $T \leq {\mathfrak C}$ (again for a possibly new value of 
${\mathfrak C}$, depending on $d$, $\kappa$, ${\mathfrak K}$, 
${\mathfrak L}$ and $({\mathbb r},{\mathbb s})$), $\delta {\mathfrak T}$ is a contraction mapping from $ {\mathscr S}^2(\mathbb R^d)$ to $ {\mathscr S}^2(\mathbb R^d)$.

 For
 two  inputs $\delta {\boldsymbol X}^{0,1}$ and $\delta {\boldsymbol X}^{0,2}$ with $\delta X^{0,1}_0= \delta X^{0,2}_0=\triangle x_0$
 as initial conditions, we denote 
  by $(\delta\tilde {\boldsymbol \mu}^1,\delta\tilde {\boldsymbol u}^1,\delta\tilde {\boldsymbol m}^{1})$ and $( \delta \tilde {\boldsymbol \mu}^2,\delta \tilde {\boldsymbol u}^2,\delta\tilde  {\boldsymbol m}^{2})$
  the corresponding solutions to 
\eqref{eq:local:linear:diff1}, and next  by $(\delta \tilde {\boldsymbol X}^{0,1},\delta \tilde {\boldsymbol Y}^{0,1},\delta \tilde {\boldsymbol Z}^{0,1})$ and $(\delta \tilde {\boldsymbol X}^{0,2},\delta \tilde {\boldsymbol Y}^{0,2},\delta \tilde {\boldsymbol Z}^{0,2})$ the corresponding solutions to \eqref{eq:local:linear:diff2}.
Using 
\cite[Proposition 4.4.5]{CardaliaguetDelarueLasryLions}
(which is licit thanks to the analysis performed in 
the first step, 
and 
with the following notations therein:
the primed solution 
is null,  
$\vartheta = 1$, 
$\tilde m_t = \mu_t$, 
$\tilde z_t (x) = \delta \tilde u_t^1(x)$, 
$\tilde z_t' (x) = \delta \tilde u_t^2(x)$,
$\tilde f_t^0(x)=  j_t(x) - \nabla_{x_0} f_t(X_t^{0},x,\tilde \mu_t) \cdot \delta X_t^{0,1}$,
$\tilde f_t^{0,\prime}(x) = j_t(x) - \nabla_{x_0} f_t(X_t^{0},x,\tilde \mu_t) \cdot \delta X_t^{0,2}$,
$\tilde V_t(x)=\tilde V_t'(x) = \Upsilon_t(x)$, 
$\tilde b_t^0(x)=\tilde b_t^{0,\prime}(x) = d_t(x)$), we get 
(thanks to 
Remark \ref{rem:thm:local:linearFB}, 
the term inside the expectation symbol in the left-hand side below is 
a random variable):
\begin{equation*}
\begin{split}
\mathbb E^0\bigg[\sup_{t\in [0,T]}\|\delta \tilde \mu^1_t-\delta \tilde \mu^2_t\|_{-\mathbb{r}+1}^2 \bigg]\leq C\mathbb E^0\bigg[\sup_{t\in [0,T]}|\delta X_t^{0,1}-\delta X_t^{0,2}|^2\bigg],
\end{split}
\end{equation*} 
where $C$ only depends  on 
$d$, $\kappa$, ${\mathfrak K}$, ${\mathfrak L}$,  and $(\mathbb{r},{\mathbb s})$.
By  standard short time stability estimates for FBSDEs, see 
\cite[Theorem 1.3]{Delarue02}, 
 we can derive (for $T \leq {\mathfrak C}$, allowing the value of 
 ${\mathfrak C}$ to vary from line to line as long as it just depends on the same parameters as before)
\begin{equation}\label{eq:local:linear:diff6}
\begin{split}
&\mathbb E^0 \biggl[ \int_0^T |\delta\tilde Z_t^{0,1}-\delta\tilde Z_t^{0,2}|^2
\ud t
\biggr]\\
&\leq C\mathbb E^0\biggl[\bigg(\int_0^T\Big(\delta_\mu f^0_t(X_t^0,\mu_t),\delta \tilde \mu_t^1-\delta \tilde \mu_t^2\Big)\ud t\bigg)^2+\Big(\delta_\mu g^0(X_T^0,\mu_T),\delta \tilde \mu_T^1-\delta \tilde \mu_T^2\Big)^2\biggr]\\
&\leq C\mathbb E^0\bigg[\sup_{t\in [0,T]}\|\delta \tilde \mu^1_t-\delta \tilde \mu^2_t\|_{-\mathbb{r}+1}^2\bigg]\leq C \mathbb E^0\bigg[\sup_{t\in [0,T]}|X_t^{0,1}-X_t^{0,2}|^2\bigg].
\end{split}
\end{equation}
Inserting the above estimate in the forward equation of 
\eqref{eq:local:linear:diff2}, we deduce that, for $T \leq {\mathfrak C}$,
\begin{equation*} 
\mathbb E^0\biggl[\sup_{t\in [0,T]}|\delta\tilde X_t^{0,1}-\delta\tilde X_t^{0,2}|^2\biggr]\leq \frac{1}{2} \mathbb E^0\biggl[\sup_{t\in [0,T]}|\delta X_t^{0,1}-\delta X_t^{0,2}|^2\biggr],
\end{equation*}
which implies $\delta {\mathfrak T}$ is a contraction mapping from $ {\mathscr S}^2(\mathbb R^d)$ to $ {\mathscr S}^2(\mathbb R^d)$. Therefore, there exists a unique fixed point 
$\delta {\boldsymbol X}^0\in  {\mathscr S}^2(\mathbb R^d)$ with $\delta X_0^0=\triangle x_0$.

Then, the coupled system comprising 
both
the (finite-dimensional) FBSDE and
the
 FBSPDE   in \eqref{eq:linear:local} admits a unique solution, which is adapted with respect to the filtration $\mathbb F^0$. The solution to the FBSPDE system is denoted 
$(\delta {\boldsymbol \mu},\delta {\boldsymbol u},\delta {\boldsymbol v}^0)$. It satisfies 
$(a)$ and $(b)$ in the statement. The 
solution to the finite-dimensional FBSDE system in  \eqref{eq:linear:local} admits a unique strong solution $(\delta {\boldsymbol X}^0, \delta {\boldsymbol Y}^0, \delta {\boldsymbol Z}^0)$. 
It satisfies $(c)$ in the statement.
\vskip 4pt

\noindent {\it Third Step.} We now estimate the 6-tuple $(\delta {\boldsymbol \mu},\delta 
{\boldsymbol u},\delta {\boldsymbol m},\delta {\boldsymbol X}^0,\delta {\boldsymbol Y}^0,\delta {\boldsymbol Z}^0)$, solution to the forward-backward system \eqref{eq:linear:local}.
The proof is similar to that one of 
\cite[Corollary 4.4.6]{CardaliaguetDelarueLasryLions}. 
The point is to compare the solution 
$(\delta {\boldsymbol \mu},\delta 
{\boldsymbol u},\delta {\boldsymbol m})$ of the FBSPDE system in 
\eqref{eq:linear:local}
with the trivial system driven by null coefficients and thus having 
a null solution. Comparison is obtained by means of  
\cite[Proposition 4.4.5]{CardaliaguetDelarueLasryLions}.
One of the key point in the proof is 
to fix one instant $t \in [0,T]$ 
and then
to apply 
\cite[Proposition 4.4.5]{CardaliaguetDelarueLasryLions}
to the restriction of the 
triple
$(\delta {\boldsymbol \mu},\delta 
{\boldsymbol u},\delta {\boldsymbol m})$
to the interval $[t,T]$, seen 
as a solution 
of the FBSPDE system in 
\eqref{eq:linear:local}
on $[t,T]$, with $\delta \mu_t$ as initial solution 
and
 under the   conditional 
probability distribution of 
${\mathbb P}^0$ given the $\sigma$-field 
${\mathcal F}_t^0$. 

We get, with probability 1 under ${\mathbb P}^0$,  
\begin{equation*} 
\begin{split} 
 \|\delta  u_t\|_{{\mathbb s}}^2
\leq 
C \|\delta  \mu_t\|_{-\mathbb{r}+1}^2
+ C {\mathbb E}^0 \Bigl[ 
\sup_{r\in [t,T]}\|d_r\|^2 _{-\mathbb{r}+2}
+
\sup_{r\in [t,T]}\|j_r\|^2 _{{\mathbb s}-1}
+
\sup_{r\in [t,T]} |\delta X_r^0|^2
+
\| k_T \|_{{\mathbb s}}^2
\, \vert \, {\mathcal F}_t^0
 \Bigr],
 \end{split} 
\end{equation*}
where $C$ 
only depends on
$d$, $\kappa$, ${\mathfrak K}$, ${\mathfrak L}$ and $(\mathbb{r},{\mathbb s})$ (recalling that $T \leq {\mathfrak C}$) and is allowed to vary from line to line. 
Measurability of the arguments in the above right-hand side 
follows from 
the third item in Remark 
\ref{rem:thm:local:linearFB}. 

Taking 
power $p \in (1,8]$ 
and then expectation (under ${\mathbb P}^0$) on both sides, we obtain
by means of Doob's inequality, 
\begin{equation*}
\begin{split}
&{\mathbb E}^0\Big[\sup_{t\in [0,T]}\|\delta u_t\|_{{\mathbb s}}^{2p}\Big] 
\\
&\leq C
{\mathbb E}^0 \Bigl[ 
\sup_{t\in [0,T]}
 \|\delta  \mu_t\|_{-\mathbb{r}+1}^{2p}
 +
 \sup_{t\in [0,T]}\|d_t\|^{2p}_{-\mathbb{r}+2}
+
\sup_{t\in [0,T]}\|j_t\|^{2p}_{{\mathbb s}-1}
+
\sup_{t\in [0,T]} |\delta X_t^0|^{2p}
+
\| k_T \|_{{\mathbb s}}^{2p}
 \Bigr].
 \end{split}
\end{equation*}
Now, by an obvious adaptation of Lemma 
\ref{le:appendix:2} (see for instance 
\cite[Lemma 3.3.1]{CardaliaguetDelarueLasryLions}),
we have, ${\mathbb P}^0$-almost surely,   
\begin{equation*} 
\sup_{t \in [0,T]}
 \|\delta  \mu_t\|_{-\mathbb{r}+1}
 \leq C \| \triangle \mu \|_{-\mathbb{r}+1} 
 +
 C T \Bigl[ \sup_{t \in [0,T]} \| \delta u_t \|_{{\mathbb s} -1} +
 \sup_{t \in [0,T]} \| d_t \|_{-\mathbb{r} +2} \Bigr]. 
 \end{equation*} 
 And, then,
 assuming $C T \leq 1/2$, 
 \begin{equation*}
\begin{split}
&{\mathbb E}^0\Big[\sup_{t\in [0,T]}\|\delta \mu_t\|_{-\mathbb{r}+1}^{2p}\Big] 
 \leq C \Bigl( 
 \| \triangle \mu \|_{-\mathbb{r}+1}^{2p}
 +
{\mathbb E}^0 \Bigl[ 
 \sup_{t\in [0,T]}\|d_t\|^{2p}_{-\mathbb{r}+2}
+
\sup_{t\in [0,T]}\|j_t\|^{2p}_{{\mathbb s}-1}
+
\sup_{t\in [0,T]} |\delta X_t^0|^{2p}
+
\| k_T \|_{{\mathbb s} }^{2p}
 \Bigr] \Bigr) .
\end{split}
\end{equation*}
And then, by \cite[Theorem A.5]{Delarue02},
\begin{equation*} 
\begin{split} 
&{\mathbb E}^0\bigg[ \biggl( \int_0^T \vert \delta Z_t^0 \vert^2 \ud t \biggr)^p \biggr]
\\
&\leq  
C
\biggl( \vert \triangle x_0 \vert^{2p}
+
{\mathbb E}^0 \biggl[ \sup_{t \in [0,T]}   \| \delta \mu_t \|^{2p}_{-\mathbb{r} +1}  +
\biggl( 
\int_0^T \bigl( \vert a_t \vert + \vert b_t \vert \bigr) \ud t  
\biggr)^{2p} 
+
\vert c_T \vert^{2p} 
\biggr]
\biggr) 
\\
&\leq 
C 
\biggl( \vert \triangle x_0 \vert^{2p} 
+ 
\|
\triangle \mu 
\|_{-\mathbb{r} +1}^{2p}
\\
&\hspace{15pt} +
{\mathbb E}^0 \biggl[ \sup_{t \in [0,T]}  \Bigl(   \|d_t\|^{2p}_{-\mathbb{r}+2}
+
 \|j_t\|^{2p}_{{\mathbb s}-1}
+
 |\delta X_t^0|^{2p}
 \Bigr) 
 +
\biggl( 
\int_0^T \bigl( \vert a_t \vert + \vert b_t \vert \bigr) \ud t  
\biggr)^{2p} 
 + 
\vert c_T \vert^{2p} 
+
\| k_T \|_{{\mathbb s}}^{2p}
\biggr]
\biggr). 
\end{split} 
\end{equation*} 
 Inserting the above bound into the equation for 
 $\delta {\boldsymbol X}^0$, we deduce that, for $C T \leq 1/2$ (and $T \leq 1$), 
 \begin{equation*} 
 \begin{split} 
&{\mathbb E}^0\Bigl[
 \sup_{0 \le t \le T}   |\delta X_t^0|^{2p}
\Bigr]
\\
&\leq  
C 
\biggl( \vert \triangle x_0 \vert^{2p} 
+ 
\|
\triangle \mu 
\|_{-\mathbb{r} +1}^{2p}
\\
&\hspace{15pt} +
{\mathbb E}^0 \biggl[ \sup_{0 \le t \le T} \Bigl(   \|d_t\|^{2p}_{-\mathbb{r}+2}
+
 \|j_t\|^{2p}_{{\mathbb s}-1}
 \Bigr) 
 +
\biggl( 
\int_0^T \bigl( \vert a_t \vert + \vert b_t \vert \bigr) \ud t  
\biggr)^{2p} 
 + 
\vert c_T \vert^{2p} 
+
\| k_T \|_{{\mathbb s}}^{2p}
\biggr]
\biggr)
\\
&=C M_p,
\end{split} 
 \end{equation*} 
and the proof is easily completed by recombining all the 
intermediary steps. 
\end{proof}

\color{black}

\subsection{First order derivatives of the master fields}

For simplicity, the results in the last two subsections are stated with $t=0$ as initial time. However, they remain true for any initial time $t\in [0,T]$ (provided 
$T$ is small enough). In particular, for any initial condition $(t,x_0,\mu)\in [0,T]\times\mathbb R^d\times\mathcal{P}(\mathbb T^d)$, the system \eqref{eq:major:FB:1:small}--\eqref{eq:minor:FB:2:small} has a unique solution, denoted $(X^{0,t,x_0,\mu},Y^{0,t,x_0,\mu}, Z^{0,t,x_0,\mu},\mu^{t,x_0,\mu},u^{t,x_0,\mu},m^{t,x_0,\mu})$, on $[t,T]$ with $T$ given in Theorem \ref{thm:local:FB}, which makes it possible to let
\begin{equation}\label{eq:U:rep}
{\mathcal U}(t,x_0,x,\mu):=u^{t,x_0,\mu}(t,x), \quad x \in {\mathbb T}^d, 
\end{equation}
and
\begin{equation}\label{eq:U0:rep}
{\mathcal U}^0(t,x_0,\mu):=Y^{0,t,x_0,\mu}_{t}.
\end{equation}
Elaborating on 
the proof of \cite[Lemma 5.1.1]{CardaliaguetDelarueLasryLions}
and using 
the uniqueness result of \eqref{eq:major:FB:1:small}--\eqref{eq:minor:FB:2:small}
together with the stability property \eqref{eq:local:diff-1}, we have, with probability 1 under 
${\mathbb P}^0$, 
\begin{equation}
\label{eq:representation:short:time}
u^{t,x_0,\mu}(r,x)={\mathcal U}(r,X_r^{0,t,x_0,\mu},x,\mu_r^{t,x_0,\mu})\quad\text{and}\quad Y_r^{0,t,x_0,\mu}={\mathcal U}^0(r,X_r^{0,t,x_0,\mu},\mu_r^{t,x_0,\mu}),\quad \text{for any }r\in [t,T].
\end{equation}
Thanks to 
\eqref{eq:local:BMO}
and
\eqref{eq:local:reg}, ${\mathcal U}^0$ and ${\mathcal U}$ are bounded. 
In fact, by \eqref{eq:local:diff-1} again, we also have
\begin{equation}
\label{eq:local:diff-1:again}
\begin{split}
&{\mathcal U}: [0,T]\times\mathbb R^d \times\mathcal{P}(\mathbb T^d)
\ni (t,x_0,\mu) 
\mapsto {\mathcal U}(t,x_0,\cdot,\mu) \in {\mathcal C}^{{\mathbb s}}({\mathbb T}^d) \ \text{ is uniformly Lipschitz continuous in $x_0,\mu$},
\\
&{\mathcal U}^0: [0,T]\times\mathbb R^d\times\mathcal{P}(\mathbb T^d)
\ni (t,x_0,\mu) 
\mapsto {\mathcal U}^0(t,x_0,\mu) \in {\mathbb R} 
\
\text{ is uniformly Lipschitz continuous in $x_0,\mu$},\\
\end{split}
\end{equation}
where the Lipschitz continuity in $\mu$ is measured by the $\mathbb W_1$ metric.
Continuity in time is ensured by the following lemma:

\begin{lemma}
\label{lem:U0:U:timereg}
Under Assumption \hyp{\^A}, there exists a constant 
${\mathfrak C}$ only depending on 
$d$, ${\mathfrak L}$, $\kappa$, $\sigma_0$ and ${\mathbb s}$,
such that the following holds true for $T
\leq {\mathfrak C}$. 
For any ${\mathbb r} \in (\lfloor {\mathbb s} \rfloor,{\mathbb s})$, we can find a constant $C$,  only depending on 
$d$, ${\mathfrak L}$, $\kappa$, $\sigma_0$, ${\mathbb r}$
and 
${\mathbb s}$, such that, for any $(t,x_0,\mu)\in [0,T]\times\mathbb R^d\times\mathcal{P}(\mathbb T^d)$ and $r\in[t,T]$,
\begin{equation}\label{eq:U:timereg}
\|{\mathcal U}(r,x_0,\cdot,\mu)-{\mathcal U}(t,x_0,\cdot,\mu)\|_{{\mathbb r}}\leq C(r-t)^{({\mathbb s}- {\mathbb r})/2},
\end{equation}
and
\begin{equation}\label{eq:U0:timereg}
|{\mathcal U}^0(r,x_0,\mu)-{\mathcal U}^0(t,x_0,\mu)|\leq C(r-t)^{1/2}.
\end{equation}
\end{lemma}
\begin{proof}
\textit{First Step.} We first show \eqref{eq:U:timereg}.
Following \cite[(4.12)]{CardaliaguetDelarueLasryLions}, note that, with probability 1 under ${\mathbb P}^0$, 
\begin{equation*}
u_t(\cdot)= \mathbb E^0\bigg[ P_{r-t}u_r(\cdot)-\int_{t}^{r}P_{s-t} \Bigl( \hat{H}(\cdot,\nabla_x u_s(\cdot)) -f_s(X_s^0,\cdot,\mu_s)\Bigr)   \ud s\bigg],\quad\text{for any }r\in [t,T],
\end{equation*}
where $(P_t)_{t\in [0,T]}$ stands here for the standard heat kernel on the torus $\mathbb T^d$. 
The  expectation in the right-hand side should be understood as follows: 
here and below, we write ${\mathbb E}^0[h(\cdot)]$ for the mapping 
$x \in {\mathbb T}^d \mapsto {\mathbb E}^0[h(x)]$ when 
$h : \Omega^0 \times {\mathbb T}^d \rightarrow {\mathbb R}$
is a (measurable) random field such that ${\mathbb E}^0[\vert h(x) 
\vert]$ is finite for any $x \in {\mathbb T}^d$. 
Of course, Lebesgue's theorem makes it possible to address the regularity of 
${\mathbb E}^0[h(\cdot)]$
when $h$ itself is regular in ${\mathbb T}^d$.  
Also, we recall 
from 
\eqref{eq:local:reg}
that there exists a constant $C$, only depending 
on the parameters in \hyp{\^A}, such that 
${\mathbb P}^0(\{ \sup_{t \leq s \le T} \| \nabla_x u_s \|_{{\mathbb s} -1} \leq C\})=1$. 
And then (with ${\rm id}$ denoting the identity mapping),
\begin{equation*}
u_t(\cdot)-\mathbb E^0[u_r(\cdot)]=\mathbb E^0\bigg[\Big(P_{r-t}-{\rm id}\Big)u_r(\cdot)-\int_{t}^{r}P_{s-t} \Bigl( \hat{H}(\cdot,\nabla_x u_s(\cdot)) -f_s(X_s^0,\cdot,\mu_s)\Bigr)   \ud s\bigg],
\end{equation*}
 which further implies
 \begin{equation*}
 \begin{split}
 \big\|u_t-\mathbb E^0[u_r]\big\|_{{\mathbb r}}&\leq\mathbb E^0\bigg[\big\|(P_{r-t}-{\rm id})u_r\big\|_{{\mathbb r}}+C\int_{t}^r(s-t)^{-1/2}\big\| \hat{H}(\cdot,\nabla_x u_s(\cdot))-f_s(X_s^0,\cdot,\mu_s)\big\|_{{\mathbb r}-1}\ud s\bigg]\\
 &\leq\mathbb E^0\bigg[\big\|(P_{r-t}-{\rm id})u_r\big\|_{{\mathbb r}}\bigg]+C(r-t)^{1/2},
 \end{split}
 \end{equation*}
 with the constant $C$ only depends on the parameters in Assumption \hyp{\^A}. 
 It is standard to have
 \begin{equation*}
 \mathbb E^0\Big[\big\|(P_{r-t}-{\rm id})u_r\big\|_{{\mathbb r}}\Big]\leq C(r-t)^{({\mathbb s}-{\mathbb r})/{2}}\mathbb E^0\big[\|u_r\|_{{\mathbb s}}\big]\leq C(r-t)^{({\mathbb s}-{\mathbb r})/{2}}.
 \end{equation*}
 Therefore,
 allowing the constant $C$ to vary from line to line, 
 \begin{equation*}
  \big\|u_t-\mathbb E^0[u_r]\big\|_{{\mathbb r}}\leq C(r-t)^{({\mathbb s}-{\mathbb r})/{2}}.
 \end{equation*}
 By \eqref{eq:U:rep}, we can derive
 \begin{equation*}
 \begin{split}
 \big\|u_t-{\mathcal U}(r,x_0,\cdot,\mu)\big\|_{{\mathbb r}}&\leq \big\|u_t-\mathbb E^0[u_r]\big\|_{r}+\Big\|\mathbb E^0[u_r]-{\mathcal U}(r,x_0,\cdot,\mu) \Big\|_{{\mathbb r}}\\
 &\leq C(r-t)^{({\mathbb s}-{\mathbb r})/{2}}+\Big\|\mathbb E^0\Big[{\mathcal U}(r,X_r^{0,t,x_0,\mu},\cdot,\mu_r^{t,x_0,\mu})-{\mathcal U}(r,x_0,\cdot,\mu)\Big]\Big\|_{{\mathbb r}}
 \\
 &\leq C\Big\{(r-t)^{({\mathbb s}-{\mathbb r})/{2}}+\mathbb E^0\big[\mathbb W_1(\mu_r^{t,x_0,\mu},\mu)+|X_r^{0,t,x_0,\mu}-x_0|\big]\Big\}
 \\
 &\leq C(r-t)^{({\mathbb s}-{\mathbb r})/{2}}.
 \end{split}
 \end{equation*}
 \vskip 2pt
 
\noindent  \textit{Second Step.} We now show \eqref{eq:U0:timereg}. We first have
\begin{equation}\label{eq:U0:diff}
\begin{split}
&\big|{\mathcal U}^0(r,x_0,\mu)-{\mathcal U}^0(t,x_0,\mu)\big|\\
&\leq\bigl\vert 
{\mathcal U}^0(r,x_0,\mu)-\mathbb E^0\bigl[{\mathcal U}^0(r,X_r^{0,t,x_0,\mu},\mu_r^{t,x_0,\mu})\bigr]
\bigr\vert
+ 
\bigl\vert \mathbb E^0\bigl[{\mathcal U}^0(r,X_r^{0,t,x_0,\mu},\mu_r^{t,x_0,\mu})-{\mathcal U}^0(t,x_0,\mu)\bigr]\bigr\vert
\\
&\leq C \Bigl( \mathbb E^0\big[\mathbb W_1(\mu_r^{t,x_0,\mu},\mu)+|X_r^{0,t,x_0,\mu}-x_0|\big]+\big|\mathbb E^0\big[Y_r^{0,t,x_0,\mu}-Y_t^{0,t,x_0,\mu}\big]\big|
\Bigr)
\\
&\leq C \Bigl( |r-t|^{1/2}+\big|\mathbb E^0\big[Y_r^{0,t,x_0,\mu}-Y_t^{0,t,x_0,\mu}\big]\big| \Bigr).
\end{split}
\end{equation}
By means of 
\eqref{eq:local:BMO}, we
have a bound for ${\mathbb E}^0 \int_t^r \vert Z_s^{0,t,x_0,\mu} \vert^2 
\ud s$, from which we obtain
\begin{equation}\label{eq:Y0:diff}
\mathbb E^0\big[\big|Y_r^{0,t,x_0,\mu}-Y_t^{0,t,x_0,\mu}\big|\big]\leq C|r-t|^{1/2}.
\end{equation}
Combining \eqref{eq:U0:diff} and \eqref{eq:Y0:diff}, we can derive
\begin{equation*}
|{\mathcal U}^0(r,x_0,\mu)-{\mathcal U}^0(t,x_0,\mu)|\leq C(r-t)^{1/2},
\end{equation*}
which completes the proof. 
\end{proof}

We now would like to establish the continuous differentiability of ${\mathcal U}$ and ${\mathcal U}_0$ in $x_0$ and $\mu$. For any initial condition $(t,x_0,\mu)\in [0,T]\times\mathbb R^d\times\mathcal{P}(\mathbb T^d)$ and $(\triangle x_0,\triangle \mu)\in \mathbb R^d\times\mathcal{C}^{-\mathbb{r}+1}(\mathbb T^d)$,
for some $\mathbb{r} \in (\lfloor {\mathbb s} \rfloor, {\mathbb s})$, 
 we consider the solution $(\delta {\boldsymbol X}^0,\delta {\boldsymbol Y}^0,\delta {\boldsymbol Z}^0,\delta {\boldsymbol \mu}, \delta {\boldsymbol u},\delta {\boldsymbol m})$ (its dependence on the initial condition is omitted for simplicity) to the linearized forward-backward 
system, set on $[t,T]$, 
\begin{equation} 
\label{eq:linear H:local}
\begin{split} 
&\ud \delta X_r^0 =  \Big[-\nabla^2_{px_0 }H^0(X_r^0,Z_r^0)\delta X_r^0-\nabla^2_{pp}H^0(X_r^0,Z_r^0)\delta Z_r^0\Big] \ud r , 
\\
&\ud \delta Y_r^0 =\Big[ \nabla_{x_0}\hat L^0(X_r^0,Z_r^0)\cdot\delta X_r^0+\nabla_{p}\hat L^0(X_r^0,Z_r^0)\cdot\delta Z_r^0
\\
& \hspace{30pt}  - \nabla_{x_0} f_r^0(X_r^0,\mu_r)\cdot \delta X_r^0 - \Big(\delta_\mu f_r^0(X_r^0,\mu_r),\delta \mu_r\Big) \Big]
 \ud r + 
\sigma_0 \delta Z_r^0 \cdot \ud B_r^0, 
\\
& \delta X_t^0= \triangle x_0, \quad \delta Y_T^0 = \nabla_{x_0}g^0(X_T^0,\mu_T)\cdot \delta X_T^0+\Big(\delta_\mu g^0(X_T^0,\mu_T),\delta \mu_T\Big);\\
&  \partial_r \delta \mu_r
-\tfrac{1}{2}\Delta_{x}\delta \mu_r
-{\rm div}_x\big(\nabla_{p} \hat{H}(x,\nabla_xu_r(x))\delta \mu_r
+\mu_r\nabla_{pp}^2 \hat{H}(x,\nabla_xu_r(x))\nabla_x\delta u_r(x)\big)=0,
\\
& \ud_r\delta u_r(x)=\Big[-\tfrac{1}{2}\Delta_{x}\delta u_r(x)+\nabla_p \hat{H}(x,\nabla_x u_r(x))\cdot\nabla_x \delta u_r(x)-\nabla_{x_0}f_r(X_r^0,x,\mu_r)\cdot\delta X_r^0\\
& \hspace{45pt}-\Big(\delta_\mu f_r(X_r^0,x,\mu_r),\delta \mu_r\Big)\Big]\ud r +\ud_r \delta m_r(x),\quad x \in\mathbb T^d,\\
&\delta \mu_t=\triangle \mu,\quad \delta u_T(x)=\nabla_{x_0}g(X_T^0,x,\mu_T)\cdot\delta X_T^0+\Big(\delta_\mu g(X_T^0,x,\mu_T),\delta \mu_T\Big)\quad \text{on } \mathbb T^d,
\end{split}
\end{equation} 
where $({\boldsymbol X}^0,{\boldsymbol Y}^0,{\boldsymbol Z}^0,{\boldsymbol \mu},{\boldsymbol u},{\boldsymbol m})$ is the unique solution to the system \eqref{eq:major:FB:1:small}--\eqref{eq:minor:FB:2:small} corresponding to the initial condition $(t,x_0,\mu)$. Applying Theorem \ref{thm:local:linearFB} to \eqref{eq:linear H:local}, we have the following theorem.
\begin{theorem}\label{thm:local:linearHFB}
Let ${\mathbb r} \in (\lfloor \mathbb{s} \rfloor,\mathbb{s})$.
Under Assumption \hyp{\^A}, 
there exist a constant 
$c$ only depending on $d$,  $\kappa$, ${\mathfrak L}$ and 
$\mathbb{s}$
and a constant ${\mathfrak C}$ only depending on $d$,  $\kappa$, 
${\mathfrak L}$, 
$\zeta(c)$, $\sigma_0$ and 
$({\mathbb r},\mathbb{s})$
such that, 
for $T \leq {\mathfrak C}$, 
the forward-backward system \eqref{eq:linear:local} admits a unique solution $(\delta {\boldsymbol \mu}, \delta {\boldsymbol u}, \delta {\boldsymbol m}, \delta {\boldsymbol X}^0, \delta {\boldsymbol Y}^0, \delta {\boldsymbol Z}^0)$, adapted with respect to the filtration $\mathbb F^0$ and with values in 
${\mathcal C}^{-\mathbb{r}+1}({\mathbb T}^d) \times {\mathcal C}^{{\mathbb s}}({\mathbb T}^d) 
\times {\mathcal C}^{{\mathbb s} -2}({\mathbb T}^d) \times {\mathbb R}^d \times {\mathbb R} \times {\mathbb R}^d$, 
and satisfying
  items (a), (b) and (c) in the statement of Theorem 
 \ref{thm:local:linearFB}. 
%

Moreover, 
there exists a constant $C$ only depending on the parameters in
Assumption \hyp{\^A} and on ${\mathbb r}$ such that, 
for any $p \in [1,8]$, 
\begin{equation}\label{eq:local:linear H:diff0}
\begin{split}
&\mathbb E^0\biggl[\sup_{t\in [0,T]} \Bigl( 
\|\delta \mu_t\|_{-\mathbb{r}+1}^{2p}
+
\|\delta u_t\|_{\mathbb s}^{2p}
+
\|\delta m_t\|_{{\mathbb s}-2}^{2p}
 \Bigr)  + \sup_{t\in [0,T]} \Bigl( |\delta X_t^0|^{2p}+ |\delta Y_t^0|^{2p} \Bigr) +
 \biggl( \int_0^T|\delta Z_t^0|^2\ud t \biggr)^p \biggr]
 \\
&\leq C\big(\|\delta \mu\|_{-\mathbb{r}+1}^{2p} +|\triangle x_0|^{2p} \big).
\end{split}
\end{equation}
\end{theorem}
\begin{proof}
The only subtlety in the application of 
Theorem \ref{thm:local:linearFB} 
lies in the verification of (v) in 
Assumption \hyp{C}. 
In fact, by Theorem 
\ref{thm:local:FB}, 
we have (for $T \leq {\mathfrak C}$) a bound for 
$\| u_t \|_{\mathbb s}$. And then, we can insert this bound in the 
function $\zeta$ in Assumption \hyp{\^A4}. 
The rest of the proof does not raise any difficulty. 
\end{proof}

Given $(t,x_0,y,\mu)\in [0,T]\times\mathbb R^d\times\mathbb T^d \times\mathcal{P}(\mathbb T^d)$, $i\in \{1,\cdots, d\}$ and a $d$-tuple $l\in \{0,\cdots, \lfloor {\mathbb s} \rfloor-1\}^d$ with $|l|:=\sum_{j=1}^dl_j\leq \lfloor {\mathbb s} \rfloor-1$, we first denote by $(\delta {\boldsymbol \mu}^i,\delta {\boldsymbol u}^i,\delta {\boldsymbol m}^{i},\delta {\boldsymbol X}^{0,i},\delta {\boldsymbol Y}^{0,i},\delta {\boldsymbol Z}^{0,i})$ the solution to the system \eqref{eq:linear H:local} on $[t,T]$ with $\triangle \mu=0$ and $\triangle x^0=e_i$, and then denote by 
$(\delta {\boldsymbol \mu}^{l,y}, \delta {\boldsymbol u}^{l,y}, \delta {\boldsymbol m}^{l,y}, \delta {\boldsymbol X}^{0,l,y}, \delta {\boldsymbol Y}^{0,l,y}, \delta {\boldsymbol Z}^{0,l,y})$ the solution to the system \eqref{eq:linear H:local} on $[t,T]$ with $\triangle \mu=(-1)^{|l|} \nabla^l \delta_y\in \mathcal{C}^{-(\mathbb{r}-1)}({\mathbb T}^d)$ ($\delta_y$ denoting the Dirac mass at point $y$) for any 
$\mathbb{r} \in (\lfloor {\mathbb s} \rfloor,{\mathbb s})$, and $\triangle x^0=0$. 
Define
\begin{equation}
\label{eq:all:the:Ks}
\begin{split}
&K_i^{x_0}(t,x_0,x,\mu) :=\delta u_t^{i}(x),\,\, K_i^{0,x_0}(t,x_0,\mu):=
\delta Y_t^{0,i},
\\
&K_l^\mu(t,x_0,x,\mu,y):=\delta u_t^{l,y}(x),\,\,  K_l^{0,\mu}(t,x_0,\mu,y):=\delta Y_t^{0,l,y}.
\end{split} 
\end{equation}
Then, choosing for instance 
${\mathbb r}=(\lfloor {\mathbb s} \rfloor + {\mathbb s})/2$, 
we can apply Theorem \ref{thm:local:linearHFB} to obtain
(for $T \leq {\mathfrak C}$, for some ${\mathfrak C}>0$ only depending 
on the parameters
$d$, $\kappa$,  ${\mathfrak L}$,  
$\zeta$ and 
$\mathbb{s}$
 in Assumption \hyp{\^A})
\begin{equation}\label{eq:Kbdd}
\|K_i^{x_0}(t,x_0,\cdot,\mu)\|_{\mathbb s}+\|K_{l}^{\mu}(t,x_0,\cdot,\mu,y)\|_{\mathbb s}+|K_i^{0,x_0}(t,x_0,\mu)|+|K_l^{0,\mu}(t,x_0,\mu,y)|\leq C,
\end{equation}
for a constant $C$ only depending on the parameters in 
Assumption \hyp{\^A}. 
In what follows, we just write $K^{x_0}(t,x_0,\cdot,\mu)$ and 
$K^{0,x_0}(t,x_0,\mu)$
for the $d$-tuples
$(K_i^{x_0}(t,x_0,\cdot,\mu))_{i=1,\cdots,d}$ and 
$(K_i^{0,x_0}(t,x_0,\mu))_{i=1,\cdots,d}$. 

The following statement is the analogue of \cite[Lemma  5.2.2]{CardaliaguetDelarueLasryLions}:

\begin{lemma}\label{lem:deltau:reginy}
Given $(t,x_0,\mu)\in [0,T]\times\mathbb R^d\times\mathcal{P}(\mathbb T^d)$ and $p \in [ 1,8]$, we have under the same generic condition $T \leq {\mathfrak C}$ as before, for any $d$-tuple $l\in \{0,\cdots,\lfloor {\mathbb s} \rfloor-1\}^d$ with $|l|:=\sum_{i=1}^dl_i\leq \lfloor {\mathbb s} \rfloor-1$ and any $y \in {\mathbb T}^d$,
\begin{equation}
 \label{eq:deltau:cony}
 \begin{split}
\lim_{h\in\mathbb R^d,h\to 0} &\mathbb E^0\Big[\sup_{r\in [t,T]}\big\|\delta u_r^{l,y+h}-\delta u_r^{l,y}\big\|_{\mathbb s}^{2p}
+\sup_{r\in [t,T]}\big\|\delta \mu_r^{l,y+h}-\delta \mu_{r}^{l,y}\big\|_{-{\mathbb s}+1}^{2p} 
\\
&\hspace{15pt} +\sup_{r\in [t,T]}\big|\delta Y_r^{0,l,y+h}-\delta Y_r^{0,l,y}\big|^{2p}\Big]=0.
\end{split} 
\end{equation}
Moreover, for any $l\in \{0,\cdots, \lfloor {\mathbb s} \rfloor-2\}^d$ with $|l|\leq  \lfloor {\mathbb s} \rfloor-2$, and any $i\in \{1,\cdots, d\}$, 
\begin{equation*}
\begin{split}
&\lim_{h\in\mathbb R,h\to 0}\mathbb E^0\Big[\sup_{r\in [t,T]}\Bigl\| \frac{1}{h}(\delta u_r^{l,y+
he_i}-\delta u_r^{l,y})-\delta u_r^{l+e_i,y}\Big\|_{\mathbb s}^{2p} +\sup_{r\in [t,T]}\Big\|\frac{1}{h}(\delta \mu_r^{l,y+
he_i}-\delta \mu_r^{l,y})-\delta \mu_r^{l+e_i,y}\Big\|_{-{\mathbb s}+1}^{2p} 
\\
& \hspace{15pt} +\sup_{r\in [t,T]}  \Big|\frac{1}{h}(\delta Y_r^{0,l,y+
he_i}-\delta Y_r^{0,l,y})-\delta Y_r^{0,l+e_i,y}\Big|^{2p} \Big]=0.
\end{split}
\end{equation*}
In particular, the function $(x,y)\in\mathbb T^d\times\mathbb T^d\mapsto K_0^{\mu}(t,x_0,x,\mu,y)$ and $y\in\mathbb T^d\mapsto K_0^{0,\mu}(t,x_0,\mu,y)$ are $(\lfloor {\mathbb s} \rfloor-1)$-times continuously differentiable with respect to $y$, and for any $l\in \{0,\cdots,\lfloor {\mathbb s} \rfloor-1\}$ with $|l|\leq \lfloor {\mathbb s} \rfloor-1$, the derivative $y \in {\mathbb T}^d \mapsto \nabla_y^l K_0^{\mu}(t,x_0,\cdot,\mu,y)\in\mathcal{C}^{\mathbb s}(\mathbb T^d)$ is continuous. It holds, for any 
$(x,y)\in\mathbb T^d\times\mathbb T^d$, 
\begin{equation*}
\nabla_y^l K^{\mu}_0(t,x_0,x,\mu,y)=K^{\mu}_l(t,x_0,x,\mu,y)\quad\text{and}\quad \nabla_y^lK_0^{0,\mu}(t,x_0,\mu,y)=K_l^{0,\mu}(t,x_0,\mu,y).
\end{equation*}
Moreover, there exists a constant $C$, only depending on the parameters in Assumption \hyp{\^A}, such that 
\begin{equation*}
\sup_{(t,x_0,\mu,y)\in [0,T]\times\mathbb R^d\times\mathcal{P}(\mathbb T^d)\times\mathbb T^d}\big\|K^{\mu}_l(t,x_0,\cdot,\mu,y)\big\|_{\mathbb s}
\leq C, 
\end{equation*}
and
\begin{equation*}
 \sup_{(t,x_0,\mu,y)\in [0,T]\times\mathbb R^d\times\mathcal{P}(\mathbb T^d)\times\mathbb T^d}\big|K_l^{0,\mu}(t,x_0,\mu,y)\big| \leq C. 
\end{equation*}
\end{lemma}
\begin{proof}
By Theorem \ref{thm:local:linearHFB}, 
for 
any fixed $\mathbb{r} \in (\lfloor {\mathbb s} \rfloor,{\mathbb s})$ 
and
any $p \in [1,8]$, 
we can find a constant $C_p$ such that, for all $(t,x,\mu,y)\in [0,T]\times\mathbb R^d\times\mathcal{P}(\mathbb T^d)\times\mathbb T^d$ and $l\in \{1,\cdots,\lfloor {\mathbb s} \rfloor-1\}$ with $|l|\leq \lfloor {\mathbb s} \rfloor-1$,
\begin{equation*}
\mathbb E^0\Big[\sup_{r\in [t,T]}\big\|\delta u_r^{l,y}\big\|_{\mathbb s}^{2p} +\sup_{r\in [t,T]}\big\|\delta \mu_r^{l,y}\big\|_{-\mathbb{r} +1}^{2p}
+\sup_{r\in [t,T]}\big|\delta Y_r^{0,l,y}\big|^{2p} \Big]\leq C_p.
\end{equation*}
In particular,
\begin{equation*}
\|K^\mu_l(t,x_0,\cdot,\mu,y)\|_{\mathbb s} +|K^{0,\mu}_l(t,x_0,\mu,y)|\leq C.
\end{equation*}
Notice that
\begin{equation*}
\lim_{h\in\mathbb R^d, h\to 0}\big\|\nabla^l_y \delta_{y+h}-\nabla_y^l\delta_y\big\|_{-\mathbb{r}+1}=0.
\end{equation*}
Therefore, Theorem \ref{thm:local:linearHFB} gives \eqref{eq:deltau:cony}. This yields
\begin{equation*}
\lim_{h\in\mathbb R^d,h\to 0}\Big(\big\|K^\mu_{l}(t,x_0,\cdot,\mu,y+h)-K^\mu_{l}(t,x_0,\cdot,\mu,y)\big\|_{\mathbb s}+\big|K_l^{0,\mu}(t,x_0,\mu,y+h)-K_l^{0,\mu}(t,x_0,\mu,y)\big|\Big)=0,
\end{equation*}
which implies
that the mappings $y\in\mathbb T^d\mapsto K^\mu_l(t,x_0,\cdot,\mu,y)\in\mathcal{C}^{\mathbb s}(\mathbb T^d)$  
and 
 $y\in\mathbb T^d\mapsto K^{0,\mu}_{l}(t,x_0,\mu,y)\in\mathbb R$ are continuous. 
Similarly, for any $l\in \{0,\cdots, \lfloor {\mathbb s} \rfloor-2\}^d$ with $|l|\leq  \lfloor {\mathbb s} \rfloor-2$, and any $i\in \{1,\cdots, d\}$, 
\begin{equation*}
\lim_{h\in\mathbb R,h\to 0}\Big\|\frac{1}{h}\big(\nabla^l\delta_{y+he_i}-\nabla^l\delta_y\big)+\nabla^{l+e_i}\delta_y\Big\|_{-\mathbb{r}+1}=0.
\end{equation*}
Therefore, again by Theorem \ref{thm:local:linearHFB}, we get
\begin{equation*}
\begin{split}
&\lim_{h\in\mathbb R,h\to 0}\bigg(\Big\|\frac{1}{h}\big(K_l^\mu(t,x_0,\cdot,\mu,y+he_i)-K_l^\mu(t,x_0,\cdot,\mu,y)\big)-K^\mu_{l+e_i}(t,x_0,\cdot,\mu,y)\Big\|_{\mathbb s}
\\
& \hspace{15pt} +\Big|\frac{1}{h}\big(K^{0,\mu}_l(t,x_0,\mu,y+he_i)-K^{0,\mu}_l(t,x_0,\mu,y)\big)-K^{0,\mu}_{l+e_i}(t,x_0,\mu,y)\Big|\bigg)=0,
\end{split}
\end{equation*}
which implies, by induction, that 
\begin{equation*}
\nabla_y^lK^\mu_0(t,x_0,x,\mu,y)=K_l^\mu(t,x_0,x,\mu,y)\quad\text{and}\quad \nabla_y^lK^{0,\mu}_0(t,x_0,\mu,y)=K_l^{0,\mu}(t,x_0,\mu,y).
\end{equation*}
The proof is completed. 
\end{proof}

The following statement is the analogue of 
\cite[Lemma  5.2.3]{CardaliaguetDelarueLasryLions}:

\begin{lemma}
Given $\triangle x_0\in\mathbb R^d$ and a finite signed measure $\triangle \mu$ on $\mathbb T^d$, the solution  $(\delta {\boldsymbol \mu}, \delta {\boldsymbol u}, \delta
{\boldsymbol m}, \delta {\boldsymbol X}^0, \delta {\boldsymbol Y}^0, \delta {\boldsymbol Z}^0)$ to the system \eqref{eq:linear H:local} has the following representation formulas
\begin{equation}\label{eq:deltau:rep}
\delta u_t(x)=K^{x_0}(t,x_0,x,\mu)\cdot\triangle x_0+\int_{\mathbb T^d}K_0^\mu(t,x_0,x,\mu,y)\triangle \mu(\ud y),
\end{equation}
and
\begin{equation}\label{eq:deltaY:rep}
\delta Y_t^0=K^{0,x_0}(t,x_0,\mu)\cdot \triangle x_0+\int_{\mathbb T^d}K_0^{0,\mu}(t,x_0,\mu,y)\triangle \mu(\ud y).
\end{equation}
\end{lemma}
\begin{proof}
By compactness of the torus, we can find, for a given $\varepsilon>0$, a covering $\{U_i\}_{1\leq i\leq N}$ of $\mathbb T^d$, made of disjoint Borel subsets with diameter less than $\varepsilon$. Choosing, for each $i\in\{1,\cdots,N\}$, $y_i\in U_i$, we then let
\begin{equation*}
\triangle \mu^\varepsilon:=\sum_{i=1}^N \triangle \mu(U_i)\delta_{y_i}.
\end{equation*}
Then,
\begin{equation*}
\begin{split} 
\|\triangle \mu-\triangle \mu^\varepsilon\|_{-1}
&=
\sup_{\|\phi\|_1\leq 1} 
\bigg|\int_{\mathbb T^d}\phi(y)\big(\triangle \mu(\ud y)-\triangle \mu^{\varepsilon}(\ud y)\big)\bigg|
\\
&=\sup_{\|\phi\|_1\leq 1}  \bigg|\sum_{i=1}^N\int_{U_i}\big(\phi(y)-\phi(y_i)\big)\triangle \mu(\ud y)\bigg|\leq C
|\triangle \mu |_{\rm TV} 
\varepsilon,
\end{split} 
\end{equation*}
where $|\triangle \mu|_{\rm TV}$ is the total mass of $\triangle \mu$. 
Consider now 
\eqref{eq:linear:local}
with
$(\triangle x_0,\triangle \mu^\varepsilon)$ as initial condition at time 
$t$
and denote the corresponding solution by 
$(\delta {\boldsymbol \mu}^{\varepsilon}, \delta {\boldsymbol u}^{\varepsilon}, \delta
{\boldsymbol m}^{\varepsilon}, \delta {\boldsymbol X}^{0,\varepsilon}, \delta {\boldsymbol Y}^{0,\varepsilon}, \delta {\boldsymbol Z}^{0,\varepsilon})$.
By linearity of the system 
\eqref{eq:linear H:local}, we observe that 
\begin{equation*} 
\begin{split}
\delta Y_t^{0,\varepsilon} &= 
K^{0,x_0}(t,x_0,\mu)\cdot \triangle x_0+\sum_{i=1}^N
K^{0,\mu}_0(t,x_0,\mu,y_i)
\triangle \mu(U_i)
\\
&=
K^{0,x_0}(t,x_0,\mu)\cdot \triangle x_0+\sum_{i=1}^N\int_{U_i}K^{0,\mu}_0(t,x_0,\mu,y_i)\triangle \mu(\ud y),
\\
\delta u_t^{\varepsilon}(x) &= 
K^{x_0}(t,x_0,x,\mu)\cdot \triangle x_0+\sum_{i=1}^N
K^\mu_0(t,x_0,x,\mu,y_i)\triangle \mu(U_i)
\\
&=
K^{x_0}(t,x_0,x,\mu)\cdot \triangle x_0+\sum_{i=1}^N\int_{U_i}K^\mu_0(t,x_0,x,\mu,y_i)\triangle \mu(\ud y), \quad x \in {\mathbb T}^d. 
\end{split} 
\end{equation*} 
And 
then, by invoking Theorem \ref{thm:local:linearHFB} in order to compare 
$(\delta {\boldsymbol u},\delta {\boldsymbol Y}^0)$
and 
$(\delta {\boldsymbol u}^{\varepsilon},\delta {\boldsymbol Y}^{0,\varepsilon})$, we obtain 
\begin{equation*}
\begin{split}
&\bigg\|\delta u_t(\cdot)-K^{x_0}(t,x_0,\cdot,\mu)\cdot \triangle x_0-\sum_{i=1}^N\int_{U_i}K^\mu_0(t,x_0,\cdot,\mu,y_i)\triangle \mu(\ud y)\bigg\|_{{\mathbb s}}\\
&+\bigg|\delta Y_t^0-K^{0,x_0}(t,x_0,\mu)\cdot\triangle x_0-\sum_{i=1}^N\int_{U_i}K_0^{0,\mu}(t,x_0,\mu,y_i)\triangle \mu(\ud y)\bigg|\\
&\leq C\|\triangle \mu- \triangle \mu^\varepsilon\|_{-{\mathbb s}+1}\leq C\| \triangle \mu- \triangle \mu^\varepsilon\|_{-1}\leq C |\triangle \mu|_{\rm TV} \varepsilon.
\end{split}
\end{equation*}
By the smoothness of $K^\mu_0$ and $K_0^{0,\mu}$ derived in Lemma \ref{lem:deltau:reginy}, we can easily derive that
\begin{equation}\label{eq:K:regy}
\begin{split}
&\bigg\|\delta u_t(\cdot)-K^{x_0}(t,x_0,\cdot,\mu)\cdot \triangle x_0-\sum_{i=1}^N\int_{U_i}K^\mu_0(t,x_0,\cdot,\mu,y)
\triangle \mu(\ud y)\bigg\|_{{\mathbb s}}
\\
&+\bigg|\delta Y_t^0-K^{0,x_0}(t,x_0,\mu)\cdot\triangle x_0-\sum_{i=1}^N\int_{U_i}K_0^{0,\mu}(t,x_0,\mu,y)\triangle \mu(\ud y)\bigg|
\\
&\leq 
C|\triangle \mu|_{\rm TV} \varepsilon.
\end{split}
\end{equation}
Letting $\varepsilon\to 0$ in \eqref{eq:K:regy}, we obtain \eqref{eq:deltau:rep} and \eqref{eq:deltaY:rep}.
\end{proof}

\subsection{Differentiability of the fields ${\mathcal U}^0$ and ${\mathcal U}$ in the variables $x_0$ and 
$\mu$}

The following statement provides a second-order expansion of the 
fields ${\mathcal U}^0$ and 
${\mathcal U}$: 

\begin{proposition}\label{prop:local:1stdiff}
Under Assumption \hyp{\^A}, consider two initial conditions $(t,x_0,\mu), (t,\hat x_0,\hat\mu)\in [0,T]\times\mathbb R^d\times\mathcal{P}(\mathbb T^d)$ and 
call 
$({\boldsymbol \mu},{\boldsymbol u},{\boldsymbol m},{\boldsymbol X}^0,{\boldsymbol Y}^0,{\boldsymbol Z}^0)$ 
and $(\hat {\boldsymbol \mu},\hat {\boldsymbol u},\hat {\boldsymbol m},\hat {\boldsymbol X}^0,\hat {\boldsymbol Y}^0,\hat 
{\boldsymbol Z}^0)$ the solutions of the system \eqref{eq:major:FB:1:small}--\eqref{eq:minor:FB:2:small} with $(t,x_0,\mu)$ and $(t,\hat x_0,\hat \mu)$ as 
respective initial conditions and $(\delta {\boldsymbol \mu},\delta 
{\boldsymbol u},\delta {\boldsymbol m},\delta {\boldsymbol X}^0,\delta 
{\boldsymbol Y}^0,\delta {\boldsymbol Z}^0)$ the solution of the system \eqref{eq:linear H:local} with $(t,\hat x_0-x_0,\hat \mu-\mu)$ as
 initial condition, so that we can let
\begin{equation*}
\begin{split}
&\triangle^{\hspace{-1pt} 2} {\boldsymbol X}^0:=\hat {\boldsymbol X}^0-{\boldsymbol X}^0-\delta {\boldsymbol X}^0,\quad
\triangle^{\hspace{-1pt} 2} {\boldsymbol Y}^0:=\hat {\boldsymbol Y}^0-{\boldsymbol Y}^0-\delta {\boldsymbol Y}^0,\quad
\triangle^{\hspace{-1pt} 2} {\boldsymbol Z}^0:=\hat {\boldsymbol Z}^0-{\boldsymbol Z}^0-\delta {\boldsymbol Z}^0,
\\
&\triangle^{\hspace{-1pt} 2} {\boldsymbol \mu}:=\hat
{\boldsymbol \mu}-{\boldsymbol \mu}-\delta{\boldsymbol \mu},\quad\triangle^{\hspace{-1pt} 2} 
{\boldsymbol u}:=\hat {\boldsymbol u}-{\boldsymbol u}-\delta {\boldsymbol u},
\quad\triangle^{\hspace{-1pt} 2} {\boldsymbol m}:=
\hat{\boldsymbol m}
- {\boldsymbol m}
-\delta {\boldsymbol m}.
\end{split}
\end{equation*}
Let $\mathbb{r} \in (\lfloor {\mathbb s} \rfloor,{\mathbb s})$.
Then,
for $T \leq {\mathfrak C}$,
for some ${\mathfrak C}>0$ only depending on the parameters 
$d$,  $\kappa$, ${\mathfrak L}$, 
$\zeta$, $\sigma_0$ and 
${\mathbb s}$
in Assumption \hyp{\^A} and on ${\mathbb r}$,  
 we can find a constant $C$, only depending
 on the parameters in Assumption \hyp{\^A} and on ${\mathbb r}$, 
   such that, for any 
   $p \in [1,8]$, 
\begin{equation}\label{eq:local:delta2}
\begin{split}
&\mathbb E^0\bigg[
\sup_{r\in [t,T]}
\Bigl( 
\|\triangle^{\hspace{-1pt} 2} \mu_r\|_{-\mathbb{r}+1}^p
+
 \|\triangle^{\hspace{-1pt} 2} u_r\|_{\mathbb s}^{p}
+
 \|\triangle^{\hspace{-1pt} 2} m_r\|_{{\mathbb s}-2}^{p} \Bigr) 
 \bigg]
 \\
&+\mathbb E^0\bigg[\sup_{r\in [t,T]}|\triangle^{\hspace{-1pt} 2} X_r^0|^p+\sup_{r\in [t,T]}|\triangle^{\hspace{-1pt} 2} Y_t^0|^p+\int_t^T|\triangle^{\hspace{-1pt} 2} Z_r^0|^p\ud r\bigg]
\leq C_p \Big(\mathbb W_1(\hat \mu,\mu)^{2p}+|\hat x_0-x_0|^{2p} \Big).
\end{split}
\end{equation}
In particular,
by choosing $p=1$, 
we get (with $C:=C_1$) 
\begin{equation*}
\begin{split}
&\bigg\|{\mathcal U} (t,\hat x_0,\cdot,\hat \mu)-{\mathcal U} (t,x_0,\cdot, \mu)-K^{x_0}(t,x_0,\cdot,\mu)\cdot (\hat x_0-x_0)-\int_{\mathbb T^d}K_0^\mu(t,x_0,\cdot,\mu,y)(\hat \mu(\ud y)-\mu(\ud y))\bigg\|_{{\mathbb s}} \\
&\leq C\Big(\mathbb W_1(\hat\mu,\mu)^2+|\hat x_0-x_0|^2\Big),
\end{split}
\end{equation*}
\begin{equation*}
\begin{split}
&\bigg| {\mathcal U}^0(t,\hat x_0,\hat \mu)- {\mathcal U}^0(t, x_0,\mu)-K^{0,x_0}(t,x_0,\mu)\cdot (\hat x_0-x_0)-\int_{\mathbb T^d}K_0^{0,\mu}(t,x_0,\mu,y)(\hat \mu-\mu)(\ud y)\bigg|\\
&\leq C\Big(\mathbb W_1(\hat\mu,\mu)^2+|\hat x_0-x_0|^2\Big),
\end{split}
\end{equation*}
and, thus, for any $(t,x)\in [0,T] \times\mathbb T^d$, the mappings
\begin{equation*}
(x_0,\mu) \in {\mathbb R}^d \times \mathcal{P}(\mathbb T^d)\mapsto {\mathcal U}(t,x_0,x,\mu)\quad 
{\rm and} \quad 
(x_0,\mu) \in {\mathbb R}^d \times \mathcal{P}(\mathbb T^d)\mapsto {\mathcal U}^0(t,x_0,\mu)
\end{equation*}
are differentiable with respect to $x_0$ and $\mu$ and the derivatives read, for any $(x_0,\mu) \in {\mathbb R}^d \times \mathcal{P}(\mathbb T^d)$,
\begin{equation*}
\nabla_{x_0} {\mathcal U}(t,x_0,x,\mu)=K^{x_0}(t,x_0,x,\mu),\quad\nabla_{x_0}
{\mathcal U}^0(t,x_0,\mu)=K^{0,x_0}(t,x_0,\mu),
\end{equation*}
and
\begin{equation*}
\delta_{\mu} {\mathcal U}(t,x_0,x,\mu,y)=K_0^\mu (t,x_0,x,\mu,y),\quad 
\delta_{\mu} {\mathcal U}^0(t,x_0,\mu,y)=K_0^{0,\mu}(t,x_0,\mu,y),\quad \text{for any }y\in\mathbb T^d.
\end{equation*}
\end{proposition}
\begin{proof}
We first note that $(\triangle^{\hspace{-1pt} 2} {\boldsymbol X}^0,\triangle^{\hspace{-1pt} 2} {\boldsymbol Y}^0,\triangle^{\hspace{-1pt} 2} {\boldsymbol Z}^0)$ solves, on 
$[t,T]$, 
\begin{equation*}
\begin{split}
&\ud_r \triangle^{\hspace{-1pt} 2} X_r^0=\big(-\nabla_{px_0}H^0(X_r^0,Z_r^0)\triangle^{\hspace{-1pt} 2} X_r^0-\nabla_{pp}H^0(X_r^0,Z_r^0)\triangle^{\hspace{-1pt} 2} Z_r^0+a_r\big) \ud r, 
\\
&\ud_r\triangle^{\hspace{-1pt} 2} Y_r^0=\bigg[-\nabla_{x_0}f^0_r(X_r^0,\mu_r)\cdot \triangle^{\hspace{-1pt} 2} X_r^0-\big(\delta_\mu f_r^0(X_r^0,\mu_r,\cdot),\triangle^{\hspace{-1pt} 2} \mu_r\big)
\\
& \hspace{15pt}+\nabla_{x_0}\hat L^0(X_r^0,Z_r^0)\triangle^{\hspace{-1pt} 2} X_r^0+\nabla_p\hat L^0(X_r^0,Z_r^0)\triangle^{\hspace{-1pt} 2} Z_r^0+b_r\bigg]\ud r+\sigma_0 \triangle^{\hspace{-1pt} 2} Z_r^0\cdot \ud B_r^0,\quad\text{on $[t,T]$,}\\
&\triangle^{\hspace{-1pt} 2}X_t^0=0,\quad \triangle^{\hspace{-1pt} 2}Y_T^0=\nabla_{x_0}g^0(X_T^0,\mu_T)\cdot\triangle^{\hspace{-1pt} 2}X_T^0+\big(\delta_\mu g^0(X_T^0,\mu_T,\cdot),\triangle^{\hspace{-1pt} 2} \mu_T\big)+c_T,
\end{split}
\end{equation*}
where
\begin{equation*}
\begin{split}
a_r=&\bigg(\nabla_{px_0}^2H^0(X_r^0,Z_r^0)-\int_0^1\nabla_{px_0}^2H^0(\theta \hat X_r^0+(1-\theta)X_r^0,\theta \hat Z_r^0+(1-\theta)Z_r^0)d\theta\bigg) (\hat X_r^0-X_r^0)\\
&+\bigg(\nabla_{pp}^2H^0(X_r^0,Z_r^0)-\int_0^1\nabla_{pp}^2H^0(\theta \hat X_r^0+(1-\theta)X_r^0,\theta \hat Z_r^0+(1-\theta)Z_r^0)d\theta\bigg)(\hat Z_r^0-Z_r^0),\\
b_r=&\bigg(\nabla_{x_0}f^0(X_r^0,\mu_r)-\int_0^1\nabla_{x_0}f^0(\theta \hat X_r^0+(1-\theta)X_r^0,\theta \hat \mu_r+(1-\theta)\mu_r)\ud \theta\bigg)\cdot ( \hat X_r^0-X_r^0)
\\
&+\bigg(\delta_\mu f^0(X_r^0,\mu_r,\cdot)-\int_0^1\delta_\mu f^0(\theta \hat X_r^0+(1-\theta)X_r^0,\theta\hat \mu_r+(1-\theta)\mu_r,\cdot)\ud \theta,\hat \mu_r-\mu_r\bigg)
\\
&-\bigg(\nabla_{x_0}\hat L^0(X_r^0,Z_r^0)-\int_0^1\nabla_{x_0}\hat L^0(\theta \hat X_r^0+(1-\theta)X_r^0,\theta \hat Z_r^0+(1-\theta)Z_r^0)d\theta\bigg) 
\cdot (\hat X_r^0-X_r^0)
\\
&-\bigg(\nabla_{p}\hat L^0(X_r^0,Z_r^0)-\int_0^1\nabla_{p}\hat L^0(\theta \hat X_r^0+(1-\theta)X_r^0,\theta \hat Z_r^0+(1-\theta)Z_r^0)d\theta\bigg)
\cdot
(\hat Z_r^0-Z_r^0)
\end{split}
\end{equation*}
and
\begin{equation*}
\begin{split}
c_T=&-\bigg(\nabla_{x_0}g^0(X_T^0,\mu_T)-\int_0^1\nabla_{x_0}g^0(\theta \hat X_T^0+(1-\theta)X_T^0,\theta\hat\mu_T+(1-\theta)\mu_T)\ud\theta\bigg)\cdot ( \hat X_T^0-X_T^0)\\
&-\bigg(\delta_\mu g^0(X_T^0,\mu_T,\cdot)-\int_0^1\delta_\mu g^0(\theta X_T^0+(1-\theta)X_T^0,\theta\hat\mu_T+(1-\theta)\mu_T,\cdot),\hat \mu_T-\mu_T\bigg).
\end{split}
\end{equation*}
We then notice that $(\triangle^{\hspace{-1pt} 2} {\boldsymbol \mu},\triangle^{\hspace{-1pt} 2} {\boldsymbol u},\triangle^{\hspace{-1pt} 2} {\boldsymbol m})$ satisfies, on $[t,T]$, 
\begin{equation*}
\begin{split}
&  \partial_r \triangle^{\hspace{-1pt} 2} \mu_r -\tfrac{1}{2}\Delta_{x}\triangle^{\hspace{-1pt} 2} \mu_r -{\rm div}_x\big(\nabla_{p}\hat H(x,\nabla_xu_r(x))\triangle^{\hspace{-1pt} 2} \mu_r +\mu_r \nabla_{pp}^2 \hat H(x,\nabla_xu_r(x))\nabla_x\delta^2 u_r(x)\big)
\\
&\hspace{15pt}-{\rm div}_x\big(d_r(x)\big)=0,\quad\text{on } \mathbb T^d,
\\
& \ud_r\delta^2 u_r(x)=\Big[-\tfrac{1}{2}\Delta_{x}\delta^2 u_r(x)+\nabla_p \hat H(x,\nabla_x u_r(x))\cdot\nabla_x \delta^2 u_r(x)-\nabla_{x_0}f_r(X_r^0,x,\mu_r)\cdot\delta^2 X_r^0
\\
& \hspace{15pt}-\Big(\delta_\mu f_r(X_r^0,x,\mu_r,\cdot),\delta^2 \mu_r\Big)+j_r(x)\Big]\ud r + \ud_r \delta^2 m_r(x),
\quad x \in \mathbb T^d,
\\
&\delta \mu_t=0,\quad \delta u_T(x)=\nabla_{x_0}g(X_T^0,x,\mu_T)\cdot\delta^2 X_T^0+\Big(\delta_\mu g(X_T^0,x,\mu_T,\cdot),\delta^2 \mu_T\Big)+k_T, \quad x \in 
 \mathbb T^d,
\end{split}
\end{equation*}
where
\begin{equation*}\begin{split}
d_r(x)=&\big(\hat \mu_r -\mu_r \big)\bigg(\int_0^1\nabla_{pp}^2 \hat H(x,\theta \nabla_x \hat u_r(x)+(1-\theta)\nabla_x u_r(x))\ud \theta\bigg) \big(\nabla_x \hat u_r(x)-\nabla_x u_r(x)\big)\\
&+\mu_r \int_0^1\int_0^1\theta\nabla_{ppp}^3\hat H\big(x,\nabla_x u_r(x)+\tilde\theta\theta(\nabla_x\hat u_r(x)-\nabla_xu_r(x))\big)\big(\nabla_x\hat u_r(x)-\nabla_x u_r(x)\big)^{\otimes 2}\ud \tilde\theta \ud \theta,
\\
j_r(x)=&-\bigg(\nabla_p \hat H(x,\nabla_xu_r(x))-\int_0^1\nabla_p \hat H\big(x,\theta\nabla_x\hat u_r(x)+(1-\theta)\nabla_xu_r(x)\big)\ud\theta\bigg)\cdot\big(\nabla_x \hat u_r(x)-\nabla_x u_r(x)\big)\\
&+\bigg(\nabla_{x_0}f_r(X_r^0,x,\mu_r)-\int_{0}^1\nabla_{x_0}f_r\big(\theta \hat X_r^0+(1-\theta)X_r^0,x,\theta \hat\mu_r+(1-\theta)\mu_r\big)\ud\theta\bigg)\cdot (\hat X_r^0-X_r^0)\\
&+\bigg(\delta_\mu f_r(X_r^0,x,\mu_r,\cdot)-\int_0^1\delta_\mu f_r\big(\theta \hat X_r^0+(1-\theta)X_r^0,x,\theta\hat \mu_r+(1-\theta)\mu_r,\cdot \big)
\ud \theta,\hat \mu_r-\mu_r\bigg),
\end{split}
\end{equation*}
and
\begin{equation*}
\begin{split}
k_T(x)=&-\bigg(\nabla_{x_0}g(X_T^0,x,\mu_T)-\int_0^1\nabla_{x_0}g\big(\theta \hat X_T^0+(1-\theta)X_T^0,x,\theta \hat \mu_T+(1-\theta)\mu_T\big)\ud \theta\bigg)\cdot (\hat X_T^0-X_T^0)\\
&-\bigg(\delta_\mu g(X_T^0,x,\mu_T)-\int_0^1\delta_\mu g(\theta X_T^0+(1-\theta)X_T^0,x,\theta\hat\mu_T+(1-\theta)\mu_T,\cdot),\hat \mu_T-\mu_T\bigg).
\end{split}
\end{equation*}
Next, we recall 
${\boldsymbol \mu}$
and 
$\hat{\boldsymbol \mu}$
take values in the space $\mathcal{P}(\mathbb T^d)$. 
We 
deduce that 
\begin{equation*}
\begin{split}
&\|d_r\|_{-1}\leq C\Bigl(\|\hat u_r-u_r \|_2\mathbb W_1(\hat \mu_t,\mu_t)+\|\hat u_r-u_r \|_1^2\Bigr),
\quad r \in [t,T]
\\
&\|k_T\|_{{\mathbb s}}\leq C\Bigl( |\hat X_T^0-X_T^0|^2+
\mathbb W_1(\hat \mu_T,\mu_T)^2\Bigr),
\\
&\|j_r\|_{{\mathbb s}-1}\leq C\Bigl(\|\hat u_r-u_r\|_{\mathbb s}^2+|\hat X_r^0-X_r^0|^2+\mathbb W_1(\hat \mu_r,\mu_r)^2\Bigr),
\\
&\int_t^T|a_r|dr\leq C\bigg[\sup_{r\in [t,T]}|\hat X_r^0-X_r^0|^2
+\int_t^T|\hat Z_r^0-Z_r^0|^2 \ud r\bigg]
\\
&\int_t^T|b_r|dr\leq C\bigg[\sup_{r\in [t,T]}|\hat X_r^0-X_r^0|^2
+\int_t^T \mathbb W_1(\hat \mu_r,\mu_r)^2 \ud r
+\int_t^T|\hat Z_r^0-Z_r^0|^2 \ud r\bigg],
\\
&|c_T|\leq C\Big(|\hat X_T^0-X_T^0|^2
+\mathbb W_1^2(\hat \mu_T,\mu_T)\Big).
\end{split}
\end{equation*}
We further apply \eqref{eq:local:diff-1} 
to obtain,
for any 
$p \geq 1$, 
\begin{equation*}
\begin{split}
&\mathbb E\Big[\sup_{r\in [t,T]}\|d_r\|_{-{\mathbb r}+2}^{p} 
\Big]\leq C_p\big(\mathbb W_1^{2p}(\hat \mu,\mu)+|\hat x_0-x_0|^{2p}\big),
\\
&\mathbb E\Big[\|k_T\|_{\mathbb s}^p \Big]
\leq 
C\big(\mathbb W_1^{2p}(\hat \mu,\mu)+|\hat x_0-x_0|^{2p}\big),
\\
&\mathbb E\Big[\sup_{r\in [t,T]}\|j_r\|_{{\mathbb s}-1}^p\Big]\leq C\big(\mathbb W_1^{2p}(\hat \mu,\mu)+|\hat x_0-x_0|^{2p}\big),
\\
&\mathbb E\bigg[\bigg(\int_t^T|a_r|dr\bigg)^p+\bigg(\int_t^T|b_r|dr\bigg)^p+c_T^2\bigg]\leq C\big(\mathbb W_1^{2p}(\hat \mu,\mu)+|\hat x_0-x_0|^{2p}\big).
\end{split}
\end{equation*}
Therefore, by Theorem \ref{thm:local:linearFB}, we have \eqref{eq:local:delta2}.
\end{proof}

\begin{corollary}
\label{corol:U:U0:C:C0}
There exist a constant ${\mathfrak C}>0$, only depending on the parameters 
$d$, $\kappa$,  ${\mathfrak L}$,  
$\zeta$, $\sigma_0$ and 
$\mathbb{s}$
in Assumption \hyp{\^A},  
and a constant $C$, only depending on the parameters 
in Assumption \hyp{\^A},  
such that, for $T \leq {\mathfrak C}$, 
for all $t \in [0,T]$, 
${\mathcal U}^0(t,\cdot,\cdot)$ satisfies \emph{(ii)} and \emph{(iii)} in the definition of 
${\mathscr C}^0(C,\lfloor \mathbb{s} \rfloor -1)$
and ${\mathcal U}(t,\cdot,\cdot,\cdot)$ satisfies \emph{(ii)}, \emph{(iii)} and 
\emph{(iv)} in the definition of 
${\mathscr C}(C,\lfloor \mathbb{s} \rfloor -1,\mathbb{s})$. 
\end{corollary} 

\begin{proof}
The bounds for $\vert {\mathcal U}^0(t,x_0,\mu)\vert$
and
$\|{\mathcal U}(t,x_0,\cdot,\mu) \|_{\mathbb s}$ follow from 
 \eqref{eq:local:BMO}
 and 
 \eqref{eq:local:reg}. 
 The bounds for $\vert \nabla_{x_0} {\mathcal U}^0(t,x_0,\mu)\vert$
and
$\|\nabla_{x_0} {\mathcal U}(t,x_0,\cdot,\mu) \|_{\mathbb s}$
follow from the inequality
\eqref{eq:Kbdd}. 
Moreover, 
the bounds for 
$\| \delta_\mu {\mathcal U}^0(t,x_0,\mu,\cdot) \|_{\lfloor {\mathbb s} \rfloor - 1}$
and
$\max_{l=0,\cdots,\lfloor {\mathbb s} \rfloor -1} 
\| \nabla_y^l \delta_\mu {\mathcal U}^0(t,x_0,\cdot,\mu,\cdot) \|_{{\mathbb s}}$
follow from 
Lemma \ref{lem:deltau:reginy}. Continuity of the derivatives as stated in 
the definition of ${\mathscr C}^0$ and ${\mathscr C}$ also follows from 
Lemma \ref{lem:deltau:reginy}. 
\end{proof} 

In the following statement, we address the spatial regularity of the 
derivatives of ${\mathcal U}$ and 
${\mathcal U}^0$: 

\begin{proposition}\label{prop:deltaU:continuity}
Given two 4-tuples $(t,x_0,\mu,y), (t,\hat x_0,\hat\mu,\hat y)\in [0,T]\times\mathbb R^d\times\mathcal{P}(\mathbb T^d) \times {\mathbb T}^d$, a 
unitary vector $e \in {\mathbb R}^d$, two scalars $\vartheta_0,\vartheta \in [0,1]$, a 
fixed $\mathbb{r} \in (\lfloor {\mathbb s} \rfloor,{\mathbb s})$
and 
a $d$-tuple $l\in \{0,\cdots,\lfloor {\mathbb s} \rfloor-1\}^d$ with $|l|:=\sum_{i=1}^dl_i\leq \lfloor {\mathbb s} \rfloor-1$, 
call 
$({\boldsymbol \mu},{\boldsymbol u},{\boldsymbol m},{\boldsymbol X}^0,{\boldsymbol Y}^0,{\boldsymbol Z}^0)$ 
and $(\hat {\boldsymbol \mu},\hat {\boldsymbol u},\hat {\boldsymbol m},\hat {\boldsymbol X}^0,\hat {\boldsymbol Y}^0,\hat 
{\boldsymbol Z}^0)$ the solutions of the system \eqref{eq:major:FB:1:small}--\eqref{eq:minor:FB:2:small} with $(t,x_0,\mu)$ and $(t,\hat x_0,\hat \mu)$ as 
respective initial conditions and then $(\delta {\boldsymbol \mu},\delta 
{\boldsymbol u},\delta {\boldsymbol m},\delta {\boldsymbol X}^0,\delta 
{\boldsymbol Y}^0,\delta {\boldsymbol Z}^0)$ and $(\delta \hat {\boldsymbol \mu},\delta\hat {\boldsymbol u},\delta\hat {\boldsymbol m},\delta \hat
{\boldsymbol X}^0,\delta \hat {\boldsymbol Y}^0,\delta \hat {\boldsymbol Z}^0)$ the solutions of the 
linearized system \eqref{eq:linear H:local} with $(t, \vartheta_0 e, \vartheta (-1)^{|l|}\nabla^l\delta_y)$ and $(t,\vartheta_0 e, \vartheta (-1)^{|l|}\nabla^l\delta_{\hat y})$ as respective initial conditions. 

Then,
for $T \leq {\mathfrak C}$,
for some ${\mathfrak C}>0$ only depending on the parameters 
$d$, $\kappa$,  ${\mathfrak L}$,  
$\zeta$, $\sigma_0$ and 
$\mathbb{s}$ in Assumption \hyp{\^A} and on 
${\mathbb r}$,  we can find a constant $C$, only depending
 on the parameters in Assumption \hyp{\^A} and on ${\mathbb r}$, 
   such that, for any 
   $p \in [1,4]$, 
\begin{equation}\label{eq:local:diff-3}
\begin{split}
&\mathbb E^0\biggl[\sup_{r\in [t,T]}
\Bigl( 
\|\delta \hat \mu_r-\delta \mu_r\|_{-\mathbb{r}+1}^{2p}
+
\|\delta \hat u_r-\delta u_r\|_{{\mathbb s}}^{2p}
+
\|\delta \hat m_r-\delta m_r\|_{{\mathbb s}-2}^{2p}
\Bigr)
\biggr]
\\
&+\mathbb E^0\bigg[\sup_{r\in [t,T]}
\Bigl( |\delta \hat X_r^0-\delta X_r^0|^{2p}+ |\delta \hat Y_t^0-\delta Y_t^0|^{2p} \Bigr) +
\biggl( \int_t^T|\delta \hat Z_r^0-\delta Z_r^0|^2\ud r\biggr)^p \bigg]\\
&\leq C \Big(|\hat x_0-x_0|^{2p}+\mathbb W_1^{2p}(\hat \mu,\mu)+|\hat y-y|^{2p(\mathbb{r} -\lfloor {\mathbb s} \rfloor)}\Big).
\end{split}
\end{equation}
In particular, 
 for any $y,\hat y\in\mathbb T^d$,
 for any $k \in \{0,\cdots,\lfloor \mathbb{s} \rfloor-1\}$, 
\begin{align}
&\big\|\nabla_{x_0} {\mathcal U}(t,x_0,\cdot,\mu)- \nabla_{x_0} {\mathcal U}(t,\hat x_0,\cdot,\hat \mu)\big\|_{{\mathbb s}}+
\big|\nabla_{x_0} {\mathcal U}^0(t,x_0,\mu)-\nabla_{x_0}  {\mathcal U}^0(t,\hat x_0,\hat \mu)\big| \leq C \Big(|\hat x_0-x_0|+\mathbb W_1(\hat \mu,\mu) \Big),
\nonumber
\\
&\big\|\nabla_y^k \delta_\mu {\mathcal U}(t,x_0,\cdot,\mu,y)-\nabla_y^k \delta_\mu {\mathcal U}(t,\hat x_0,\cdot,\hat \mu,\hat y)\big\|_{{\mathbb s}}+\big|\nabla_y^k \delta_\mu {\mathcal U}^0(t,x_0,\mu,y)-\nabla_y^k \delta_\mu {\mathcal U}^0(t,\hat x_0,\hat \mu,\hat y)\big| \label{eq:local:diff-3bb}
\\
&\leq C \Big(|\hat x_0-x_0|+\mathbb W_1(\hat \mu,\mu)+|\hat y-y|^{\mathbb{r}-\lfloor {\mathbb s} \rfloor}\Big),
\nonumber
\end{align}
and, for any $t \in [0,T]$, ${\mathcal U}^0(t,\cdot,\cdot) \in {\mathscr D}^0(C, {\mathbb r}  - 1)$
and ${\mathcal U}(t,\cdot,\cdot,\cdot) \in {\mathscr D}(C,  {\mathbb r}   - 1,{\mathbb s})$. 
\end{proposition}
\begin{proof}
Applying Theorem \ref{thm:local:FB}, we obtain
(for $T \leq {\mathfrak C}$)
\begin{equation}\label{eq:local:diff-2}
\begin{split}
&\mathbb E^0\bigg[\sup_{r\in[t,T]}\Big( \|\hat u_r-u_r\|_s^{4p} + \mathbb W_1^{4p}(\hat\mu_r,\mu_r) + |\hat X_r^{0}-X_r^{0}|^{4p}+|\hat Y_r^{0}-Y_r^{0}|^{4p}
\Bigr) + 
\biggl( 
\int_t^T|\hat Z^{0}_r-Z_r^{0}|^2\ud r
\biggr)^{2p} \bigg]\\
&\leq C \Big[|\hat x_0-x_0|^{4p} +\mathbb W_1^{4p} (\hat \mu,\mu)\Big].
\end{split}
\end{equation}
By Theorem \ref{thm:local:linearHFB}, 
\begin{equation}
\label{eq:local:diff-2bb} 
\begin{split}
&\mathbb E^0\biggl[\sup_{r\in [t,T]} \Bigl( 
\|\delta \mu_r\|_{-\mathbb{r}+1}^{4p}
+\|\delta \hat{\mu}_r\|_{-\mathbb{r}+1}^{4p}
+
\|\delta u_r\|_{\mathbb s}^{4p}
+
\|\delta \hat{u}_r\|_{\mathbb s}^{4p}
+  |\delta X_r^0|^{4p}+ |\delta \hat{X}_r^0|^{4p} 
+
 |\delta Y_r^0|^{4p}+ |\delta \hat{Y}_r^0|^{4p} 
\Bigr) 
\\
&\hspace{15pt} +
 \biggl( \int_0^T \bigl( |\delta Z_t^0|^2
 +
 |\delta \hat{Z}_t^0|^2
 \bigr)
 \ud t \biggr)^{2p} \biggr]
 \\
&\leq C. 
\end{split}
\end{equation}
It remains to see that, since $|l|\leq\lfloor {\mathbb s} \rfloor-1$, we have
\begin{equation*}
\|\nabla^l\delta_{\hat y}-\nabla^l\delta_y\|_{-\mathbb{r}+1}\leq |\hat y-y|^{\mathbb{r}-\lfloor {\mathbb s} \rfloor}.
\end{equation*}
In order to conclude, we follow the proof of Proposition \ref{prop:local:1stdiff}. 
The idea is to write 
 $(\delta \hat {\boldsymbol \mu}- 
 \delta {\boldsymbol \mu},
 \delta\hat {\boldsymbol u}
 -\delta {\boldsymbol u},
 \delta\hat {\boldsymbol m}
 -
 \delta {\boldsymbol m}, 
 \delta \hat
{\boldsymbol X}^0
-
 \delta {\boldsymbol X}^0,
 \delta \hat {\boldsymbol Y}^0-\delta 
{\boldsymbol Y}^0,\delta \hat {\boldsymbol Z}^0-\delta {\boldsymbol Z}^0)$ 
as a solution of a linear system of the form 
\eqref{eq:linear:local}, with suitable coefficients
$(a_t)_{0 \le t \le T}$, 
$(b_t)_{0 \le t \le T}$,
$c_T$, $(d_t)_{0 \le t \le T}$, 
$(j_t)_{0 \le t \le T}$ 
and $k_T$. 
Briefly, all these terms can be bounded from above by 
products of two terms: one of the norms appearing in 
\eqref{eq:local:diff-2}
and one of the norms 
of 
\eqref{eq:local:diff-2bb}. Cauchy-Schwarz inequality permits to handle 
the $2p$ moment of any of these products. By \eqref{eq:local:diff-2} and 
\eqref{eq:local:diff-2bb}, we obtain \eqref{eq:local:diff-3}.

Then, we get the first line in 
\eqref{eq:local:diff-3bb}
by choosing $p=1$, $\vartheta_0=1$, $\vartheta=0$ and $y=\hat{y}$. The second inequality is
obtained by choosing 
$p=1$, $\vartheta_0=0$ and $\vartheta=1$. 
\end{proof}

\subsection{Regularity in time of the derivatives}

We now address the time continuity of the derivatives. We start with the following statement: 

\begin{proposition}
\label{prop:se:5:continuity:time}
Let ${\mathbb r} \in (\lfloor {\mathbb s} \rfloor,{\mathbb s})$
and let Assumption \hyp{\^A} be in force. 
Then, there exists a constant ${\mathfrak C}>0$, only depending on $d$,   $\kappa$, ${\mathfrak L}$,
$\zeta$, $\sigma_0$ and 
$\mathbb{s}$
in Assumption \hyp{\^A} and on ${\mathbb r}$, such that, for 
$T \leq {\mathfrak C}$, for any $(t,x_0,\mu)\in [0,T]\times\mathbb R^d\times\mathcal{P}(\mathbb T^d)$, it holds
\begin{equation}\label{eq:U:contintime}
\begin{split}
&\lim_{h\to 0}\Big(\big\|\nabla_{x_0} {\mathcal U}(t+h,x_0,\cdot,\mu)-\nabla_{x_0}{\mathcal U}(t,x_0,\cdot,\mu)\big\|_{\mathbb{r}}
 + 
\big\| \delta_\mu {\mathcal U}(t+h,x_0,\cdot,\mu,\cdot)- \delta_\mu {\mathcal U}(t,x_0,\cdot,\mu,\cdot)\big\|_{\mathbb{r},\mathbb{r}-1}
\Bigr) =0,
\end{split}
\end{equation}
where 
$\| \cdot \|_{{\mathbb r},{\mathbb r}-1 }$ denotes the standard H\"older norm\footnote{\label{foo:2:26}
Let $r,s>0$ with $r,s\not\in\mathbb N$. If a function $h : (x,y) \in {\mathbb T}^d \times {\mathbb T}^d \mapsto h(x,y)$ has 
mixed derivatives up to the order $\lfloor r \rfloor$ in $x$ and $\lfloor s \rfloor$ in $y$, we set
\begin{equation*}
\|g\|_{\lfloor r \rfloor,\lfloor s \rfloor}:=\sup_{k=0,\cdots, \lfloor r \rfloor}\sup_{l=0,\cdots \lfloor s \rfloor}\sup_{x,y\in\mathbb T^d}|\nabla_{x}^k\nabla_{y}^lg(x,y)|,
\end{equation*}
and, if  $\nabla_{x}^{\lfloor r \rfloor}\nabla_y^{\lfloor s \rfloor} h$ is $(r-\lfloor r \rfloor,s-\lfloor s \rfloor)$-H\"older continuous, we also set
\begin{equation*}
\|g\|_{(r,s)}:=\|g\|_{\lfloor r \rfloor,\lfloor s \rfloor}+\sup_{(x,y),(x',y') \in {\mathbb T}^d\times {\mathbb T}^d : (x,y) \not = (x',y')}\frac{|\nabla_x^{\lfloor r \rfloor}\nabla_y^{\lfloor s \rfloor}g(x,y)-\nabla_x^{\lfloor r \rfloor}\nabla_y^{\lfloor s \rfloor}g(x',y')|}{|x-x'|^{r-\lfloor r \rfloor}+|y-y'|^{s-\lfloor s \rfloor}}.
\end{equation*}} 
on 
the space ${\mathcal C}^{{\mathbb r},{\mathbb r}-1}({\mathbb T}^d \times {\mathbb T}^d)$ of functions $h : (x,y) \in {\mathbb T}^d \times {\mathbb T}^d
\mapsto h(x,y)$ that have crossed derivatives in $(x,y)$ up the order $\lfloor {\mathbb r} \rfloor $ in $x$ and $\lfloor {\mathbb r} \rfloor -1$ in $y$, 
all the derivatives being 
${\mathbb r}-\lfloor {\mathbb s}\rfloor$
H\"older continuous in $(x,y)$. Similarly, 
\begin{equation}\label{eq:U0:contintime}
\begin{split}
&\lim_{h\to 0}\Big(\big|\nabla_{x_0} {\mathcal U}^0(t+h,x_0,\mu)-\nabla_{x_0}
{\mathcal U}^0(t,x_0,\mu) \bigr\vert +
\big\| \delta_\mu {\mathcal U}^0(t+h,x_0,\mu,\cdot)- \delta_\mu {\mathcal U}^0(t,x_0,\mu,\cdot)\big\|_{\mathbb{r}-1  }\Big)=0.
\end{split}
\end{equation}
\end{proposition}
\begin{proof}
Given $(x_0,\mu),(\hat x_0,\hat \mu)\in\mathbb R^d\times\mathcal{P}(\mathbb T^d)$, we can derive from Proposition \ref{prop:local:1stdiff} that, for
$T \leq {\mathfrak C}$ and for any $t\in [0,T]$,
\begin{equation}\label{eq:U:taylor1st}
\begin{split}
&{\mathcal U}(t,\hat x_0,\cdot,\hat \mu)-{\mathcal U}(t,x_0,\cdot, \mu)
\\
&=\nabla_{x_0} {\mathcal U}(t,x_0,\cdot,\mu)\cdot (\hat x_0-x_0)+\int_{\mathbb T^d}\delta_\mu {\mathcal U}(t,x_0,\cdot,\mu,y)\big(\hat \mu(\ud y)-\mu(\ud y)\big)+O\big(|\hat x_0-x_0|^2+\mathbb W_1^2(\hat \mu,\mu)\big),
\end{split}
\end{equation}
where the equality holds true in $\mathcal{C}^{\mathbb s}(\mathbb T^d)$ and the Landau notation $O(\cdot)$ is uniform in $(t,x_0,\mu)$, and
\begin{equation}\label{eq:U0:taylor1st}
\begin{split}
&{\mathcal U}^0(t,\hat x_0,\hat \mu)-{\mathcal U}^0(t,x_0,\mu)
\\
&=\nabla_{x_0} {\mathcal U}^0(t,x_0,\mu)\cdot (\hat x_0-x_0)+\int_{\mathbb T^d}\delta_\mu {\mathcal U}^0(t,x_0,\mu,y)\big(\hat \mu(\ud y)-\mu(\ud y)\big)+O\big(|\hat x_0-x_0|^2+\mathbb W_1^2(\hat \mu,\mu)\big).
\end{split}
\end{equation}

We now fix $(x_0,\mu) \in {\mathbb R}^d \times {\mathcal P}({\mathbb T}^d)$
and let $\mathbb{u} := ({\mathbb r} + {\mathbb s})/2$. 
Recall \eqref{eq:Kbdd} and the fact that $\nabla_{x_0}
{\mathcal U}^0(t,x_0,\mu)=K^{0,x_0}(t,x_0,\mu)$. Therefore,   the 
collection  $\{ \nabla_{x_0} {\mathcal U}^0(t,x_0,\mu), \ t\in [0,T]\}$ 
is bounded and thus relatively compact in $\mathbb R^d$. By Proposition \ref{prop:deltaU:continuity}
(applied with the indices $({\mathbb u},{\mathbb s})$ in lieu of $({\mathbb r},{\mathbb s})$)
and under the same generic condition $T \leq {\mathfrak C}$ as therein,  
the collection of functions
$\{ x\in\mathbb T^d\mapsto \nabla_{x_0} {\mathcal U}(t,x_0,x,\mu),  t\in [0,T]\}$
is relatively compact in $\mathcal{C}^{\mathbb{r}}(\mathbb T^d)$,
the collection of functions $\{ 
y\in\mathbb T^d\mapsto \delta_\mu {\mathcal U}^0(t,x_0,\mu,y), t\in [0,T]\}$
is relatively compact in $\mathcal{C}^{\mathbb{r}-1}(\mathbb T^d)$, and
 the collection of functions
$\{(x,y)\in \mathbb T^d \times \mathbb T^d \mapsto  \delta_{\mu}{\mathcal U}(t,x_0,x,\mu,y), t\in [0,T]\}$ 
is relatively compact in $\mathcal{C}^{\mathbb{r},{\mathbb r}-1}(\mathbb T^d \times \mathbb T^d)$. Any limits $\Phi^{0,x_0}\in\mathbb R^d$, $\Phi^{x_0},\Phi^{0,\mu}:\mathbb T^{d}\mapsto \mathbb R$ and $\Phi^{\mu}:\mathbb \mathbb T^d \times \mathbb T^d \mapsto \mathbb R$ obtained by letting $t$ tend to some $t_0\in [0,T]$ in \eqref{eq:U:taylor1st} and \eqref{eq:U0:taylor1st} must satisfy
\begin{equation*}
\begin{split}
&{\mathcal U}(t_0,\hat x_0,\cdot,\hat \mu)-{\mathcal U}(t_0,x_0,\cdot, \mu)
\\
&=\Phi^{x_0}(\cdot)\cdot (\hat x_0-x_0)+\int_{\mathbb T^d}\Phi^{\mu}(\cdot,y)\big(\hat \mu(\ud y)-\mu(\ud y)\big)+O\big(|\hat x_0-x_0|^2+\mathbb W_1^2(\hat \mu,\mu)\big),
\end{split}
\end{equation*}
where the equality holds true in $\mathcal{C}^{\mathbb{r}}(\mathbb T^d)$, and
\begin{equation*} 
\begin{split}
&{\mathcal U}^0(t_0,\hat x_0,\hat \mu)-{\mathcal U}^0(t_0,x_0, \mu)
\\
&=\Phi^{0,x_0}\cdot (\hat x_0-x_0)+\int_{\mathbb T^d}\Phi^{0,\mu}(y)\big(\hat \mu(\ud y)-\mu(\ud y)\big)+O\big(|\hat x_0-x_0|^2+\mathbb W_1^2(\hat \mu,\mu)\big).
\end{split}
\end{equation*}
The last two displays show that
\begin{equation*}
\begin{split}
&\nabla_{x_0}{\mathcal U}(t_0,x_0,x,\mu)\cdot (\hat x_0-x_0)+\int_{\mathbb T^d}\delta_\mu {\mathcal U}(t_0,x_0,x,\mu,y)\big(\hat \mu(\ud y)-\mu(\ud y)\big)
\\
&=\Phi^{x_0}(x)\cdot (\hat x_0-x_0)+\int_{\mathbb T^d}\Phi^{\mu}(x,y)\big(\hat \mu(\ud y)-\mu(\ud y)\big),
\quad x \in {\mathbb T}^d, 
\end{split}
\end{equation*}
and
\begin{equation*}
\begin{split}
&\nabla_{x_0} {\mathcal U}^0(t,x_0,\mu)\cdot (\hat x_0-x_0)+\int_{\mathbb T^d}\delta_\mu {\mathcal U}^0(t,x_0,\mu,y)\big(\hat \mu(\ud y)-\mu(\ud y)\big)\\
=&\Phi^{0,x_0}\cdot (\hat x_0-x_0)+\int_{\mathbb T^d}\Phi^{0,\mu}(y)\big(\hat \mu(\ud y)-\mu(\ud y)\big),
\end{split}
\end{equation*}
which implies that
\begin{equation*}
\nabla_{x_0} {\mathcal U}(t_0,x_0,x,\mu)=\Phi^{x_0}(x), \quad\delta_\mu {\mathcal U}(t_0,x_0,x,\mu,y)=\Phi^{\mu}(x,y),
\end{equation*}
and
\begin{equation*}
\nabla_{x_0} {\mathcal U}^0(t_0,x_0,\mu)=\Phi^{0,x_0},\quad \delta_\mu {\mathcal U}^0(t_0,x_0,\mu,y)=\Phi^{0,\mu}(y).
\end{equation*}
This proves \eqref{eq:U:contintime} and \eqref{eq:U0:contintime}. 
%
\end{proof}

We now refine Proposition 
\ref{prop:se:5:continuity:time}. (using the same notation as therein for the norm 
$\| \cdot \|_{{\mathbb r},{\mathbb r}-1}$, 
see footnote
\ref{foo:2:26}):
\begin{proposition}
\label{prop:reg:Holder:time:deribatives}
Let $\mathbb{r} \in (\lfloor {\mathbb s} \rfloor,{\mathbb s})$.
Under Assumption \hyp{\^A}, there 
exist a constant ${\mathfrak C}>0$, only depending on the parameters 
$d$, ${\mathfrak L}$,  $\kappa$, 
$\zeta$, $\sigma_0$ and 
$\mathbb{s}$
in Assumption \hyp{\^A} and on ${\mathbb r}$, and a constant 
$C>0$, only depending on 
the parameters in 
Assumption \hyp{\^A}
and on ${\mathbb r}$,
such that, for $T \leq {\mathfrak C}$, 
for all $(t,x_0,\mu)\in [0,T]\times\mathbb R^d\times\mathcal{P}(\mathbb T^d)$ 
and $h>0$ such that $t+h \leq T$, 
\begin{equation}
\label{eq:U:holder:contintime}
\begin{split}
&\big\|\nabla_{x_0} {\mathcal U}(t+h,x_0,\cdot,\mu)-\nabla_{x_0}{\mathcal U}(t,x_0,\cdot,\mu)\big\|_{\mathbb{r}}
+ \big\| \delta_\mu {\mathcal U}(t+h,x_0,\cdot,\mu,\cdot)- \delta_\mu {\mathcal U}(t,x_0,\cdot,\mu,\cdot)\big\|_{\mathbb{r},\mathbb{r}-1 }
\\
&\quad 
\leq C h^{1/2},
\end{split}
\end{equation}
and
\begin{equation}\label{eq:U0:holder:contintime}
\begin{split}
&\big|\nabla_{x_0} {\mathcal U}^0(t+h,x_0,\mu)-\nabla_{x_0}
{\mathcal U}^0(t,x_0,\mu)\big|
+ \big\| \delta_\mu {\mathcal U}^0(t+h,x_0,\mu,\cdot)- \delta_\mu {\mathcal U}^0(t,x_0,\mu,\cdot)\big\|_{\mathbb{r}-1}
  \leq C h^{1/2}.
\end{split}
\end{equation}
\end{proposition}

In fact, one could have directly established Proposition 
\ref{prop:reg:Holder:time:deribatives}, but 
we felt better to start with Proposition 
\ref{prop:se:5:continuity:time}, because we use repeatedly the two 
properties \eqref{eq:U:contintime} and \eqref{eq:U0:contintime}
in the rest of the analysis.

\begin{proof}
We start with the proof of the second inequality in \eqref{eq:U:holder:contintime}. To do so, we recall 
from Lemma 
\ref{lem:deltau:reginy} and
Proposition \ref{prop:local:1stdiff} that 
$\nabla_y^{l}\delta_\mu {\mathcal U}(t,x_0,x,\mu,y)
= K_l^{\mu}(t,x_0,x,\mu,y)$,
for $y \in {\mathbb T}^d$ and for 
$l \in \{0,\cdots,\lfloor {\mathbb s} \rfloor -1 \}^d$ with 
$\vert l \vert \leq \lfloor {\mathbb s} \rfloor -1$, 
with 
$K^{x_0}$
and 
$K^{\mu}_l$ being defined in 
\eqref{eq:all:the:Ks}. By the latter, we have
(with the same notation as therein and under the generic notation $T \leq {\mathfrak C}$)
\begin{equation*} 
K_l^{\mu}(t,x_0,x,\mu,y) - 
K_l^{\mu}(t+h,x_0,x,\mu,y)
= \delta u_t^{l,y}(x) - \nabla_y^l \delta_\mu {\mathcal U}(t+h,x_0,x,\mu,y),
\end{equation*}
and then, by recalling 
\eqref{eq:linear H:local}, 
\begin{equation}
\label{eq:U:holder:contintime:step:1}
\begin{split} 
&K_l^{\mu}(t,x_0,x,\mu,y) - 
K_l^{\mu}(t+h,x_0,x,\mu,y)
\\
&= {\mathbb E}^0 
\bigl[ \delta u_{t+h}^{l,y}(x) \bigr]
- \nabla_y^l \delta_\mu {\mathcal U}(t+h,x_0,x,\mu,y)
\\
&\hspace{15pt}+ 
{\mathbb E}^0 \int_t^{t+h} 
\Big[-\tfrac{1}{2}\Delta_{x}\delta u_r^{l,y}(x)+\nabla_p \hat H(x,\nabla_x u_r (x))\cdot\nabla_x \delta u_r^{l,y} (x)-\nabla_{x_0}f_r(X_r^0,x,\mu_r  )\cdot\delta X_r^{0,l,y}
\\
& \hspace{45pt}-\Big(\delta_\mu f_r(X_r^0,x,\mu_r),\delta \mu_r^{l,y} \Big)\Big]\ud r. 
\end{split}
\end{equation} 
Again, it is worth stressing that the pair of triples 
$( \delta {\boldsymbol X}^{0,l,y}, \delta {\boldsymbol Y}^{0,l,y},
\delta {\boldsymbol Z}^{0,l,y})$ and 
$(\delta {\boldsymbol \mu}^{l,y},\delta {\boldsymbol u}^{l,y}, 
\delta {\boldsymbol m}^{l,y})$ solve
the linearized system 
\eqref{eq:linear H:local}
with $\triangle \mu=(-1)^{|l|} \nabla^l \delta_y\in \mathcal{C}^{1-\mathbb{r}}({\mathbb T}^d)$
(for the same  
$\mathbb{r}$ as in the statement) and $\triangle x^0=0$ as initial conditions at
time $t$, when ${\boldsymbol X}^0$ 
starts from $x_0$ at time $t$ and  
${\boldsymbol \mu}$ from $\mu$. 
Moreover, 
using \eqref{eq:local:linear H:diff0}, 
we can bound the 
$L^2(\Omega^0,{\mathcal F}^0,{\mathbb P}^0)$-norms
of all the 
terms entering the integrand 
of 
\eqref{eq:U:holder:contintime:step:1}. 
As a result, 
\begin{equation}
\label{eq:U:holder:contintime:step:1:b}
\begin{split} 
&\Bigl\vert 
\nabla_y^{l}\delta_\mu {\mathcal U}(t,x_0,x,\mu,y)
-
\nabla_y^{l}\delta_\mu {\mathcal U}(t+h,x_0,x,\mu,y)
\Bigr\vert 
\\
&\leq 
{\mathbb E}^0 \Bigl[ \Bigl\vert 
\nabla_y^{l}\delta_\mu {\mathcal U}(t+h,X_{t+h}^0,x,\mu_{t+h},y)
-
\nabla_y^{l}\delta_\mu {\mathcal U}(t+h,x_0,x,\mu,y)
\Bigr\vert 
\Bigr] + Ch. 
\end{split}
\end{equation} 
Using
the forward equation in 
\eqref{eq:major:FB:1:small}, we have
\begin{equation*} 
{\mathbb E}^0 \bigl[ \vert X_{t+h}^0 -x_0 \vert^2 \bigr] 
=
{\mathbb E}^0 \bigl[ \vert X_{t+h}^0 - X_t^0 \vert^2 \bigr] 
\leq C h {\mathbb E}^0 \int_t^{t+h} \bigl( 1 + 
\vert Z_s^0 \vert^2 
\bigr) \ud s. 
\end{equation*} 
By \eqref{eq:local:BMO}, the right-hand
side above is bounded by $Ch$ for a possibly new value of $C$. 

Also, using 
the forward equation in 
\eqref{eq:minor:FB:2:small}
together with the bound \eqref{eq:local:reg}, we have 
${\mathbb W}_1(\mu_{t+h},\mu_t) \leq C h$. And then, using the Lipschitz regularity of 
$\nabla_y^{l}\delta_\mu {\mathcal U}$ in the variables $x_0$ and $\mu$
(see 
\eqref{eq:local:diff-3bb}), 
we obtain 
\begin{equation}
\label{eq:U:holder:contintime:step:1:c}
\begin{split} 
&\Bigl\vert 
\nabla_y^{l}\delta_\mu {\mathcal U}(t,x_0,x,\mu,y)
-
\nabla_y^{l}\delta_\mu {\mathcal U}(t+h,x_0,x,\mu,y)
\Bigr\vert 
\leq Ch^{1/2},
\end{split}
\end{equation} 
which proves the second inequality in 
\eqref{eq:U:holder:contintime}. 

We prove 
the second inequality in 
a similar manner, using the representation 
$\nabla_y^l 
\delta_{\mu} {\mathcal U}^0(t,x_0,\mu,y)=K_l^{0,\mu}(t,x_0,\mu,y)=\delta Y_t^{0,l,y}$
(with the same definition as above for 
$\delta {\boldsymbol Y}^{0,l,y}$). 
We have
\begin{equation*} 
\begin{split} 
&\nabla_y^l 
\delta_{\mu} {\mathcal U}^0(t,x_0,\mu,y)
- 
\nabla_y^l 
\delta_{\mu} {\mathcal U}^0(t+h,x_0,\mu,y)
\\
&= K_l^{0,\mu}(t,x_0,\mu,y) - K_l^{0,\mu}(t+h,x_0,\mu,y)
\\
&= {\mathbb E}^0 \bigl[ \delta Y_{t+h}^{0,l,y} \bigr] 
- K_l^{0,\mu}(t+h,x_0,\mu,y)
\\
&\hspace{15pt} - 
{\mathbb E}^0 \int_t^{t+h}
\Big[ \nabla_{x_0}\hat L^0(X_r^0,Z_r^0)\cdot\delta X_r^0+\nabla_{p}\hat L^0(X_r^0,Z_r^0)\cdot\delta Z_r^0
 \\
 &\hspace{30pt} - \nabla_{x_0} f_r^0(X_r^0,\mu_r)\cdot \delta X_r^0 - \Big(\delta_\mu f_r^0(X_r^0,\mu_r),\delta \mu_r\Big) \Big]
 \ud r.
\end{split}
\end{equation*}
This is the same as before except for 
the fact that the bound for 
%
the term $\delta Z_r^0$ in 
\eqref{eq:local:linear H:diff0} is weaker than the bound satisfied by the other terms therein. 
Here is how we remedy this difficulty: it suffices to say (thanks 
to Cauchy-Schwarz inequality) that 
${\mathbb E} \int_t^{t+h} \vert \delta Z_r^0 \vert \ud r 
\leq C h^{1/2}$. Proceeding as in the derivation of 
\eqref{eq:U:holder:contintime:step:1:b}, we deduce 
\begin{equation}
\label{eq:U:holder:contintime:step:2:b}
\begin{split} 
&\Bigl\vert 
\nabla_y^{l}\delta_\mu {\mathcal U}^0(t,x_0,\mu,y)
-
\nabla_y^{l}\delta_\mu {\mathcal U}^0(t+h,x_0,\mu,y)
\Bigr\vert 
\\
&\leq 
{\mathbb E}^0 \Bigl[ \Bigl\vert 
\nabla_y^{l}\delta_\mu {\mathcal U}^0(t+h,X_{t+h}^0,\mu_{t+h},y)
-
\nabla_y^{l}\delta_\mu {\mathcal U}^0(t+h,x_0,\mu,y)
\Bigr\vert 
\Bigr] + Ch^{1/2}. 
\end{split}
\end{equation} 
And then, following the proof of 
\eqref{eq:U:holder:contintime:step:1:c}, we get 
\begin{equation}
\label{eq:U:holder:contintime:step:2:b}
\begin{split} 
&\Bigl\vert 
\nabla_y^{l}\delta_\mu {\mathcal U}^0(t,x_0,\mu,y)
-
\nabla_y^{l}\delta_\mu {\mathcal U}^0(t+h,x_0,\mu,y)
\Bigr\vert 
 \leq   Ch^{1/2}. 
\end{split}
\end{equation}
This completes the proof. 
%
%
\end{proof}

\subsection{Identification of the martingale representation term}

The next step is to identify the process ${\boldsymbol Z}^0$ in 
\eqref{eq:major:FB:1}:

\begin{proposition}
\label{prop:representation:Z0}
Under Assumption \hyp{\^A}, 
there exists a constant ${\mathfrak C}>0$, only depending on the parameters 
$d$, ${\mathfrak L}$,  $\kappa$, 
$\zeta$, $\sigma_0$ and 
$\mathbb{s}$
in Assumption \hyp{\^A},  
such that, for $T \leq {\mathfrak C}$,
the process ${\boldsymbol Z}^0$ in 
 \eqref{eq:major:FB:1:small} admits the following representation:
 \begin{equation*} 
 \forall t \in [0,T], \quad {\mathbb P} \bigl( \{ Z_t^0 = \nabla_{x_0} {\mathcal U}^0(t,X_t^0,\mu_t) \} \bigr) =1,
 \end{equation*} 
 and this, for any choice of initial condition $(x_0,\mu)$ in 
  \eqref{eq:major:FB:1:small}.
\end{proposition} 

The result is in fact a consequence of the more general 
representation property (the reader can skip the demonstration on first reading): 
\begin{lemma}
\label{lema:representation:Z0:gen}
Under Assumption \hyp{\^A}, 
consider ${\mathfrak C}>0$ as in Theorem 
\ref{thm:local:FB}. Then, for $T \leq {\mathfrak C}$ and
\begin{enumerate}
\item for 
any initial condition 
$(x_0,\mu)$ in 
  \eqref{eq:major:FB:1:small}, 
 \item for any continuous 
function ${\mathcal V}^0 : [0,T] \times {\mathbb R}^d \times 
{\mathcal P}({\mathbb T}^d) \rightarrow {\mathbb R}$ 
with the property that, for some $C_0>0$ and some ${\mathbb r}
\in (\lfloor {\mathbb s} \rfloor,{\mathbb s})$,  
${\mathcal V}^0(t,\cdot,\cdot) \in {\mathscr D}^0(C_0, {\mathbb r}  - 1)$
for all $t \in [0,T]$, and 
\begin{equation*}
\begin{split}
&\forall (t,x_0,\mu) 
\in [0,T) \times {\mathbb R}^d \times {\mathcal P}({\mathbb T}^d), 
\\
&\quad \lim_{h\to 0}\Big(\big|\nabla_{x_0} {\mathcal V}^0(t+h,x_0,\mu)-\nabla_{x_0}
{\mathcal V}^0(t,x_0,\mu) \bigr\vert +
\big\| \delta_\mu {\mathcal U}^0(t+h,x_0,\mu,\cdot)- \delta_\mu {\mathcal U}^0(t,x_0,\mu,\cdot)\big\|_{\mathbb{r}-1  }\Big)=0,
\end{split}
\end{equation*}
\item for any 
triplet $({\boldsymbol U}^0,{\boldsymbol V}^0,{\boldsymbol \ell}^0)$
in ${\mathscr S}^2(\Omega^0,{\mathcal F}^0,{\mathbb F}^0,{\mathbb P}^0;{\mathbb R})
\times 
{\mathscr H}^2(\Omega^0,{\mathcal F}^0,{\mathbb F}^0,{\mathbb P}^0;{\mathbb R})
\times
{\mathscr H}^2(\Omega^0,{\mathcal F}^0,{\mathbb F}^0,{\mathbb P}^0;{\mathbb R})$
 satisfying 
\begin{equation} 
\label{eq:lem:representation:bsde:U0:V0}
\ud U^0_t = \ell^0_t \ud t + \sigma_0 V^0_t \cdot \ud B_t^0, \quad U_T^0 = {\mathcal V}^0(T,X_T^0,\mu_T), 
\end{equation} 
and, ${\mathbb P}^0$-almost surely, $U_t^0 = {\mathcal V}^0(t,X_t^0,\mu_t)$, 
\end{enumerate}
the process ${\boldsymbol V}^0$ admits the following representation:
 \begin{equation*} 
 \forall t \in [0,T], \quad {\mathbb P} \bigl( \{ V_t^0 = \nabla_{x_0} {\mathcal V}^0(t,X_t^0,\mu_t) \} \bigr) =1.
 \end{equation*} 
\end{lemma} 

\begin{proof} 
The proof is standard. Here, we take it from 
\cite[Lemma 4.11]{CarmonaDelarue_book_I}. We consider a uniform subdivision 
${\boldsymbol \tau}:= (t_i)_{i=0,\cdots,n}$ with $0=:t_0<t_1< \cdots < t_n:=T$ of stepsize $h$ together with a simple 
${\mathbb R}^d$-valued 
adapted 
process 
${\boldsymbol \pi}=(\pi_t)_{0 \le t \le T}$ of the form 
\begin{equation}
 \label{eq:prop:5:11:representation:Z0:0}
\pi_t = \sum_{i=0}^{n-1} \pi^{(i)} {\mathbf 1}_{(t_i,t_{i+1}]}(t), \quad t \in [0,T],
\end{equation} 
where, for some 
$K \geq 0$ and for each 
$i=0,\cdots,n-1$, $\pi^{(i)}$ belongs to $L^\infty(\Omega,{\mathcal F}_{t_i}^0,{\mathbb P}^0;{\mathbb R}^d)$
with 
${\mathbb P}^0(\{ \vert \pi^{(i)} \vert \leq K\})=1$.
For $i=0,\cdots,n-1$, we then have
\begin{equation*} 
\begin{split}
{\mathbb E}^0 \biggl[ \biggl( \int_{t_i}^{t_{i+1}} \pi_t \cdot 
\ud B_t^0 
\biggr) \times U_{t_{i+1}}^0
\biggr] 
= {\mathbb E}^0 \Bigl[ \bigl( \pi^{(i)} \cdot ( B^0_{t_{i+1}} - B^0_{t_i}) \bigr) 
\times U_{t_{i+1}}^0
\Bigr].
\end{split} 
\end{equation*}
\vskip 4pt

\noindent  \textit{First Step.} Using the backward equation   
\eqref{eq:lem:representation:bsde:U0:V0},
it is standard to prove that 
there exists a constant $C$
(only depending on ${\boldsymbol \pi}$ through $K$)
such that, for all $i=0,\cdots,n-1$, 
\begin{equation} 
\label{eq:prop:5:11:representation:Z0:1}
\begin{split}
&\biggl\vert 
{\mathbb E}^0 \biggl[ \biggl( \int_{t_i}^{t_{i+1}} \pi_t \cdot \ud B_t^0 
\biggr) \times U_{t_{i+1}}^0 \biggr] - 
\sigma_0
 {\mathbb E}^0 \biggl[ \int_{t_i}^{t_{i+1}} \pi_t \cdot V_t^0 \ud t
\biggr]
\biggr\vert 
\\
&= \biggl\vert 
{\mathbb E}^0 \biggl[ \biggl( \int_{t_i}^{t_{i+1}} \pi_t \cdot \ud B_t^0 
\biggr) \times 
\int_{t_i}^{t_{i+1}}
\ell_t^0 \ud t
 \biggr] 
\biggr\vert 
\\
&\leq 
(t_{i+1}-t_i)^{1/2}
{\mathbb E}^0 \biggl[   \int_{t_i}^{t_{i+1}} \vert \pi_t \vert^2
\ud t
\biggr]^{1/2} 
{\mathbb E}^0 \biggl[ \int_{t_i}^{t_{i+1}}
 \vert \ell_t^0
 \vert^2 \ud t
 \biggr]^{1/2} 
\\
&\leq C (t_{i+1} - t_i) 
\biggl\{  \int_{t_i}^{t_{i+1}} 
{\mathbb E}^0 
\bigl[ \vert \ell_t^0 \vert^2
\bigr] 
\ud t 
\biggr\}^{1/2}
\leq  C 
h
\epsilon (h), 
\end{split}
\end{equation} 
where we have let
\begin{equation} 
\label{eq:def:espilon:representation:mart}
\epsilon(h) := 
\biggl\{  
h
+
\sup_{0 \leq s < t < T : \vert t-s \vert \leq h} 
\int_s^t {\mathbb E}^0 \bigl[ 
\vert Z_r^0 \vert^2 
+
\vert \ell_r^0 \vert^2 
\bigr] \ud r \biggr\}^{1/2}, \quad h >0. 
\end{equation}
Because $\int_0^T {\mathbb E}^0[ \vert Z_r^0 \vert^2 + \vert \ell_r^0 \vert^2] 
\ud r < \infty$, we have 
$\lim_{h \rightarrow 0} \epsilon(h)=0$. 

Now, we observe that the left-hand side 
in the top line of 
\eqref{eq:prop:5:11:representation:Z0:1}
 is equal to 
\begin{align} 
&{\mathbb E}^0 \biggl[ \biggl( \int_{t_i}^{t_{i+1}} \pi_t\cdot \ud B_t^0 \biggr) \times U_{t_{i+1}}^0 \biggr]
\nonumber
\\
&= 
{\mathbb E}^0 \biggl[ \biggl( \int_{t_i}^{t_{i+1}} \pi_t\cdot \ud B_t^0 \biggr) 
\times 
\Bigl( {\mathcal V}^0(t_{i+1},X_{t_{i+1}}^0,\mu_{t_{i+1}} ) - 
{\mathcal V}^0(t_{i+1},X_{t_{i}}^0,\mu_{t_{i}} )
 \Bigr) \biggr]
\label{eq:prop:5:11:representation:Z0:2}
 \\
 &=
 {\mathbb E}^0 \biggl[ \biggl( \int_{t_i}^{t_{i+1}} \pi_t\cdot \ud B_t^0 \biggr) 
 \times
\biggl( 
\int_0^1 
\biggl[
\nabla_{x_0} {\mathcal V}^0 \Bigl(t_{i+1}, r X_{t_{i+1}}^0 + (1-r) 
X_{t_i}^0, r \mu_{t_{i+1}} + (1-r) \mu_{t_i} \Bigr) 
\cdot 
\bigl( X_{t_{i+1}}^0 - X_{t_i}^0 \bigr)
\nonumber 
\\
&\hspace{15pt} + 
\int_{{\mathbb T}^d} \delta_{\mu} {\mathcal V}^0 
 \Bigl(t_{i+1}, r X_{t_{i+1}}^0 + (1-r) 
X_{t_i}^0, r \mu_{t_{i+1}} + (1-r) \mu_{t_i}, y \Bigr) 
\ud \bigl( \mu_{t_{i+1}} - \mu_{t_{i}} \bigr)
(y)
 \biggr] \ud r \biggr)\biggr].
\nonumber
\end{align} 
\vskip 4pt

\noindent  \textit{Second Step.} We first handle the term on the last line of 
\eqref{eq:prop:5:11:representation:Z0:2}. 
Using the forward equation in 
\eqref{eq:minor:FB:2:small}, we have, for any 
$r \in [0,1]$, 
\begin{equation*}
\begin{split}
&\int_{{\mathbb T}^d} \delta_{\mu} {\mathcal V}^0 
 \Bigl(t_{i+1}, r X_{t_{i+1}}^0 + (1-r) 
X_{t_i}^0, r \mu_{t_{i+1}} + (1-r) \mu_{t_i}, y \Bigr) 
\ud \bigl( \mu_{t_{i+1}} - \mu_{t_{i}} \bigr)
(y)
\\
&=
\int_{t_i}^{t_{i+1}}
\int_{{\mathbb T}^d} 
\tfrac12 \Delta_{y} \delta_{\mu} {\mathcal V}^0 
 \Bigl(t_{i+1}, r X_{t_{i+1}}^0 + (1-r) 
X_{t_i}^0, r \mu_{t_{i+1}} + (1-r) \mu_{t_i}, y \Bigr) 
\ud   \mu_{s} 
(y) 
ds
\\
&\hspace{15pt} -
\int_{t_i}^{t_{i+1}}
\int_{{\mathbb T}^d} 
\nabla_y \delta_{\mu} {\mathcal V}^0 
 \Bigl(t_{i+1}, r X_{t_{i+1}}^0 + (1-r) 
X_{t_i}^0, r \mu_{t_{i+1}} + (1-r) \mu_{t_i}, y \Bigr) 
\cdot 
\nabla_p \hat H \bigl(y, \nabla_x u_s(y) \Bigr) 
\ud   \mu_{s} 
(y) 
ds. 
\end{split}
\end{equation*} 
By item 2 in the statement, 
the two functions
$\Delta_{y} \delta_{\mu}
{\mathcal V}^0$
and
$\nabla_y  \delta_\mu {\mathcal V}^0$ 
are bounded. Therefore, 
we can modify the value of  
$C$ such that, for all 
$i = 0,\cdots,n-1$, with probability 1 under 
${\mathbb P}^0$, 
\begin{equation}
\label{eq:prop:5:11:representation:Z0:3}
\begin{split}
\biggl\vert \int_{{\mathbb T}^d} \delta_{\mu} {\mathcal V}^0 
 \Bigl(t_{i+1}, r X_{t_{i+1}}^0 + (1-r) 
X_{t_i}^0, r \mu_{t_{i+1}} + (1-r) \mu_{t_i}, y \Bigr) 
\ud \bigl( \mu_{t_{i+1}} - \mu_{t_{i}} \bigr)
(y) \biggr\vert \leq C (t_{i+1} - t_i). 
\end{split}
\end{equation} 
And then, by Cauchy-Schwarz' inequality (and recalling 
\eqref{eq:def:espilon:representation:mart}), 
\begin{equation}
\label{eq:prop:5:11:representation:Z0:5}
\begin{split} 
&\biggl\vert {\mathbb E}^0 \biggl[ \biggl( \int_{t_i}^{t_{i+1}} \pi_t\cdot \ud B_t^0 \biggr) 
\\
&\hspace{30pt} \times
\biggl( 
\int_0^1 
\biggl[ 
\int_{{\mathbb T}^d} \delta_{\mu} {\mathcal V}^0 
 \Bigl(t_{i+1}, r X_{t_{i+1}}^0 + (1-r) 
X_{t_i}^0, r \mu_{t_{i+1}} + (1-r) \mu_{t_i}, y \Bigr) 
\ud \bigl( \mu_{t_{i+1}} - \mu_{t_{i}} \bigr)
(y)
 \biggr] \ud r
\biggr)
\biggr]
\biggr\vert 
\\
&\leq C  (t_{i+1} - t_i)^{3/2}
\leq C h \epsilon(h).  
\end{split}
\end{equation}
\vskip 4pt
 
\noindent  \textit{Third Step.} We now address 
the term on the third line of \eqref{eq:prop:5:11:representation:Z0:2}.
Proceeding as 
in the derivation of 
\eqref{eq:prop:5:11:representation:Z0:3}, 
we have, for all $i=0,\cdots,n-1$, 
with probability 1 under 
${\mathbb P}^0$, 
\begin{equation*} 
{\mathbb W}_1\bigl( \mu_{t_{i}}, \mu_{t_{i+1}}\bigr) 
\leq C 
(t_{i+1} - t_i)^{1/2}.
\end{equation*}
Therefore, thanks to the boundedness and Lipschitz regularity of 
$\nabla_{x_0} {\mathcal V}^0 $ (see 
item 2 in the statement), we get 
(recalling again 
\eqref{eq:def:espilon:representation:mart})
\begin{equation}
\label{eq:prop:5:11:representation:Z0:6}
\begin{split} 
& \biggl\vert  {\mathbb E}^0 \biggl[
\int_0^1 
\biggl[
\nabla_{x_0} {\mathcal V}^0  \Bigl(t_{i+1}, r X_{t_{i+1}}^0 + (1-r) 
X_{t_i}^0, r \mu_{t_{i+1}} + (1-r) \mu_{t_i} \Bigr) 
\cdot 
\bigl( X_{t_{i+1}}^0 - X_{t_i}^0 \bigr)
 \biggr] dr 
 \\
 &\hspace{30pt} 
 - 
 \sigma_0
\nabla_{x_0} {\mathcal V}^0  \Bigl(t_{i+1}, X_{t_{i}}^0,  \mu_{t_i} \Bigr) 
\cdot 
\bigl( B_{t_{i+1}}^0 - B_{t_i}^0 \bigr)
 \biggr\vert^2\biggr]^{1/2}
 \\
 &\leq C
 (t_{i+1} - t_i)^{1/2} 
 \epsilon(t_{i+1} - t_i) 
 \\
&\hspace{15pt} +  \sigma_0
 {\mathbb E}^0 \biggl[ \biggl\vert 
\int_0^1 
\biggl[
\nabla_{x_0} {\mathcal V}^0  \Bigl(t_{i+1}, r X_{t_{i+1}}^0 + (1-r) 
X_{t_i}^0, r \mu_{t_{i+1}} + (1-r) \mu_{t_i} \Bigr) 
\cdot 
\bigl( B_{t_{i+1}}^0 - B_{t_i}^0 \bigr)
 \biggr] \ud r 
 \\
 &\hspace{30pt} 
 - 
\nabla_{x_0} {\mathcal V}^0  \Bigl(t_{i+1}, X_{t_{i}}^0,  \mu_{t_i} \Bigr) 
\cdot 
\bigl( B_{t_{i+1}}^0 - B_{t_i}^0 \bigr)
 \biggr\vert^2\biggr]^{1/2}
\\
&\leq  
C (t_{i+1} - t_i)^{1/2} 
 \epsilon(t_{i+1} - t_i) 
 \\
  &\hspace{15pt}
+ 
C {\mathbb E}^0 \Bigl[ 
\Bigl\{ 1 
\wedge 
\Bigl(   
\bigl\vert X_{t_{i+1}}^0 - X_{t_i}^0 \bigr\vert^2
+
{\mathbb W}_1\bigl( \mu_{t_{i}}, \mu_{t_{i+1}}\bigr)^2 
\Bigr) 
\Bigr\}
\bigl\vert B_{t_{i+1}}^0 - B_{t_i}^0 \bigr\vert^2
\Bigr]^{1/2}
\\
&\leq C (t_{i+1} - t_i)^{1/2} 
 \epsilon(t_{i+1} - t_i) 
+ 
C (t_{i+1} - t_i)^{1/2}  {\mathbb E}^0 \biggl[ 
  1 
\wedge 
\biggl(   
(t_{i+1} - t_i)
\int_{t_i}^{t_{i+1}}
  \vert Z_r^0 \vert^2 \ud r 
\biggr)^2
\biggr]^{1/4}
\\
&\leq C (t_{i+1} - t_i )^{1/2} 
\epsilon(t_{i+1} - t_i) 
+ C (t_{i+1} - t_i)^{1/2}  {\mathbb E}^0 \biggl[ 
  1 
\wedge 
\biggl(   
(t_{i+1} - t_i)
\int_{t_i}^{t_{i+1}}
  \vert Z_r^0 \vert^2 \ud r 
\biggr)
\biggr]^{1/4}
\\
&\leq  C (t_{i+1} - t_i )^{1/2} 
\epsilon(t_{i+1} - t_i)
\leq C h^{1/2} \epsilon(h), 
 \end{split}
 \end{equation} 
 with the last line following from Young's inequality. 
\vskip 4pt 

\noindent  \textit{Fourth Step.} 
Finally, by combining
\eqref{eq:prop:5:11:representation:Z0:1},\eqref{eq:prop:5:11:representation:Z0:2},
\eqref{eq:prop:5:11:representation:Z0:5}
and
\eqref{eq:prop:5:11:representation:Z0:6}.
\begin{equation*}
\begin{split} 
&\biggl\vert  {\mathbb E}^0 \biggl[   
\int_{t_i}^{t_{i+1}} 
\pi_s  \cdot \Bigl( \sigma_0 V_s^0
\Bigr) 
\ud s 
-
\int_{t_i}^{t_{i+1}}  \pi_s
\cdot \Bigl( 
\sigma_0
\nabla_{x_0} {\mathcal V}^0 \bigl(t_{i+1}, X_{t_{i}}^0,  \mu_{t_i} \bigr) 
\Bigr) 
\ud s
 \biggr]
  \biggr\vert
  \leq C  h\epsilon(h),
 \end{split}
 \end{equation*} 
which yields 
\begin{equation}
\begin{split} 
&\lim_{h \rightarrow 0} 
\biggl\vert  {\mathbb E}^0 \biggl[   
\int_{0}^{T} 
\pi_s \cdot  
  \Bigl( 
  \sigma_0  V_s^0 
\Bigr) 
 \ud s
- \int_{0}^{T}
\pi_s \cdot 
\Bigl( \sigma_0
\nabla_{x_0} {\mathcal V}^0 \bigl(t^+({\boldsymbol \tau},s), X_{t^-({\boldsymbol \tau},s)}^0,  \mu_{t^-({\boldsymbol \tau},s)} \bigr) 
\Bigr) 
\ud s
 \biggr]
  \biggr\vert
 \leq C  \epsilon(h),
 \end{split}
\label{eq:prop:5:11:representation:Z0:7}
 \end{equation} 
 with 
 \begin{equation*}
 t^+({\boldsymbol \tau},s)
: = \sum_{i=0}^{n-1} t_{i+1} {\mathbf 1}_{(t_i,t_{i+1}]}(s), 
\quad
t^-({\boldsymbol \tau},s)
:= \sum_{i=0}^{n-1} t_{i} {\mathbf 1}_{[t_i,t_{i+1})}(s),    
 \quad s \in [0,T]. 
 \end{equation*}
We observe that, as the stepsize $h$ of ${\boldsymbol \tau}$ tends to $0$, 
$ t^+({\boldsymbol \tau},s)$ and 
$ t^-({\boldsymbol \tau},s)$ tend to $s$. 
 
 By boundedness and continuity of $\nabla_{x_0} {\mathcal V}^0$
 (see item 2) and by Lebesgue's dominated convergence theorem,
 it is straightforward to prove that 
\begin{equation} 
\label{eq:prop:5:11:representation:Z0:8}
\lim_{h \rightarrow 0} 
{\mathbb E}^0 \biggl[
\int_{0}^{T} 
\Bigl\vert 
\nabla_{x_0} {\mathcal V}^0 \bigl(s, X_{s}^0,  \mu_{s} \bigr) 
-
\nabla_{x_0} {\mathcal V}^0 \bigl( t^+({\boldsymbol \tau},s), X_{ t^-({\boldsymbol \tau},s)}^0,  \mu_{ t^-({\boldsymbol \tau},s)} \bigr) 
\Bigr\vert 
\ud s
 \biggr]
= 0. 
\end{equation}   
 
It remains to choose ${\boldsymbol \pi}$ as a bounded simple process 
(as in 
\eqref{eq:prop:5:11:representation:Z0:0})
but with respect to a fixed partition
${\boldsymbol \tau}^0:=(t_i^0)_{i=0,\cdots,n_0}$, 
with $0=:t_0^0 < t_1^0 < \cdots < t_{n_0}:=T$.
The stepsize of ${\boldsymbol \tau}^0$ is denoted $h_0$ and is kept fixed. 
We then refine the partition ${\boldsymbol \tau}^0$ considering 
another partition ${\boldsymbol \tau} \supset {\boldsymbol \tau}^0$.
The stepsize of ${\boldsymbol \tau}$ is denoted $h$ and is sent to zero 
in the next lines. 
The main observation is that  ${\boldsymbol \pi}$ 
is also a bounded simple process with respect to the partition 
${\boldsymbol \tau}$. This makes it possible to apply \eqref{eq:prop:5:11:representation:Z0:7}
and 
\eqref{eq:prop:5:11:representation:Z0:8}, from which we deduce
(sending $h$ to $0$) that 
\begin{equation*} 
{\mathbb E}^0 \biggl[   
\int_{0}^{T} 
\pi_t \cdot  
  \Bigl( 
  \sigma_0 \Bigl[ 
  V_t^0 
-
\nabla_{x_0} {\mathcal V}^0 \bigl(t, X_t,  \mu_{t} \bigr)
\Bigr] 
\Bigr) 
 \ud t 
 \biggr]
 =0. 
 \end{equation*}
 The above is true for any bounded simple process ${\boldsymbol \pi}$. By a density argument, 
it is true for any bounded progressively measurable process, from which we deduce that, 
 for ${\rm Leb} \times {\mathbb P}^0$-almost every $(t,\omega^0)$, 
$$ \sigma_0 \Bigl[ 
  V_t^0 
-
\nabla_{x_0} {\mathcal V}^0 \bigl(t, X_t,  \mu_{t} \bigr)
\Bigr]=0.$$ 
This completes the proof. 
 \end{proof}

\subsection{Derivation of the master equations}
We now address the two PDEs (called master equations) 
that are satisfied by ${\mathcal U}^0$ and ${\mathcal U}$. 
The derivation of the master equation for ${\mathcal U}^0$ relies on the following formula (which is very much in the spirit of a Kolmogorov formula): 

\begin{proposition}
\label{prop:representation:U0:PDE}
Under Assumption \hyp{\^A}, there 
exists a constant ${\mathfrak C}>0$, only depending on the parameters 
$d$,  $\kappa$,  ${\mathfrak L}$,
$\zeta$, $\sigma_0$ and 
$\mathbb{s}$
in Assumption \hyp{\^A} such that, for $T \leq {\mathfrak C}$, 
we have the following representation of 
${\mathcal U}^0(t_0,x_0,\mu)$, for any $(t_0,x_0,\mu) \in [0,T] 
\times {\mathbb R}^d \times {\mathcal P}({\mathbb T}^d)$:
\begin{equation} 
\label{prop:representation:U0:PDE:00}
\begin{split}
{\mathcal U}^0(t_0,x_0,\mu) 
&= {\mathbb E}^0 \bigl[ g^0(\tilde X_T^0,\mu) \bigr]
 - 
 {\mathbb E}^0 \biggl[ \tfrac12 \int_{t_0}^{T} \int_{{\mathbb T}^d} 
\Delta_y \delta_\mu {\mathcal U}^0\bigl(s,\tilde X_{s}^{0},\mu,y\bigr) \ud
\mu(y)  \ud s  
\biggr] 
\\
&\hspace{15pt}  +  {\mathbb E}^0 \biggl[ \int_{t_0}^{T} \int_{{\mathbb T}^d} 
\nabla_y \delta_\mu {\mathcal U}^0\bigl(s,\tilde X_{s}^0,\mu,y\bigr) 
\cdot 
\nabla_p \hat H\bigl(y,\nabla_x {\mathcal U}(s,\tilde X_s^0,y,\mu) \bigr) 
 \ud
\mu(y)  \ud s  
\biggr] 
\\
&\hspace{15pt} 
 -
 {\mathbb E}^0 \biggl[ \int_{t_0}^{T}  
  \Bigl( f_s^0(\tilde X_s^0,\mu) + \hat L^0\bigl(\tilde X_s^0, 
  \nabla_{x_0} {\mathcal U}^0(s,\tilde X_s^0,\mu)
  \bigr)
\Bigr) \ud s   \biggr], 
\end{split} 
\end{equation} 
where $\tilde{\boldsymbol X}^0=(\tilde X^0_t)_{t_0 \le t \leq T}$ solves the 
SDE
\begin{equation}
\label{eq:tilde:X:0:U0} 
\ud \tilde X_t^0 = - \nabla_p H^0 \bigl( \tilde X_t^0, \nabla_{x_0} {\mathcal U}^0(t,\tilde X_t^0,\mu) \bigr) \ud t 
+ \ud B_t^0, \quad t \in [t_0,T]; \quad \tilde X_{t_0}^0=x_0.    
\end{equation} 
\end{proposition}

Pay attention that $\mu$ is fixed in 
\eqref{eq:tilde:X:0:U0} 
(equivalently, 
\eqref{eq:tilde:X:0:U0} is not a McKean-Vlasov equation). 

\begin{proof}
Throughout the proof, $T$ is implicitly taken less than ${\mathfrak C}$ so that all the results already proved in this section can be freely applied.

For simplicity, we establish the representation at time $t_0=0$. 
Thanks to the regularity of $\nabla_{x_0} {\mathcal U}^0$ (see 
Proposition 
\ref{prop:deltaU:continuity}) 
the equation 
\eqref{eq:tilde:X:0:U0}
has a unique solution. 
\vskip 4pt

\noindent  \textit{First Step.} 
We fix $t \in [0,T]$. With $\tilde{\boldsymbol X}^0$ as in the statement, we consider 
the solution 
$({\boldsymbol X}^{0},{\boldsymbol \mu})=(X_s^0,\mu_s)_{t \le s \le T}$
to the two forward equations in 
\eqref{eq:major:FB:1:small}
and
\eqref{eq:minor:FB:2:small}
with $X^0_t=\tilde X^0_t$ and $\mu_t=\mu$ as initial condition 
at time $t$. 
By 
Proposition 
\ref{prop:representation:Z0}, the equation for 
${\boldsymbol X}^0$ can be rewritten as
\begin{equation}
\label{eq:proof:edp:X0:0}  
\ud X_s^0 = - \nabla_p H^0 \bigl( X_s^0, \nabla_{x_0} {\mathcal U}^0(s,X_s^0,\mu_s) \bigr) \ud s
+ \ud B_s^0, \quad s \in [t,T]; \quad X_t^0 = \tilde X_t^0. 
\end{equation} 
In particular, by Lipschitz regularity of $\nabla_{x_0}  {\mathcal U}^0$
and from the identity $\mu_t=\mu$, 
we obtain, for a fixed (small) $h>0$, 
\begin{equation}
\label{eq:proof:edp:X0:1} 
\begin{split}
\sup_{t \leq s \leq t+h} \vert \tilde X_s^0  - X_s^{0} \vert
&\leq C \int_t^{t+h} {\mathbb W}_1(\mu_s,\mu_t) \ud s
 \leq C h^{3/2}. 
\end{split} 
\end{equation} 
Now, 
by 
\eqref{eq:proof:edp:X0:1}, 
\begin{equation} 
\label{eq:proof:edp:X0:2}
\begin{split}
{\mathbb E}^0 \bigl[ {\mathcal U}^0(t+h,\tilde X_{t+h}^{0},\mu) 
\, \vert \, {\mathcal F}_t^0 \bigr] 
&=
{\mathbb E}^0 \bigl[ {\mathcal U}^0(t+h,\tilde X_{t+h}^{0},\mu_t) 
\, \vert \, {\mathcal F}_t^0 \bigr] 
\\
&= {\mathbb E}^0 \bigl[ {\mathcal U}^0(t+h, X_{t+h}^{0},\mu_{t+h}) 
\, \vert \, {\mathcal F}_t^0 \bigr] 
\\
&\hspace{15pt} +
{\mathbb E}^0 \bigl[ {\mathcal U}^0(t+h,\tilde X_{t+h}^{0},\mu_t) 
 - {\mathcal U}^0(t+h, X_{t+h}^0,\mu_{t+h}) 
\, \vert \, {\mathcal F}_t^0 \bigr] 
\\
&= {\mathbb E}^0 \bigl[{\mathcal U}^0(t+h, X_{t+h}^{0},\mu_{t+h}) 
\, \vert \, {\mathcal F}_t^0 \bigr] 
\\
&\hspace{15pt} +
{\mathbb E}^0 \bigl[ {\mathcal U}^0(t+h, X_{t+h}^{0},\mu_t) 
- {\mathcal U}^0(t+h, X_{t+h}^0,\mu_{t+h}) 
\, \vert \, {\mathcal F}_t^0 \bigr] 
+ O(h^{3/2}), 
\end{split}
\end{equation} 
where 
$\vert O(h^{3/2}) \vert \leq C h^{3/2}$. 
\vskip 4pt

\noindent \textit{Second Step.} 
We handle the term on the last line of 
\eqref{eq:proof:edp:X0:2}. Using the regularity of 
${\mathcal U}^0$ in the third argument 
together with 
the equation for ${\boldsymbol \mu}$ 
in 
\eqref{eq:minor:FB:2:small}, we have (the reader may compare the expansion below with \cite[Theorem 5.99]{CarmonaDelarue_book_I}):
\begin{equation*} 
\begin{split} 
&{\mathcal U}^0(t+h, X_{t+h}^0,\mu_{t+h}) 
- 
{\mathcal U}^0(t+h, X_{t+h}^{0},\mu_t) 
\\
&= \tfrac12 \int_{t}^{t+h} \int_{{\mathbb T}^d} 
\Delta_y \delta_\mu {\mathcal U}^0(t+h,X_{t+h}^0,\mu_s,y) \ud
\mu_s(y)  \ud s 
\\
&\hspace{15pt} - \int_t^{t+h} \int_{{\mathbb T}^d} 
\nabla_y \delta_\mu {\mathcal U}^0(t+h,X_{t+h}^0,\mu_s,y) 
\cdot 
\nabla_p \hat H\bigl(y,\nabla_x {\mathcal U} (s,X_s^0,y,\mu_s) \bigr) 
 \ud
\mu_s(y)  \ud s.
\end{split}
\end{equation*}  
And then, using 
\eqref{eq:proof:edp:X0:1} 
together with 
the Lipschitz regularity 
of $\Delta_y \delta_\mu {\mathcal U}^0$, $\nabla_y \delta_\mu {\mathcal U}^0$ and $\nabla_x {\mathcal U}$
in the variables $x_0$ and $\mu$ and the H\"older regularity of the same functions in $y$ (see \eqref{eq:local:diff-3bb}), we deduce that 
\begin{align} 
&{\mathbb E}^0 \Bigl[ {\mathcal U}^0(t+h, X_{t+h}^0,\mu_{t+h}) 
- 
{\mathcal U}^0(t+h, X_{t+h}^{0},\mu_t) \, \vert \,  {\mathcal F}_t^0 \Bigr]
\nonumber
\\
&= {\mathbb E}^0 \biggl[ \tfrac12 \int_{t}^{t+h} \int_{{\mathbb T}^d} 
\Delta_y \delta_\mu {\mathcal U}^0(t+h,\tilde X_{s}^{0},\mu_t,y) \ud
\mu_t(y)  \ud s 
\label{eq:proof:edp:X0:4} 
\\
&\hspace{15pt} - \int_t^{t+h} \int_{{\mathbb T}^d} 
\nabla_y \delta_\mu {\mathcal U}^0(t+h,\tilde X_{s}^0,\mu_t,y) 
\cdot 
\nabla_p \hat H\bigl(y,\nabla_x {\mathcal U} (s,\tilde X_s^0,y,\mu_t) \bigr) 
 \ud
\mu_t(y)  \ud s \, \vert \, {\mathcal F}_t^0 
\biggr] + O\bigl(h^{1+({\mathbb r}-\lfloor {\mathbb s}\rfloor)/2}\bigr). \nonumber
\end{align} 
We then handle the term on the penultimate line of 
\eqref{eq:proof:edp:X0:2}. 
Using the 
backward equation in \eqref{eq:major:FB:1:small}, 
we have
\begin{equation}
\label{eq:proof:edp:X0:3}
\begin{split}
&{\mathbb E}^0 \bigl[ {\mathcal U}^0(t+h, X_{t+h}^{0},\mu_{t+h}) 
\, \vert \, {\mathcal F}_t^0 \bigr]
 =  {\mathbb E}^0 \bigl[ Y_{t+h}^0 
\, \vert \, {\mathcal F}_t^0 \bigr]
\\
&= Y_t^0
  -
 {\mathbb E}^0 \biggl[ \int_t^{t+h}  
  \Bigl( f_s^0(X_s^0,\mu_s) + \hat L^0\bigl(X_s^0, 
  \nabla_{x_0} {\mathcal U}^0(s,X_s^0,\mu_s)
  \bigr)
\Bigr) \ud s \, \vert \, {\mathcal F}_t^0 \biggr]
\\
&= {\mathcal U}^0(t,X_t^0,\mu_t)  
  -
 {\mathbb E}^0 \biggl[ \int_t^{t+h}  
  \Bigl( f_s^0(\tilde X_s^0,\mu_t) + \hat L^0\bigl(\tilde X_s^0,  \nabla_{x_0} {\mathcal U}^0(s,\tilde X_s^0,\mu_t)
  \bigr)  
\Bigr) \ud s \, \vert \, {\mathcal F}_t^0 \biggr] 
  + O(h^{3/2}). 
\end{split}
\end{equation}
\vskip 4pt

\noindent \textit{Third Step.} 
We now insert
\eqref{eq:proof:edp:X0:4}
and
\eqref{eq:proof:edp:X0:3}
into 
\eqref{eq:proof:edp:X0:2}. 
Recalling 
$X_t^0 = \tilde X_t^0$ and 
$\mu_t=\mu$, 
we obtain 
\begin{equation} 
\label{eq:proof:edp:X0:5}
\begin{split} 
&{\mathbb E}^0 \bigl[ {\mathcal U}^0(t+h,\tilde X_{t+h}^{0},\mu) 
\, \vert \, {\mathcal F}_t^0 \bigr] 
= {\mathcal U}^0(t,\tilde X_t^0,\mu)  
\\
&\hspace{15pt} - 
 {\mathbb E}^0 \biggl[ \tfrac12 \int_{t}^{t+h} \int_{{\mathbb T}^d} 
\Delta_y \delta_\mu {\mathcal U}^0(t+h,\tilde X_{s}^{0},\mu,y) \ud
\mu(y)  \ud s \, \vert \, {\mathcal F}_t^0 
\biggr] 
\\
&\hspace{15pt} +  {\mathbb E}^0 \biggl[ \int_t^{t+h} \int_{{\mathbb T}^d} 
\nabla_y \delta_\mu {\mathcal U}^0(t+h,\tilde X_{s}^0,\mu,y) 
\cdot 
\nabla_p \hat H\bigl(y,\nabla_x {\mathcal U}(s,\tilde X_s^0,y,\mu) \bigr) 
 \ud
\mu(y)  \ud s \, \vert \, {\mathcal F}_t^0 
\biggr] 
\\
&\hspace{15pt} 
 -
 {\mathbb E}^0 \biggl[ \int_t^{t+h}  
  \Bigl( f_s^0(\tilde X_s^0,\mu) + \hat L^0\bigl(\tilde X_s^0,   \nabla_{x_0} {\mathcal U}^0(s,\tilde X_s^0,\mu)
  \bigr)  
\Bigr) \ud s \, \vert \, {\mathcal F}_t^0 \biggr] + O\bigl(h^{1+({\mathbb r}-\lfloor {\mathbb s}\rfloor)/2}\bigr).
\end{split}
\end{equation}  

We are now given  a subdivision
${\boldsymbol \tau} = (t_i)_{i=0,\cdots,n}$
of $[0,T]$, with 
 $0=t_0<t_1<\cdots<t_n=T$
 and 
with stepsize $h$. Then, \eqref{eq:proof:edp:X0:5} says 
\begin{equation*} 
\begin{split} 
&{\mathbb E}^0 \bigl[ {\mathcal U}^0(t_{i+1},\tilde X_{t_{i+1}}^{0},\mu)  \bigr] 
= {\mathbb E}^0 \bigl[ {\mathcal U}^0(t_i,\tilde X_{t_i}^0,\mu) \bigr] 
\\
&\hspace{15pt} - 
 {\mathbb E}^0 \biggl[ \tfrac12 \int_{t_i}^{t_{i+1}} \int_{{\mathbb T}^d} 
\Delta_y \delta_\mu {\mathcal U}^0(t_{i+1},\tilde X_{s}^{0},\mu,y) \ud
\mu(y)  \ud s  
\biggr] 
\\
&\hspace{15pt} +  {\mathbb E}^0 \biggl[ \int_{t_i}^{t_{i+1}} \int_{{\mathbb T}^d} 
\nabla_y \delta_\mu {\mathcal U}^0(t_{i+1},\tilde X_{s}^0,\mu,y) 
\cdot 
\nabla_p \hat H\bigl(y,\nabla_x {\mathcal U}(s,\tilde X_s^0,y,\mu) \bigr) 
 \ud
\mu(y)  \ud s  
\biggr] 
\\
&\hspace{15pt} 
 -
 {\mathbb E}^0 \biggl[ \int_{t_i}^{t_{i+1}}  
  \Bigl( f_s^0(\tilde X_s^0,\mu) + \hat L^0\bigl(\tilde X_s^0,  
  \nabla_{x_0} {\mathcal U}^0(s,\tilde X_s^0,\mu)
 \bigr)
\Bigr) \ud s   \biggr] + O\bigl(h^{1+({\mathbb r}-\lfloor {\mathbb s}\rfloor)/2}\bigr).
\end{split}
\end{equation*}  
Letting $t({\boldsymbol \tau},s) := t_{i+1}$ if 
$s \in (t_i,t_{i+1}]$, for $i=0,\dots,n-1$, 
and summing over $i$ 
the above expansion, 
we get 
\begin{equation*} 
\begin{split}
&{\mathcal U}^0(0,x_0,\mu) 
= {\mathbb E}^0 \bigl[ g^0(\tilde X_T^0,\mu) \bigr]
\\
&\hspace{15pt}  - 
 {\mathbb E}^0 \biggl[ \tfrac12 \int_{0}^{T} \int_{{\mathbb T}^d} 
\Delta_y \delta_\mu {\mathcal U}^0\bigl(t({\boldsymbol \tau},s),\tilde X_{s}^{0},\mu,y\bigr) \ud
\mu(y)  \ud s  
\biggr] 
\\
&\hspace{15pt}  +  {\mathbb E}^0 \biggl[ \int_{0}^{T} \int_{{\mathbb T}^d} 
\nabla_y \delta_\mu {\mathcal U}^0\bigl(t({\boldsymbol \tau},s),\tilde X_{s}^0,\mu,y\bigr) 
\cdot 
\nabla_p \hat H\bigl(y,\nabla_x {\mathcal U}(s,\tilde X_s^0,y,\mu) \bigr) 
 \ud
\mu(y)  \ud s  
\biggr] 
\\
&\hspace{15pt} 
 -
 {\mathbb E}^0 \biggl[ \int_{0}^{T}  
  \Bigl( f_s^0(\tilde X_s^0,\mu) + \hat L^0\bigl(\tilde X_s^0,   \nabla_{x_0} {\mathcal U}^0(s,\tilde X_s^0,\mu)
  \bigr)  
\Bigr) \ud s   \biggr] + O\bigl(h^{({\mathbb r}-\lfloor {\mathbb s}\rfloor)/2}\bigr). 
\end{split} 
\end{equation*} 
Using the
boundedness and 
time continuity of 
$\Delta_y \delta_\mu {\mathcal U}^0$
and 
$\nabla_y \delta_\mu {\mathcal U}^0$ (recall in particular 
Proposition 
\ref{prop:reg:Holder:time:deribatives}) and observing 
that $\lim_{h \rightarrow 0} t({\boldsymbol \tau},s) = s$, we can easily let $h$ tend to $0$ in the formula.
By Lebesgue's dominated convergence theorem,  we get the expected formula. 
\end{proof} 

We have a similar result for ${\mathcal U}$: 

\begin{proposition}
\label{prop:representation:U:PDE}
Under Assumption \hyp{\^A}, there 
exists a constant ${\mathfrak C}>0$, only depending on the parameters 
$d$, $\kappa$, ${\mathfrak L}$,  
$\zeta$, $\sigma_0$ and 
$\mathbb{s}$
in Assumption \hyp{\^A} such that, for $T \leq {\mathfrak C}$, 
the following representation of 
${\mathcal U}(t_0,x_0,\mu,x)$ holds true, for any $(t_0,x_0,x,\mu) \in [0,T] 
\times {\mathbb R}^d \times {\mathbb T}^d \times {\mathcal P}({\mathbb T}^d)$:
\begin{equation}
\label{prop:representation:U0:PDE:11} 
\begin{split}
{\mathcal U}(t_0,x_0,x,\mu) 
&= {\mathbb E}^0 \bigl[ g(\tilde X_T^0,x,\mu) \bigr]
 - 
 {\mathbb E}^0 \biggl[ \tfrac12 \int_{t_0}^{T} \int_{{\mathbb T}^d} 
\Delta_y \delta_\mu {\mathcal U}\bigl(s,\tilde X_{s}^{0},x,\mu,y\bigr) \ud
\mu(y)  \ud s  
\biggr] 
\\
&\hspace{15pt}  +  {\mathbb E}^0 \biggl[ \int_{t_0}^{T} \int_{{\mathbb T}^d} 
\nabla_y \delta_\mu {\mathcal U}\bigl(s,\tilde X_{s}^0,x,\mu,y\bigr) 
\cdot 
\nabla_p \hat H\bigl(y,\nabla_x {\mathcal U}(s,\tilde X_s^0,y,\mu) \bigr) 
 \ud
\mu(y)  \ud s  
\biggr] 
\\
&\hspace{15pt} 
 + {\mathbb E}^0 \biggl[ \int_{0}^{T}  
  \Bigl( 
      -  \tfrac12 \Delta_x {\mathcal U}(s,\tilde X_s^0,x,\mu) + \hat H\bigl(x,\nabla_x 
  {\mathcal U}(s,\tilde X_s^0,x,\mu)
   \bigr)  - f_s(\tilde X_s^0,x,\mu)  
\Bigr) \ud s   \biggr],
\end{split} 
\end{equation} 
where $\tilde{\boldsymbol X}^0=(\tilde X^0_t)_{t_0 \le t \leq T}$ solves the 
SDE
\eqref{eq:tilde:X:0:U0}.
\end{proposition}

\begin{proof}
The proof is similar to the proof of 
Proposition 
\ref{prop:representation:U0:PDE}. 
In particular, we just establish the representation at time $t_0=0$. 
\vskip 4pt

\noindent \textit{First Step.} 
For a fixed $t \in [0,T]$, we reuse the notation
\eqref{eq:proof:edp:X0:0}
and the bound
\eqref{eq:proof:edp:X0:1}. 
Following 
\eqref{eq:proof:edp:X0:2},
we deduce from the Lipschitz regularity of 
${\mathcal U}$
(see 
\eqref{eq:local:diff-1:again})
 that  
\begin{align} 
{\mathbb E}^0 \bigl[ {\mathcal U}(t+h,\tilde X_{t+h}^{0},x,\mu) 
\, \vert \, {\mathcal F}_t^0 \bigr] 
&=
{\mathbb E}^0 \bigl[ {\mathcal U}(t+h,\tilde X_{t+h}^{0},x,\mu_t) 
\, \vert \, {\mathcal F}_t^0 \bigr] \nonumber
\\
&= {\mathbb E}^0 \bigl[ {\mathcal U}(t+h, X_{t+h}^{0},x,\mu_{t+h}) 
\, \vert \, {\mathcal F}_t^0 \bigr] \nonumber
\\
&\hspace{15pt} +
{\mathbb E}^0 \bigl[ {\mathcal U}(t+h,\tilde X_{t+h}^{0},x,\mu_t) 
 - {\mathcal U}(t+h, X_{t+h}^0,x,\mu_{t+h}) 
\, \vert \, {\mathcal F}_t^0 \bigr] 
\label{eq:proof:edp:X:2}
\\
&= {\mathbb E}^0 \bigl[{\mathcal U}(t+h, X_{t+h}^{0},x,\mu_{t+h}) 
\, \vert \, {\mathcal F}_t^0 \bigr] 
\nonumber
\\
&\hspace{15pt} +
{\mathbb E}^0 \bigl[ {\mathcal U}(t+h, X_{t+h}^{0},x,\mu_t) 
- {\mathcal U}(t+h, X_{t+h}^0,x,\mu_{t+h}) 
\, \vert \, {\mathcal F}_t^0 \bigr] 
+ O(h^{3/2}), 
\nonumber
\end{align} 
where 
$\vert O(h^{3/2}) \vert \leq C h^{3/2}$. 
\vskip 4pt

\noindent \textit{Second Step.} 
We follow the second step in the proof of Proposition 
\ref{prop:representation:U0:PDE} and 
address the term on the last line of 
\eqref{eq:proof:edp:X:2}. Using the regularity of 
${\mathcal U}$ in the third argument 
(see Proposition 
\ref{prop:deltaU:continuity})
together with 
the equation for ${\boldsymbol \mu}$ 
in 
\eqref{eq:minor:FB:2:small}, we have 
\begin{equation*} 
\begin{split} 
&{\mathcal U}(t+h, X_{t+h}^0,x,\mu_{t+h}) 
- 
{\mathcal U}(t+h, X_{t+h}^{0},x,\mu_t) 
\\
&= \tfrac12 \int_{t}^{t+h} \int_{{\mathbb T}^d} 
\Delta_y \delta_\mu {\mathcal U}^0(t+h,X_{t+h}^0,x,\mu_s,y) \ud
\mu_s(y)  \ud s 
\\
&\hspace{15pt} - \int_t^{t+h} \int_{{\mathbb T}^d} 
\nabla_y \delta_\mu {\mathcal U}^0(t+h,X_{t+h}^0,x,\mu_s,y) 
\cdot 
\nabla_p \hat H\bigl(y,\nabla_x {\mathcal U} (s,X_s^0,y,\mu_s) \bigr) 
 \ud
\mu_s(y)  \ud s.
\end{split}
\end{equation*}  
And then, using 
\eqref{eq:proof:edp:X0:1} 
together with 
the Lipschitz regularity 
of $\Delta_y \delta_\mu {\mathcal U}$, $\nabla_y \delta_\mu {\mathcal U}$ and $\nabla_x {\mathcal U}$
in the variables $x_0$ and $\mu$ and their H\"older regularity in $y$, we deduce that 
\begin{align}
&{\mathbb E}^0 \Bigl[ {\mathcal U}(t+h, X_{t+h}^0,x,\mu_{t+h}) 
- 
{\mathcal U}(t+h, X_{t+h}^{0},x,\mu_t) \, \vert \,  {\mathcal F}_t^0 \Bigr]
\nonumber
\\
&= {\mathbb E}^0 \biggl[ \tfrac12 \int_{t}^{t+h} \int_{{\mathbb T}^d} 
\Delta_y \delta_\mu {\mathcal U}(t+h,\tilde X_{s}^{0},x,\mu_t,y) \ud
\mu_t(y)  \ud s \label{eq:proof:edp:X:4} 
\\
&\hspace{15pt} - \int_t^{t+h} \int_{{\mathbb T}^d} 
\nabla_y \delta_\mu {\mathcal U}(t+h,\tilde X_{s}^0,x,\mu_t,y) 
\cdot 
\nabla_p \hat H\bigl(y,\nabla_x {\mathcal U} (s,\tilde X_s^0,y,\mu_t) \bigr) 
 \ud
\mu_t(y)  \ud s \, \vert \, {\mathcal F}_t^0 
\biggr] + O\bigl(h^{1+({\mathbb r}-\lfloor {\mathbb s}\rfloor)/2}\bigr). \nonumber
\end{align} 
We then handle the term on the penultimate line of 
\eqref{eq:proof:edp:X:2}. 
Using the 
backward equation in \eqref{eq:minor:FB:2:small} (with the prescribed initial conditions for the two forward equations in \eqref{eq:major:FB:1:small} and \eqref{eq:minor:FB:2:small}) and 
\eqref{eq:local:diff-1:again}
again, 
we have
\begin{equation}
\label{eq:proof:edp:X:3}
\begin{split}
&{\mathbb E}^0 \bigl[ {\mathcal U}(t+h, X_{t+h}^{0},x,\mu_{t+h}) 
\, \vert \, {\mathcal F}_t^0 \bigr]
 =  {\mathbb E}^0 \bigl[ u_{t+h}(x)
\, \vert \, {\mathcal F}_t^0 \bigr]
\\
&= u_t(x)
  +
 {\mathbb E}^0 \biggl[ \int_t^{t+h}  
  \Bigl( 
  -  \tfrac12 \Delta_x u_s(x) + \hat  H\bigl(x,\nabla_x u_s(x) \bigr)  - f_s(X_s^0,x,\mu_s)  \Bigr) \ud s  
 \, \vert \, {\mathcal F}_t^0 \biggr]
\\
&= {\mathcal U}(t,X_t^0,x,\mu_t)
\\
&\hspace{10pt} 
  +
 {\mathbb E}^0 \biggl[ \int_t^{t+h}  
  \Bigl( 
  -  \tfrac12 \Delta_x {\mathcal U}(s,X_s^0,x,\mu_s) + \hat H\bigl(x,\nabla_x 
  {\mathcal U}(s,X_s^0,x,\mu_s)
   \bigr)  - f_s(X_s^0,x,\mu_s)  \Bigr) \ud s  
 \, \vert \, {\mathcal F}_t^0 \biggr]
 \\
&= {\mathcal U}(t,X_t^0,x,\mu_t)  
\\
&\hspace{10pt} 
  +
 {\mathbb E}^0 \biggl[ \int_t^{t+h}  
  \Bigl( 
  -  \tfrac12 \Delta_x {\mathcal U}(s,\tilde X_s^0,x,\mu_t) + \hat H\bigl(x,\nabla_x 
  {\mathcal U}(s,\tilde X_s^0,x,\mu_t)
   \bigr)  - f_s(\tilde X_s^0,x,\mu_t)  
\Bigr) \ud s \, \vert \, {\mathcal F}_t^0 \biggr] 
 + O(h^{3/2}). 
\end{split}
\end{equation}
\vskip 4pt

\noindent \textit{Third Step.} 
We now insert
\eqref{eq:proof:edp:X:4}
and
\eqref{eq:proof:edp:X:3}
into 
\eqref{eq:proof:edp:X:2}. 
We obtain 
\begin{equation} 
\label{eq:proof:edp:X:5}
\begin{split} 
&{\mathbb E}^0 \bigl[ {\mathcal U}(t+h,\tilde X_{t+h}^{0},x,\mu_t) 
\, \vert \, {\mathcal F}_t^0 \bigr] 
= {\mathcal U}^0(t,\tilde X_t^0,x,\mu)  
\\
&\hspace{15pt} - 
 {\mathbb E}^0 \biggl[ \tfrac12 \int_{t}^{t+h} \int_{{\mathbb T}^d} 
\Delta_y \delta_\mu {\mathcal U}(t+h,\tilde X_{s}^{0},x,\mu,y) \ud
\mu(y)  \ud s \, \vert \, {\mathcal F}_t^0 
\biggr] 
\\
&\hspace{15pt} +  {\mathbb E}^0 \biggl[ \int_t^{t+h} \int_{{\mathbb T}^d} 
\nabla_y \delta_\mu {\mathcal U}(t+h,\tilde X_{s}^0,x,\mu,y) 
\cdot 
\nabla_p \hat H\bigl(y,\nabla_x {\mathcal U}(s,\tilde X_s^0,y,\mu) \bigr) 
 \ud
\mu(y)  \ud s \, \vert \, {\mathcal F}_t^0 
\biggr] 
\\
&\hspace{15pt} 
 +
 {\mathbb E}^0 \biggl[ \int_t^{t+h}  
  \Bigl(
    -  \tfrac12 \Delta_x {\mathcal U}(s,\tilde X_s^0,x,\mu) + \hat H\bigl(x,\nabla_x 
  {\mathcal U}(s,\tilde X_s^0,x,\mu)
   \bigr)  - f_s(\tilde X_s^0,x,\mu)  
\Bigr) \ud s \, \vert \, {\mathcal F}_t^0 \biggr] 
\\
&\hspace{15pt}  + O\bigl(h^{1+({\mathbb r}-\lfloor {\mathbb s}\rfloor)/2}\bigr).
\end{split}
\end{equation}  

We are now given  a subdivision
${\boldsymbol \tau} = (t_i)_{i=0,\cdots,n}$
of $[0,T]$, with 
 $0=t_0<t_1<\cdots<t_n=T$
 and 
with stepsize $h$. Then, \eqref{eq:proof:edp:X:5} says 
\begin{equation*} 
\begin{split} 
&{\mathbb E}^0 \bigl[ {\mathcal U}(t_{i+1},\tilde X_{t_{i+1}}^{0},x,\mu)  \bigr] 
= {\mathbb E}^0 \bigl[ {\mathcal U}(t_i,\tilde X_{t_i}^0,x,\mu) \bigr] 
\\
&\hspace{15pt} - 
 {\mathbb E}^0 \biggl[ \tfrac12 \int_{t_i}^{t_{i+1}} \int_{{\mathbb T}^d} 
\Delta_y \delta_\mu {\mathcal U}(t_{i+1},\tilde X_{s}^{0},x,\mu,y) \ud
\mu(y)  \ud s  
\biggr] 
\\
&\hspace{15pt} +  {\mathbb E}^0 \biggl[ \int_{t_i}^{t_{i+1}} \int_{{\mathbb T}^d} 
\nabla_y \delta_\mu {\mathcal U}(t_{i+1},\tilde X_{s}^0,x,\mu,y) 
\cdot 
\nabla_p \hat H\bigl(y,\nabla_y {\mathcal U}(s,\tilde X_s^0,y,\mu) \bigr) 
 \ud
\mu(y)  \ud s  
\biggr] 
\\
&\hspace{15pt} 
 +
 {\mathbb E}^0 \biggl[ \int_{t_i}^{t_{i+1}}  
  \Bigl( 
      -  \tfrac12 \Delta_x {\mathcal U}(s,\tilde X_s^0,x,\mu) + \hat H\bigl(x,\nabla_x 
  {\mathcal U}(s,\tilde X_s^0,x,\mu)
   \bigr)  - f_s(\tilde X_s^0,x,\mu)  
\Bigr) \ud s   \biggr] +  O\bigl(h^{1+({\mathbb r}-\lfloor {\mathbb s}\rfloor)/2}\bigr).
\end{split}
\end{equation*}  
Letting $t({\boldsymbol \tau},s) := t_{i+1}$ if 
$s \in (t_i,t_{i+1}]$, for $i=0,\dots,n-1$, 
and summing over $i$ 
the above expansion, we get 
\begin{equation*} 
\begin{split}
&{\mathcal U}(0,x_0,x,\mu) 
= {\mathbb E}^0 \bigl[ g(\tilde X_T^0,x,\mu) \bigr]
\\
&\hspace{15pt} - 
 {\mathbb E}^0 \biggl[ \tfrac12 \int_{0}^{T} \int_{{\mathbb T}^d} 
\Delta_y \delta_\mu {\mathcal U}(t({\boldsymbol \tau},s),\tilde X_{s}^{0},x,\mu,y) \ud
\mu(y)  \ud s  
\biggr] 
\\
&\hspace{15pt} +  {\mathbb E}^0 \biggl[ \int_{0}^{T} \int_{{\mathbb T}^d} 
\nabla_y \delta_\mu {\mathcal U}(t({\boldsymbol \tau},s),\tilde X_{s}^0,x,\mu,y) 
\cdot 
\nabla_p \hat H\bigl(y,\nabla_y {\mathcal U}(s,\tilde X_s^0,y,\mu) \bigr) 
 \ud
\mu(y)  \ud s  
\biggr] 
\\
&\hspace{15pt} 
 +
 {\mathbb E}^0 \biggl[ \int_{0}^{T}  
  \Bigl( 
      -  \tfrac12 \Delta_x {\mathcal U}(s,\tilde X_s^0,x,\mu) +\hat H\bigl(x,\nabla_x 
  {\mathcal U}(s,\tilde X_s^0,x,\mu)
   \bigr)  - f_s(\tilde X_s^0,x,\mu)  
\Bigr) \ud s   \biggr]  + O\bigl(h^{({\mathbb r}-\lfloor {\mathbb s}\rfloor)/2}\bigr).
\end{split} 
\end{equation*} 
Using the
boundedness and 
time continuity 
of all the derivatives 
appearing above (see 
Lemma
\ref{lem:U0:U:timereg}
and
Proposition 
\ref{prop:reg:Holder:time:deribatives}) 
and 
 observing 
that $\lim_{h \rightarrow 0} t({\boldsymbol \tau},s) = s$, we can easily let $h$ tend to $0$ in the formula.
By Lebesgue's dominated convergence theorem,  we get the expected formula. 
\end{proof}

The following result makes the connection between Propositions \ref{prop:representation:U0:PDE}
and 
\ref{prop:representation:U:PDE}
and the master equation.
\begin{corollary}
\label{corollary:U0:PDE}
Let  Assumption \hyp{\^A}
be in force. Assume also that:
\begin{enumerate}[i.]
\item  the coefficient $(t,x_0,\mu) \mapsto 
f_t^0(x_0,\mu)$ is H\"older continuous in time, uniformly in 
$(x_0,\mu)$; the coefficient $(t,x_0,x,\mu) \mapsto 
f_t(x_0,x,\mu)$ is H\"older continuous in time, 
uniformly in $(x_0,x,\mu)$;
\item $x_0 \mapsto g^0(x_0,\mu)$ has 
H\"older continuous second-order derivatives,
uniformly in $\mu$; $x_0 \mapsto g(x_0,x,\mu)$
has H\"older continuous second-order derivatives,  
uniformly in $(x,\mu)$. 
 \end{enumerate}
 Then, 
there 
exists a constant ${\mathfrak C}>0$, only depending on the parameters 
$d$, ${\mathfrak L}$,  $\kappa$, 
$\zeta$, $\sigma_0$ and 
$\mathbb{s}$
in Assumption \hyp{\^A} such that, for $T \leq {\mathfrak C}$, 
the system of two master equations
\begin{equation} 
\label{eq:major:1:small}
\begin{split} 
&\partial_t V^0(t,x_0,\mu) + 
\tfrac{1}2 \sigma_0^2 \Delta_{x_0} 
V^0(t,x_0,\mu) - \nabla_p H^0 \bigl( x_0, 
  \nabla_{x_0} V^0(t,x_0,\mu) \bigr)  \cdot  \nabla_{x_0} V^0(t,x_0,\mu) 
  \\
&\hspace{15pt}  +
  \hat{L}^0 \bigl( x_0, 
  \nabla_{x_0} V^0(t,x_0,\mu) \bigr)  
  + f_t^0(x_0,\mu)
\\
&\hspace{15pt}+ 
\int_{{\mathbb T}^d}
\bigl\{ \tfrac{1}2
{\rm div}_y(\partial_{\mu}V^0(t,x_0,\mu,y)) -\partial_\mu V^0(t,x_0,\mu,y) \cdot \nabla_p \hat H \bigl(y, \nabla_x V(t,x_0,y,\mu) \bigr) \bigr\} \ud \mu(y) =0, 
\\
&V^0(T,x_0,\mu) = g^0(x_0,\mu),
\end{split}
\end{equation}
for $(t,x_0,\mu) \in [0,T] \times {\mathbb R}^d \times {\mathcal P}({\mathbb T}^d)$, and 
\begin{equation} 
\label{eq:minor:1:small}
\begin{split} 
&\partial_t V(t,x_0,x,\mu) + 
\tfrac{1}2 
\Delta_x 
V(t,x_0,x,\mu) +\tfrac{1}{2} \sigma_0^2 \Delta_{x_0}V(t,x_0,x,\mu)- \hat H \bigl( x_0,x, \nabla_x V(t,x_0,x,\mu) \bigr)+  f_t(x_0,x,\mu)
\\
&\hspace{15pt} - \nabla_p H^0\big( x_0,  \nabla_{x_0} V^0(t,x_0,\mu)\big) \cdot \nabla_{x_0} V(t,x_0,x,\mu) 
\\
&\hspace{15pt} + \int_{{\mathbb T}^d} \bigl\{ \tfrac12 {\rm div}_y(\partial_\mu V(t,x_0,x,\mu,y)) -\partial_\mu V(t,x_0,x,\mu,y) \cdot 
\nabla_p \hat H \bigl( y, \nabla_x V(t,x_0,y,\mu)\bigr) \bigr\} \ud \mu(y) =0, 
\\
&V(T,x_0,x,\mu) = g(x_0,x,\mu),
\end{split}
\end{equation} 
for $(t,x_0,x,\mu) \in [0,T] \times {\mathbb R}^d \times {\mathbb T}^d \times {\mathcal P}({\mathbb T}^d)$,
has a classical solution $(V^0,V)$ in the sense that 
\begin{enumerate}
\item $(t,x_0,\mu) \mapsto (\partial_t V^0(t,x_0,\mu),  \nabla_{x_0} V^0(t,x_0,\mu), \nabla^2_{x_0} V^0(t,x_0,\mu))$ is continuous on $[0,T] \times {\mathbb R}^d \times {\mathcal P}({\mathbb T}^d)$; 
$(t,x_0,\mu,y) \mapsto (\partial_\mu V^0(t,x_0,\mu,y), \nabla_y \partial_\mu V^0(t,x_0,\mu,y))$ is continuous on $[0,T] \times {\mathbb R}^d \times {\mathcal P}({\mathbb T}^d) \times {\mathbb T}^d$; 
\item $(t,x_0,x,\mu) \mapsto (\partial_t V(t,x_0,x,\mu),  \nabla_{x_0} V(t,x_0,x,\mu), \nabla^2_{x_0} V(t,x_0,x,\mu), 
 \nabla_{x} V(t,x_0,x,\mu), \nabla^2_{x} V(t,x_0,x,\mu))$ is continuous on $[0,T] \times {\mathbb R}^d \times {\mathbb T}^d \times {\mathcal P}({\mathbb T}^d)$; 
$(t,x_0,x,\mu,y) \mapsto (\partial_\mu V(t,x_0,x,\mu,y), \nabla_y \partial_\mu V(t,x_0,x,\mu,y))$ is continuous on $[0,T] \times {\mathbb R}^d \times {\mathbb T}^d \times {\mathcal P}({\mathbb T}^d) \times {\mathbb T}^d$, 
\item $(t,x_0,\mu) \mapsto \nabla_{x_0} V^0(t,x_0,\mu)$ and $(t,x_0,x,\mu) \mapsto ( \nabla_{x_0} V(t,x_0,x,\mu),
 \nabla_x V(t,x_0,x,\mu))$ are Lipschitz 
continuous with respect to $(x_0,\mu)$ and $(x_0,x,\mu)$ respectively 
(using the distance ${\mathbb W}_1$ to handle the argument $\mu$), uniformly in $t$.
\end{enumerate}
One solution is given by the master fields 
${\mathcal U}^0$ and 
${\mathcal U}$ 
introduced in 
\eqref{eq:U0:rep}
and
\eqref{eq:U:rep}. 
In particular, 
${\mathcal U}^0$ and ${\mathcal U}$ have 
first-order derivative in $t$ and 
second-order derivatives in $x_0$ and these derivatives are jointly continuous (with respect to all their arguments). 
Moreover, for any $t \in [0,T]$, 
the mapping 
$x_0 \mapsto {\mathcal U}^0(t,x_0,\mu)$ has 
H\"older continuous second-order derivatives,
uniformly in $\mu$, and the mapping $x_0 \mapsto {\mathcal U}^0(t,x_0,x,\mu)$
has H\"older continuous second-order derivatives, 
uniformly in $(x,\mu)$. 
\end{corollary}

\begin{proof}
We use Propositions \ref{prop:representation:U0:PDE}
and 
\ref{prop:representation:U:PDE}
to prove that ${\mathcal U}^0$ and 
${\mathcal U}$ solve equations \eqref{eq:major:1:small} and \eqref{eq:minor:1:small}
respectively. 

We start with \eqref{eq:major:1:small}. 
We claim that all the terms entering the time integrals in the right-hand side of 
\eqref{prop:representation:U0:PDE:00}
are H\"older continuous in the pair $(s,\tilde X_s^0)$, uniformly in the other parameters. 
For the term on the first line of 
\eqref{prop:representation:U0:PDE:00}, this follows from the fact that, for all 
$t \in [0,T]$, 
${\mathcal U}^0(t,\cdot,\cdot)$ belongs to ${\mathscr D}^0(C,2)$ for some $C \geq 0$
(see Proposition \ref{prop:deltaU:continuity})
and that its derivatives are H\"older continuous in time (see 
Proposition 
\ref{prop:reg:Holder:time:deribatives}).
For the second line of
\eqref{prop:representation:U0:PDE:00}, 
the regularity of the term 
$\nabla_y \delta_\mu {\mathcal U}^0(s,\tilde X_s^0,\mu,y)$ follows from the same argument. 
The regularity 
of the term 
$\nabla_x  {\mathcal U}(s,\tilde X_s^0,y,\mu)$ also follows from the fact that ${\mathcal U}(t,\cdot,\cdot,\cdot)$ belongs to ${\mathscr D}(C,1,2)$
for all $t \in [0,T]$,
and that its derivatives in $x$ are H\"older continuous in time (see Lemma \ref{lem:U0:U:timereg}).
The argument is the same for the 
terms on 
the third line of 
\eqref{prop:representation:U0:PDE:00}. 

Moreover, we observe that the drift of the SDE
\eqref{eq:tilde:X:0:U0} is also H\"older continuous in the arguments 
$(t,\tilde X_t^0)$. 

And then we can interpret  
\eqref{prop:representation:U0:PDE:00}
as the Kolmogorov formula for 
representing ${\mathcal U}^0(t_0,x_0,\mu)$
as the solution to a second-order PDE, denoted $(E^{\mu})$, in the variables $(t_0,x_0) \in [0,T] \times {\mathbb R}^d$. 
This second-order PDE $(E^{\mu})$ is driven by the generator of 
\eqref{eq:tilde:X:0:U0}. It has $g^0(\cdot,\mu)$ as terminal condition and the remaining three terms 
in the right-hand side of \eqref{prop:representation:U0:PDE:00}
as source terms. By standard Schauder's theory for uniformly parabolic PDEs with H\"older continuous coefficients, the PDE $(E^{\mu})$ has a classical solution with second H\"older derivatives that are H\"older continuous 
in time and space. 
This says two things. First, for a fixed $\mu$, $(t_0,x_0) \mapsto {\mathcal U}^0(t_0,x_0,\mu)$ solves the 
PDE $(E^{\mu})$. 
By expanding all the terms, one easily recovers the master equation 
 \eqref{eq:major:1:small}. 
 In particular, ${\mathcal U}^0$ has a first-order derivative in time and
 second-order derivatives in $x_0$. 
 Second, those derivatives are jointly continuous in $(t,x_0,\mu)$ (the proof is given a few lines below). In particular, 
 ${\mathcal U}^0$ satisfies items 1 and 3 in the statement 
 (properties of the first order derivatives of ${\mathcal U}^0$ in $x_0$ and $\mu$ 
 follow directly from Propositions 
\ref{prop:deltaU:continuity}
and
\ref{prop:reg:Holder:time:deribatives}). 
Joint regularity of $\partial_t {\mathcal U}^0$ and $\nabla^2_{x_0} {\mathcal U}^0$ 
in $(t,x_0,\mu)$ follows
from Schauder's theory again: the $(t,x_0)$-H\"older norms 
of $\partial_t {\mathcal U}^0$ and $\nabla^2_{x_0} {\mathcal U}^0$
can be bounded independently of $\mu$; therefore, the latter two derivatives are locally compact 
(in the space of H\"older functions) when $\mu$ varies; using continuity of the coefficients 
of 
$(E^{\mu})$ with respect to the argument $\mu$, one easily deduces that 
$\partial_t {\mathcal U}^0$ and $\nabla^2_{x_0} {\mathcal U}^0$ 
are indeed jointly continuous (using in particular the fact that 
${\mathcal U}^0(t,\cdot,\cdot) \in {\mathscr D}^0(C,2+({\mathbb s}-\lfloor {\mathbb s} \rfloor)/2)$ for some $C \geq 0$). 

Thanks to Proposition \ref{prop:representation:U:PDE}, one can proceed in a similar way to prove that 
${\mathcal U}$ solves the master equation 
\eqref{eq:minor:1:small} and satisfies the properties 
claimed in the statement.

\end{proof}

\subsection{Back to the original coefficients}
\label{subse:original:coeff}

We now want to come back to the original system 
\eqref{eq:major:FB:1}--\eqref{eq:minor:FB:2}. 
To do this, we use the following lemma: 
\begin{lemma}
\label{lem:H:approx}
Assume that 
$L^0$ and 
$L$ 
are as in 
Assumption \hyp{A}
and define the corresponding Hamiltonians 
$H^0$ and $H$. 
Then, for any given $R >0$, there exist two 
coefficients $\hat{L}^0$ and $\hat{H}$ satisfying Assumption \hyp{\^A} with respect to 
a new constant $\kappa$
and a function $\zeta$ only depending on the parameters 
$\kappa$ and $\lambda$ 
in \hyp{A},  such that, 
for any $(x,x_0,p) \in {\mathbb T}^d \times {\mathbb R}^d \times {\mathbb R}^d$, 
with $\vert p \vert \leq R$, 
\begin{equation}
\label{eq:lem:H:approx}
\hat{H}(x,p) = H(x,p), \quad \hat{L}^0(x_0,p) = L^0(x_0,-\nabla_p H^0(x_0,p)).
\end{equation}
\end{lemma}

\begin{proof}
Without any loss of generality, we can assume $R>2$. We then introduce two auxiliary functions. We call $\zeta$ a smooth cut-off function 
from ${\mathbb R}^d$ to $[0,1]$ such that $\zeta(p)=1$ for $\vert p\vert \leq R$, 
$\zeta(p)=0$ for $\vert p\vert \geq R^2$, 
and
$\vert \nabla_{p} \zeta(p) \vert \leq c/R^2$ 
 and 
$\vert \nabla_{pp}^2 \zeta(p) \vert \leq c/R^4$ for any $p$, where $c>0$ is a universal constant. 
We call $\varphi$ an even smooth function from ${\mathbb R}$ to ${\mathbb R}$ such that 
$\varphi''$ is non-decreasing on $[0,R^2]$ and non-increasing
on $[R^2,+\infty)$ with 
$\varphi''(r)=1$ for 
$r \in [0,R]$, 
$\varphi''(r)=r/R$
for $r \in [2R,R^2]$,
and 
 $\varphi''(r) = 2R^5/r^2$ for $r \geq 2R^2$. Assuming that $\varphi(0)=0$ and $\varphi'(0)=0$, we have 
$\varphi(r)=r^2/2$ for $r \in [-R,R]$, 
$\varphi'(r) \geq r^2/(2R)$ for $r \in [2R,R^2]$ and 
$ \varphi'(r) \in (0, 4R^3]$ for all $r \in {\mathbb R} \setminus \{0\}$. We then write $\Phi(p)$ for $\varphi(\vert p\vert)$. 
The function $\Phi$ is smooth on ${\mathbb R}^d$ and 
\begin{equation} 
\label{eq:approx:new:construction:Phi}
\nabla^2_{pp} \Phi(p) =  
\frac{ \varphi'(\vert p \vert)}{\vert p \vert}  
\Bigl[ 1 - 
\frac{p}{\vert p \vert} \otimes  \frac{p}{\vert p \vert} \Bigr]
+
\varphi''(\vert p \vert)  \frac{p}{\vert p \vert} \otimes  \frac{p}{\vert p \vert}, \quad p \in {\mathbb R}^d \setminus \{0\}. 
\end{equation} 
In particular, $\nabla^2_{pp} \Phi(p)=I_d$ for $\vert p \vert \leq R$.
For $\vert p \vert \in [R,R^2]$, 
$\nabla^2_{pp} \Phi(p)\geq [\vert p \vert/(2R)]I_d$
(in the sense of comparison of symmetric matrices) and, 
for $\vert p \vert > R^2$, $\nabla^2_{pp} \Phi(p)$ is positive definite. 

With $\lambda'$ as in \hyp{A5}, we consider 
\begin{equation*} 
\hat{H}(x,p) := \Bigl( H(x,p) - \tfrac{\lambda'}4 \vert p \vert^2 \Bigr) \zeta(p) + 
\tfrac{\lambda'}{2}
\Phi(p). 
\end{equation*} 
For $\vert p \vert \leq R$, 
\begin{equation*} 
\hat{H}(x,p) = 
H(x,p) - \tfrac{\lambda'}4 \vert p \vert^2 + 
\tfrac{\lambda'}4 \vert p \vert^2 
= H(x,p). 
\end{equation*} 
Moreover, 
\begin{equation*} 
\begin{split}
\nabla^2_{pp} 
\hat{H}(x,p) 
&= \Bigl( \nabla^2_{pp} H(x,p) - \tfrac{\lambda'}2 I_d \Bigr) \zeta(p) + 
\Bigl( \nabla_{p} H(x,p) - \tfrac{\lambda'}2 p \Bigr) 
\otimes \nabla \zeta(p) 
+
\nabla \zeta(p) 
\otimes 
\Bigl( \nabla_{p} H(x,p) - \tfrac{\lambda'}2 p \Bigr) 
\\
&\hspace{15pt}
+ \Bigl( H(x,p) - \tfrac{\lambda'}4 \vert p \vert^2 \Bigr) \nabla^2_{pp} \zeta(p)
+ \tfrac{\lambda'}{2}
\nabla^2_{pp} \Phi(p).
\end{split}
\end{equation*}
And then, using \hyp{A5}, we know that 
(in the sense of comparison of symmetric matrices) 
\begin{equation*} 
\begin{split}
&\nabla^2_{pp} 
\hat{H}(x,p) \geq - 
C \Bigl(  \vert p \vert \times \frac{1}{R^2}
+
  \vert p \vert^2 \times \frac{1}{R^4}
  \Bigr) 
{\mathbf 1}_{\{\vert p \vert \in [R,R^2]\}}
I_d
+  \tfrac{\lambda'}{2} \nabla^2_{pp} \Phi(p),
\end{split}
\end{equation*}
for a constant $C$ only depending on the properties of $H$. 
And then, by 
\eqref{eq:approx:new:construction:Phi}, we get, 
for $\vert p \vert \in [R,R^2]$, 
$\nabla^2_{pp} 
\hat{H}(x,p) \geq 
[\lambda' \vert p \vert/(4R)
- C (  \vert p \vert /R^2
+
  \vert p \vert^2 /R^4 
)]  I_d = 
(\vert p \vert/R) [\lambda'/4
- C /R - 
 C \vert p \vert /R^3 
] I_d.$ 
For $R$ large enough, the right-hand side is positive. And, then, it is straightforward to deduce that 
$\hat{H}$ satisfies \hyp{\^A4}. Property 
\hyp{\^A3} follows from the fact that $\varphi$ has bounded derivatives (of order greater than 1) and thus $\Phi$ also has bounded derivatives (of order greater than 1). 

As for the construction of $\hat{L}^0(x_0,p)$, it suffices to consider
$L^0(x_0,-\zeta(p) \nabla_p H^0(x_0,p) )$. 
\end{proof}

\begin{proposition}
\label{prop:H:approx}
Assume that 
$(f_t^0)_{0 \le t \le T}$, 
$(f_t)_{0 \le t \le T}$, 
$L^0$ and 
$L$ 
are as in 
Assumptions \hyp{A1}, \hyp{A2} and \hyp{A4} 
and define the corresponding Hamiltonians 
$H^0$ and $H$. 
Consider $g^0$ and $g$ as in 
\hyp{\^A2}, for a certain ${\mathbb s} \in (3,\lfloor \curss \rfloor- (d/2+1)]$. 
Given $R_0\geq 2 {\mathfrak L}$ (with ${\mathfrak L}$ as in 
\hyp{\^A2}), there exists a constant
${\mathfrak C} >0$, only depending on 
$d$, $\kappa$, ${\mathfrak L}$, $\lambda$, $R_0$, $\sigma_0$
and $\mathbb{s}$, such that, for $T \leq {\mathfrak C}$, 
the system 
\eqref{eq:major:FB:1:small}--\eqref{eq:minor:FB:2:small}
with $\hat{H}$ and $\hat{L}^0$ given by Lemma 
\ref{lem:H:approx}
  is uniquely solvable (in the sense of 
Theorem 
\ref{thm:local:FB}) and satisfies the 
following conclusions for ${\mathbb r} :=( \lfloor {\mathbb s} \rfloor + {\mathbb s})/2$
and for another constant $C$, only depending
on the parameters in \hyp{A}:
\begin{enumerate}[i.]
\item For all $t \in [0,T]$, ${\mathcal U}^0(t,\cdot,\cdot) \in {\mathscr D}^0(C, {\mathbb r}  - 1)$
and 
${\mathcal U}(t,\cdot,\cdot,\cdot) \in {\mathscr D}(C,   {\mathbb r}  - 1,{\mathbb s})$;
\item With the same notation for the norm 
$\| \cdot \|_{{\mathbb r},{\mathbb r}-1}$ as in footnote
\ref{foo:2:26}, 
\begin{equation*}
\begin{split}
&\lim_{h\to 0}\Big(\big\|\nabla_{x_0} {\mathcal U}(t+h,x_0,\cdot,\mu)-\nabla_{x_0}{\mathcal U}(t,x_0,\cdot,\mu)\big\|_{\mathbb{r}}
 +
\big\| \delta_\mu {\mathcal U}(t+h,x_0,\cdot,\mu,\cdot)- \delta_\mu {\mathcal U}(t,x_0,\cdot,\mu,\cdot)\big\|_{\mathbb{r},\mathbb{r}-1}
\Bigr) =0,
\end{split}
\end{equation*}
and
\begin{equation*}
\begin{split}
&\lim_{h\to 0}\Big(\big|\nabla_{x_0} {\mathcal U}^0(t+h,x_0,\mu)-\nabla_{x_0}
{\mathcal U}^0(t,x_0,\mu)\big| +
\big\| \delta_\mu {\mathcal U}^0(t+h,x_0,\mu,\cdot)- \delta_\mu {\mathcal U}^0(t,x_0,\mu,\cdot)\big\|_{\mathbb{r}-1}\Big)=0;
\end{split}
\end{equation*}
\item The gradients
 $(t,x_0,\mu) \mapsto \nabla_{x_0} {\mathcal U}^0(t,x_0,\mu)$ 
 and
 $(t,x_0,x,\mu) \mapsto \nabla_{x} {\mathcal U}(t,x_0,x,\mu)$ 
 are bounded by 
$R_0$. In particular, the solution to \eqref{eq:major:FB:1:small}--\eqref{eq:minor:FB:2:small}
is also a solution to 
\eqref{eq:major:FB:1}--\eqref{eq:minor:FB:2} and any other solution to 
\eqref{eq:major:FB:1}--\eqref{eq:minor:FB:2} satisfying 
$\| \int_0^{\cdot} Z_s^0 \cdot \ud B_s^0 \|_{\rm BMO} \leq 
R_0$
and 
$\sup_{t \in [0,T]} \vert \nabla_x u_t\vert 
\leq R_0$
 coincides with the unique 
solution of 
\eqref{eq:major:FB:1:small}--\eqref{eq:minor:FB:2:small};
\item If the coefficients $f^0$, $f$, $g^0$ and $g$ satisfy the additional conditions of 
Corollary 
\ref{corollary:U0:PDE}, then 
$({\mathcal U}^0,{\mathcal U})$ solves the 
master equations 
\eqref{eq:major:1}--\eqref{eq:minor:1}. 
\end{enumerate}
\end{proposition} 

\begin{proof}
We let 
${\mathbb s}:=( \lfloor \curss \rfloor -(d/2+1)+3)/2$. For $R \geq R_0$ and $T \leq {\mathfrak C}$, 
with ${\mathfrak C}$ depending on $d$, ${\mathfrak L}$, $\lambda$, $\kappa$, $R$, 
$\sigma_0$ and 
${\mathbb s}$, the system 
\eqref{eq:major:FB:1:small}--\eqref{eq:minor:FB:2:small}
with $\hat{H}$ and $\hat{L}^0$ given by Lemma 
\ref{lem:H:approx} has a unique solution.  
Item (i) follows from 
Proposition 
\ref{prop:deltaU:continuity}. 
Item (ii) follows from 
\eqref{eq:U:contintime}
and
\eqref{eq:U0:contintime}.

The main challenge is to prove item (iii). 
We 
first prove the bound for $(t,x_0,x,\mu) \mapsto \nabla_{x_0} {\mathcal U}(t,x_0,x,\mu)$. To do so, we come back to the proof of the bound 
\eqref{eq:local:reg}. In fact, now that ${\boldsymbol u}$ has been found, 
we can regard the term $\hat H(x,\nabla_x u_t(x))$ 
as part of the source term in 
\eqref{eq:minor:FB:2:small}, with an explicit bound (depending on $R$). 
The idea is that the influence of the source term is very small as $t$ gets close to $T$. 
In particular, we can use any standard estimates for linear backward SPDEs, see for instance 
\cite[Lemma 4.3.7]{CardaliaguetDelarueLasryLions}, to  
a get a bound for $\vert \nabla_x u_t \vert$ of the form ${\mathfrak L} + C_R(T-t)$, where 
$C_R$ depends on $R$. Briefly, ${\mathfrak L}$ comes from the boundary condition 
$g$ and $C_R$ comes from the source term. For $T-t$ small enough, 
$C_R(T-t)$ can be bounded by ${\mathfrak L}$. Equivalently, we can decrease the value of 
${\mathfrak C}$ so that $(t,x_0,x,\mu) \mapsto \nabla_{x} {\mathcal U}(t,x_0,x,\mu)$ is bounded by $R_0$. 
The proof of the bound for 
$(t,x_0,\mu) \mapsto \nabla_{x_0} {\mathcal U}^0(t,x_0,\mu)$
is very similar, even though slightly more difficult. 
The first step is to use the representation of 
$\nabla_{x_0} {\mathcal U}^0$
in terms of $K^{0,x_0}$, see
Proposition \ref{prop:local:1stdiff},
and then to return to the interpretation of the latter 
in terms of $\delta {\boldsymbol Y}^0$, 
see
\eqref{eq:all:the:Ks}. 
The point is thus to come back to the equation 
\eqref{eq:linear H:local}
for $\delta {\boldsymbol Y}^0$. 
At this stage, we already have a bound for 
$\delta {\boldsymbol X}^0$ and 
$\delta {\boldsymbol Z}^0$. Therefore, we can also obtain a bound 
for $\vert \delta {Y}^0_t \vert$ of the form 
${\mathfrak L} + C_R(T-t)$
and then 
complete the proof as done for 
$\vert \nabla_x {\boldsymbol u} \vert$.

We now turn to the rest of the claim (iii). 
The fact that the solution to 
\eqref{eq:major:FB:1:small}--\eqref{eq:minor:FB:2:small}
is also a solution to 
\eqref{eq:major:FB:1}--\eqref{eq:minor:FB:2} 
 is easily checked and follows 
from 
\eqref{eq:lem:H:approx} (by the way,
the argument is similar for item (iv)). 
It thus remains 
to prove that uniqueness to 
\eqref{eq:major:FB:1}--\eqref{eq:minor:FB:2} 
holds true within the class of 
solutions satisfying 
$\| \int_0^{\cdot} Z_s^0 \cdot \ud B_s^0 \|_{\rm BMO} \leq 
R_0$
(which condition is obviously satisfied by the solution 
to 
\eqref{eq:major:FB:1:small}--\eqref{eq:minor:FB:2:small}
since $T \leq 1$)
and 
$\sup_{t \in [0,T]} \vert \nabla_x u_t\vert 
\leq R_0$.
Under the latter bound, we can replace $H(x,\nabla_x u_t)$
by $\hat{H}(x,\nabla_x u_t)$ in 
the system \eqref{eq:minor:FB:2}.
This makes it possible to identify 
the two systems \eqref{eq:minor:FB:2}
and 
\eqref{eq:minor:FB:2:small} (i.e., the systems for the minor player). The difficulty 
comes from the fact that 
we cannot do the same for the major player and identify the systems 
\eqref{eq:major:FB:1}
and 
\eqref{eq:major:FB:1:small}: indeed, we cannot replace $L^0(X_t^0,-\nabla_p H^0(X_t^0,Z_t^0))$
by 
$\hat{L}^0(X_t^0,Z_t^0)$ because the 
BMO bound on ${\boldsymbol Z}^0$ 
is too weak to do so. To overcome this issue, 
we must revisit the proof of 
Theorem 
\ref{thm:local:FB}
with 
$\hat{L}^0(x_0,p)$ being replaced by 
$L^0(x_0,-\nabla_p H^0(x_0,p))$ (but with the same $\hat{H}$ as therein):
since the difficulty only comes from the new `major' system
\eqref{eq:major:FB:1:small}, 
 the only point is to prove that
\eqref{eq:local:diff2}
remains true
in this new setting with the additional
assumption that $\tilde {\boldsymbol X}^{0,i} = {\boldsymbol X}^{0,i}$ for 
$i=1,2$ (since we are dealing with uniqueness,
we can restrict the analysis to fixed points
of the mapping ${\mathfrak T}$, 
with the latter being
defined right below 
\eqref{eq:major:FB:tilde2}).
Here is the way we get the analogue of 
 \eqref{eq:local:diff2}. Removing the `tilde' in
the second step of 
the proof of Theorem 
\ref{thm:local:FB}, we first observe 
from the Lipschitz property of 
$\nabla_p H^0$ and Gronwall's lemma that
\begin{equation}
\label{eq:new:proof:uniqueness??}
\sup_{t \in [0,T]} 
\vert X_t^{0,1} - X_t^{0,2} \vert 
\leq C \int_0^T \vert Z_t^{0,1} - Z_t^{0,2} 
\vert \ud t,
\end{equation} 
for a constant $C$ 
only 
depending 
on $d$, $\kappa$,
${\mathfrak L}$,  $\lambda$, $\sigma_0$ and ${\mathbb s}$, where here and below we assume without any loss of generality that $T \leq 1$ (equivalently ${\mathfrak C} \leq 1$). 
As for the 
backward equation, 
we
notice that there exist two progressively measurable processes $(\theta_t)_{0 \le t \le T}$ and 
$(\vartheta_t)_{0 \le t \le T}$, with values in 
${\mathbb R}^d$, such that $\vert \theta_t \vert \leq C (1+ \vert Z_t^{0,1} 
\vert + \vert Z_t^{0,2} \vert)$ and $\vert \vartheta_t \vert \leq C$ for all $t \in [0,T]$ and a 
possibly new value of $C$, such that 
\begin{equation*} 
\begin{split}
\ud \bigl( Y_t^{0,1} - Y_t^{0,2} 
\bigr) &= - \Bigl[ f_t(X_t^{0,1},\mu_t^1) - f_t(X_t^{0,2},\mu_t^2) + \vartheta_t \cdot \bigl( X_t^{0,1} - X_t^{0,2} \bigr) \Bigr]
\ud t 
\\
&\hspace{15pt} - \theta_t \cdot \bigl( Z_t^{0,1} - Z_t^{0,2} \bigr) \ud t + 
\sigma_0 \bigl( Z_t^{0,1} - Z_t^{0,2} \bigr) \cdot \ud B_t^0, \quad t \in [0,T]. 
\end{split} 
\end{equation*} 
We then let
\begin{equation*} 
{\mathcal E}_t :=
{\mathcal E}_t \biggl(  \sigma_0^{-1} \int_0^{\cdot} \theta_s \cdot \ud B_s^0
\biggr), \quad t \in [0,T]. 
\end{equation*} 
And then, by the BMO properties of ${\boldsymbol Z}^{0,1}$ and ${\boldsymbol Z}^{0,2}$, 
we can apply Girsanov theorem and then obtain 
\begin{equation*} 
\begin{split}
 Y_t^{0,1} - Y_t^{0,2} 
&= 
{\mathcal E}_t^{-1} 
{\mathbb E}^0 
\biggl\{ {\mathcal E}_T  \biggl[ g^0(X_T^{0,1},\mu_T^1) - 
g^0(X_T^{0,2},\mu_T^2) 
+ 
\int_t^T \bigl[ f^0_s(X_s^{0,1},\mu_s^1)
-
f^0_s(X_s^{0,2},\mu_s^2)
\bigr] \ud s
\\
&\hspace{15pt} + \int_t^T \vartheta_s \cdot \bigl( X_s^{0,1} - X_s^{0,2} 
\bigr) \ud s \biggr] \, \vert \, {\mathcal F}_t^0 \biggr\}, \quad t \in [0,T]. 
\end{split}
\end{equation*} 
In fact, the BMO norm of ${\boldsymbol \theta}$ is bounded by a known constant. We deduce from 
\cite[Theorem 3.1]{Kazamaki} 
that there exist two conjugate exponents $p,q>1$, only depending on 
$d$, $\kappa$,
${\mathfrak L}$,  $\lambda$, $\sigma_0$ and ${\mathbb s}$, such that
\begin{equation*} 
{\mathbb E}^0 \Bigl[ \Bigl( {\mathcal E}_t^{-1} {\mathcal E}_T\Bigr)^q \vert \, {\mathcal F}_t^0 \Bigr] \leq C,
\end{equation*} 
the value of $C$ being allowed to vary from line to line. 
Then, by H\"older inequality,
\begin{equation} 
\label{eq:new:estimate:HHH}
\begin{split}
\vert Y_t^{0,1} - Y_t^{0,2} \vert
&\leq C 
{\mathbb E}^0 
\biggl[ \sup_{t \leq s \leq T} \vert X_s^{0,1} - X_s^{0,2} \vert^p 
+ 
\sup_{t \leq s \leq T}   {\mathbb W}_1(\mu_s^1,\mu_s^2)^p
 \, \vert \, {\mathcal F}_t^0 \biggr]^{1/p}
 \\
&\leq C 
{\mathbb E}^0 
\biggl[ \sup_{t \leq s \leq T} \vert X_s^{0,1} - X_s^{0,2} \vert^{2p} 
+ 
\sup_{t \leq s \leq T}   {\mathbb W}_1(\mu_s^1,\mu_s^2)^{2p}
 \, \vert \, {\mathcal F}_t^0 \biggr]^{1/(2p)}. 
\end{split}
\end{equation} 
Also, by squaring the difference 
${\boldsymbol Y}^{0,1} - {\boldsymbol Y}^{0,2}$ (in \eqref{eq:major:FB:1}),
we obtain the following variant of \cite[Proposition 2.2]{PARDOUX199055}: 
${\mathbb P}^0$-almost surely, 
for all $t \in [0,T]$, 
\begin{equation*}
\begin{split}
&\vert Y_t^{0,1} - Y_t^{0,2} \vert^2 
+ \tfrac12 \sigma_0^2 {\mathbb E}^0 
\biggl[ \int_t^T \vert Z_s^{0,1} - Z_s^{0,2} 
\vert^2 \ud s \, \vert \, {\mathcal F}_t^0
\biggr]
\\
&\leq C {\mathbb E} \biggl[ 
\sup_{s \in [0,T]} \vert X_s^{0,1} -X_s^{0,2} 
\vert^2 + 
\sup_{s \in [0,T]} 
{\mathbb W}_1(\mu_s^1,\mu_s^2)^2
+ 
  \sup_{s \in [t,T]} \vert Y_s^{0,1} -Y_s^{0,2} 
\vert^2 \int_t^T \bigl( 1+  \vert Z_s^{0,1} \vert 
+
\vert Z_s^{0,2} \vert 
\bigr)^2 \ud s 
\, \vert \, {\mathcal F}_t^0
\biggr].
\end{split}
\end{equation*} 
Here, using again the BMO property together with \cite[Theorem 2.2]{Kazamaki}, we can assume without any loss of generality that 
\begin{equation*} 
{\mathbb E} \biggl[ 
\biggl( 
 \int_t^T \bigl( 1+  \vert Z_s^{0,1} \vert 
+
\vert Z_s^{0,2} \vert 
\bigr)^2 \ud s \biggr)^{q}  
\, \vert \, {\mathcal F}_t^0
\biggr]
\leq C. 
\end{equation*}
And then, by a new application of H\"older inequality (assuming without any loss of generality that $p \in (2,+\infty)$ and $q \in (1,2)$), 
\begin{equation*}
\begin{split}
&\vert Y_t^{0,1} - Y_t^{0,2} \vert^2 
+ \tfrac12 \sigma_0^2 {\mathbb E}^0 
\biggl[ \int_t^T \vert Z_s^{0,1} - Z_s^{0,2} 
\vert^2 \ud s \, \vert \, {\mathcal F}_t^0
\biggr]
\\
&\leq C {\mathbb E} \biggl[ 
\sup_{s \in [0,T]} \vert X_s^{0,1} -X_s^{0,2} 
\vert^2 + 
\sup_{s \in [0,T]} 
{\mathbb W}_1(\mu_s^1,\mu_s^2)^2
\, \vert \, {\mathcal F}_t^0
\biggr]
+ C
{\mathbb E} \biggl[ 
  \sup_{s \in [t,T]} \vert Y_s^{0,1} -Y_s^{0,2} 
\vert^{2p}
\, \vert \, {\mathcal F}_t^0
\biggr]^{1/p}.
\end{split}
\end{equation*} 
Inserting \eqref{eq:new:estimate:HHH} and using Doob's inequality, we get 
\begin{equation*}
\begin{split}
&\vert Y_t^{0,1} - Y_t^{0,2} \vert^2 
+ \tfrac12 \sigma_0^2 {\mathbb E}^0 
\biggl[ \int_t^T \vert Z_s^{0,1} - Z_s^{0,2} 
\vert^2 \ud s \, \vert \, {\mathcal F}_t^0
\biggr]
\\
&\leq C {\mathbb E} \biggl[ 
\sup_{s \in [0,T]} \vert X_s^{0,1} -X_s^{0,2} 
\vert^{2p} + 
\sup_{s \in [0,T]} 
{\mathbb W}_1(\mu_s^1,\mu_s^2)^{2p}
\, \vert \, {\mathcal F}_t^0
\biggr]^{1/p}.
\end{split}
\end{equation*} 
Taking the power $p$ and then the supremum over $t \in [0,T]$, we get
\begin{equation*} 
\begin{split}
&{\mathbb E}^0 \biggl[ \sup_{t \in [0,T]} 
 {\mathbb E}^0 
\biggl[ \int_t^T \vert Z_s^{0,1} - Z_s^{0,2} 
\vert^2 \ud s \, \vert \, {\mathcal F}_t^0 
\biggr]^{p}
\, \vert \, {\mathcal F}_0^0 
\biggr]
\\
&\leq C T^{p} {\mathbb E}^0 \biggl[
\biggl( 
\int_0^T \vert Z_s^{0,1} -Z_s^{0,2} 
\vert^2
\ud s
\biggr)^{p}
\, \vert \, {\mathcal F}_0^0 
\biggr] +
C {\mathbb E}^0 \Bigl[ 
\sup_{s \in [0,T]} 
{\mathbb W}_1(\mu_s^1,\mu_s^2)^{2p}
\, \vert \, {\mathcal F}_0^0 
\Bigr],
\end{split}
\end{equation*} 
with the second line following from 
\eqref{eq:new:proof:uniqueness??}. 
Bounding from below the left-hand side by 
the value at $t=0$, 
we get 
(for a possibly new value of ${\mathfrak C} \leq 1$):
\begin{equation*} 
\begin{split}
&{\mathbb E}^0 \Bigl[ 
\sup_{s \in [0,T]} \vert X_s^{0,1} -X_s^{0,2} 
\vert^{2p} 
\Bigr]
\leq C T
{\mathbb E}^0 \Bigl[
\sup_{s \in [0,T]} 
{\mathbb W}_1(\mu_s^1,\mu_s^2)^{2p}
\Bigr].
\end{split}
\end{equation*} 
The fact that we use here 
$2p$ as exponent (and not $2$) makes a difference with 
\eqref{eq:local:diff2}. However, we can easily complete the proof of uniqueness 
by inserting the above bound in 
\eqref{eq:local:diff5} (with a possibly new value of $p$). 
\end{proof} 

\section{Appendix}

\subsection{Convenient form of the chain rule over 
${\mathbb R}^d \times {\mathcal P}({\mathbb T}^d)$}

In this subsection, we consider
 two filtered probability spaces 
$(\Omega^0,{\mathcal F}^0,{\mathbb F}^0,{\mathbb P}^0)$
and
$(\Omega,{\mathcal F},{\mathbb F},{{\mathbb P}})$ equipped with two 
Brownian motions $(B_t^0)_{0 \leq t \leq T}$ 
and $(B_t)_{0 \leq t \leq T}$ with values in ${\mathbb R}^d$.
We are also given two ${\mathbb R}^d$-valued 
It\^o processes
\begin{equation*} 
\begin{split} 
&\ud X_t^0 = b_t^0 \ud t + \varsigma_t^0 \ud B_t^0,
\\
&\ud X_t = b_t \ud t + \varsigma_t \ud B_t, \quad t \geq 0, 
\end{split}
\end{equation*}
with square-integrable conditions $X_0^0$ and $X_0$, respectively 
${\mathcal F}_0^0$ and ${\mathcal F}_0$-measurable.  
The processes $(b_t^0)_{t \geq 0}$ and $(\varsigma_t^0)_{t \geq 0}$ 
are constructed on $\Omega^0$ and are ${\mathbb F}^0$-progressively measurable
(with values in ${\mathbb R}^d$ and ${\mathbb R}^d \otimes {\mathbb R}^d$ respectively), and the processes 
$(b_t)_{t \geq 0}$ and $(\varsigma_t)_{t \geq 0}$ 
are constructed on 
$\Omega^0 \times \Omega$ 
and are 
${\mathbb F}^0 \otimes {\mathbb F}$-progressively measurable
(also with values in ${\mathbb R}^d$ and ${\mathbb R}^d \otimes {\mathbb R}^d$ respectively). 

We also assume that, for any $T>0$,  the processes 
$(\varsigma_t^0)_{0 \le t \le T}$ and 
$(\varsigma_t)_{0 \le t \le T}$
are bounded by a deterministic constant. And, we assume that 
\begin{equation*} 
{\mathbb E}^0 \Bigl[ \sup_{0 \le t \le T}  
\vert b_t^0 \vert^2 
\Bigr]
+
{\mathbb E}^0{\mathbb E}\Bigl[ \sup_{0 \le t \le T} 
\vert b_t \vert^2 
\Bigr]
< \infty. 
\end{equation*} 

Lastly, 
for any $t \geq 0$, 
we let 
\begin{equation*} 
\mu_t(\omega^0) := 
{\mathcal L}^0(X_t)(\omega^0), \quad \omega^0 \in \Omega^0, 
\end{equation*}
when $(X_t)_{t \geq 0}$ is regarded as a random variable with values in 
${\mathbb T}^d$. Equivalently, 
$\mu_t(\omega^0)$ is seen as an element of 
${\mathcal P}({\mathbb T}^d)$. 

\begin{proposition}
\label{prop:ito:formula}
Let $\ell : [0,+\infty) \times {\mathbb R}^d \times {\mathcal P}({\mathbb T}^d) \ni (t,x_0,\mu) \mapsto \ell(t,x_0,\mu) \in {\mathbb R}$ be differentiable with respect to $t$, 
$x_0$ and $\mu$ such that 
\begin{enumerate}
\item The functions $(t,x_0,\mu) \mapsto \partial_t \ell(t,x_0,\mu)$ and 
$(t,x_0,\mu) \mapsto \nabla_{x_0} \ell(t,x_0,\mu)$
are jointly continuous (with respect to ${\mathbb W}_1$ in the argument $\mu$); 
\item The function $[0,+\infty) \times {\mathbb R}^d \times {\mathcal P}({\mathbb T}^d) \times {\mathbb T}^d \ni (t,x_0,\mu,y) \mapsto \partial_\mu \ell(t,x_0,\mu,y)$
is jointly continuous (with respect to ${\mathbb W}_1$ in the argument  $\mu$)
and is differentiable with respect to $y$, 
the derivative 
$[0,+\infty) \times {\mathbb R}^d \times {\mathcal P}({\mathbb T}^d) \times {\mathbb T}^d \ni (t,x_0,\mu,y) \mapsto \nabla_y \partial_\mu \ell(t,x_0,\mu,y)$
being also jointly continuous. 
\end{enumerate}
Then, ${\mathbb P}^0$-almost surely, for any $t \geq 0$, 
\begin{equation*} 
\begin{split}
\ud_t \bigl[ \ell \bigl(t, X_t^0 , \mu_t \bigr) \bigr]
&=
\biggl[ 
\partial_t \ell \bigl(t, X_t^0 , \mu_{t} \bigr)
+
{\mathbb E} \bigl[ b_t \cdot \partial_\mu \ell (t,X_t^0,\mu_t,X_t)
\bigr] 
+
\tfrac12
{\mathbb E} \bigl[ {\rm Tr} \bigl( \varsigma_t \varsigma_t^\dagger  \nabla_y \partial_\mu \ell (t,X_t^0,\mu_t,X_t)
\bigr) 
\bigr] 
\\
&\hspace{5pt} + 
b_t^0 \cdot \nabla_{x_0} \ell(t,X_t^0,\mu_t) 
+
\tfrac12 
{\rm Tr} \Bigl( 
\varsigma_t^0 \bigl( \varsigma_t^0 \bigr)^\dagger \nabla_{x_0}^2 \ell(t,X_t^0,\mu_t)  \Bigr)
\biggr] \ud t + 
\bigl( \varsigma_t^0 \bigr)^{\dagger} \nabla_{x_0} \ell (t,X_t^0,\mu_t) \cdot \ud B_t^0.
\end{split}
\end{equation*}

\end{proposition}

\begin{proof}
Fix $T>0$. It suffices to prove the formula for $t \in [0,T]$. 
For a mesh $0=t_0<t_1<\cdots<t_n=t$ of the interval $[0,t]$, we 
have
\begin{equation*}
\begin{split}
&\ell \bigl(t, X_t^0 , \mu_t \bigr)  - 
\ell \bigl(0, X_0^0 , \mu_0 \bigr)
\\
&= \sum_{i=1}^n 
\Bigl[ 
\bigl\{ 
\ell \bigl(t_i, X_{t_i}^0 , \mu_{t_i} \bigr) 
-
\ell \bigl(t_{i-1}, X_{t_i}^0 , \mu_{t_{i-1}} \bigr) 
\bigr\}
+
\bigl\{ 
\ell \bigl(t_{i-1}, X_{t_i}^0 , \mu_{t_{i-1}} \bigr) 
-
\ell \bigl(t_{i-1}, X_{t_{i-1}}^0 , \mu_{t_{i-1}} \bigr) 
\bigr\}
\Bigr].
\end{split}
\end{equation*}
By freezing $\omega^0$, we can expand 
$\ell (t_i, X_{t_i}^0 , \mu_{t_i}  ) 
-
\ell (t_{i-1}, X_{t_i}^0 , \mu_{t_{i-1}} )$ by means of the standard chain rule 
on $[0,+\infty) \times {\mathcal P}({\mathbb T}^d)$, see \cite[Theorem 5.99]{CarmonaDelarue_book_I}. 
For $i \in \{1,\cdots,n\}$,
we obtain 
\begin{equation*}
\begin{split} 
&\ell \bigl(t_i, X_{t_i}^0 , \mu_{t_i} \bigr) 
-
\ell \bigl(t_{i-1}, X_{t_i}^0 , \mu_{t_{i-1}} \bigr) 
\\
&= 
\int_{t_{i-1}}^{t_i} 
\partial_t \ell \bigl(s, X_{t_i}^0 , \mu_{s} \bigr)
\ud s  
+
\int_{t_{i-1}}^{t_i} 
{\mathbb E} \bigl[ b_s \cdot \partial_\mu \ell (s,X_{t_i}^0,\mu_s,X_s) 
\bigr] \ud s
+
\frac12
\int_{t_{i-1}}^{t_i} 
{\mathbb E} \bigl[ {\rm Tr} \bigl(  \varsigma_s \varsigma_s^\dagger  \nabla_y \partial_\mu \ell (s,X_{t_i}^0,\mu_s,X_s)
\bigr) 
\bigr] \ud s
\\
&= \int_{t_{i-1}}^{t_i} 
\partial_t \ell \bigl(s, X_s^0 , \mu_{s} \bigr)
\ud s  
+
\int_{t_{i-1}}^{t_i} 
{\mathbb E} \bigl[ b_s \cdot \partial_\mu \ell (s,X_s^0,\mu_s,X_s) 
\bigr] \ud s
+
\frac12
\int_{t_{i-1}}^{t_i} 
{\mathbb E} \bigl[ {\rm Tr} \bigl(  \varsigma_s \varsigma_s^\dagger  \nabla_y \partial_\mu \ell (s,X_s^0,\mu_s,X_s)
\bigr) 
\bigr] \ud s
\\
&\hspace{15pt} + \varpi^0_{t_{i-1},t_{i}},
\end{split}
\end{equation*}  
where $(\varpi^0_{r,s})_{0 \leq r \leq s \leq T}$ is a collection of ${\mathcal F}^0_T$-measurable random variables satisfying 
$\sum_{i=1}^n \vert 
\varpi^0_{t_{i-1},t_{i}} \vert \rightarrow 0$ in ${\mathbb P}^0$-probability as 
$n$ tends to $\infty$ (and the step size of the mesh tends to $0$). 
The derivation of the above identity relies on the fact that the 
 path $(X^0_s)_{0 \leq s \leq T}$ 
 is continuous and the derivatives
 $(s,\xi,\mu,y) \mapsto  \partial_\mu \ell(s,\xi,\mu,y)$ 
 and
 $(s,\xi,\mu,y) \mapsto \nabla_y \partial_\mu \ell(s,\xi,\mu,y)$ 
 are continuous (and thus) bounded on (compact) sets of the form 
$[0,T] \times \{ \xi \in {\mathbb R}^d : \vert \xi \vert \leq a \}
\times {\mathcal P}({\mathbb T}^d) \times {\mathbb T}^d$, for any $a>0$
(and similarly for the derivative 
$\partial_t \ell$).
Moreover, by standard It\^o formula, we have in a similar manner:
\begin{equation*}
\begin{split}
&\ell \bigl(t_{i-1}, X_{t_i}^0 , \mu_{t_{i-1}} \bigr) 
-
\ell \bigl(t_{i-1}, X_{t_{i-1}}^0 , \mu_{t_{i-1}} \bigr) 
\\
&= \int_{t_{i-1}}^{t_i}
b_s^0 \cdot \nabla_{x_0} \ell(t_{i-1},X_s^0,\mu_{t_{i-1}}) \ud s 
+
\frac12 
\int_{t_{i-1}}^{t_i}
{\rm Tr} \Bigl( 
\varsigma_s^0 \bigl( \varsigma_s^0 \bigr)^\dagger \nabla_{x_0}^2 \ell(t_{i-1},X_s^0,\mu_{t_{i-1}})  \Bigr) \ud s
\\
&\hspace{15pt} +
\int_{t_{i-1}}^{t_i}
\bigl( \varsigma_s^0 \bigr)^{\dagger} \nabla_{x_0} \ell (t_{i-1},X_s^0,\mu_{t_{i-1}}) \cdot \ud B_s^0
\\
&= \int_{t_{i-1}}^{t_i}
b_s^0 \cdot \nabla_{x_0} \ell(s,X_s^0,\mu_s) \ud s 
+
\frac12 
\int_{t_{i-1}}^{t_i}
{\rm Tr} \Bigl( 
\varsigma_s^0 \bigl( \varsigma_s^0 \bigr)^\dagger \nabla_{x_0}^2 \ell(s,X_s^0,\mu_s)  \Bigr) \ud s
\\
&\hspace{15pt} + \int_{t_{i-1}}^{t_i}
\bigl( \varsigma_s^0 \bigr)^{\dagger} \nabla_{x_0} \ell (s,X_s^0,\mu_{s}) \cdot \ud B_s^0+   \varpi^0_{t_{i-1},t_i},
\end{split}
\end{equation*}
for a possible new choice of $(\varpi^0_{r,s})_{0 \le r < s \leq T}$. 
The above identity follows from the continuity of the path 
$(\mu_s)_{0 \le s \le T}$ (with respect to ${\mathbb W}_1$). 

Combining the last two displays, summing over $i$ and letting 
the step size of the mesh tend to $0$, we complete the proof. 
\end{proof} 

\subsection{Estimates for transport-diffusion equations}
\color{black} 

We here collect several results regarding the long-time behaviour of transport-diffusion equation on the torus. 
We start with the following first lemma:

\begin{lemma}
\label{le:appendix:1}
Let $T >0$, $\mathbb{r}\geq 1$ and $b : [0,\infty) \times {\mathbb T}^d \rightarrow {\mathbb R}^d$ be a measurable function such that 
$\sup_{0 \leq t \leq T} \| b(t,\cdot)\|_{\mathbb{r}-1} < \infty$.
Then, 
there exist two constants $C$ and $\gamma >0$, depending on the quantity 
$\sup_{0 \leq t \leq T} \| b(t,\cdot)\|_{\mathbb{r}-1}$
 but not on $T$, such that 
the solution to the transport-diffusion equation 
(set 
on $[0,T] \times {\mathbb T}^d$)
\begin{equation*} 
\begin{split}
&\partial_t \varphi_t + \tfrac12 \Delta \varphi_t + b(t,\cdot) \cdot \nabla \varphi_t = 0,
\quad t \in [0,T]; \quad \varphi_T=\phi,
\end{split}
\end{equation*}
for $\phi \in {\mathcal C}^{\mathbb{r}}({\mathbb T}^d)$, 
satisfies 
\begin{equation}
\label{eq:exp:decay:varphi:lem:6.2} 
\| \varphi_t - \bar \varphi_t \|_{\mathbb{r}} \leq C \exp \bigl( - \gamma (T-t) \bigr) \| \phi \|_{\mathbb{r}}, 
\end{equation}
where 
$\bar \varphi_t = \int_{{\mathbb T}^d} \varphi_t(x) \ud x$.
\end{lemma}

\begin{proof}
Notice that the result is standard for $t$ close to $T$ (say $T-t \leq 1$). 
It just provides a control of the ${\mathcal C}^{\mathbb r}$-norm of the solution 
in term of the ${\mathcal C}^{\mathbb r}$-norm of the terminal condition 
and the ${\mathcal C}^{{\mathbb r}-1}$-norm of the velocity field. 

In order to get the exponential decay when $t$ gets away from $T$, we recall from \cite[Lemma 7.4]{CLLP} that 
\begin{equation}
\label{eq:exp:decay:varphi:lem:6.2} 
\| \varphi_t - \bar \varphi_t \|_{L^\infty} \leq C \exp \bigl( - \gamma (T-t) \bigr) \| \phi \|_{L^\infty}, 
\quad t \in [0,T],
\end{equation}
with $C$ and $\gamma$ as in the statement. 
Take now $\delta \in (0,1)$. For $t \in [0,T-\delta]$, we write 
\begin{equation*}
 \varphi_t= P_{\delta} \varphi_{t+\delta}  + \int_t^{t+\delta} P_{r-t} \bigl[ b(r,\cdot) \cdot  \nabla\varphi_r \bigr] \ud r,
\end{equation*} 
where $(P_s)_{s \geq 0}$ stands here for the standard heat kernel on the torus 
${\mathbb T}^d$, namely $(P_s)_{s \geq 0}$  is the semi-group generated by the 
operator
$\tfrac12 \Delta$. 
And then, 
\begin{equation*} 
 \varphi_t - \bar{\varphi}_{t+\delta} = P_{\delta} \bigl( \varphi_{t+\delta} - \bar \varphi_{t+\delta} \bigr)  + \int_t^{t+\delta} P_{r-t} \bigl[ b(r,\cdot) \cdot  \nabla\varphi_r \bigr] \ud r.
\end{equation*}
In particular, for any 
integer $k \in \{0,\cdots,\lfloor \mathbb{r} \rfloor-1\}$ and any real $\eta \in (0,1)$,
\begin{equation*} 
\| \nabla^{k+1} \varphi_t   \|_{\eta}
=
\bigl\| \nabla^{k+1} \bigl( \varphi_t  - \bar{\varphi}_{t+\delta} \bigr)  \bigr\|_{\eta}
 \leq C_\delta \| \varphi_{t+\delta}- \bar{\varphi}_{t+\delta}  \|_k 
+ C \int_t^{t+\delta} (r-t)^{-(1+\eta)/2 } \| b(r,\cdot) \cdot  \nabla\varphi_r \|_k \ud r, 
\end{equation*}
where $C_\delta$ depends on $\delta$. 

Assuming that 
\eqref{eq:exp:decay:varphi:lem:6.2} holds true with respect to $\| \cdot \|_k$ instead of 
$\| \cdot \|_\mathbb{r}$, we can bound 
$\| b(r,\cdot) \cdot  \nabla\varphi_r
 \|_{k} = \| b(r,\cdot) \cdot  \nabla( \varphi_r
 -
\bar \varphi_r)
 \|_{k} $ by $C_k(\exp(-\gamma (T-r)) \| \phi \|_{\mathbb r} +\| \nabla^{k+1}  \varphi_r\|_{\eta})$. We obtain
 \begin{equation*}
\begin{split}
\| \nabla^{k+1} \varphi_t   \|_{\eta} &\leq C_{k,\delta}
\exp\bigl(-\gamma (T-t)\bigr) \| \phi \|_{\mathbb r} 
+
C_k
\int_t^{t+\delta}  
 (r-t)^{-(1+\eta)/2 }
\| \nabla^{k+1}  \varphi_r\|_{\eta}
\ud r,
\end{split} 
\end{equation*} 
and then, for $t+\delta \leq T$, 
\begin{equation*} 
 \exp \bigl(  \gamma(T-t) \bigr)  
\,  \| \nabla^{k+1} \varphi_t   \|_{\eta}
 \leq C_{k,\delta} \| \phi \|_{\mathbb r} + C_k \delta^{(1-\eta)/{2}}  
  \exp \bigl(  \gamma \delta  \bigr)  
\sup_{t \leq r \leq t + \delta} 
\Bigl[ \exp \bigl(  \gamma(T-r) \bigr)  
 \| \nabla^{k+1} \varphi_r \|_\eta \Bigr]. 
 \end{equation*} 
Choosing $\delta$ small enough and then taking the supremum over $t \leq T-\delta$, we get a bound for the left-hand side. 
We then get 
the result by iterating on the value of $k$. 
\end{proof}

\begin{lemma}
\label{le:appendix:2}
Let $T >0$, $\mathbb{r} \geq 1$ and $b : [0,\infty) \times {\mathbb T}^d \rightarrow {\mathbb R}^d$ be a measurable function such that 
$\sup_{0 \leq t \leq T} \| b(t,\cdot)\|_\mathbb{r} < \infty$. 
Then, 
there exist two constants $C$ and $\gamma >0$, depending on 
$\sup_{0 \leq t \leq T} \| b(t,\cdot)\|_\mathbb{r} < \infty$
 but not on $T$, such that 
the solution to the conservation equation 
(set on $[0,T] \times {\mathbb T}^d$) 
\begin{equation} 
\label{le:eq:6:2}
\begin{split}
&\partial_t q_t - \tfrac12 \Delta q_t + \textrm{\rm div}_x \bigl( b(t,\cdot) q_t \bigr) = 0,
\quad t \in [0,T]; \quad q_0=q,
\end{split}
\end{equation}
for a smooth initial condition $q : {\mathbb T}^d \rightarrow {\mathbb R}$ 
with $\int_{{\mathbb T}^d} q(x) \ud x =0$,  
satisfies 
\begin{equation*} 
\| q_t \|_{-\mathbb{r}} \leq C \exp \bigl( - \gamma t \bigr) \| q \|_{-\mathbb{r}}, \quad t \in [0,T]. 
\end{equation*}
When 
$\bar q:= \int_{{\mathbb T}^d} q(x) \ud x \not =0$, 
we deduce from the conservative structure that 
\begin{equation*} 
\| q_t - \bar q \|_{-\mathbb{r}} \leq C   \| q - \bar q\|_{-\mathbb{r}},
\quad 
\int_{{\mathbb T}^d} q_t(x) \ud x = \bar q, \quad t \in [0,T].  
\end{equation*}
\end{lemma} 

\begin{proof}
The proof is done by duality. For $\phi$ and $(\varphi_t)_{0 \le t \le T}$ as in the statement of Lemma 
 \ref{le:appendix:1}, we compute 
 (by decomposing $q$ in positive and negative parts, it suffices to derive the identity below 
 when $q$ is a probability measure, in which case the result follows on It\^o-Krylov formula
 for It\^o processes with a non-degenerate diffusion coefficient and  abounded drift)
 \begin{equation*} 
 \ud_t \bigl( \varphi_t, q_t \bigr) = 0, 
 \quad t \in [0,T], 
 \end{equation*} 
where $(\cdot, \cdot)$ is here understood as the duality bracket between ${\mathcal C}^\mathbb{r}({\mathbb T}^d)$ and 
${\mathcal C}^{-\mathbb{r}}({\mathbb T}^d)$,
so that 
\begin{equation*} 
( \phi,q_T) = (\varphi_0,q) \leq  C \exp \bigl( - \gamma T \bigr) \| q \|_{-\mathbb{r}} \| \phi \|_\mathbb{r}.
\end{equation*}
The result follows by maximizing over $\phi \in {\mathcal C}^\mathbb{r}({\mathbb T}^d)$ satisfying 
$\| \phi \|_\mathbb{r} \leq 1$. 
\end{proof}

 \section*{Acknowledgment}

François Delarue acknowledges the financial support of the European Research Council (ERC) under the European Union’s Horizon Europe research and innovation programme (AdG ELISA project, Grant Agreement No. 101054746). Views and opinions expressed are however those of the authors only and do not necessarily reflect those of the European Union or the European Research Council Executive Agency. Neither the European Union nor the granting authority can be held responsible for them.

Chenchen Mou acknowledges the financial support of the Hong Kong Research Grants Council (RGC) under Grants No. GRF 11311422 and GRF 11303223.

\bibliographystyle{myamsplain} 

\providecommand{\bysame}{\leavevmode\hbox to3em{\hrulefill}\thinspace}
\providecommand{\MR}{\relax\ifhmode\unskip\space\fi MR }
\providecommand{\MRhref}[2]{%
  \href{http://www.ams.org/mathscinet-getitem?mr=#1}{#2}
}
\providecommand{\href}[2]{#2}

\end{document}